\documentclass[11pt]{imsart}

\RequirePackage[OT1]{fontenc}
\usepackage{a4wide}
\usepackage{xr-hyper}
\usepackage{hyperref}
\usepackage{amsmath,amsthm,amssymb}
\usepackage[isolatin]{inputenc}
\usepackage{amsfonts,dsfont}
\usepackage{pstricks}
\usepackage{comment}
\usepackage{natbib}
\usepackage{colordvi}
\usepackage{graphicx}
 \usepackage{pgf, tikz}
 \usepackage{ulem}
  \usetikzlibrary{fit,patterns,decorations.pathreplacing}
\newtheorem{thm}{Theorem}[section]
\newtheorem{prp}[thm]{Proposition}
\newtheorem{lem}[thm]{Lemma}
\newtheorem{cor}[thm]{Corollary}

\newtheorem{remark}[thm]{Remark}

\newtheorem{definition}[thm]{Definition}

\newcommand{\I}{\mathbf{I}} 
\newcommand{\II}{\mathbf{I\hspace{-0.5mm}I}} 

\newcommand{\R}{\mathbb{R}} 

\renewcommand{\P}{\mathbb{P}}
\newcommand{\E}{\mathbb{E}} 

\renewcommand{\l}{\ell}

\newcommand{\FDR}{\mbox{FDR}}
\newcommand{\BH}{\mbox{BH}}

\newcommand{\mtc}{\mathcal}

\newcommand{\wt}[1]{{\widetilde{#1}}}
\newcommand{\wh}[1]{{\widehat{#1}}}
\newcommand{\ol}[1]{\overline{#1}}

\newcommand{\ind}{{\mathds{1}}}

\newcommand{\cB}{{\mtc{B}}}
\newcommand{\cH}{{\mtc{H}}}

\newcommand{\cN}{{\mtc{N}}}

\usepackage{graphics}
\newcommand{\FDP}{\mbox{FDP}}

\newcommand{\TDP}{\mbox{TDP}} 
\newcommand{\TDR}{\mbox{TDR}}

\renewcommand{\star}{*}

\def\cB{\mathcal{B}}

\def\cG{\mathcal{G}}
\def\cH{\mathcal{H}}

\def\cN{\mathcal{N}}

\def\cP{\mathcal{P}}

\def\c\Gamma(Y){\mathcal{\Gamma(Y)}}

\def\b\Gamma(Y){\mathbf{\Gamma(Y)}}

\def\beq{\begin{equation}}
\def\eeq{\end{equation}}
\def\beqn{\begin{eqnarray*}}
\def\eeqn{\end{eqnarray*}}

\begin{document}

\begin{frontmatter}

\title{False discovery rate control with unknown null distribution: is it possible to mimic the oracle?}

\runtitle{FDR control with unknown null distribution}

\begin{aug}
\author{Etienne Roquain and Nicolas Verzelen}
\runauthor{Roquain, E. and Verzelen, N.}
\end{aug}

\begin{abstract}
Classical multiple testing theory prescribes the null distribution, which is often a too stringent assumption for nowadays large scale experiments. This paper presents theoretical foundations to understand the limitations caused by ignoring the null distribution, and how it can be properly learned from the (same) data-set, when possible. We explore this issue in the case where the null distributions are Gaussian with an unknown rescaling parameters (mean and variance) and the alternative distribution is let arbitrary. While an oracle procedure in that case is  the Benjamini Hochberg procedure applied with the true (unknown) null distribution, we pursue the aim of building a procedure that asymptotically mimics the performance of the oracle (AMO in short). 
Our main result states that an AMO procedure exists if and only if the sparsity parameter $k$ (number of false nulls) is of order less  than $n/\log(n)$, where $n$ is the total number of tests. 
Further sparsity boundaries are derived for general location models where the shape of the null distribution is not necessarily Gaussian. Given our impossibility results, we also pursue a weaker objective, which is to find a confidence region for the oracle. To this end, we develop a distribution-dependent confidence region for the null distribution. As practical by-products, this provides a goodness of fit test for the null distribution, as well as a visual method assessing the reliability of empirical null multiple testing methods. Our results are illustrated with numerical experiments and a companion vignette \cite{RVvignette2020}. 
\end{abstract}



\end{frontmatter}

\section{Introduction}

\subsection{Background}

In large-scale data analysis, the practitioner routinely faces the problem of simultaneously testing a large number $n$ of null hypotheses. 
In the last decades, a wide spectrum of multiple testing procedures have been developed. Theoretically-founded control of the amount of false rejections are provided notably by controlling the false discovery rate (FDR), that is, the average proportion of errors among the rejections, as done by the famous Benjamini Hochberg procedure (BH), introduced in \cite{BH1995}. 
Among these procedures, various types of power enhancements have been proposed by taking into account the underlying structure of the data.
For instance, let us mention adaptation to the quantity of signal \cite{BKY2006,BR2009,Sar2008,LB2019}, to the signal strength \cite{RW2009,CS2009,HZZ2010,IH2017,Dur2019}, to the spatial structure \cite{PGVW2004,SC2009,RBWJ2019,DBNR2018}, or to data dependence structure \cite{LS2008,FKC2009,Fan2012,GHS2013,DR2015b,Fan2017},  among others.

\begin{figure}[h!]
\includegraphics[scale=0.45]{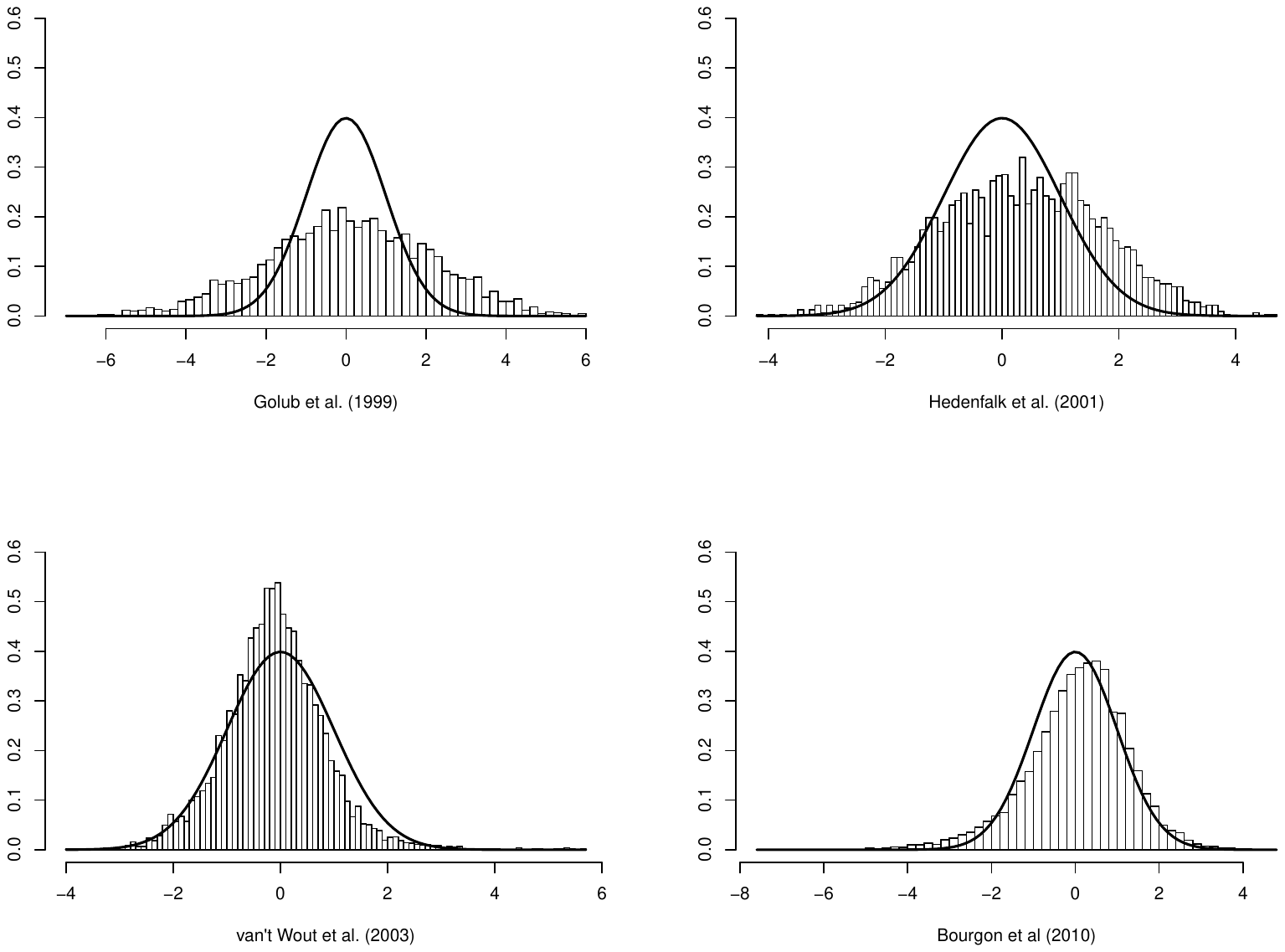}
\vspace{-0.5cm}
\caption{
{Histograms of the test statistics (rescaled to be all marginally standard Gaussian), for three datasets presented by Efron: \cite{Golub99}  (top-left); \cite{Hed2001} (top-right); \cite{van2003} (bottom-left); and \cite{BGH2010} (bottom-right). Pictures reproducible from the vignette \cite{RVvignette2020}.
}} \label{fig:nullwrong}
\end{figure}

Most of these theoretical studies --- and in general, of the FDR controlling procedures developed in the multiple testing literature --- rely on the fact that the null distribution of the test statistics is exactly known, either for finite $n$ or asymptotically. 
However, in common practice, this null distribution is often {\it mis-specified}: 
\begin{itemize}
\item {\it The null distribution can be wrong}. This phenomenon, pointed out in a series of pioneering papers \cite{Efron2004, Efron2007b,  Efron2008,Efron2009b} and studied further in \cite{Sch2010, AS2015, Ste2017,SS2018} is illustrated in Figure~\ref{fig:nullwrong} for four classical datasets.  
As one can see,  the theoretical null distribution $\mathcal{N}(0,1)$ does not describe faithfully the overall behavior of the measurements. One reason invoked by the aforementioned papers is the presence of correlations, that is, co-factors, that modify the shape of the null.
As a result, using this theoretical null distribution into a standard multiple testing procedure (e.g., BH) can lead to an important resurgence of false discoveries. Markedly, this effect can be even more severe than simply ignoring the multiplicity of the tests (see \cite{RVvignette2020}),  and thus the benefit of using a multiple testing correction can be lost.  
\item {\it The null distribution can be unknown}.
Data often come from raw measurements that have been ``cleaned" via many sophisticated normalization processes, and  
the practitioner has no prior belief in what the null distribution should be. Hence,  the null distribution is implicitly defined as the "background noise" of the measurements and searching signal in the data turns out to make some assumption on this background (typically Gaussian) and find outliers, defined as items that significantly deviate from the background. This occurs for instance in astrophysics \cite{SCS1999,Astro2001,SMB2017}, for which devoted procedures are developed, but without full theoretical justifications.
\end{itemize}
To address these issues,  Efron popularized the concept of {\it empirical null distribution}, that is, of a null distribution estimated from the data, in the works \cite{ETST2001,Efron2004, Efron2007b,  Efron2008,Efron2009b} notably through the {\it two-group mixture model} and the {\it local fdr} method. This paved the way for many extensions \cite{JC2007,SC2009,CS2009,CJ2010,PB2012,NM2014,heller2014,CSWW2019,RRV2019}, which make this type of technics widely used nowadays, mostly in genomics \cite{encode2007,zablocki2014,JY2016,amar2017extracting} but also in other applied fields, as neuro-imaging, see, e.g., \cite{lee2016}. However, this approach suffers from a lack of  theoretical justification, even in the original setting described by Efron, where the null distribution is assumed to be Gaussian.

\subsection{Aim}
In this work, we propose to fill this gap: we consider the issue of controlling the FDR when the null distribution is Gaussian $\mathcal{N}(\theta,\sigma^2)$, with {\it unspecified} scaling parameters $\theta\in\R$, $\sigma>0$ (general location models will be also dealt with). In addition, according to the original framework considered in \cite{BH1995}, the alternative distributions are let arbitrary, which  provides a setting both general and simple. 
We address the following issue: 
\begin{center}
{\it {When the null distribution is unknown}, is it possible to build a procedure that both control the FDR at the nominal level and has a power asymptotically mimicking the oracle?}
\end{center}
In addition, as classically considered in multiple testing theory (see, e.g., \cite{Dic2014}), we aim here for a {\it strong} FDR control, valid for any data distribution  (in a given sparsity range). 
For short, a procedure enjoying the two properties delineated above is called an {\it AMO} procedure in the sequel. To achieve this aim,  we should choose an appropriate notion of ``oracle", which corresponds to the default procedure that one would perform if the scaling parameters $\theta,\sigma^2$ were known. Due to the popularity of the BH procedure, we define the oracle procedure as the BH procedure using the null distribution  $\mathcal{N}(\theta,\sigma^2)$ with the true parameters $\theta,\sigma^2$. This choice is also suitable because the BH procedure controls the FDR under arbitrary alternatives \citep{BH1995} while it has optimal power against classical alternatives \citep{AC2017,RRJW2020}.

\subsection{Our contribution}
We consider  a setting where the statistician observes  {\it independent} real random variables $Y_i$, $1\leq i \leq n$. 
Among these $n$ random variables, {$n-k$ follow the unknown null distribution $\cN(\theta,\sigma^2)$ and the remaining $k$ follow arbitrary and unknown distributions. The upper bound $k$ on the number of false nulls is referred henceforth as the sparsity parameter. 
The latter plays an important role in our results. A reason is that having many observations under the null (that is, $k$  small) facilitates {\it de facto} the problem of estimating the null distribution.
As argued in Section~\ref{sec:scalingnull}, this setting is both an extension of the two-group model of~\cite{Efron2008} and of Huber's contamination model.

In this manuscript, we first establish that, when $k$ is much larger than $n/\log(n)$, no AMO exists. Hence, any multiple testing procedure
 either violates the FDR control or is less powerful than the BH oracle procedure. In particular, any local fdr-type procedure that controls the FDR has a sub-oracle power. Conversely, any  local fdr-type procedure that mimics the oracle power violates the FDR control.
Hence, the usual protocole of applying blindly such approaches  
 is questionable when the data contain more than a constant portion of signal (say, when the number of alternatives is of order $10\%$). On the other hand, when  $k$  is much smaller than $n/\log(n)$, a simple procedure that first computes corrected $p$-values by plugging robust estimators of $\theta$ and $\sigma^2$ and then applies  a classical Benjamini-Hochberg (BH) procedure \cite{BH1995}, is established to be AMO.  This type of procedures is
referred below as plug-in BH procedures. Figure~\ref{fig:interpretbis} displays the behavior of these procedures for different plugged values $u,s^2$ for $\theta,\sigma^2$ (true, mis-specified or estimated). This simple example illustrates that using a wrong scaling of the null distribution can lead to poor  performances, with either an uncontrolled increase of false discoveries (top-right panel), or an uncontrolled decrease of true discoveries (bottom-left panel). By contrast, fitting the null distribution with robust estimator of of $\theta$ and $\sigma^2$ seems to nearly mimic the oracle BH procedure.

\begin{figure}[h!]
\begin{tabular}{cc}
\hspace{-8mm}True null $\mathcal{N}(1,4)$ &\hspace{-8mm}Mis-specified null $\mathcal{N}(0,1)$\vspace{-1cm} \\
\hspace{-8mm}\includegraphics[scale=0.55]{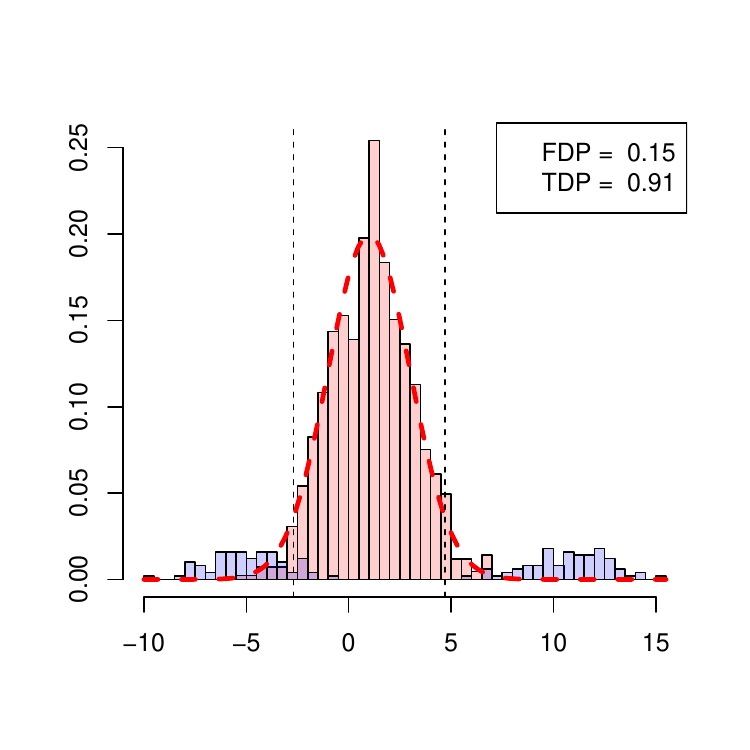}&\hspace{-8mm}\includegraphics[scale=0.55]{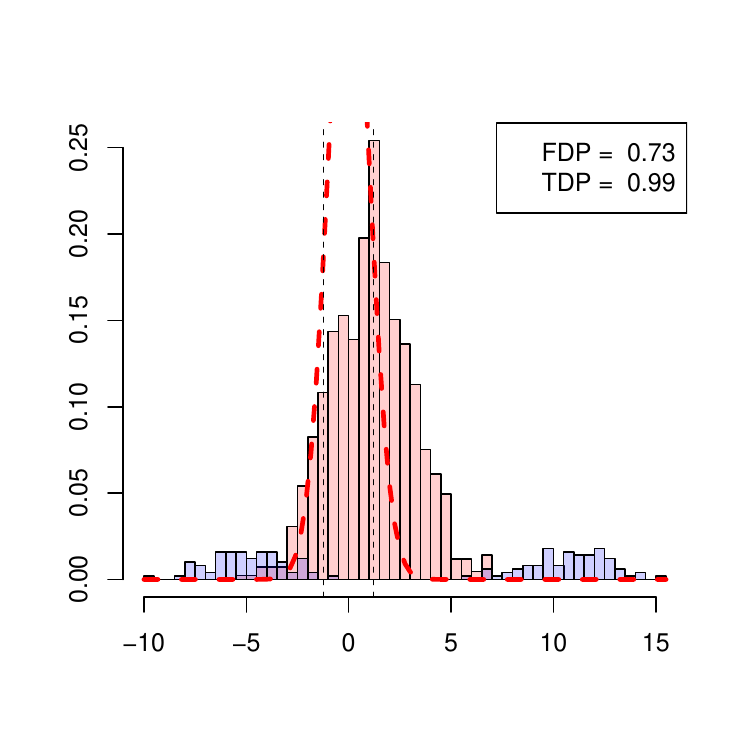}\vspace{-5mm} \\
\hspace{-8mm}Mis-specified null $\mathcal{N}(1,16)$ &\hspace{-8mm} Estimated null $
  \mathcal{N}(1.12,5.85)$ \vspace{-1cm}\\
\hspace{-8mm}\includegraphics[scale=0.55]{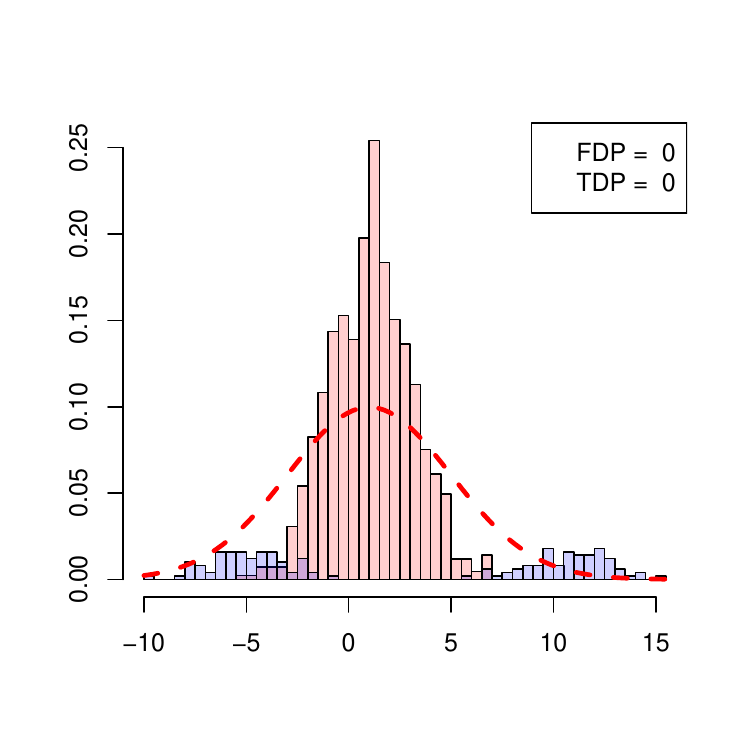}&\hspace{-8mm}\includegraphics[scale=0.55]{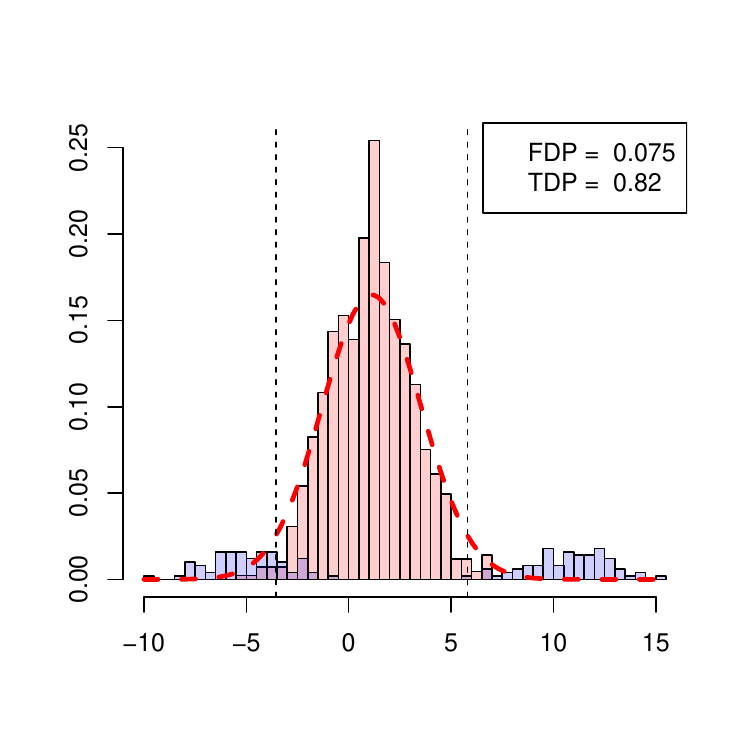}\vspace{-5mm} 
 \end{tabular}
\vspace{-0.5cm}
\caption{
Illustration of the plug-in BH procedure with different plugged null distributions. The data have been generated as independent $Y_i\sim \mathcal{N}(\theta+\mu_i,\sigma^2)$, for $\mu_i=0$, $1\leq i \leq n_0$, $\mu_i=5$, $n_0+1\leq i \leq n_0+n_1/2$, and $\mu_i=-3$, $n_0+n_1/2+1\leq i \leq n$ for $n=1000$, $n_0=850$, $n_1=150$, $\theta=1$, $\sigma^2=4$, $\alpha=0.2$.
Each panel displays the same overlap of the two following histograms of the data: colored in pink, the histogram of the $Y_i$, $1\leq i \leq n_0$, generated under the null; colored in blue, the (rescaled) histogram of the $Y_i$, $n_0+1\leq i\leq n$, generated under the alternative. 
The plug-in BH procedure is applied at level $\alpha=0.2$ and its rejection threshold is displayed by the vertical dashed lines:  the rejected null hypotheses correspond to the $Y_i$'s above the most-right vertical dashed line and below the most-left vertical dashed line.  The FDP is the ratio of the false rejection number to the total rejection number, see \eqref{equ-FDP} below. The TDP is the ratio of the true rejection number to the total number of false nulls ($n_1$), see \eqref{equ-TDP} below. The plug-in BH procedure uses rescaled $p$-values $p_i(u,s)=2\overline{\Phi}(|Y_i-u|/s)$, $1\leq i\leq n$, where $\overline{\Phi}$ is the tail distribution of the standard normal distribution, see \eqref{equ-pvalues} below, using different values of $u,s$. 
Top-left: $u=1$,  $s^2=4$; top-right: $u=0$,  $s^2=1$; 
bottom-left: $u=1$,  $s^2=16$;
bottom-right: $u=\wt{\theta}\approx 1.12$, $s^2=\wt{\sigma}^2\approx 5.85$, which are values derived from standard robust estimators, see  \eqref{equ:estimators} below.
} \label{fig:interpretbis}
\end{figure}
 
 This analysis is the central result of the paper. It is then extended in several directions:  first, we show a stronger impossibility result: FDR control and power mimicking are two aims that are incompatible across the delineated boundary; typically, any procedure achieving an oracle power below the boundary entails FDR violation above the boundary. 
 Second, we pinpoint the boundary where AMO is possible when only the mean parameter $\theta$ is unknown or, more generally,  when the density of the null distribution is arbitrary but only known up to a location parameter. 

Finally, given our impossibility results, one can legitimately ask whether obtaining AMO procedures is not too demanding. Hence, we also investigate the weaker aim of obtaining confidence regions for the oracle procedure. 
 This is achieved by developing a confidence region for the null distribution. Then, candidate region sets for the oracle procedure are rejection sets of plug-in procedures using nulls of that region. 
The stability of these rejection sets then provides to the statistician a visual method to assess whether they can be confident in the plugged BH procedure. 
Interestingly, this confidence region can take various forms depending on the considered data set, as shown in  Section~\ref{sec:adaptive} and in the attached vignette \cite{RVvignette2020}. This illustrates that this region is distribution-dependent and goes beyond a minimax analysis that is based on worst-case distributions.
 In addition, this approach can be easily adapted to  any plug-in method--not necessarily of the BH type-- and to  any candidate distribution family for the null, which increases the application scope of our approach.

\subsection{Work related to our AMO boundaries}

Many work are related to the theory developed here.
First, as already mentioned, a wide literature has grown around the concept of "empirical null distribution", elaborating upon the work of Efron. While his proposal originally relies on the Gaussian null class, more sophisticated classes have been proposed later \cite{Sch2010, AS2015,SS2018} to better modeling null coming from a  multivariate correlated Gaussian vector. This results in a parametric family with much more parameters to fit. Related to this,  
estimating the parameters of the null has been considered in a more challenging multivariate factor model, see, e.g., \cite{LS2008,FKC2009,Fan2012,Fan2017}.  While the authors provide 
 error bounds for the inferred factor models, none of these work establish FDR controls of the corresponding corrected BH procedure.

In fact, it turns out that only few work have provided theoretical guarantees for using an empirical null distribution into a multiple testing procedure, even for the simple Gaussian case. 
 \cite{JC2007} and \cite{CJ2010} proposed a method to estimate the null in a particular context, but without evaluating the impact of such an operation when plugged into a multiple testing procedure. Such an attempt has been made by \cite{Gho2012}, who showed that the FDR control is maintained under the assumption that incorporating the empirical null distribution is an operation that can make the BH procedure only more conservative. Nevertheless, this assumption is admittedly difficult to check. Other studies have been developed in the one-sided context, for which contaminations (that is non-null measurements) are assumed to come only from the right-side (say) of the global measurement distribution. In that case, the left-tail of the  distribution can be used to learn the null. Such an idea has been exploited in \cite{CDRV2018} to estimate the scaling parameters $\theta$ and $\sigma^2$ within the null $\mathcal{N}(\theta,\sigma^2)$ from the left-quantiles of the observed data. 
Doing so, they show that the plug-in BH procedure has performances close (asymptotically in $n$)  to those of the BH procedure using the true unknown scaling. 
In addition, relaxing the Gaussian-null assumption, an FDR controlling procedure has been introduced in  \cite{BC2015,AC2017}, by only assuming the symmetry of the null. In that case, the null is implicitly learned by estimating the number of false discoveries occurring at the right-side of the null from its left-side.
However, the one-sided contamination model is not the most common practical situation where signal can arise at both sides of the null distribution. 
The case for which the alternative distributions are let arbitrary and potentially two-sided is more difficult than the one-sided case and will be considered throughout the paper. 

 Let us mention few additional related studies with mis-specified null: in \cite{BLS2010}, the null is unknown and estimated from an independent sample, so the setting is completely different. In \cite{BKZ2014}, the authors  study the effect of non-normality over the BH procedure using $p$-values  calibrated with the Gaussian distribution. This is substantially different from our problem, where the null is assumed Gaussian with an uncertainty in the parameters. Next,  \cite{pollard2004choice} also discuss the choice of a null distribution, but the aim is to build a null that ensures a valid FWER-type error rate control, which is a goal markedly different from here. 

Next, maybe on a more conceptual side, our work can be seen as a {\it frequentist minimax robust} study of empirical null distributions: first, we do assume that {there exists a true null distribution} and we try to estimate it to produce our inference. Second, we let the alternative be arbitrary, which means that the AMO properties should hold whatever the alternatives.  For the FDR control, this is classically referred to as  {\it strong} control of the FDR, see, e.g., \cite{Dic2014}.  For the power mimicking, this is new to our knowledge and requires to use a suitable notion power, compatible with a worst-case analysis.   
Third, the proofs of our impossibility results borrow some ideas from the literature on robust estimation and classical Huber contamination model \citep{huber1964robust, huber2011robust}.

Finally, as in many statistical studies in high-dimension, see, e.g., \cite{DJ2004,bickel2009simultaneous,bogdan2011asymptotic,javanmard2019false}, the sparsity plays an important role in our results.

\subsection{Notation and presentation of the paper}

\noindent 
\smallskip

\noindent{\bf Notation}. For two sequences  $u_n$ and $v_n$, $u_n\gg v_n$ means $v_n=o(u_n)$. Given a real number $x$, $\lfloor x\rfloor$ and $\lceil x\rceil$ respectively denote the lower and upper integer parts of $x$. 
Given a finite set $A$, its cardinal is denoted $|A|$.
Given $x$, $y$, $x\wedge y$ (resp. $x\vee y$) stands for the minimum (resp. maximum) of $x$ and $y$.  For $Y\sim P$, the corresponding probability is denoted $\P_{Y\sim P}$ or simply $\P$ when there is no confusion. The density of the standard normal distribution is denoted $\phi$ whereas $\ol{\Phi}$ stands for its tail distribution function, that is, $\ol{\Phi}(z)=\P(Z\geq z)$, $z\in\R$, $Z\sim\mathcal{N}(0,1)$. Finally, given a vector $v\in \R^n$,  we denote by $v_{(i)}$ the $i$-th order statistic of $v$, that is, the $i$-th smallest entry of $v$. \\

\noindent 
{\bf Organization of the paper}. 
The setting and the main results are described in Section~\ref{sec:scalingnull}.
While they are formulated in an asymptotical manner for simplicity, more accurate non-asymptotical counterparts  are provided in Section~\ref{sec:non-asymp}: an impossibility result is given in Section~\ref{sec:lb} (with a corollary given in Section~\ref{sec:complementlb}) and a matching upper-bound is provided in Section~\ref{sec:ub}. 
Section~\ref{sec:sigmaknown} is devoted to study the situation where the variance of the null is known, while Section~\ref{sec:generallocation} provides extensions to a general location model. The null confidence region  is presented in Section~\ref{sec:adaptive}, which is  illustrated on synthetic and real data sets.
A discussion is given in Section~\ref{sec:discussion}.
Numerical experiments, proofs, lemmas, and auxiliary results are deferred to the appendix. An application of our approach on real data sets is developed carefully in a devoted vignette, see \cite{RVvignette2020}.

 \section{Setting and presentation of the main results}\label{sec:scalingnull}

\subsection{Framework for testing an unknown null} 
To formalize our setting, 
 we resort to a variation of Huber's model \citep{huber1964robust}. Let us observe  {\it independent} real random variables $Y_i$, $1\leq i \leq n$. The distribution of the vector $Y=(Y_i)_{1\leq i \leq n}$ in $\R^n$ is denoted by $P=\otimes_{i=1}^n P_i$. We assume that most of the $P_i$'s follows the same (null) distribution while the others are  ``contaminated" and can be arbitrary. 
Also, following the setting used by Efron \citep{Efron2004}, we shall assume in this manuscript  
that this null distribution is of the form $\mathcal{N}(\theta,\sigma^2)$ for some unknown scaling $(\theta,\sigma)\in \R\times (0,\infty)$ (except in Sections~\ref{sec:generallocation}~and~\ref{sec:adaptive}  where different or more general nulls are considered).
Formally, this leads to assume that $P=\otimes_{i=1}^n P_i$ belongs to the collection $\mathcal{P}$ of all 
distributions satisfying
\begin{equation}\label{equ-assumpfond}
\mbox{there exists } (\theta,\sigma)\in \R \times (0,\infty) \mbox{ such that } |\{i\in\{1,\dots,n\} \::\: P_i = \mathcal{N}(\theta,\sigma^2)\}|>n/2.
\end{equation}
In other words, \eqref{equ-assumpfond}
 ensures that there exists a scaling $(\theta,\sigma)$ such that more than half of the $P_i$'s are $ \mathcal{N}(\theta,\sigma^2)$.
While \eqref{equ-assumpfond} may be surprising at first sight, it is a minimal condition to make the problem identifiable with respect to the unknown null distribution.
 For $P\in \mathcal{P}$, we denote 
  by $(\theta(P),\sigma(P))$ this unique couple.
This allows us to  formulate the multiple testing problem:
\begin{center}
$H_{0,i}: ``P_i =  \mathcal{N}(\theta(P),\sigma^2(P))"$ against $H_{1,i}: ``P_i \neq  \mathcal{N}(\theta(P),\sigma^2(P))"$,
\end{center}
for all $1\leq i \leq n$.
We underline that $H_{0,i}$ is not a point mass null hypotheses, that is, ``$P_i=P^0$", for some known distribution $P^0$, nor a composite null of the type ``$P_i$ is a Gaussian distribution", but  a point mass null hypothesis with value depending on all the marginals $(P_j,1\leq j\leq n)$.

{
Let us introduce some notation. We denote by $\cH_0(P) = \{ 1\leq i\leq n \::\: P \mbox{ satisfies } H_{0,i}\}$ the set of true null hypotheses, by $n_0(P)=|\cH_0(P)|$ its cardinal and by $\cH_1(P)$ its complement in $\{1,\dots,n\}$. We also let  $n_1(P)=|\cH_1(P)|=n-n_0(P)$, so that $n_1(P)<n/2$ by \eqref{equ-assumpfond}. 
As an illustration, if $P=\otimes_{i=1}^n P_i$ is given by 
$$
(P_i,1\leq i\leq n) = \big(\:P_1\:,\: P_2 \:,\: \mathcal{N}(1,4)\:,\: \mathcal{N}(1,4) \:,\: P_5\:,\:  \mathcal{N}(1,4)\:,\: \mathcal{N}(1,4)\:\big)
$$
for $n=7$ and some distributions $P_1,P_2,P_5$ on $\R$ that are all different from $ \mathcal{N}(1,4)$, we have $\theta(P)=1$, $\sigma^2(P)=4$, $\cH_0(P) =\{3,4,6,7\}$ and $n_1(P)=3$.
}

{
Finally, we will sometimes consider an asymptotic situation where $n$ tends to infinity. In that case, the quantities $\mathcal{P}$, $P$, $Y$ (and those related) are all depending on $n$, but we remove such dependences in the notation for the sake of clarity.}

\subsection{Comparison with classical Huber's model and two-group models} 

Originally, Huber's contamination model was introduced in robust statistics as a mixture model with density $h=(1-\pi)\phi_{\theta,\sigma} + \pi f$ where $\pi\in [0,1/2)$, $\phi_{\theta,\sigma}$ stands for the density of a $\mathcal{N}(\theta, \sigma^2)$ and $f$ is an arbitrary density. When sampling according to this model, one observes a proportion close to $(1-\pi)$ of data sampled from the normal distribution and a proportion close to $\pi$ of contaminated data. In our framework, the contaminated data account for the false null hypotheses and non-contaminated data for the true null hypotheses. Note that this mixture model interprets as a specific random instance of our model introduced in the previous section. Indeed, one can sample $n$ observations according to $h$
by first generating $n$ Bernoulli random variables $Z_i\in \{0,1\}$ with parameter $\pi$. Next, if $Z_i=0$, then $Y_i$ is sampled according to $\phi_{\theta,\sigma}$, whereas if, $Z_i=1$,  $Y_i$ is sampled according to $f$. As a consequence, conditionally to the $Z_i$'s, the distribution $P$ of $Y$ satisfies $n_1(P)= \sum_{i=1}^n Z_i$ and $(\theta(P),\sigma(P))=(\theta,\sigma)$, at least when $\sum_{i=1}^n Z_i < n/2$. Besides, all false null distributions $P_i$ are identically distributed according to $f$. This random instance representation of our model is central for proving impossibility results in Section~\ref{sec:lb}. 
In the multiple testing terminology, Huber's contamination model can be interpreted as a two-group model where the null distribution is a normal distribution with unknown parameters and the alternative distribution is let completely arbitrary. Let us mention that many versions of null/alternative families have been considered in the literature, as mixture of Gaussian distributions \cite{SC2007,JC2007,CJ2010} or well chosen parametric families \cite{SS2018}. The one we choose is in accordance with the original proposition of \cite{Efron2008} while being suitable to ensure the strong FDR control. 
Furthermore, note that, in general, letting the alternative distribution arbitrary leads to identifiability problem for the null in the corresponding two-group model, see Figure~\ref{lbsigma_intro} below. However, one advantage of using a model with a fixed mixture is that identifiability is always ensured under Assumption~\ref{equ-assumpfond}.

\subsection{Criteria}

A multiple testing procedure is defined as a measurable function $R$ taking as input the data $Y$ and returning a subset $R(Y)\subset \{1,\dots,n\}$ corresponding to the set of rejected null hypotheses among $(H_{0,i},1\leq i\leq n)$. 
The amount of false positives of $R$ (type I errors) is classically measured by the false discovery proportion   of $R$:
\begin{equation}\label{equ-FDP}
\FDP(P, R(Y)) = \frac{|R(Y) \cap \cH_0(P) |}{|R(Y)|\vee 1},
\end{equation}
see \cite{BH1995}.
The expectation $\FDR(P,R)=\E_{Y\sim P} [\FDP(P, R(Y))]$ is the 
false discovery rate of the procedure $R$.
The amount of true positives of $R$ is measured by
\begin{equation}\label{equ-TDP}
\TDP(P, R(Y)) = \frac{|R(Y) \cap \cH_1(P)|}{n_1(P)\vee 1}\ ,
\end{equation}
 and  corresponds to the proportion of (correctly) rejected nulls among the set of false null hypotheses.
 It has been often used as a power metric for multiple testing procedures, see, e.g. \cite{BH1995,RW2009, AC2017,RRJW2020}.

\subsection{Plug-in BH procedures}

In our study, an important class of procedure are the BH procedures with rescaled $p$-values, that we call the plug-in BH procedures. This corresponds to first estimating the null distribution $(\theta(P),\sigma(P))$ and then plugging it into BH. 

Since Benjamini-Hochberg (BH) procedure is defined through the $p$-value family, we first define, for $u\in \mathbb{R}$ and $s>0$, the rescaled $p$-values 
\begin{equation}\label{equ-pvalues}
p_i(u,s) =2\ol{\Phi}\left(\frac{|Y_i-u|}{s}\right)  , \:\: u\in\R, \:\: s>0, \:\:1\leq i \leq n ,
\end{equation}
which corresponds to the situation where $\theta(P)$, $\sigma(P)$ have been estimated by $u$, $s$, respectively.
By convention, the value $s=+\infty$ is allowed here, which gives a rescaled $p$-value always equal to $1$.
The oracle $p$-values are then given by
 \begin{equation}\label{equ-pvaluesperfect}
p^{\star}_i = p_i(\theta(P),\sigma(P))  ,\:\:1\leq i \leq n\ .
\end{equation}

\begin{definition}\label{def:rescaledBH}
Let $\alpha\in (0,1)$, $u\in\R$, $s>0$ and $P\in\mathcal{P}$.
The plug-in BH procedure of level $\alpha$ with scaling $u$ and $s$ is given by
\begin{align}
\BH_\alpha(Y;u,s)&=\{1\leq i \leq n \::\: p_i(u,s)\leq T_\alpha(Y;u,s) \};  \label{BHdef}\\
&=\{1\leq i \leq n \::\: p_i(u,s)\leq T_\alpha(Y;u,s) \vee (\alpha/n)\};  \nonumber\\
T_\alpha(Y;u,s) &= \max\left\{ t \in[0,1]\::\: \sum_{i=1}^n \mathds{1}\{p_i(u,s) \leq t\} \geq  n t/\alpha\right\}.
\label{Tdef}
\end{align}
In particular, the oracle BH procedure (of level $\alpha$) is defined as the plug-in BH procedure (of level $\alpha$)  with scaling $\theta(P)$ and $\sigma(P)$, that is, is defined by $\BH^\star_\alpha(Y)=\BH_\alpha(Y;\theta(P),\sigma(P))$.
\end{definition}

When not ambiguous, we will sometimes drop $Y$ in the notation $\BH_\alpha(Y;u,s)$, $T_\alpha(Y;u,s)$, $\BH^\star_\alpha(Y)$ for short.
The oracle procedure $\BH^\star_\alpha$ corresponds to the situation where the true scaling $(\theta(P),\sigma(P))$ is directly plugged into the BH procedure and is therefore  the oracle procedure in our study.  
In our framework,  the $p$-values $p_i^\star$ are all independent, with the property $p_i^\star\sim U(0,1)$ whenever $i\in\cH_0(P)$. Hence, it is well known  \citep{BH1995,BY2001} that its FDR satisfies the following:
\begin{align}
\forall P \in \mathcal{P} ,\:\:\: \FDR(P,\BH^\star_\alpha) = \alpha n_0(P)/n \label{FDRcontrolBH}\ .
\end{align}
To mimic $\BH^\star_\alpha$, natural candidates  are the plug-in BH procedures $\BH_\alpha(\wh{\theta},\wh{\sigma})$, for some suitable estimators $\wh{\theta}$, $\wh{\sigma}$ of $\theta(P)$, $\sigma(P)$ (by convention, the value $\wh{\sigma}=\infty$ is allowed here). In the sequel, $(\wh{\theta}$, $\wh{\sigma})$ is called a rescaling. 

\subsection{AMO procedures}

To evaluate how a procedure is mimicking  $\BH^\star_\alpha$ on some sparsity range, let us define the following notation: for any procedure $R(Y)\subset\{1,\dots,n\}$,  any sparsity parameter $k\in [1,n/2]$ and any level $\alpha\in(0,1)$, we let
\begin{align}
\I(R,k)
&=\sup_{\substack{P \in \mtc{P}\\ n_1(P)\leq k}} \{\FDR(P,R)  \};
\label{defI}\\
\II(R,k,\alpha)&=\sup_{\substack{P \in \mtc{P}\\ n_1(P)\leq k}} \left\{\P_{Y\sim P}\left(  \TDP(P,R)< \TDP(P,\BH^\star_{\alpha})\right)\right\}.\label{defII}
\end{align}
Note that $\I(\BH^\star_\alpha,k)= \alpha$ for any $k$ by \eqref{FDRcontrolBH}.
In particular, the control $\I(\BH^\star_\alpha,n)\leq \alpha$ is uniform on $P\in \mtc{P}$ meaning that any least favorable configuration does not deteriorate the FDR. This is strong FDR control, on the range of distributions with at most $k$ false nulls. 
The criterion $\II(R,k,\alpha)$ is a type II risk defined relatively to $\BH^\star_\alpha$: it is small when the TDP of R  is at least as large as the one of $\BH^\star_\alpha$, with a large probability. In particular, the map $\alpha\mapsto \II(R,k,\alpha)$ is nondecreasing. 
Then, a procedure is said to mimic the oracle if 
it maintains the strong FDR control while having a small relative type II risk.

\begin{definition}\label{defmimic}
Let $R=(R_\alpha)_{\alpha\in (0,1)}$ be a sequence of multiple testing procedure, both depending on the nominal level $\alpha$ and of the number  $n$ of tests. 
For a given sparsity sequence $k_n\in[1,n/2)$, the procedure sequence $R$ 
 is  asymptotically mimicking the oracle BH procedure, AMO in short, whenever the two following properties hold: there exists a positive sequence 
$\eta_n\to 0$
 such that
 \begin{align}
\limsup_n \sup_{\alpha\in(1/n,1/2)}\{\I(R_\alpha,k_n)- \alpha\} &\leq 0;\label{IRkn}\\
 \lim_n \sup_{\alpha\in(1/n,1/2)}\{ \II(R_\alpha,k_n,\alpha(1-\eta_n))\} &=0\label{IIRkn}. 
 \end{align}
 Furthermore, if $\wh{\theta}$ and $\wh{\sigma}$ are two (sequence of) estimators of $\theta(P)$ and $\sigma(P)$, respectively, the rescaling $(\wh{\theta},\wh{\sigma})$  is said to be AMO if the sequence of plug-in BH procedure $(\BH_\alpha(\wh{\theta},\wh{\sigma}))_{\alpha\in (0,1)}$ is AMO. 
\end{definition} 

{In this definition, the performances of the oracle BH procedure are mimicked both in terms of FDR and TDP.}
{Note that the power statement is made slightly weaker than one could expect at first sight, with a slight decrease of the level in $\BH^\star_{\alpha(1-\eta_n)}$. 
Since $\eta_n$ converges to $0$, this modification is very light. In addition, if one wants a comparison with the  oracle procedure $\BH^\star_{\alpha}$ (without modification of the level), the convergence \eqref{IIRkn} can be equivalently replaced by  $ \lim_n \sup_{\alpha\in(1/n,1/2)}\{ \II(R_{\alpha(1+\eta_n)},k_n,\alpha))\}=0.$ This would not change our results.
}
Also, we underline that, while the statements \eqref{IRkn} and \eqref{IIRkn} are formulated in an asymptotic manner for compactness, all our results will be non-asymptotic.

\begin{remark}
{The oracle BH procedure is the reference procedure here. Since there is yet no general theory proving that it is optimal in a universal way, one can legitimately ask whether this choice is reasonable. Other proposals have been made, see e.g.   \cite{rosset2020optimal}, { but the procedures there are much more complex and rely on distributional assumptions under the alternative}. Here, we choose the BH procedure because: 1) it is widely used and thus is a meaningful benchmark 2) it is simple and thus allows for a full theoretical analysis and 3) it has been shown to be optimal in some specific regime, see \cite{AC2017,RRJW2020}.} 
\end{remark}

\begin{remark}
Instead of stochastically comparing the true discovery proportions in~\eqref{defII}, an alternative could have been to compare their expectations. 
{The expectation of the TDP, called the true discovery rate (TDR), is the standard notion of power in the literature, see for instance~\cite{RW2009, AC2017,RRJW2020}, where specific classes of alternative distributions are considered. Here, the TDR is not a suitable measure of power, because the alternative distribution is let  completely free in \eqref{defII}. As a result, in some cases, the TDR is maximized by trivial procedures that typically reject no null hypothesis with probability $1-\alpha$ and reject all null hypotheses with probability $\alpha$. As such procedures are obviously undesirable, we focus on the stronger asymptotic stochastic domination TDP property required in~\eqref{IIRkn}. }
\end{remark}

\subsection{Robust estimation of $(\theta(P),\sigma(P))$}

Since our framework allows arbitrary 
 alternative distributions, we consider simple robust estimators for $(\theta(P),\sigma(P))$ defined by 
\begin{equation}\label{equ:estimators}
\wt{\theta} = Y_{(\lceil n/2\rceil)};\:\:
\wt{\sigma} =  U_{(\lceil n/2\rceil)}/\ol{\Phi}^{-1}(1/4),
\end{equation}
where $U_i=|Y_i-Y_{(\lceil n/2\rceil)}|$ and $\ol{\Phi}^{-1}(1/4)\approx 0.674$. 
{While  $\wt{\theta}$ is the sample median, $\wt{\sigma}$ corresponds to a suitable rescaling of  $U_{(\lceil n/2\rceil)}$, the median absolute deviation (MAD) of the sample.  }
 Under the null, the variables $|Y_i-\theta|/\sigma $ are i.i.d. and distributed as the absolute value of a standard Gaussian variable. Hence, taking the median of the $|Y_i-\theta|$ should be a robust estimator of $\sigma$ times the median of the absolute value of a standard Gaussian variable, that is, of $\sigma\:\ol{\Phi}^{-1}(1/4)$. Rescaling suitably this quantity and replacing $\theta$ by $\wt{\theta} $ leads to the definition of $\wt{\sigma}$.
The two estimators defined by \eqref{equ:estimators} are   minimax optimal; see e.g.~\cite{chen2018robust} for a result  in a slightly different mixture model. We will use here specific properties of these estimators, to be found in Section~\ref{sec:estimationrate}.

\subsection{Presentation of the results}\label{sec:presentationresults}

\subsubsection{Main result}

We now state the main result of the paper.
\begin{thm}\label{mainth}
In the setting of Section~\ref{sec:scalingnull} and according to Definition~\ref{defmimic}, the following holds:
\begin{itemize}
\item[(i)] for a sparsity $k_n\gg n/\log(n)$,
there exists no (sequence of) AMO procedure that is AMO;
\item[(ii)] for a sparsity $k_n\ll n/\log(n)$, the sequence of plug-in BH procedure $(\BH_\alpha(\wt{\theta},\wt{\sigma}))_{\alpha\in (0,1)}$ is AMO, for the scaling $(\wt{\theta},\wt{\sigma})$ given by standard robust estimators \eqref{equ:estimators}. 
\end{itemize}
\end{thm}

Part (i) of Theorem~\ref{mainth} (lower bound) { means that, when the proportion of true alternative hypotheses is much largen than $1/\log(n)$, it is not possible to perform as well as an oracle that knows the null distribution in advance.}
Obtaining negative results on FDR control has received recently some attention in multiple testing literature \cite{AC2017,RRJW2020,CR2018} in various contexts, and in restriction to the class of thresholding-based procedures. 
Here, our impossibility holds for {\it any} multiple testing procedure. 
The proof of our lower bound relies on a Le Cam's  two-point reduction scheme. Namely, it is derived by identifying two mixture distributions on $\R^n$ that are indistinguishable while corresponding to  distant null distributions (see Figure~\ref{lbsigma_intro}) and by studying the impact of such 
fuzzy configuration on the FDR and TDP metrics. 
While this argument is classical in the estimation or (single) testing literature (see, e.g., \citealp{MR2724359} and \citealp{DJ2006}), it is to our knowledge new in the multiple testing context.

Part (ii) of Theorem~\ref{mainth} (upper bound) is proved in Section~\ref{sec:ub}. For this, we extend the ideas used in \cite{CDRV2018} to accommodate the new two-sided geometry of the test statistics. In particular, correcting the $Y_i$'s by $\wh{\theta}$ changes the order of the $p$-values, which was not the case in the one-sided situation. Our proof relies on the symmetry of the Gaussian distribution and on special properties of the BH procedure rejection set when removing one element of the $p$-value family, see, e.g., \cite{FZ2006}.
Also note that  the scaling $(\wt{\theta},\wt{\sigma})$ does not use the knowledge of $k_n$, which means that these estimators are adaptive with respect to the sparsity $k_n$ on the range $k_n\ll n/\log(n)$.

\subsubsection{Extending the scope of the main result}\label{complresults}

We provide three complementary results. First, in the testing literature, type I error rate controls are generally favored over type II error rate controls. In our framework, we can always design a plug-in BH procedure that controls the FDR 
by simply setting $\wh{\sigma}=\infty$, which is equivalent to taking $R(Y)=\emptyset$ (no rejection). In view of this remark, we can re-interpret the statement of Theorem~\ref{mainth} as follows:
\begin{itemize}
\item[(i)] in the dense regime ($k_n\gg n/\log (n)$), it is possible to achieve \eqref{IRkn} but not with \eqref{IIRkn};
\item[(ii)] in the sparse regime ($k_n\ll n/\log (n)$), it is possible to achieve both \eqref{IRkn} and \eqref{IIRkn}.
\end{itemize}
A natural question is then: can we achieve the best of the two worlds? Is that possible to find a rescaling satisfying \eqref{IRkn} in the dense regime and both \eqref{IRkn} and \eqref{IIRkn} in the sparse regime?
We establish in Section~\ref{sec:complementlb} that such a procedure does not exist, see Corollary~\ref{cor:adaptation}. As a consequence, any procedure controlling the FDR in the dense regime is not {AMO} in the sparse regime. Conversely, 
any {AMO} procedure  in the sparse regime is not able to control the FDR in the dense regime.
This is the case in particular for the plug-in procedure $\BH_\alpha(\wt{\theta},\wt{\sigma})$ considered in Theorem~\ref{mainth} (ii). More formally, combining Corollary~\ref{cor:adaptation} ($\alpha=c_3/2$) and Theorem~\ref{THMBORNESUP}  below establishes the following result.

\begin{cor}\label{corbadnews}
There exist numerical constants $\alpha_0\in(0,1/2)$ and $c>0$ such that for any sequence $u_n \to \infty$, 
$$\liminf_n \{\I(\BH_{\alpha_0}(\wt{\theta},\wt{\sigma}),u_n n/\log (n))- \alpha_0\}>c.$$
\end{cor}

Second, in Section~\ref{sec:sigmaknown}, we show an analogue of Theorem~\ref{mainth} when $\sigma=\sigma(P)$ is supposed to be known. Hence, the only unknown null parameter is $\theta$ and the class of rescaling is restricted to those of the form $(\wh{\theta}, \sigma)$, where $\wh{\theta}$ is an estimator of $\theta$. 
We establish that the sparsity boundary is slightly modified  in this case: impossibility is shown for $k_n  \gg n/\log^{1/2}(n)$, while $(\wt{\theta},\sigma)$  is {AMO} for $k_n  \ll n/\log^{1/2}(n)$ (Theorem~\ref{mainthsigmaknown}).
While the upper-bound part is similar to the upper-bound part of Theorem~\ref{mainth} above, the lower bound arguments have to be adapted to the case where only the location parameter is unknown. More precisely, we establish two types of lower-bounds. We first develop a lower bound valid for any multiple testing procedure  (Theorem~\ref{thm:lower_sigmaknown}), which follows the same philosophy as the lower-bound developed in Theorem~\ref{mainth} (via Theorem~\ref{THM:LOWER_GENERAL}). Next, we provide a refined lower bound specifically tailored to plug-in BH type procedures. Contrary to the previous lower bounds, it does not state type I error/type II error trade-offs but it establishes that uniform control of the FDR is alone already out of reach.
Namely, this result shows that, on the sparsity range $k_n  \gg n/\log^{1/2}(n)$, any plug-in procedure exhibits a FDP close to $1/2$ and makes around $n^{3/4}$ false discoveries, this on an event of probability close to $1/2$ (see Theorem~\ref{THM:THETACHAPFORFDR}). 
{Intuitively, this comes from the fact that $\wh{\sigma}=\sigma$ is fixed to the true value and thus cannot compensate the estimation error of $\wh{\theta}$, which irremediably leads to many false discoveries in that regime.}

Third, we extend our results to the case where the null distribution has a known symmetric density $g$ with an unknown location parameter, see Section~\ref{sec:generallocation}. Therein, we derive lower bounds in two different regimes, when $k_n/n$ tends to zero (Theorem~\ref{thm:lower_location}) and when $k_n/n$ is of order constant (Theorem~\ref{thm:lower_location2}). Also, we provide a general upper bound matching the lower bounds under assumptions on $g$ (Theorem~\ref{THMBORNESUPgeneral}).
As expected, the sparsity boundary depends on $g$. 
For instance, for $\zeta$-Subbotin null $g(x)=L^{-1}_\zeta \:e^{-|x|^\zeta/\zeta}$, $\zeta>1$, the boundary is proved to be $k_n\asymp n/ (\log(n))^{1-1/\zeta}$ (Corollary~\ref{cor:subottin}), which recovers the Gaussian case for $\zeta=2$. {For the Laplace distribution $g(x)=e^{-|x|}/2$, {AMO} scaling  is possible as long as $k_n\ll n$ (Corollary~\ref{cor:laplace}). Finally, we further explore the behavior of any procedure for the Laplace distribution  on the boundary when $k_n$ is of the same order as $n$ (Proposition~\ref{prp:lower_laplace}).}

\subsubsection{Confidence region for the null and applications}\label{region}
Our previous analysis shows that, when the sparsity is not strong enough, we cannot hope to build a procedure that mimics the properties of the oracle BH procedure. This holds in the minimax sense, that is, this impossibility is shown to be met under a least favorable configuration (see Figure~\ref{lbsigma_intro} below). However, if the underlying distribution $P$ is reasonably far from this distribution, it is not necessarily impossible to mimic the oracle. {Hence, for some specific data sets that are not sparse, one can possibly reliably estimate the null distribution and plug a BH procedure. This raises the issue of deriving data-driven and distribution-dependent measures of the reliability of plug-in null estimation methods. This is the topic of} Section~\ref{sec:adaptive}. 
Our main result there is a general, non-asymptotic, confidence region for the null distribution (Theorem~\ref{thmconfidence}). The latter holds without any assumption on the sparsity and on the null distribution.  It only requires an upper-bound $k$ on the number of false nulls.  
This induces a goodness of fit test for any given null distribution (Corollary~\ref{cor:testnull}) or even any family of null distributions (Corollary~\ref{cor:testfamilynull}). As shown in the vignette, for several data sets, the theoretical null $\mathcal{N}(0,1)$ is rejected while the family of Gaussian null is accepted. This reinforces the interest in using Gaussian empirical nulls, as Efron suggested in the first place. 
Another application of the confidence region is  a confidence set for the rejection set (or number) of the oracle BH procedure (Corollary~\ref{cor:region}). This result is certainly weaker than the aforementioned AMO property, but is valid for all distributions $P$, even those being above the boundary. It can be use to make practical recommandations:
if the rejection set (or number) is fairly unchanged over the confidence region, then the user knows approximately the rejection set (or number) of the oracle. By contrast, if the empty rejection set lies in that region, the user should probably make no rejection. We suggest to visualize this phenomenon via colored/annotated confidence region (Figure~\ref{fig:heatmap}).
Finally, we underline that this region can be applied with any type of null distribution, not necessarily Gaussian, because our confidence region is nonparametric.

\subsubsection{Numerical experiments}

In Section~\ref{sec:numexp}, we provide numerical experiments that illustrate our results. The simulations corroborate the theoretical findings and can be summarized as follows: 
\begin{itemize}
\item the plug-in BH procedure used with robust estimators is mimicking the oracle for $k$ small enough; 
\item it is improved by local fdr-type methods for standard alternatives. However, the latter are less robust to extreme alternatives; 
\item all procedures fail to mimic the oracle when $k$ is large;
\item these results are qualitatively similar under weak dependence between the measurements.
 \end{itemize}
 In addition, simulations are made under equi-correlation of the $Y_i$, $1\leq i\leq n$, which is an elementary factor model see \cite{Fan2012}. Interestingly, the methods here mimic the {\it conditional} true null, that is, the null distribution {\it conditionally} on the factor. Hence, they are able both to remove the dependence and to reduce the variance of the noise. This corroborates previous results in the literature, see, e.g., \cite{Efron2007,Fan2012}. In addition, this phenomenon also holds with complex real data dependencies, as we illustrate in the vignette
 \cite{RVvignette2020}. Markedly, we also show there that this convenient property is not met when estimating the null via a classical permutation-based approach. 

\section{Non asymptotical bounds}\label{sec:non-asymp}

\subsection{Lower bound}\label{sec:lb}

To prove part (i) of Theorem~\ref{mainth}, we establish a more general, non asymptotic, impossibility result.

\begin{thm}\label{THM:LOWER_GENERAL}
There exist numerical positive constants  $c_1$--$c_5$  such that the following holds for  all $n\geq c_1$ and any $\alpha\in (0,1)$. Consider any two positive numbers $ k_1\leq k_2$ satisfying
\beq\label{eq:condition_lower_k_0_general}
c_2\frac{n\log\left(2/\alpha\right)}{\log(n)}\left[1+ \log\left(\frac{k_2}{k_1}\right)\right]\leq k_2  < n/2 \ .
\eeq
For any multiple testing procedure $R$ 
such that 
$$
\FDR(P,R)\leq c_3 \ , \mbox{ for any $P \in \mtc{P}$ with $n_1(P)\leq k_2$ },
$$
there exists some $P\in \cP$ with  $n_1(P)\leq k_1$ such that we have 
\begin{align}
\P_{Y\sim P}(|R(Y)\cap \cH_1(P)|= 0)&\geq 2/5\,;\nonumber\\
\P_{Y\sim P}\left[|\BH^{\star}_{\alpha/2}\cap \cH_1(P)|\geq c_4\alpha^{-1}\Big\{\frac{n}{\log n}\Big\}^{1/2}\right]&\geq 1 - e^{-c_5 \alpha^{-1}\{n/\log(n)\}^{1/2}}\geq 4/5\ .
\label{THM:LOWER_GENERAL_TDP}
\end{align}
In particular, 
we have that $\I(R,k_2)\leq c_3$ implies $\II(R,k_1,\alpha/2)\geq 1/5$.
\end{thm}

 Theorem \ref{THM:LOWER_GENERAL} states that, for any procedure $R$, either the FDR is not controlled at the nominal level $\alpha\leq c_3$ for all $P$ with $n_1(P)\leq k_2$ or that there exists a distribution $P$ with $n_1(P)\leq k_1$ such that $R$ does not make any correct rejection with positive probability while the oracle procedure $\BH^{\star}_{\alpha/2}$ make at least (of the order of) $\{n/\log(n)\}^{1/2}$ correct rejections with probability close to one.

Now, let us show that Theorem~\ref{THM:LOWER_GENERAL} implies part (i) of Theorem~\ref{mainth}. 
Consider any sequence $k_n$ with $n/2> k_n \gg n/(\log n)$, any sequence $\eta_n\to 0$, an arbitrary sequence of procedure $(R_\alpha)_{\alpha\in (0,1)}$,  and choose $\alpha =(c_3\wedge 1)/2$. Clearly, for $n$ large enough, the sparsity parameters $k_1=k_2=k_n$ satisfy the requirements of Theorem~\ref{THM:LOWER_GENERAL} and thus, for $n$ large,  either 
$\I(R_\alpha,k_n)-\alpha> (c_3\wedge 1)/2$ or $\II(R_\alpha,k_n,\alpha/2)\geq 1/5$.
This entails that \eqref{IRkn} and \eqref{IIRkn} cannot hold simultaneously.

\medskip 

{
Let us provide some high-level ideas of the proof of Theorem~\ref{THM:LOWER_GENERAL};  we refer to Section~\ref{sec:proofTHM:LOWER_GENERAL} for the details. As explained in Section~\ref{sec:scalingnull}, two-group models can be viewed as random instances of our setting and it suffices to prove that no AMO procedure in this setting. Let us assume that $Y_i$, $1\leq i \leq n$, are i.i.d. and have a common distribution given by the mixture density
\begin{equation}\label{equ-twogroup_LF}
h=(1-\pi)\phi + \pi f_1,
\end{equation}
where $\phi$ is the density of the standard Gaussian distribution, $f_1$ is the density of the alternative and 
 $\pi\in (0,1)$ is a prescribed proportion of signal. This density is depicted in the left panel of Figure~\ref{lbsigma_intro}, for some specific choice of $\pi$ and $f_1$. Here, the alternative density looks nicely separated from the null $\phi$, which indicates that the oracle procedure should  typically make some rejections.  By contrast, consider the situation depicted in the right panel of Figure~\ref{lbsigma_intro}, where the null density is given by $\phi_{\sigma_2}(\cdot)=\sigma_2^{-1} \phi(\cdot/\sigma_2)$ and the alternative is given by a density $f_2$, concentrated near $0$. In that situation, the alternative density are not well distinguishable from the null so that the oracle procedure, which ``knows" what is the null distribution,  makes no rejection with high probability to ensure a correct FDR control. 

\begin{figure}[h!]
\begin{tabular}{ccc}
$h=(1-\pi)\phi + \pi f_1$  & $h=(1-\pi) \phi_{\sigma_2}+\pi f_2$ \vspace{-0.5cm}\\
 \includegraphics[scale=0.45]{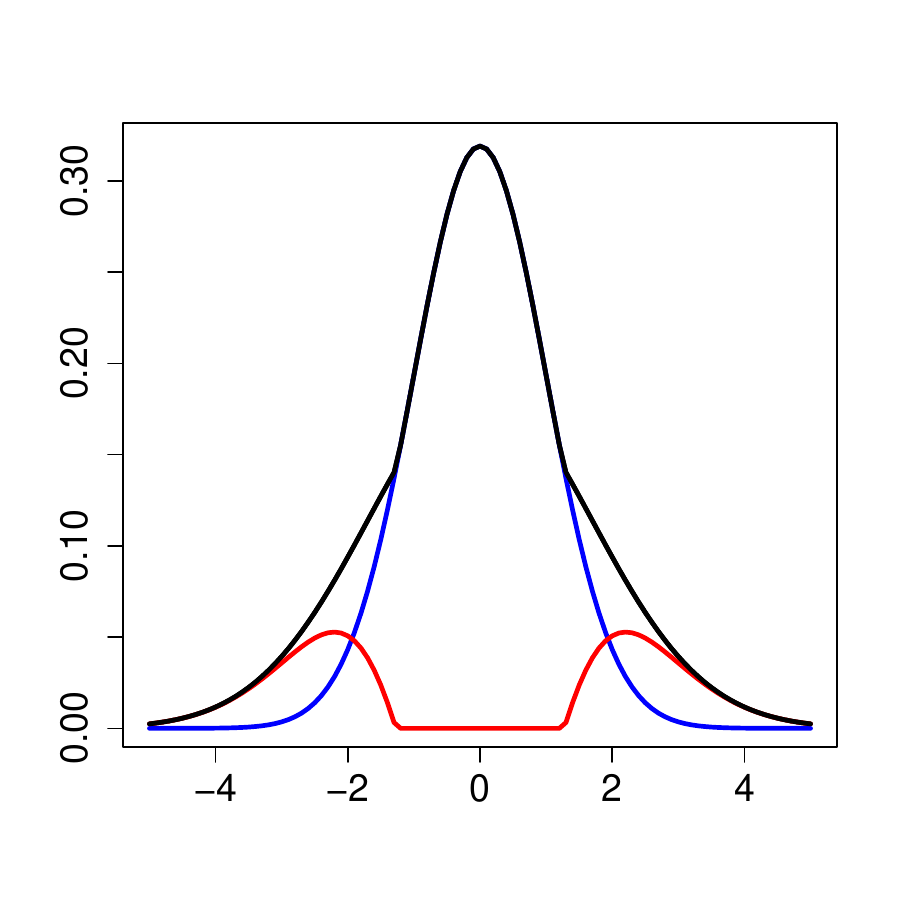}&\hspace{-0.5cm}\includegraphics[scale=0.45]{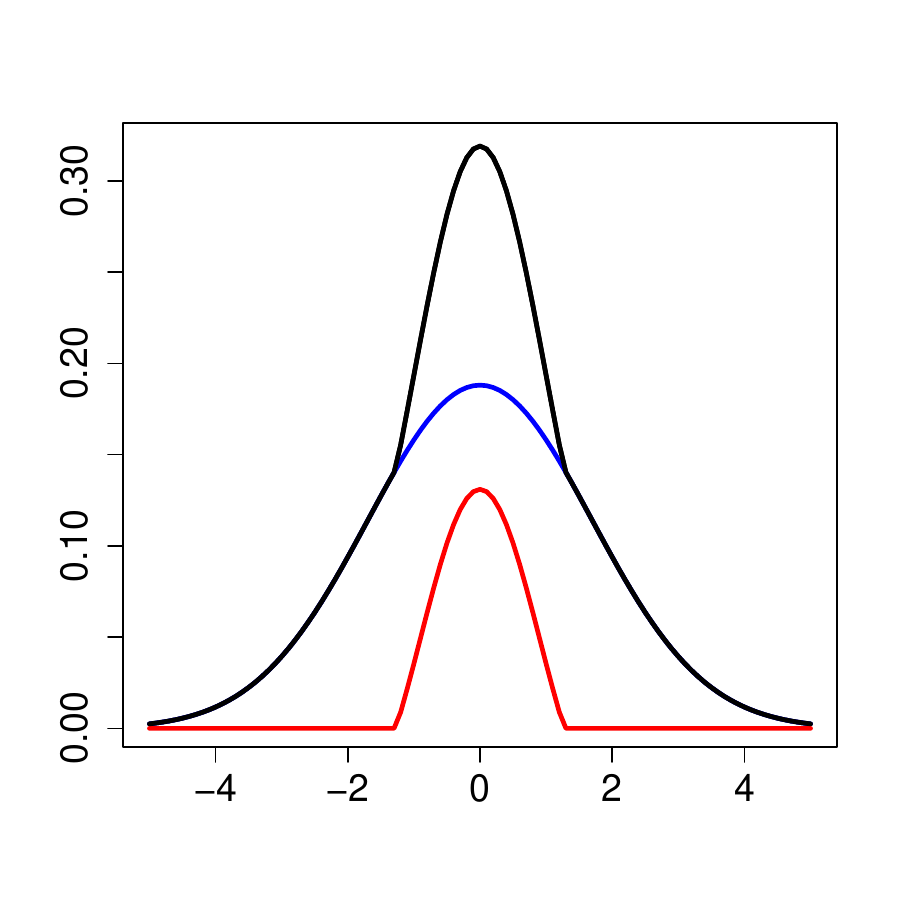}
 \end{tabular}
\vspace{-0.5cm}
\caption{Left: the density $h$ given by \eqref{equ-twogroup_LF} (black), interpreted as a mixture between the null $\mathcal{N}(0,1)$ ($(1-\pi)\phi$ in blue) and the alternative $f_1$ ($\pi f_1$ in red). Right: the same $h$ interpreted as a mixture between the null $\mathcal{N}(0,\sigma_2^2)$ ($(1-\pi)\phi_{\sigma_2} $ in blue) and the alternative $f_2$ ($\pi f_2$ in red). $\pi=1/5$, $\sigma_2^2\approx 2.88$. 
} \label{lbsigma_intro}
\end{figure}
The point is that $f_1$ and $f_2$ are chosen so that the two mixture densities in the left and right panel coincide. Hence, when the data are generated by this mixture, {\it any data-driven procedure} cannot decipher whether the data arise from $(\phi,f_1)$ or $(\phi_{\sigma_2},f_2)$. As a result, a data-driven procedure is not able to mimic the behavior of the oracle as soon as the distribution of the rejection number of the oracle highly differs in the two situations. Quantifying precisely the latter provides a condition on the sparsity parameter $\pi=\pi_n$, namely $\pi_n \gg 1/\log(n)$, under which no AMO procedure exists. 

}

\subsection{Upper bound}\label{sec:ub}

In this section, we prove Part (ii) of Theorem~\ref{mainth}. 
The following result states FDR and power {oracle inequalities}.

\begin{thm}\label{THMBORNESUP}
In the setting of Section~\ref{sec:scalingnull}, there exist universal constants $c_1$, $c_2>0$ such that the following holds for all $n\geq c_1$ and  $\alpha\in (0,0.5)$. Consider any number $k \leq 0.1 n$ such that $\eta=c_2 \log(n/\alpha) \left((k/n)\vee n^{-1/6}\right)\leq 0.05$. Then, we have
\begin{align}
& \I(\BH_\alpha(\wt{\theta},\wt{\sigma}),k)  \leq   \alpha (1+\eta)+ e^{-n^{1/2}}\ ;\label{eq:upper_FDR_thm} \\
& \II(\BH_{\alpha}(\wt{\theta},\wt{\sigma}),k,\alpha(1-\eta)) \leq e^{-n^{1/2}} \ .\label{eq:lower_TDP_thm}
\end{align}
\end{thm}

Let us check that Theorem~\ref{THMBORNESUP} implies (ii) of Theorem~\ref{mainth}. If $\log(n)k_n/n$ tends to zero and $\alpha\in(1/n,1/2)$,  we have $\eta\leq 2 c_2 \log(n) \left(\frac{k_n}{n}\vee n^{-1/6}\right)$ which is smaller than $0.05$ for $n$ large enough, and by \eqref{eq:upper_FDR_thm} above and \eqref{FDRcontrolBH},
\begin{align*}
\sup_{\alpha\in(1/n,1/2)}\{\I(\BH_\alpha(\wt{\theta},\wt{\sigma}),k)-\alpha\} 
 \leq \eta+ e^{-n^{1/2}},
\end{align*}
which converges to $0$ as $n$ grows to infinity. 
This gives \eqref{IRkn} for $(\wh{\theta},\wh{\sigma})=(\wt{\theta},\wt{\sigma})$. 
Similarly,  \begin{align*}
\sup_{\alpha\in(1/n,1/2)}\{  \II(\BH_{\alpha}(\wt{\theta},\wt{\sigma}),k,\alpha(1-\eta))\} 
\leq e^{- n^{1/2}}\rightarrow 0,
 \end{align*}
which gives \eqref{IIRkn} for $(\wh{\theta},\wh{\sigma})=(\wt{\theta},\wt{\sigma})$ and $\eta_n=\eta$. \\

The proof of Theorem~\ref{THMBORNESUP} is given in Section~\ref{sec:proofupperbound}. The general argument can be summarized as follows. Observing  that the estimators $\wt{\theta},\wt{\sigma}$ converge at the rate $n_1(P)/n+ n^{-1/2}$  (Lemma~\ref{lem:esti}), we mainly have to quantify the impact of these errors on the FDR/TDP metrics.
To show \eqref{eq:upper_FDR_thm}, we establish that the FDR metric is at worst perturbed by the estimation rate  multiplied by $\log(n/\alpha)$. Here, $\alpha/n$ corresponds to the smallest $p$-value threshold of the BH procedure. This can be shown by studying how the $p$-value process is affected by misspecifying the scaling parameters (Lemma~\ref{lem:forFNR}). A difficulty stems from the fact that the FDR metric is not monotonic in the rejection set, so that specific properties of BH procedure and of the estimators $\wt{\theta},\wt{\sigma}$ are required (Lemmas~\ref{cor:pvaluei} and \ref{lem:indep}).
The second result \eqref{eq:lower_TDP_thm} is proved similarly, the main difference being that we need a  slight decrease in the level $\alpha$ (Lemma~\ref{lem:forFNR}) of the oracle procedure $\BH^{\star}_{\alpha}$ to compare the BH thresholds $T_\alpha({\theta}(P),{\sigma}(P))$ and $T_{\alpha}(\wh{\theta},\wh{\sigma})$. 
This results in a level $\alpha(1-\eta)$ instead of $\alpha$ in \eqref{eq:lower_TDP_thm}.

\subsection{Relation between FDR and power across the boundary}
\label{sec:complementlb}
Theorem~\ref{mainth} establishes that it is impossible to perform as well as the oracle BH procedure when $k_n\gg n/\log(n)$. As simultaneously controlling the FDR and power mimicking is out of reach, one may require that, at least, the FDR is controlled. 
Theorem~\ref{THM:LOWER_GENERAL}, applied with $k_1<k_2$, shows that controlling the FDR in the dense case has consequences on the relative type II risk in the sparse case. More precisely, for some $\epsilon>0$, Condition~\eqref{eq:condition_lower_k_0_general} and $k_1\leq k_2$ is satisfied for 
 $k_2=  \log(1/\epsilon) n/\log(n)$ and $k_1=  \epsilon^{(c_2\log(2/\alpha))^{-1}}\log(1/\epsilon) e \frac{n}{ \log n}$ (for $\epsilon$ in a specific range), which entails the following result.

\begin{cor}\label{cor:adaptation}
Consider the same numerical constants $c_1$--$c_3$  as in Theorem~\ref{THM:LOWER_GENERAL} above. Take  any  $\alpha\in (0,c_3)$, any $n\geq c_1$ 
and fix any 
$
 \epsilon\in (n^{-1/2} ;(\alpha/2)^{c_2}).
$
Then for any procedure $R$ with $\I(R,k_2)\leq c_3$ for a sparsity $k_2=  \log(1/\epsilon) n/\log(n)$, we have $\II(R,k_1,\alpha/2)\geq 1/5$ for a sparsity $k_1=\epsilon^{(c_2\log(2/\alpha))^{-1}}\log(1/\epsilon) e n/ \log n$.
In particular, if $n^{-1/4}<(\alpha/2)^{c_2}$,  we have  for any procedure $R$,
\begin{itemize}
\item if $\I(R,n/4)\leq \alpha$, then $\II(R,n^{1-\delta} e/4,\alpha/2)\geq 1/5$;
\item if  $\II(R,n^{1-\delta} e/4,\alpha/2)< 1/5$, then $\I(R,n/4)>c_3$,
\end{itemize}
where we let $\delta=1/(4c_2\log(2/\alpha))>0$.
\end{cor}

In plain words, the above corollary entails that a procedure $R$ controlling the FDR up to a sparsity $\log(1/\epsilon) \frac{n}{\log n}$  (that is of order larger than or equal to the boundary $n/\log(n)$ of Theorem~\ref{mainth}), suffers from a {power loss} in a sparse setting where $n_1(P)$ is of order $\epsilon^{(c_2\log(2/\alpha))^{-1}}\log(1/\epsilon)  \frac{n}{ \log n} $, for which {AMO} is theoretically possible (as stated in Theorem~\ref{THMBORNESUP}). As $\epsilon$ decreases, $R$ is assumed to control the FDR in denser settings and becomes over-conservative in sparser settings. The case $\epsilon=n^{-1/4}$, requiring that the FDR is controlled at the nominal level up to a sparsity $n/4$ enforces a power loss in some "easy" settings where $n_1(P)/n$ is polynomially small. 
In other words, if we require FDR control in the dense regime, we will pay a high power price in the "easy" regime where {AMO} is achievable. Conversely, any AMO procedure in sparse regime violates the FDR control in the dense regime. Corollary~\ref{corbadnews} formalizes this fact with the plug-in BH procedure of Theorem~\ref{THMBORNESUP}.

\section{Known variance}\label{sec:sigmaknown}

This section is dedicated to the simpler case where $\sigma(P)$ is known to the statistician, so that only the mean $\theta(P)$ has to be estimated. In this setting, it turns out that the boundary  for {AMO} is $n/\log^{1/2}(n)$ instead of $n/\log(n)$. 

\begin{thm}\label{mainthsigmaknown}
In the setting of Section~\ref{sec:scalingnull} and according to Definition~\ref{defmimic}, the following holds:
\begin{itemize}
\item[(i)] for a sparsity $k_n$ with $k_n \log^{1/2}(n)/n\gg 1$,  there exists no (sequence of) procedure that is AMO;  
\item[(ii)] for a sparsity $k_n$ with $k_n \log^{1/2}(n)/n = o(1)$, 
 the scaling $(\wt{\theta},\sigma(P))$ given by \eqref{equ:estimators}  is {AMO}.
\end{itemize}
\end{thm}

The upper bound (ii) is proved similarly to the upper bound of Theorem~\ref{mainth}, but with the weaker condition $k_n \log^{1/2}(n)/n =o(1)$. For this, one readily checks that Theorem~\ref{THMBORNESUP} extends to the case where $\wh{\sigma}=\sigma(P)$ up to replacing $\eta$ by $\eta=c_2 \log^{1/2}(n/\alpha) \left(\frac{n_1(P)+1}{n}+n^{-1/6}\right)$ (and possibly modifying the constants $c_1$ and $c_2$). The proofs are exactly the same, except that Lemma~\ref{lem:forFNR} has to be replaced by  Lemma~\ref{lem:forFNRsigmaknown}. See Section~\ref{sec:proof_mainth_sigma_known_2} for details. 
Let us additionally provide here a heuristic to explain the value of the boundary. Roughly, the oracle BH procedure is equivalent to the plug-in BH procedure if the corrected observations $Y_i-\wh{\theta}$ can be compared to the Gaussian quantiles $\overline{\Phi}^{-1}(\alpha k/(2n))$ in the same way as the $Y_i-\theta$ do. Hence, the plug-in operation will mimic the oracle if 
$$
| \hat{\theta}-\theta| \ll \min_k\left\{\overline{\Phi}^{-1}(\alpha k/(2n))-\overline{\Phi}^{-1}(\alpha (k-1)/(2n))\right\} \asymp \frac{\alpha/n}{\phi(\overline{\Phi}^{-1}(\alpha /n))},
$$
which leads to $k/n \ll 1/{\log^{1/2} n }$, by using the standard properties on the Gaussian tail distribution (Section~\ref{sec:aux}) and the estimation rate of $\wt{\theta}$ (Section~\ref{sec:estimationrate}).

 {In the remainder of this section, we focus on the impossibility results. We first establish in Theorem~\ref{thm:lower_sigmaknown} the counterpart of Theorem~\ref{THM:LOWER_GENERAL}. This lower bound is valid non asymptotically and for arbitrary testing procedures. 
 Next, we provide a sharper lower bound for plug-in procedures. }

\subsection{Lower bound for a general procedure}\label{sec:lbFDR}

\begin{thm}\label{thm:lower_sigmaknown}
There exist numerical positive constants  $c_1$--$c_5$  such that the following holds for  all $n\geq c_1$ and any $\alpha\in (0,1)$. Consider two positive numbers $ k_1\leq k_2$ satisfying
\beq\label{eq:condition_lower_k_0_bus}
c_2 \frac{n \log\left(2/\alpha\right)}{{\log^{1/2}(n)}}\left\{1+ \log\left(\frac{k_2}{k_1}\right)\right\}^{1/2}\leq k_2  < n/2 \ , 
\eeq
For any multiple testing procedure $R$ 
satisfying  
\[
\FDR(P,R)\leq c_3 \ ,  \mbox{ for any $P \in \mtc{P}$ with $n_1(P)\leq k_2$ },
\]
there exists some $P\in \cP$ with  $n_1(P)\leq k_1$ such that we have  
\begin{align}
\P_{Y\sim P}(|R(Y)\cap \cH_1(P)|= 0)&\geq 2/5\,;\nonumber\\
\P_{Y\sim P}\left[|\BH^{\star}_{\alpha/2}\cap \cH_1(P)|\geq c_4\alpha^{-1}\Big\{\frac{n}{\log n}\Big\}^{1/2}\right]&\geq 1 - e^{-c_5 \alpha^{-1}\{n/\log(n)\}^{1/2}}\geq 4/5\ . 
\end{align}
In particular, we have that $\I(R,k_2)\leq c_3$ implies $\II(R,k_1,\alpha/2)\geq 1/5$.

\end{thm}

This result is qualitatively similar to Theorem~\ref{THM:LOWER_GENERAL}, up to the change the boundary condition \eqref{eq:condition_lower_k_0_general} into \eqref{eq:condition_lower_k_0_bus}. Taking $k_1=k_2=k_n \gg n/{\log^{1/2}(n)}$, we deduce part (i) of Theorem~\ref{mainthsigmaknown}.

As in Section~\ref{sec:complementlb}, we also deduce from Theorem~\ref{thm:lower_sigmaknown} that no procedure $R$ can simultaneously control the $\FDR$ at the nominal level  up to some $k_n\gg n/{\log^{1/2}(n)}$ while being also  {AMO} for all sequences $k_n\ll n/{\log^{1/2}(n)}$. 

\begin{cor}\label{cor:adaptation_known}
Consider the same numerical constants $c_1$--$c_5$  as in Theorem~\ref{thm:lower_sigmaknown} above. Take  any  $\alpha\in (0,c_3)$, any $n\geq c_1$ 
and fix any 
$
 \epsilon\in (n^{-1/4} ;(\alpha/2)^{c_2} e^{-(c_2 \log(2/\alpha))^2}).
$
Then for any procedure $R$ with $\I(R,k_2)\leq c_3$ for a sparsity $k_2= { \log^{1/2}(1/\epsilon)} \frac{n}{{\log^{1/2} n}}$, we have $\II(R,k_1,\alpha/2)\geq 1/5$ for a sparsity $k_1=\epsilon^{(c_2\log(2/\alpha))^{-2}}{\log^{1/2}(1/\epsilon)} e \frac{n}{{ \log^{1/2} n}}$.
In particular, if $n^{-1/16}<(\alpha/2)^{c_2} e^{-(c_2 \log(2/\alpha))^2}$,  we have  for any procedure $R$,
\begin{itemize}
\item if $\I(R,n/4)\leq \alpha$, then $\II(R,n^{1-\delta} e/4,\alpha/2)\geq 1/5$;
\item if  $\II(R,n^{1-\delta} e/4,\alpha/2)< 1/5$ then $\I(R,n/4)>c_3$,
\end{itemize}
where we let $\delta=1/(16 c^2_2\log^2(2/\alpha))>0$.
\end{cor}

\subsection{Lower bound for plug-in procedures}

In the previous section, we established an impossibility result for all multiple testing procedures $R$. In this section, we turn our attention to the special case of plug-in procedures $\BH_\alpha(\wh{\theta},\sigma(P))$ where $\widehat{\theta}$ is any estimator of $\theta(P)$.

\begin{thm}\label{THM:THETACHAPFORFDR}
There exist positive numerical constants $c_1$--$c_3$ such that the following holds for all $\alpha\in (0,1)$, all $n\geq N(\alpha)$, any estimator $\widehat{\theta}$,  and all $k$ satisfying 
\beq\label{eq:condition_lower_k_0}
 c_1\frac{n\log(2/\alpha)}{{\log^{1/2}(n)}}\leq   k  < \frac{n}{2}\ .
\eeq
There exists $P\in \cP$ with $n_1(P)\leq k$ and an event $\Omega$ of probability higher than $1/2-c_2/n$ such that, on $\Omega$, the plug-in procedure $\BH_\alpha(Y;\wh{\theta},\sigma(P))$ satisfies both
\begin{eqnarray}
\big|\BH_\alpha(Y;\wh{\theta},\sigma(P))\cap \cH_0(P)\big|&\geq& 0.5 n^{3/4}\ ; \nonumber \\
\FDP(P,\BH_\alpha(Y;\wh{\theta},\sigma(P))) &\geq& \frac{1}{2+ c_3n^{-1/5}}\ . \label{eq:lower_fdp_sigma_known}
\end{eqnarray}
\end{thm}

This theorem enforces that  no plug-in procedure $\BH_\alpha(Y;\wh{\theta},\sigma(P))$ is able  to control the $\FDR$ at the nominal level in dense settings ($k_n\gg n/{\log^{1/2}(n)}$). In fact, the FDP of plug-in procedures $\BH_\alpha(Y;\wh{\theta},\sigma(P))$  is even shown to be at least of the order of  $1/2$ with probability close to $1/2$. On the same event, the plug-in procedure $\BH_\alpha(Y;\wh{\theta},\sigma(P))$ makes many false rejections. This statement is much stronger than the one of Theorem~\ref{thm:lower_sigmaknown} (in the case $k_1=k_2$).

In contrast to the previous lower bounds, the proof of Theorem~\ref{THM:THETACHAPFORFDR} relies on a tighter control of the shifted $p$-value process and quantifies its impact on the BH threshold.

\section{Extension to general location models} \label{sec:generallocation}

In this section, we generalize our approach to the case where the null distribution is not necessarily Gaussian.
For simplicity, we focus here on the location model. Let $\cG$ denote the collection of densities on $\R$ that are  symmetric, continuous and non-increasing on $\R_+$. Given any $g\in \cG$, we extend the setting  of Section~\ref{sec:scalingnull}, by now assuming that $P=\otimes_{i=1}^n P_i$ belongs to the collection $\mathcal{P}_g$ of all distributions on $\R^n$ satisfying
\begin{equation}\label{equ-assumpfondgeneral}
\mbox{there exists } \theta\in \R \mbox{ such that } |\{i\in\{1,\dots,n\} \::\: P_i \mbox{ has density } g(\cdot-\theta)\}|>n/2.
\end{equation}
 In other words, we assume that there exists $\theta$ such that at least half of the $P_i$'s have for density $g(\cdot-\theta)$. Such $\theta$ is therefore uniquely defined from $P$, and we denote it again by $\theta(P)$.
The testing problem becomes
\begin{center}
$H_{0,i}: ``P_i \sim g(\cdot-\theta(P))"$ against $H_{1,i}: ``P_i \nsim g(\cdot-\theta(P))"$, for all $1\leq i \leq n$.
\end{center}

The rescaled $p$-values are now defined by 
\begin{equation}\label{equ-pvaluesgeneral}
p_i(u) =2\ol{G}\left(|Y_i-u|\right)  , \:\: u\in\R, \:\:1\leq i \leq n ,
\end{equation}
where $\ol{G}(y)=\int_y^{+\infty} g(x)dx$, $y\in\R$. The oracle $p$-values are given by $p_i^{\star}=2\ol{G}\left(|Y_i-\theta(P)|\right)$, $1\leq i \leq n$. The BH procedure at level $\alpha$ using $p$-values $p_i(u)$, $1\leq i\leq m$, is denoted $\BH_\alpha(u)$, whereas  the oracle version is still  denoted $\BH^\star_\alpha$.

\medskip 

For a given sparsity sequence $k_n\in[1,n/2)$, the  sequence of procedure $R=(R_\alpha)_{\alpha\in (0,1)}$ %
 is said to be {AMO} if there exists a positive sequence 
$\eta_n\to 0$
  such that
 \begin{align}
 \limsup_n \sup_{\alpha\in(1/n,1/2)}\{\I_g(R_\alpha,k_n)- \alpha\} &\leq 0;\label{IRkng}\\
 \lim_n \sup_{\alpha\in(1/n,1/2)}\{ \II_g(R_\alpha,k_n,\alpha(1-\eta_n))\} &=0\label{IIRkng}, 
 \end{align}
 where $\I_g(\cdot)$ and $\II_g(\cdot)$ are respectively defined as \eqref{defI} and \eqref{defII}, except that $\mathcal{P}$  is replaced by $\mathcal{P}_g$ therein.  
Similarly, for any sequence of estimators $\wh{\theta}$ of $\theta(P)$, the rescaling $\wh{\theta}$ is said to be AMO if $(R_\alpha)_{\alpha\in (0,1)}=(\BH_\alpha(\wh{\theta}))_{\alpha\in (0,1)}$ is AMO.

\subsection{Lower bounds}

We first state two conditions under which \eqref{IRkng} and \eqref{IIRkng} cannot hold together.

\begin{thm}\label{thm:lower_location}
Consider any $g\in \cG$.  There exist numerical positive constants  $c_1$ and $c_2$ and a constant $c_{g}$ (only depending on $g$)  such that the following holds for  all $n> 2k \geq  c_1$ and any $\alpha\in (0,1/2)$. Assume that
\beq\label{eq:condition_scaling}
\frac{k}{n c_{g}} \geq  \min_{t\in [\frac{\alpha}{2n};\frac{\alpha}{12}]} \left[\overline{G}^{-1}\left(\frac{t}{2}\right)- \ol{G}^{-1}\left(\frac{12 t}{\alpha}\right)\right]\ ,
\eeq
and consider 
\[
 t_0=\max\left\{t\in \bigg[\frac{\alpha}{2n};\frac{\alpha}{12}\bigg]\ \text{s.t.} \quad  \overline{G}^{-1}\left(\frac{t}{2}\right)- \ol{G}^{-1}\left(\frac{12 t}{\alpha}\right)\leq \frac{k}{n c_g}  \right\}\ . 
\]
For any multiple testing procedure $R$ 
satisfying  
\[
 \FDR(P,R)\leq \frac{1}{5} \ , \mbox{ for all $P \in \mtc{P}_g$ with $ n_1(P)\leq k$ },
\]
there exists some $P\in \cP_g$ with  $n_1(P)\leq k$ such that we have 
\begin{align}
\P_{Y\sim P}(|R(Y)\cap \cH_1(P)|= 0)&\geq 2/5\,;\nonumber\\
\P_{Y\sim P}\left[|\BH^{\star}_{\alpha/2}\cap \cH_1(P)|\geq \frac{2nt_0}{\alpha}\right]&\geq 1 - e^{-c_2 \alpha^{-1} nt_0}\ . 
\end{align}
In particular,  $\I_g(R,k)\leq 1/5$ implies $\II_g(R,k,\alpha/2)\geq 2/5-e^{-c_2 \alpha^{-1} nt_0}$.
\end{thm}

A consequence of Theorem~\ref{thm:lower_location} is that, for some sparsity sequence $k_n$, if for  all $n> 2k_n \geq  c_1$, Condition~\eqref{eq:condition_scaling} holds with $e^{-c_2  nt_0}\leq 1/5$, it is not possible to achieve any AMO procedure in the sense defined above. Interestingly, Condition~\eqref{eq:condition_scaling}  depends on the variations of $\ol{G}^{-1}(t)$ for small $t>0$. Taking $g= \phi$ and $t=1/\{n\log (n)\}^{1/2}$ and using the relations stated in Lemma~\ref{lem:quantile}, we recover Theorem~\ref{thm:lower_sigmaknown} (case $k_1=k_2$) obtained in the Gaussian location model and the corresponding {sharp} condition $k\gg n/\{\log(n)\}^{1/2}$.

Now consider the Laplace function $g(x)= e^{-|x|}/2$, so that $\ol{G}^{-1}(t)=\log(1/(2t))$. Then Condition~\eqref{eq:condition_scaling} cannot be guaranteed even when $k/n$ is of the order of a constant. More generally, Theorem~\ref{thm:lower_location} is silent for any $g$ such that 
$\min_{t\in [\frac{\alpha}{2n};\frac{\alpha}{12}]} [\overline{G}^{-1}(\frac{t}{2})- \ol{G}^{-1}(\frac{12 t}{\alpha})]$ is of the order of a constant.  

The next result is dedicated to this case. 
Remember that, when $k/n\geq 1/2$, $\theta(P)$ is not identifiable. We show that there exists a threshold $\pi_{\alpha}< 1/2$, such that deriving a {AMO} scaling is impossible when $k/n$ belongs to the region $(\pi_{\alpha},1/2)$. Markedly, $\pi_{\alpha}$ does not depend on $g$. For $\alpha\in(0,1)$, it is defined by
\beq\label{eq}
  \pi_{\alpha} = \frac{\sqrt{(1-\alpha)} -(1-\alpha)}{\alpha}\in (0,1/2)\ . 
\eeq
\begin{thm}\label{thm:lower_location2}
Consider any $\alpha\in (0,1)$ and $  \pi_{\alpha}$ given by \eqref{eq}. There exist a positive constant $c_{\alpha}$ (only depending on $\alpha$)  such that following holds for any $\ol{\pi}\in (\pi_{\alpha}, 1/2)$, any $g\in \cG$ and $n$ larger than a constant depending on $\alpha$ and $\ol{\pi}$. 
For any multiple testing procedure $R$ satisfying
\[
\FDR(P,R)\leq 1/4 \,,\mbox{ for all $P \in \mtc{P}_g$ and $ n_1(P)\leq \ol{\pi}n$ } ,
\]
there exists $P\in \cP_g$ with  $n_1(P)\leq \ol{\pi} n$ 
such that we have 
\begin{align}
\P_{Y\sim P}(|R(Y)|= 0)&\geq 1/3\,;\nonumber\\
\P_{Y\sim P}\left[|\BH^{\star}_{\alpha}\cap \cH_1(P)|\geq n \frac{\overline{\pi}}{4} \right]&\geq 1 - 10 e^{-c_{\alpha} n(\ol{\pi}-\pi_{\alpha})^2}\geq 3/4\ .
\end{align}
In particular, $\I_g(R,\ol{\pi}n)\leq 1/4$ implies $\II_g(R,\ol{\pi}n,\alpha)\geq 1/12$.
\end{thm}

To illustrate the above result, take $\alpha\in (0, 1/4]$ and $\ol{\pi}\in (\pi_{\alpha}, 1/2)$. Applying the above result for $\alpha'<\alpha$ with  $\pi_{\alpha'}<\ol{\pi}$, we obtain that, for any procedure $R$ with $\I_g(R,\ol{\pi}n)\leq \alpha$, we have $\II_g(R,\ol{\pi}n,\alpha')\geq 1/12$. In particular, this shows that there exists no {AMO} scaling  in the regime $k_n= n \ol{\pi}$, for $\ol{\pi}\in  (\pi_{\alpha}, 1/2)$. In addition, this holds uniformly over all $g$ in the class $\mathcal{G}$.

\subsection{Upper bound}
Since any $g\in \cG$ is symmetric {and puts a mass around ``$0$"}, $\theta(P)$ also corresponds to  the median of the null distribution. We consider, again, $\wt{\theta}=Y_{\lceil n/2\rceil}$ as the estimator of $\theta(P)$  and plug it into BH to build $\BH_\alpha(\wt{\theta})$.  
The following result holds. 

\begin{thm}\label{THMBORNESUPgeneral}
Consider any $g\in \cG$. 
There exist  constants $c_1(g), c_2(g) >0$ only depending on $g$ such that the following holds for all $n\geq c_1(g)$ and  $\alpha\in (0,0.5)$. Consider an integer $k \leq 0.1 n$ such that
\begin{equation}\label{relationGbar}
 \eta=c_2(g) \left((k/n)\vee n^{-1/6}\right) \max_{t\in [0.95 \alpha/n,\alpha]} \left\{ \frac{1}{\ol{G}^{-1}(t/2)- \ol{G}^{-1}(t )}   \:\frac{g(\ol{G}^{-1}(t))}{g(\ol{G}^{-1}(t/2))}\right\} \leq 0.05.
\end{equation}
Then, we have
\begin{align}
 &\I_g(\BH_\alpha(\wt{\theta}),k) \leq   \alpha (1+\eta) + e^{-n^{1/2}}\ ;\label{eq:upper_FDR_thmgeneral} \\
 &\II_g(\BH_\alpha(\wt{\theta}),k, \alpha(1-\eta))\leq  e^{-n^{1/2}} \ .\label{eq:lower_TDP_thmgeneral}
\end{align}
\end{thm}
If we consider any asymptotic setting where $\eta$ in~\eqref{relationGbar} converges to 0, then it follows from the above theorem that $\wt{\theta}$ is a {AMO} scaling.

Comparing \eqref{eq:condition_scaling} of the lower bound in the previous section with~\eqref{relationGbar}, we observe that those are matching up to the term
$$
\max_{t\in [0.95 \alpha/n,\alpha]} \left\{\frac{g(\ol{G}^{-1}(t))}{g(\ol{G}^{-1}(t/2))}\right\} .
$$
The latter is of the order of a constant for the Subbotin-Laplace cases as illustrated below. 

\subsection{Application to Subbotin distributions}

We now apply our general results to the class of Subbotin distributions. 

\begin{cor}\label{cor:subottin}
Consider the location Subbotin null model for which $g(x)=L^{-1}_\zeta \:e^{-|x|^\zeta/\zeta}$, for some fixed $\zeta>1$ and the normalization constant $L_\zeta=2 \Gamma(1/\zeta)\zeta^{1/\zeta-1}$. Then 
\begin{itemize}
\item[(i)] for a sparsity $k_n\gg n/ (\log(n))^{1-1/\zeta}$,   
there exists no (sequence of) procedure that is AMO. 
\item[(ii)] for a sparsity $k_n\ll n/ (\log(n))^{1-1/\zeta}$,   
 the scaling $\wt{\theta}=Y_{\lceil n/2\rceil}$ is AMO.
 \end{itemize}
\end{cor}

\begin{cor}\label{cor:laplace}
Let us consider the Laplace density $g(x)=0.5 \:e^{-|x|}$. Then 
 for a sparsity $k_n\ll n$,   
 the scaling $\wt{\theta}=Y_{\lceil n/2\rceil}$ is AMO.
\end{cor}

\subsection{An additional result for the Laplace location model}\label{sec:laplace}

Our general theory implies that, in the Laplace location model, an AMO scaling is possible when $k_n\ll n$ (Corollary~\ref{cor:laplace}) and is impossible if $\lim \inf k_n / n > \pi_{\alpha}$ (Theorem~\ref{thm:lower_location2}). However, it is silent when $k_n/n$ converges to a small constant $\pi\in (0,1)$. In this section, {we investigate  the scaling problem in this regime. We}
establish that AMO scaling is impossible and that one needs to incur a small but yet non negligible loss.
Define, for any $\alpha\in (0,1)$, 
\begin{equation}\label{eq2}
\pi^{*}_{\alpha}=\frac{1-\sqrt{\alpha}}{2-\sqrt{\alpha}}\:\: {\in (\pi_\alpha,1/2)}\enspace .
\end{equation}

\begin{prp}[Lower Bound for the Laplace distribution]\label{prp:lower_laplace}
There exists a positive and increasing function $\zeta: (0,1/2)\mapsto \mathbb{R}_+$ with $\lim_{1/2} \zeta = +\infty$ such that the following holds for any $\alpha\in (0,1)$, any $\ol{\pi}  < \pi^*_{\alpha}$ \and for any $n$ larger than a constant depending only on $\alpha$ and $\ol{\pi}$. 
For any procedure $R$ satisfying
 \[
 \FDR[P,R]\leq \alpha n_0(P)/n\,,\mbox{ for all $P\in\mtc{P}_g$ with $n_1(P)\leq \ol{\pi} n$ },
 \]
there exists a distribution  $P\in \mtc{P}_g$ with $n_1(P)\leq \ol{\pi} n$ such that
\[
  \P_{Y\sim P}[|\BH^*_{\alpha}|>0] - \P_{Y\sim P}[|R(Y)|>0] \geq \alpha \zeta(\overline{\pi})  - c_{\overline{\pi}} n^{-1/3}\ ,
\] 
where $ c_{\overline{\pi}} $ only depends on $\ol{\pi}$. 
\end{prp}

Recall that, for any distribution $P$, the FDR of $\BH^\star_{\alpha}$ is equal to $\alpha n_0(P)/n$, see \eqref{FDRcontrolBH}. Hence, the above proposition states that any procedure achieving the same FDR bound as the oracle procedure is strictly more conservative than the oracle, in the sense that $ \P_{Y\sim P}[|\BH^*_{\alpha}|>0,|R(Y)|=0] \geq \alpha \zeta(\overline{\pi})+o(1)>0$. In addition, the amplitude of $\alpha \zeta(\overline{\pi})$ is 
increasing with $\ol{\pi}$, which is expected. 
Also, the assumption $\ol{\pi}  < \pi^*_{\alpha}$ is technical. In particular, we can easily prove that, for larger $\ol{\pi}$, the result remains true by replacing $\zeta(\ol{\pi})$ by $\zeta(\ol{\pi} \wedge \pi^\star_{\alpha})$. 

\begin{remark}
On the feasibility side,  
we can show that in the regime where $n_1(P)/n$ converges to a small constant, {the plug-in BH procedure at level $\alpha$ is yet not AMO, but is comparable to oracle BH procedures with modified nominal levels $\alpha'\neq \alpha$.} Recall that 
$ p_i(u)= 2\ol{G}(|Y_i-u|)= e^{-|Y_i-u|}$ and $p_i^{\star}= e^{-|Y_i-\theta(P)|}$. As a consequence, given an estimator $\widehat{\theta}$, the ratio 
$p_i(\widehat{\theta})/p_i^{\star}$ belongs to 
$[e^{-|\widehat{\theta}-\theta(P)|}; e^{|\widehat{\theta}-\theta(P)|}]$. Assuming that $\alpha e^{|\widehat{\theta}-\theta(P)|}< 1$, it follows from the definition of $\BH_{\alpha}(u)$ that 
\[
 \BH^{\star}_{\alpha e^{-|\widehat{\theta}-\theta(P)|}} \subset \BH_{\alpha}(\widehat{\theta})\subset \BH^{\star}_{\alpha e^{|\widehat{\theta}-\theta(P)|}}\ . 
\]
As a consequence, as long as $|\widehat{\theta}-\theta(P)|\leq \log(1/\alpha)$, $\BH_{\alpha}(\widehat{\theta})$ is sandwiched between two oracle BH procedures with modified type I errors. As an example, the median estimator $\widehat{\theta}=Y_{\lceil n/2 \rceil }$ satisfies $|\widehat{\theta}-\theta|\leq c n_1(P)/n$ with high probability when $n_1(P)/n$ is small enough (see the proof of  Theorem~\ref{THMBORNESUPgeneral}). As a consequence, with high probability, we have 
\[
 \BH^{\star}_{\alpha  e^{-cn_1(P)/n}} \subset \BH_{\alpha}(\widehat{\theta})\subset \BH^{\star}_{\alpha e^{cn_1(P)/n}}\ .
\] 
{Conversely, Proposition~\ref{prp:lower_laplace} entails that no multiple testing procedure can be sandwiched by oracle procedures with level  $\alpha (1-o(1))$ and $\alpha(1+o(1))$. }
\end{remark}

\section{Confidence region for the null and applications}\label{sec:adaptive}
In this section, we tackle the issue of building a confidence superset on the possible null distribution(s) for $P$, which has not been considered yet in the literature to the best of our knowledge. 

\subsection{A confidence region for the null}

We come back to the general Huber model described in Section~\ref{sec:scalingnull}, although we do not assume that the null distribution is necessarily Gaussian. That is, the observations $Y_i$, $1\leq i \leq n$ are only assumed to be independent. Their respective c.d.f.'s are denoted by $F_i$, $1\leq i \leq n$,  and we let 
\begin{equation}\label{equ-F0kP}
\mathcal{F}_{0,k}(P) = \left\{F_0 \mbox{ c.d.f. } \::\: |\{i\in \{1,\dots,n\}\::\: F_i = F_0 \}|\geq n-k  \right\}
\end{equation}
the set of all plausible null c.d.f.'s for $P$, for some prescribed, known, maximum amount of contaminated marginals $k\in [0,n-1]$. As before, if $k<n/2$, the set $\mathcal{F}_{0,k}(P)$ is of cardinal at most $1$. Otherwise, several null  c.d.f.'s are possible for $P$.

Our inference is based on the empirical c.d.f. $\wh{F}_n$ of the sample $Y$: the idea is that for any $F_0\in \mathcal{F}_{0,k}(P)$, the function $(\wh{F}_n- (1-k/n)F_0)/(k/n)$ should be close to be a c.d.f., which induces some constraints. 
This idea bears similarities with the existing literature, in particular with \cite{GW2004}, that derived confidence interval for the proportion of signal when the true null is known and uniform. 

For some $\alpha\in(0,1)$, let us denote 
\begin{align}
\mathcal{F}_{1-\alpha}(Y)&=\left\{F_0\mbox{ c.d.f. }\::\: \forall j \in \{0,\dots,n\}, \wh{a}_n(j;F_0) \leq \wh{b}_n(j;F_0)\right\}\label{equ-conf-F0};
\\
\wh{a}_n(j;F_0)&=0\vee \frac{\max_{0\leq \l\leq j}\left\{\l/n-(1-k/n)F_0(Y_{(\l)})\right\}-c_{n,\alpha}}{k/n} \label{achapnF0};
\\
\wh{b}_n(j;F_0)&=1\wedge \frac{\min_{j\leq \l\leq n}\left\{\l/n-(1-k/n)F_0(Y_{(\l+1)}^{-})\right\}+c_{n,\alpha}}{k/n} \label{bchapnF0},
\end{align}
where $c_{n,\alpha}=\{-(1-k/n)\log(\alpha/2)/(2n)\}^{1/2}$, $F_0(y^{-})=\lim_{x\to y^-}F_0(x)$ and $Y_{(1)}\leq \dots\leq Y_{(n)}$ denote the order statistics ($Y_{(0)}=-\infty$, $Y_{(n+1)}=+\infty$) of the observed sample $(Y_i, 1\leq i \leq n)$.
Note that all these quantities depend on $k$, but we have omitted it in the notation for short. The following result holds.

\begin{thm}\label{thmconfidence}
For a given sparsity parameter $k\in [0,n-1]$, the region $\mathcal{F}_{1-\alpha}(Y)$ defined by \eqref{equ-conf-F0} is a $(1-\alpha)$-confidence superset of the set  $\mathcal{F}_{0,k}(P)$ \eqref{equ-F0kP} of possible nulls for $P$ with at most $k$ contaminations, in the following sense:
$$
\P(\mathcal{F}_{0,k}(P)\subset \mathcal{F}_{1-\alpha}(Y))\geq 1-\alpha.
$$
\end{thm}

Compared to our previous results, this result is less demanding on the sparsity parameter: it  only assumes $n_1(P)\leq k$ and not that $n_1(P)/n$ tends to zero at some rate. In particular, it goes beyond the worst-case boundary effect delineates in our main results. Also, it is non parametric, and usable in combination with any possible modeling for the null.

\subsection{Application 1: a goodness of fit test for a given null distribution}

Considering any known c.d.f. $F_0$,  the confidence region derived in Theorem~\ref{thmconfidence} provides a way to test the null hypothesis $H_0': $ ``$F_0\in \mathcal{F}_{0,k}(P)$", that is, ``$F_0$ is a plausible null c.d.f. for  $P$ with at most $k$ contaminations".

\begin{cor}\label{cor:testnull}
Consider $k\in [0,n-1]$ and the sets  $\mathcal{F}_{1-\alpha}(Y)$ \eqref{equ-conf-F0} and  $\mathcal{F}_{0,k}(P)$ \eqref{equ-F0kP}. Consider any c.d.f. $F_0$. 
Then the test rejecting the null hypothesis $H_0': $ ``$F_0\in \mathcal{F}_{0,k}(P)$" whenever $F_0\notin \mathcal{F}_{1-\alpha}(Y)$, that is, if there exists $ j \in \{0,\dots,n\}$, such that $\wh{a}_n(j;F_0) > \wh{b}_n(j;F_0)$, is of level $\alpha$.
\end{cor}

Since for any $F_0\in \mathcal{F}_{0,k}(P)$, $\P(F_0\notin \mathcal{F}_{1-\alpha}(Y))\leq 1-\P(\mathcal{F}_{0,k}(P) \subset \mathcal{F}_{1-\alpha}(Y))$, the proof is straightforward from Theorem~\ref{thmconfidence}.
In particular, this test can be used to test whether the {\it theoretical} null distribution $\mathcal{N}(0,1)$  is suitable for some data set, given some maximum proportion of contaminations, say $k/n=10\%$. As shown in the vignette \cite{RVvignette2020}, this test rejects $H_0' $ for many data sets.

Next, we can also build a goodness of fit test of level $\alpha$ for a family $(F_{0,\vartheta})_{\vartheta\in\Theta}$ of null c.d.f.'s. This corresponds to consider the null hypothesis $H_0'': $ ``$\exists \vartheta\in\Theta \::\: F_{0,\vartheta} \in \mathcal{F}_{0,k}(P) $", that is, ``in the family $(F_{0,\vartheta})_{\vartheta\in\Theta}$ there is at least a plausible null c.d.f. for  $P$ with at most $k$ contaminations". 

\begin{cor}\label{cor:testfamilynull}
Consider $k\in [0,n-1]$ and the sets  $\mathcal{F}_{1-\alpha}(Y)$ \eqref{equ-conf-F0} and  $\mathcal{F}_{0,k}(P)$ \eqref{equ-F0kP}. Consider any family of c.d.f.'s  $(F_{0,\vartheta})_{\vartheta\in\Theta}$. 
Then the test rejecting the null hypothesis $H_0'': $ ``$\exists \vartheta\in\Theta \::\: F_{0,\vartheta} \in \mathcal{F}_{0,k}(P) $" whenever $\forall \vartheta\in\Theta,$ $F_{0,\vartheta} \notin\mathcal{F}_{1-\alpha}(Y)$, that is, if the confidence region does not contains any null distribution of the family, is of level $\alpha$.
\end{cor}

Since for any $\vartheta_0\in \Theta$, with $ F_{0,\vartheta_0} \in \mathcal{F}_{0,k}(P)$, we have $\P(\forall \vartheta\in\Theta,F_{0,\vartheta} \notin\mathcal{F}_{1-\alpha}(Y))\leq \P(F_{0,\vartheta_0} \notin\mathcal{F}_{1-\alpha}(Y))\leq 1-\P(\mathcal{F}_{0,k}(P) \subset \mathcal{F}_{1-\alpha}(Y))$,  the proof is straightforward from Theorem~\ref{thmconfidence}.
 As typical instance, this can be used to build a goodness of fit test to the family of Gaussian null distribution with arbitrary scaling.
 Interestingly, this test never rejects this null hypothesis for the data sets used in the vignette, which shows that considering Gaussian null can be suitable for these data. 

The two aforementioned tests provide a way to validate Efron's paradigm 
who discarded the theoretical null $\mathcal{N}(0,1)$, while still using empirical Gaussian nulls.


\subsection{Application 2: a reliability indicator for empirical null procedures}

For simplicity, let us focus on the case of Gaussian null distributions as in the setting of Section~\ref{sec:scalingnull}, for which $k<n/2$. The confidence region derived in Theorem~\ref{thmconfidence} induces a confidence region for the true scaling $(\theta(P),\sigma(P))$ given by
{
\begin{equation}\label{confregion}
\mathcal{S}_{k,\alpha} = \left\{ (\theta,\sigma)\in \R\times (0,\infty) \::\: \forall j\in\{0,\dots,n\}, \wh{a}_n(j;\Phi((\cdot-\theta)/\sigma))\leq \wh{b}_n(j;\Phi((\cdot-\theta)/\sigma)) \right\}.
\end{equation}
\begin{cor}\label{cor:region}
Consider the setting of Section~\ref{sec:scalingnull}.  Provided that $n_1(P)\leq k$, the set 
$\mathcal{S}_{k,\alpha}$ is a $(1-\alpha)$-confidence region for the true scaling $(\theta(P),\sigma(P))$.
In particular, with probability at least $1-\alpha$, the oracle BH procedure is one of the procedures $\BH_\alpha(u,s)$, $(u,s)\in \mathcal{S}_{k,\alpha}$.
\end{cor}
}
Corollary~\ref{cor:region} indicates a way to guess the rejection set of the oracle BH procedure, by inspecting how the rejection set of the procedure $\BH_\alpha(u,s)$ varies across the values $(u,s)$ in the region $\mathcal{S}_{k,\alpha}$. This suggests practical recommandations to assess the reliability of the BH plug-in procedure. Typically, for ``stable" rejection sets, the plug-in BH procedure $\BH_\alpha(\wt{\theta},\wt{\sigma})$ can be used, while for ``variable" rejection sets, only null hypotheses belonging to all rejection sets of the procedures $\BH_\alpha(u,s)$, $(u,s)\in \mathcal{S}_{k,\alpha}$ can be safely rejected.

\begin{figure}[h!]
\begin{tabular}{cc}
  Lower bound 
  & Gaussian alternative 
  \\
\hspace{-6mm}
\includegraphics[scale=0.33]{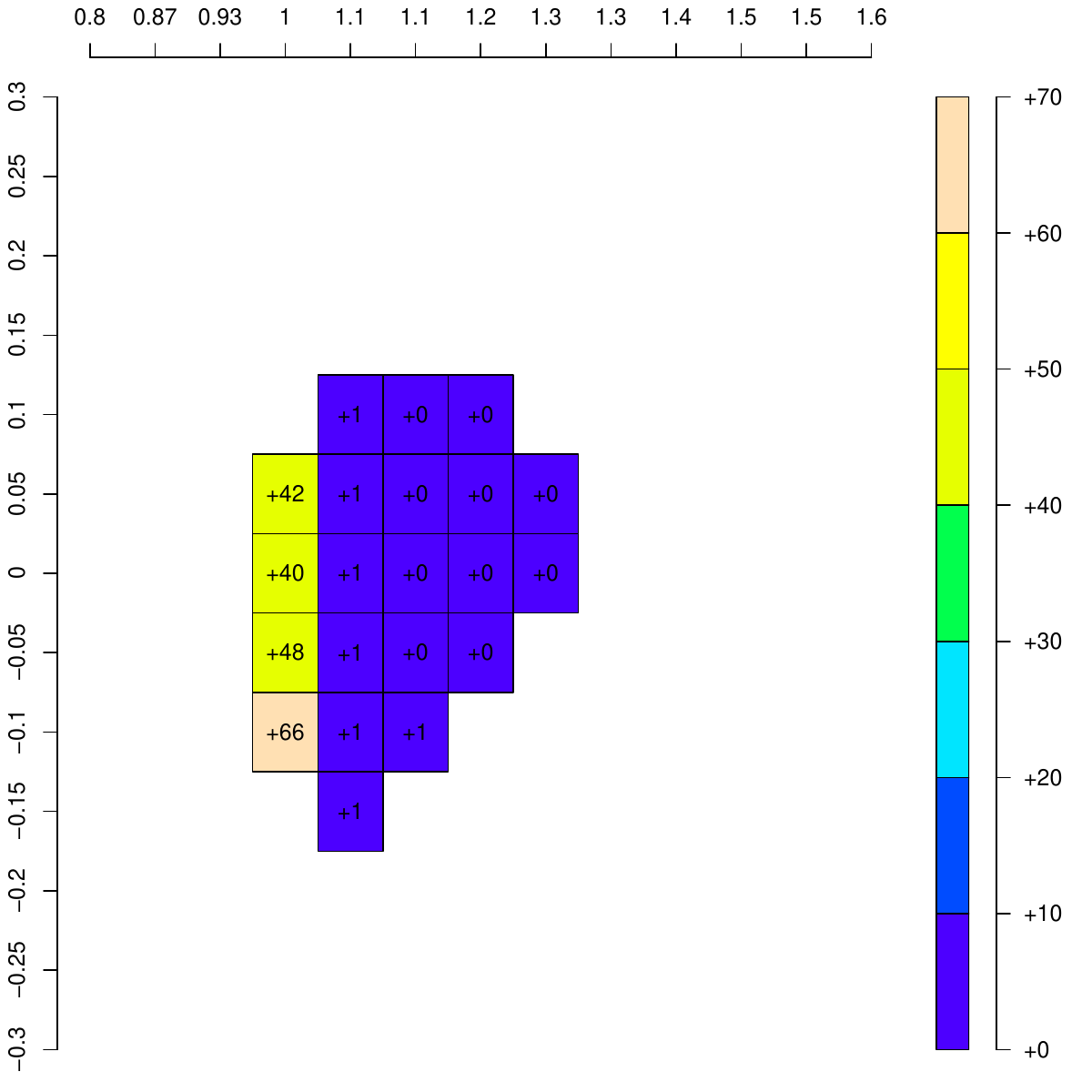}
&
\hspace{-6mm}
\includegraphics[scale=0.33]{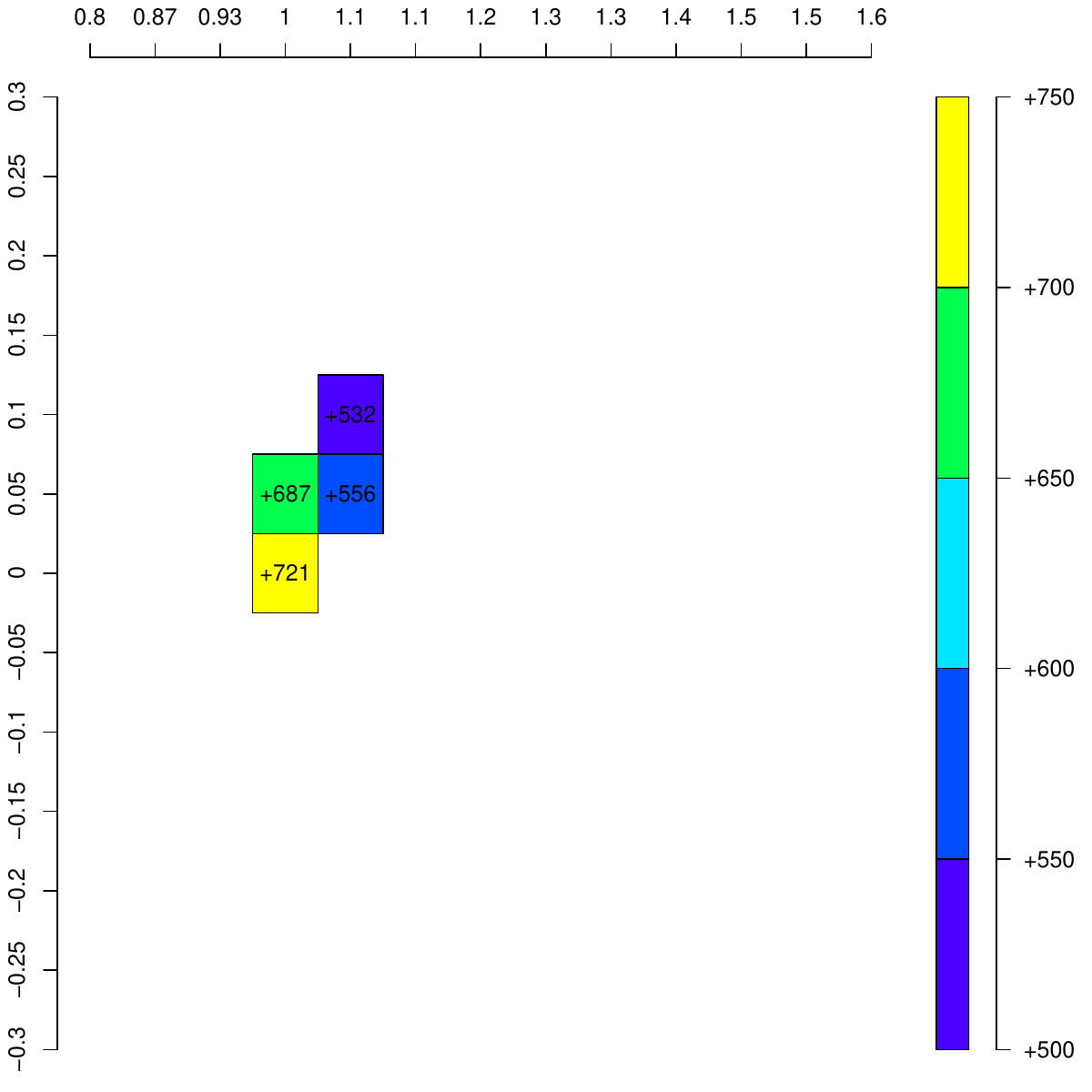}
\\
Leukemia dataset \cite{Golub99}&  HIV data set \cite{vanWout2003}  \\
\hspace{-6mm}
\includegraphics[scale=0.33]{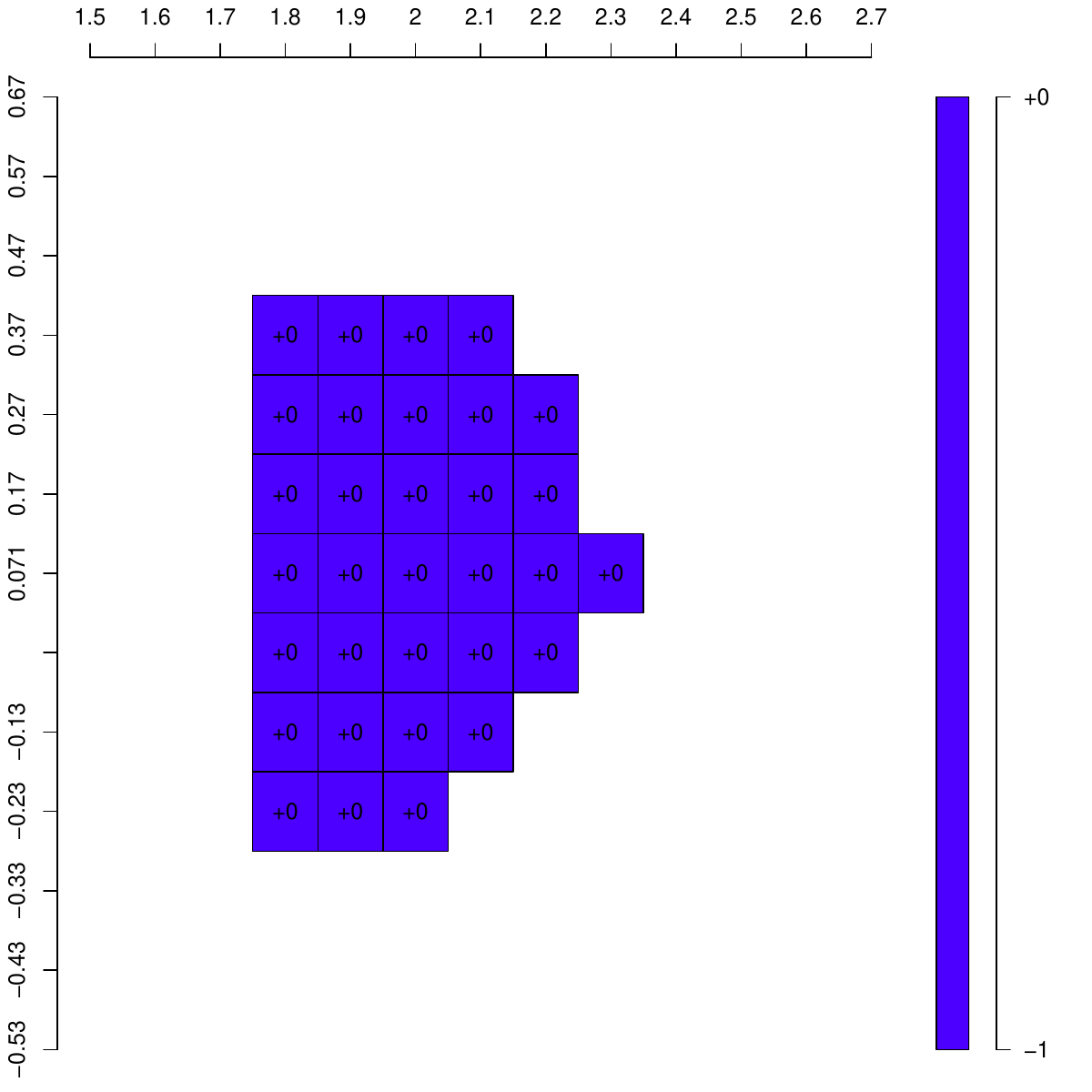}&
\hspace{-6mm}
\includegraphics[scale=0.33]{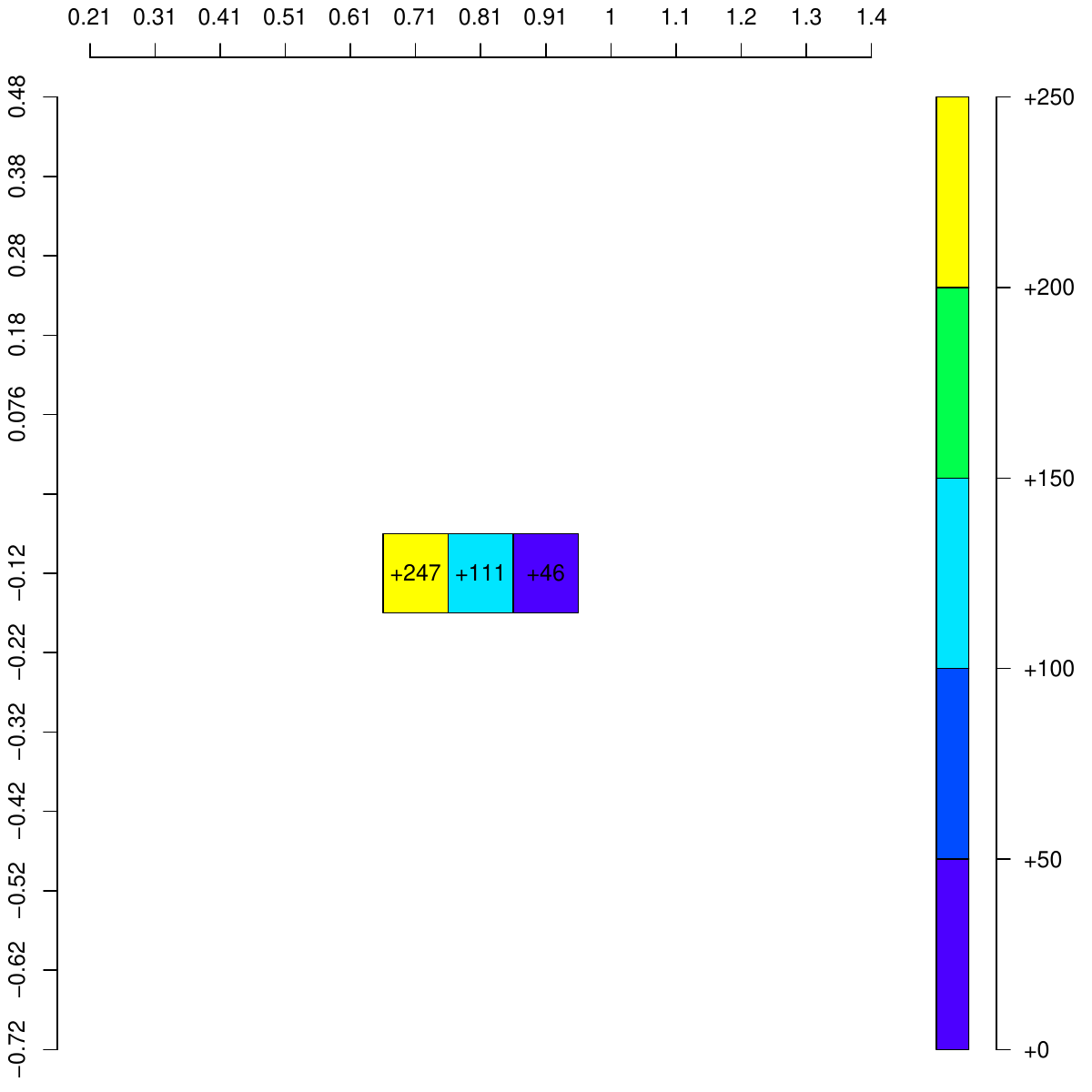}

\end{tabular}
\caption{Plot of the confidence region  $\mathcal{S}_{k,\alpha}$ \eqref{confregion} for $\alpha=0.1$, $k/n=0.1$, in the Gaussian Huber model. Top: two simulated data sets with $n=10\,000$. Bottom: the two real data sets \cite{Golub99}  (bottom-left) and \cite{van2003} (bottom right) for which $n=3051$ and $n=7680$, respectively. In each pixel of the region, the depicted number is the rejection number of the plug-in BH procedure at level $\alpha=0.1$ using the corresponding scaling. More details are given in the text. The bottom panels are reproducible from the vignette \cite{RVvignette2020}\label{fig:heatmap}.}
\end{figure}

As an illustration, we have displayed the obtained region $\mathcal{S}_{k,\alpha}$ \eqref{confregion} (colored area) in Figure~\ref{fig:heatmap} for different data sets. 
Also, in each point of this area, we have added the number of rejections of the plug-in BH procedure using the corresponding null (here, only the rejection number is reported for short).
The two top ones corresponds to simulated observations (only 1 run each time). In the ``lower bound" setting, the null is $\mathcal{N}(0,\sigma^2)$ and the alternative density is $\frac{0.9}{0.1}\left[\phi-\phi(\cdot/\sigma)/\sigma\right]_+$,  
which gives $\sigma\approx 1.26$. 
The region is wide in this setting and contains the true null $\mathcal{N}(0,\sigma^2)$ (no rejection for the plug-in BH procedure) but also the erroneous null $\mathcal{N}(0,1)$ (some rejections for the plug-in BH procedure). This is in accordance with the special shape of the lower bound, as discussed in Section~\ref{sec:lb}. 
By contrast, in the ``Gaussian alternative" setting, the null is  $\mathcal{N}(0,1)$ and the alternative is $\mathcal{N}(3,1)$. The region looks much more narrowed and 
contains only scalings for which the corresponding plug-in BH procedures make many findings.  Since the user knows that the oracle BH procedure is one of these procedures (with high probability, see Corollary~\ref{cor:region}), but could be any of these procedures, the plot indicates that the user should better make no rejection in the top-left situation, while make some rejections (at least $532$) in the top-right situation.

The bottom panels in Figure~\ref{fig:heatmap} display the region for two classical real data sets. 
For the leukemia data set,  the region contains scalings for which the plug-in BH procedure makes no rejection (this turns out to be true for all of them). Hence, declaring any variable as significant can be foolhardy for this data set. By contrast, for the HIV data set, there is evidence that the oracle BH procedure is able to reject some variables there, and using the plug-in BH procedure can be legitimately used 
 in that case. 

Finally, note that the indicator displayed in the region concerns plug-in BH procedures. However, depending on the aim of the practitioner, it could be any other classical procedures using a prescribed null.
We also underline that the above analysis can be done with any null family $(F_{0,\vartheta})_{\vartheta\in\Theta}$, not necessarily Gaussian. The $(1-\alpha)$-confidence region for the true null parameter(s) $\vartheta(P)$ is then replaced by
$$
{\Theta}_{k,\alpha} = \left\{ \vartheta\in \Theta \::\: \forall j\in\{0,\dots,n\}, \wh{a}_n(j;F_{0,\vartheta})\leq \wh{b}_n(j;F_{0,\vartheta}) \right\}.
$$

\section{Discussion}\label{sec:discussion}

Elaborating upon Efron's problem, we have presented a general theory to assess whether one can estimate the null and use it into a plug-in BH procedure, while keeping properties similar to the oracle BH procedure in terms of FDP and TDP. 
As expected, the sparsity parameter $k$ played a central role, and   matching lower bounds and upper bounds were established. The obtained sparsity boundaries were shown to depend i) on the fact that the null variance is known or not and ii) on the variations of the quantile function of the null distribution.
We eventually went beyond the worst case analysis by designing a confidence region for the null distribution, that is valid for any possible null, and that only relies on an independence assumption, together with some known upper bound on the signal proportion. This led to goodness of fit tests for null distributions and a visualization method for assessing the reliability of empirical null procedures  that are both useful for  a practical use. This is illustrated in detail in the vignette \cite{RVvignette2020}. 

This work paves the way for several extensions. 
First, one direction is to investigate the sparsity boundary when the model is reduced, e.g., by considering more constrained alternatives. 
A first hint has been given for one-sided alternatives in \cite{CDRV2018}, where both a uniform FDR control and power results can be achieved in dense settings, e.g., $k=n/2$ (say), which is markedly different from what we obtained here. 
In future work, many more structured setting can be considered, e.g., decreasing alternative  densities,  temporal/spatial structure on the signal, and so on.
Second, the problem could also be made more difficult 
by considering more complex model for the null,  for instance, dropping the assumption that $g$ is known, but assuming instead that it belongs to some parametric or non-parametric class. 
Each of this setting should come with a new boundary that is worth to investigate.
Third, the independence between the tests is a strong assumption that is essential in all our results. Our numerical experiments suggest that our results could be extended to weak dependencies. Supporting this fact with a theoretical statement is probably challenging, but deserves to be explored. 
Finally, the new proposed confidence region is based on the DKW inequality, which is not always an accurate tool. Reducing the size of the confidence region is an interesting avenue for future investigations.

\section*{Acknowledgements} 
This work has been supported by ANR-16-CE40-0019 (SansSouci), ANR-17-CE40-0001 (BASICS) and by the GDR ISIS through the "projets exploratoires" program (project TASTY). {We are grateful to Ery Arias-Castro, David Mary and to anonymous referees for insightful comments that helped us to improve the presentation of the manuscript. }

\bibliographystyle{apalike}
\bibliography{biblio}

\appendix

\section{Numerical experiments}\label{sec:numexp}


In this section, we illustrate our results with a numerical experiments. We consider a classical setting that allows to evaluate the performances of multiple testing procedures under dependence, see, e.g., \cite{fithian2020conditional}.

\subsection{Description}

In all our experiments, we consider $Y\sim \mathcal{N}(\mu,\Sigma)$, for some unknown mean $\mu_i=0$, $1\leq i \leq n-k$, with three different alternative structures, combined with three dependence structures (correlation matrices) $\Sigma$ and $k/n\in\{0.03,0.1,0.3,0.5\}$. This gives $3\times3\times4=36$ settings, each of them being summarized with a $(\FDR,\TDR)$ ROC-type plot. The error rates are all estimated with $100$ replications.
Note that for simplicity, we use as power criterion the $\TDR$, which is the average of the $\TDP$ here. While this  is slightly different than the power notion used in our theory, this is fair enough for these numerical experiments.
Also, without loss of generality, the true values of the scaling parameters are $\theta=0$ and $\sigma=1$.

We considered correlation structures ranging from independence to strong dependence:
\begin{enumerate}
\item Independence: $\Sigma=I_n$;
\item Block-correlation: $\Sigma_{i,i}=1$ and $\Sigma_{i,j}=0.5$ if $\lceil i/20\rceil = \lceil j/20\rceil$  and $0$ otherwise;
\item Equi-correlation: $\Sigma_{i,i}=1$ and $\Sigma_{i,j}=\rho=0.2$ for $i\neq j$;
\end{enumerate}

The following  structures are explored for the alternatives:
\begin{enumerate}
\item Standard: the $\mu_i\neq 0$ are generated as i.i.d. variable uniform on $[-5,-2]\cup[2,5]$;
\item Cauchy: in addition to standard alternatives, the values $X_i$, $n-k+1\leq i\leq n-k/2$, are replaced by i.i.d. Cauchy variables; 
\item Zero-located: in addition to standard alternatives, the values $X_i$, $n-k+1\leq i\leq n-k/2$, are replaced by i.i.d. variables generated according to the density $f_2(x)=40\max(0,\phi(x)-\phi(x/\sigma)/\sigma)$  with the  value of $\sigma>1$ that makes $f_2$ being a density ($\sigma\approx 1.053$).
\end{enumerate}
Note that all the alternative structures keep a portion of standard alternatives. Then, (Half-)Cauchy allows very high alternatives and is thus devoted to challenge the robustness of the methods, while (Half-)zero-located is inspired by the lower bound, with alternatives very close to $0$ (see right panel of Figure~\ref{lbsigma}), and is thus devoted to challenge the methods (especially in terms of FDR control) in a case close to the worse configuration. 
\\

On each plot, we consider the five following methods:
\begin{enumerate}
\item~[Oracle]: plug-in BH at level $\alpha/\pi_0$, that uses the true value of the scaling parameters, here $\theta=0$, $\sigma=1$;
\item~[MedianMAD]: plug-in BH at level $\alpha/\pi_0$ with $\wt{\theta}$ being the median and $\wt{\sigma}$ being the median absolute deviation, see \eqref{equ:estimators};
\item~[TrimMAD]: same but with $\wt{\theta}$ being the trimmed mean (50\% of trimmed data);
\item~[MLE]: plug-in BH at level $\alpha/\pi_0$ with estimators given by a maximum likelihood method in the truncated Gaussian model, as done in package \texttt{locfdr}; 
\item~[LocfdrSC]: procedure that first computes the local fdr values computed as done in package \texttt{locfdr} and that combines them to provide a FDR control as in \cite{SC2007}.
\end{enumerate}
Note that we performed the plug-in BH procedure at level $\alpha/\pi_0$ instead of $\alpha$ to obtain an Oracle with an FDR that can be well compared to $\alpha$, rather than $\pi_0\alpha$, which makes plots more readable across different values of $k$. (In the Gaussian case, the theory is unchanged with this oracle).
[MedianMAD] is the default procedure studied in the paper, while [TrimMAD] and [MLE] are variations thereof, using different scaling estimators. [LocfdrSC] is a more sophisticated procedure, not of a plug-in BH type, that learns the null distribution, the null proportion and the alternative distribution.

Each of these methods are computed for $\alpha\in \{0.05,0.2\}$. So there are $2\times 5=15$ $(\FDR,\TDR)$-points per plot. 
\subsection{Results}

The results are displayed in Figures~\ref{fig:indep},\ref{fig:block}~and~\ref{fig:equi}, for the independent, block-dependent and equi-correlated case, respectively. We summarize below the findings in each case.

First, in the independent case (Figure~\ref{fig:indep}):
\begin{itemize}
\item [Oracle] has FDR close to $\alpha$, as the theory provides.
\item First line of Figure~\ref{fig:indep}: in the case of standard alternative all methods mimic the oracle for $k$ small but all fails for $k$ too large. Methods based on the \texttt{locfdr} package ([MLE] and LocfdrSC) are slightly better than [MedianMAD] and [TrimMAD], because they include a more sophisticated estimation of the null or/and alternative distributions.
\item Methods based on the \texttt{locfdr} package ([MLE] and [LocfdrSC]) lack of robustness. First, consider the second line of Figure~\ref{fig:indep}: both method lead to an error when running the code, which seems due to the too extreme values taken by the Cauchy alternatives, so both methods displayed as $(0,0)$ points for Cauchy alternative. Second, consider the third line of Figure~\ref{fig:indep}: their FDR are not controlled anymore for zero-located alternatives. 
\item The two procedures [TrimMAD] and [MedianMAD] behaves similarly on all experiments.
\item Globally, the different observed effects are more visible for $\alpha=0.2$ than for $\alpha=0.05$.
\end{itemize}

The results in the block-correlation case are qualitatively similar to the independence case, see Figure~\ref{fig:block}. This means that the different processes involved in the problem seem to have a similar behavior under (this kind of) weak dependence. 

In the equi-correlation case, however, the situation is markedly different; the dependence is strong and distributed simultaneously among all the measurements. It is worth to note that the measurements can be written as 
$$
Y_i   =  \mu_i + \rho^{1/2}  U  + (1-\rho)^{1/2} \xi_i, \:\:1\leq i\leq n, 
$$
where $U,\xi_1,\dots,\xi_n$ are all i.i.d. $\mathcal{N}(0,1)$. Hence, conditionally on the factor $U$, we have that $Y$ follows the distribution $\mathcal{N}(\mu_i+ \rho^{1/2}  U, (1-\rho)I_n)$. This means that, while the {\it true unconditional null} is $\mathcal{N}(0,1)$, the {\it true conditional null} is $\mathcal{N}(\rho^{1/2}  U, 1-\rho)$. Based on this,  [Oracle], which is oracle under independence,  is not the oracle anymore in the equi-correlated setting; it is improvable by learning  the conditional null. 
In the equi-correlation case, the results are displayed in Figure~\ref{fig:equi} and can be summarized as follows:
 \begin{itemize}
\item Overall, [MedMAD], [TrimMAD] are more robust than [MLE], [LocfdrSC] but less powerful for standard alternatives, again.
\item When $k$ is too large, all methods fail to mimic the oracle, again.
\item The empirical procedures are able to remove the equi-correlation for $k\in\{30,100\}$; the TDR of all methods [MedMAD], [TrimMAD], [MLE] and [LocfdrSC] are larger than the one of [Oracle]. This is due to what is mentioned above: they all learn the conditional null distribution, rather than the unconditional null.
 \end{itemize}
 
Overall, this numerical experiments reinforces that, in a minimax sense, we should focus on the robust procedures   [MedMAD], [TrimMAD] while for more narrowed and favorable alternatives, they can be improved by the more sophisticated procedures [MLE] and [LocfdrSC].




\begin{figure}[h!]
\begin{tabular}{cccc}
 $k=30$ & $k=100$ & $k=300$ & $k=500$\\
\includegraphics[scale=0.17]{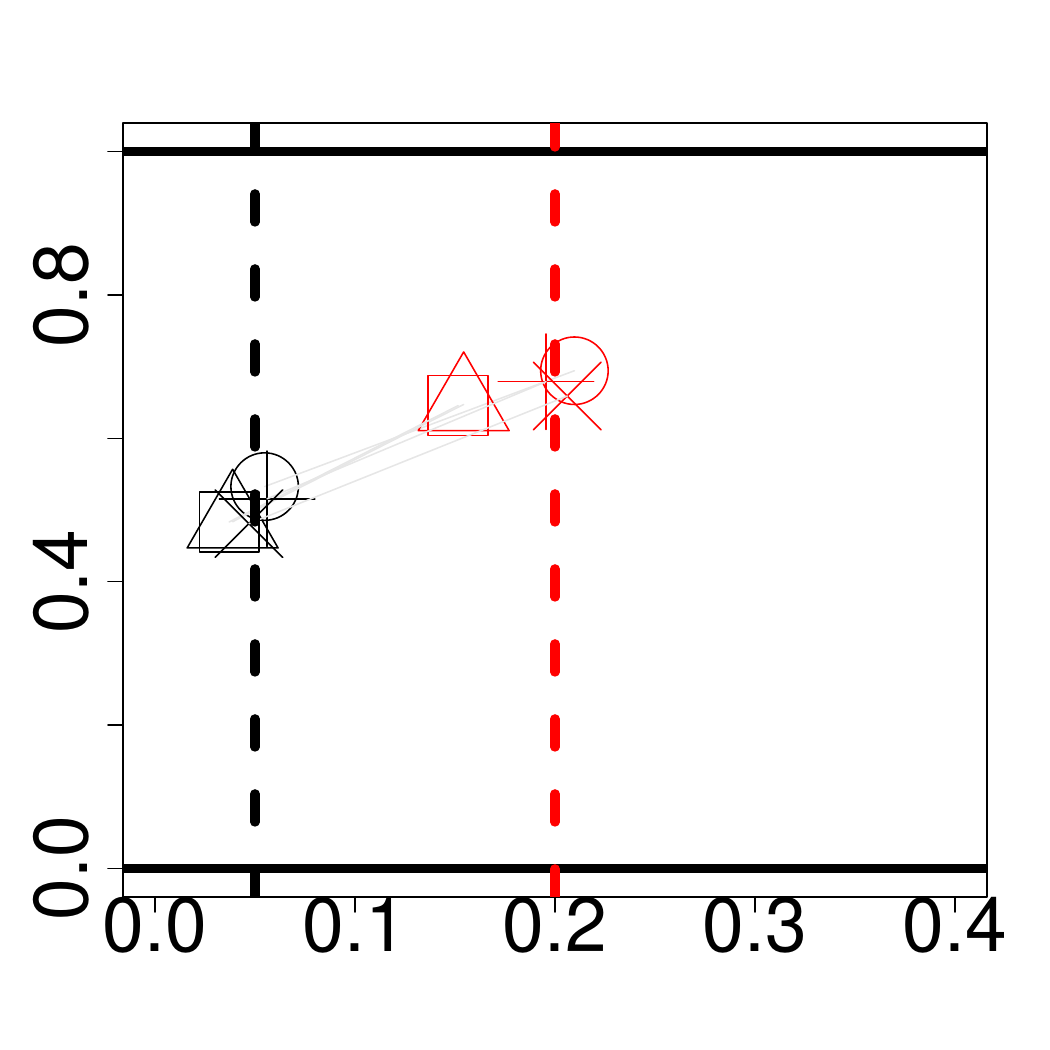}&
\hspace{-6mm}
\includegraphics[scale=0.17]{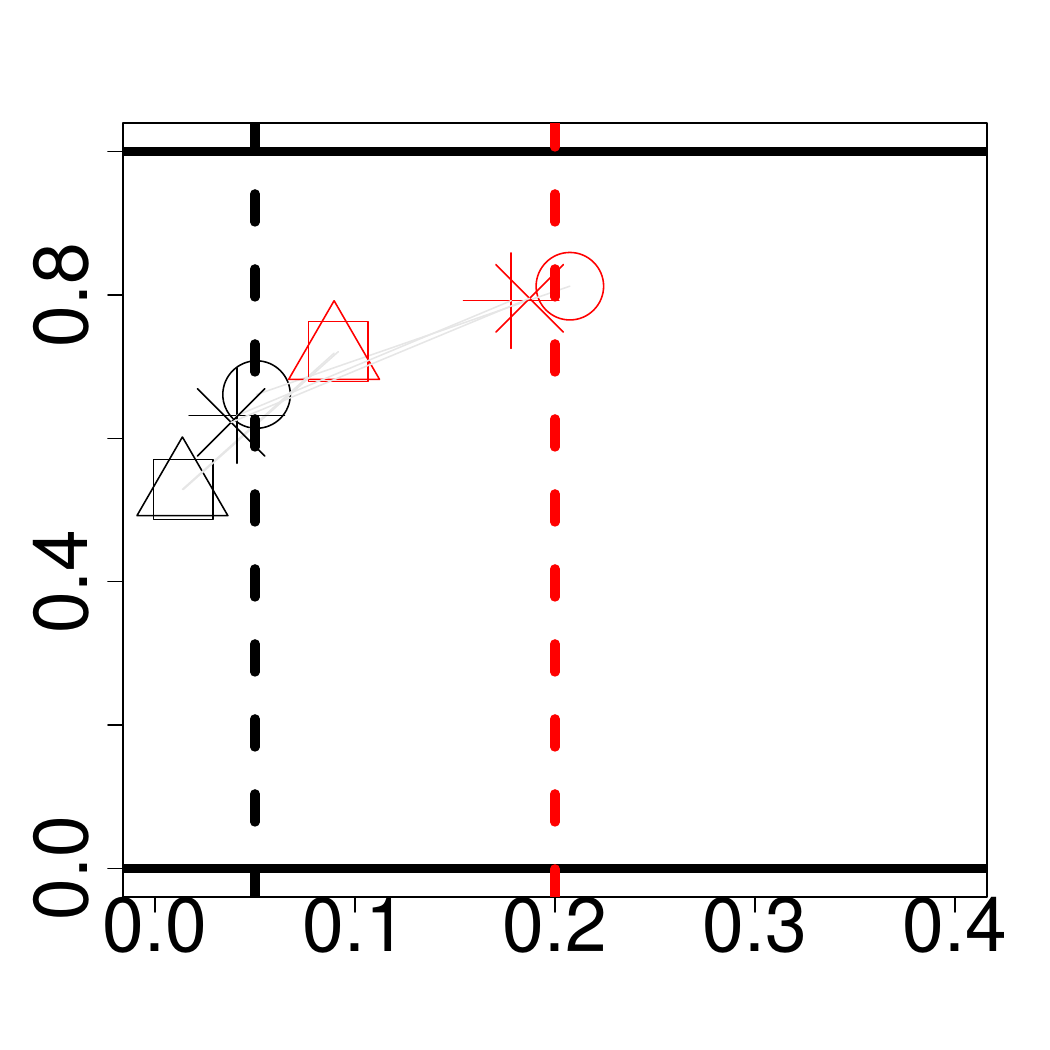}&
\hspace{-6mm}
\includegraphics[scale=0.17]{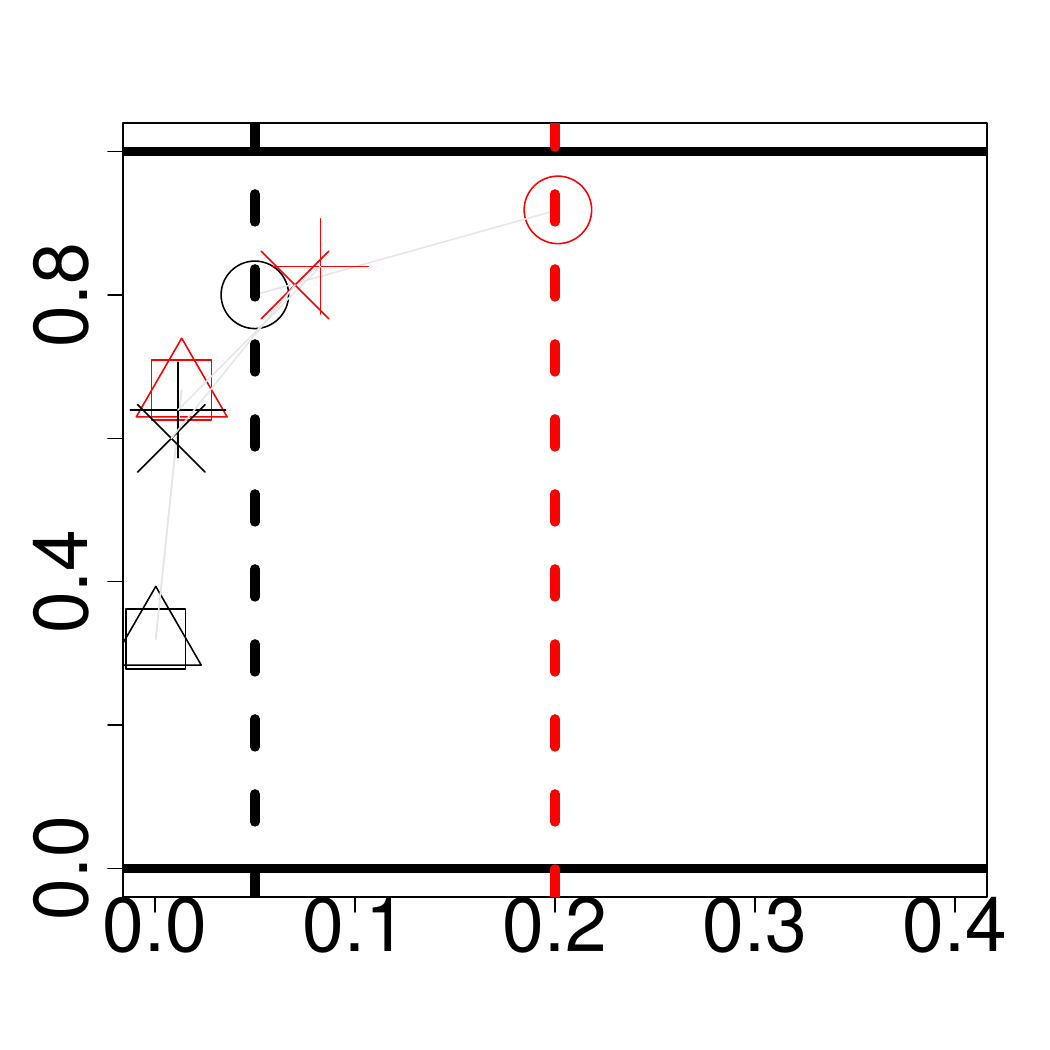}&
\hspace{-6mm}
\includegraphics[scale=0.17]{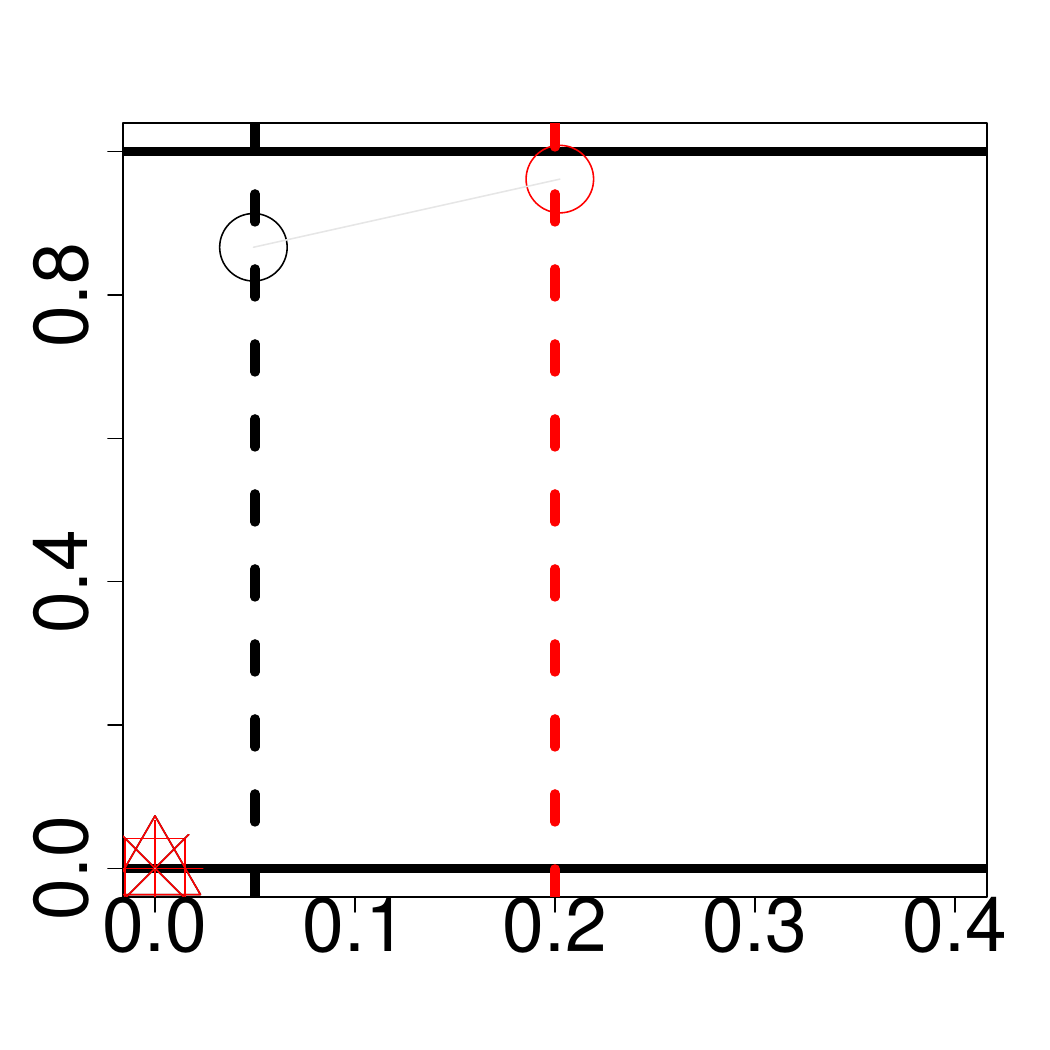}\\
\includegraphics[scale=0.17]{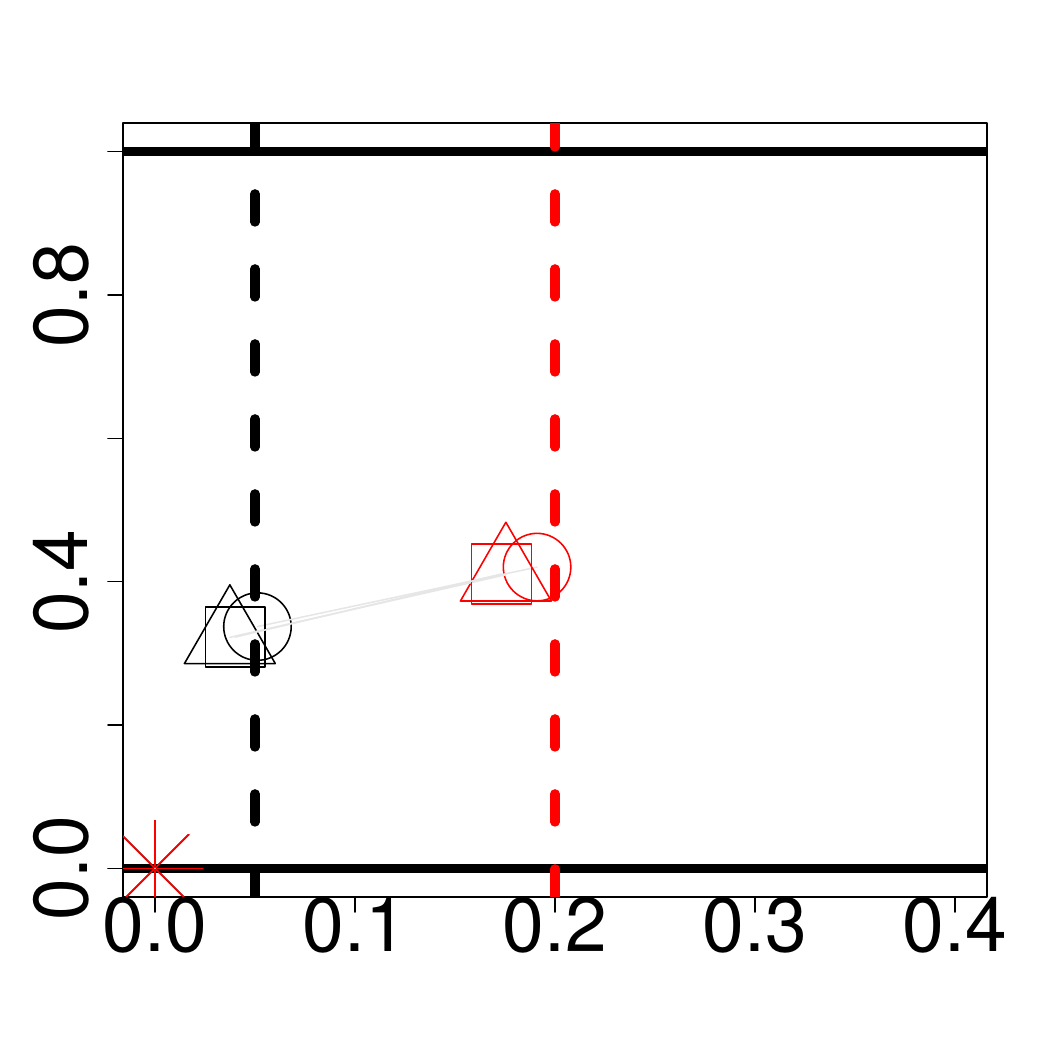}&
\hspace{-6mm}
\includegraphics[scale=0.17]{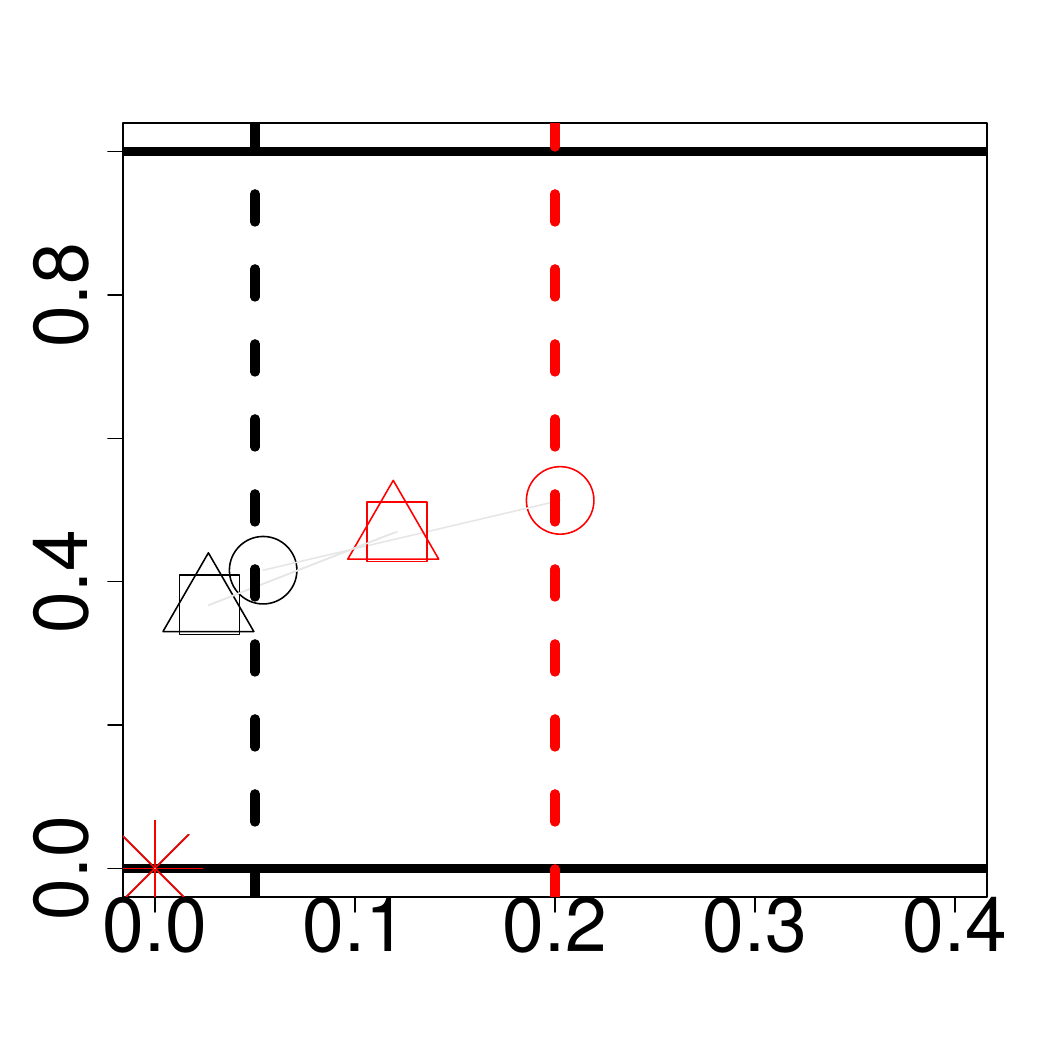}&
\hspace{-6mm}
\includegraphics[scale=0.17]{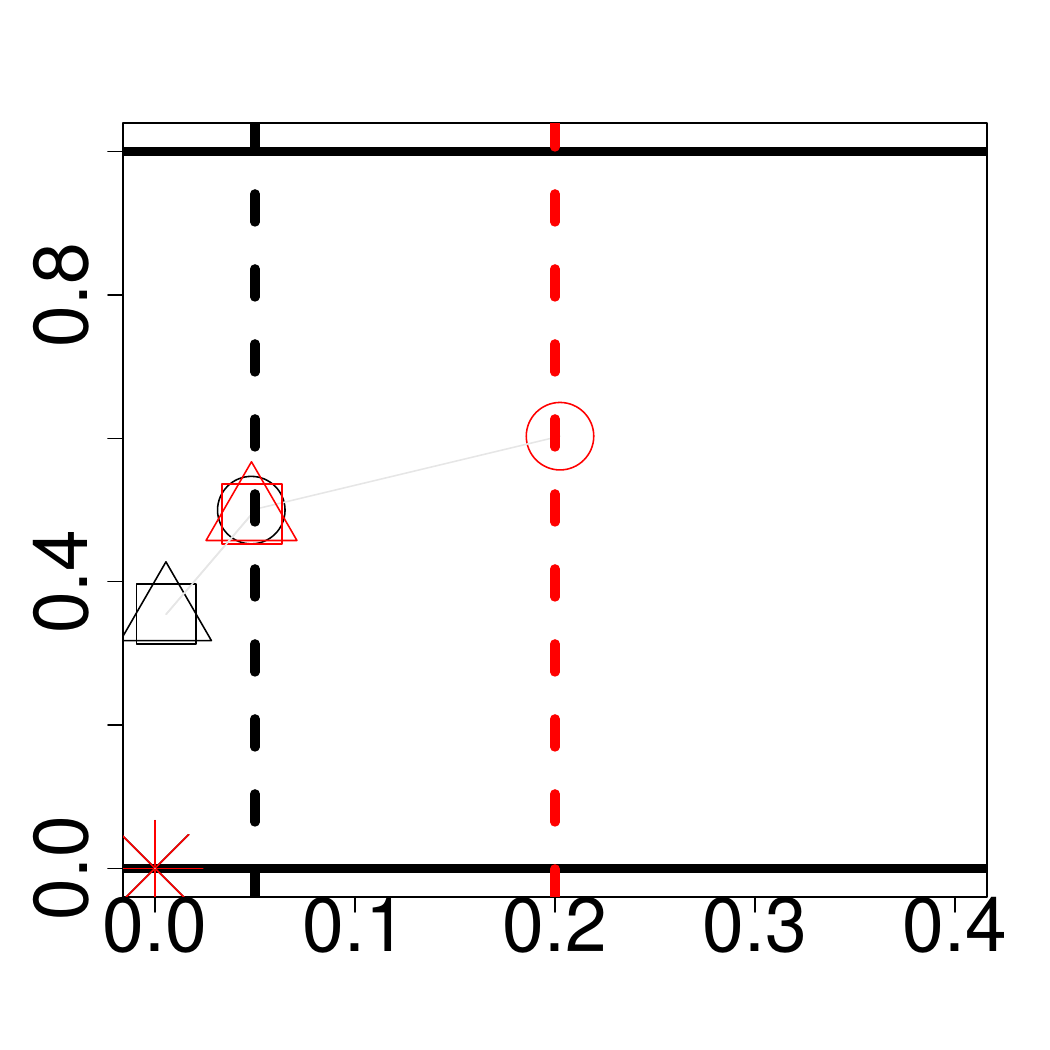}&
\hspace{-6mm}
\includegraphics[scale=0.17]{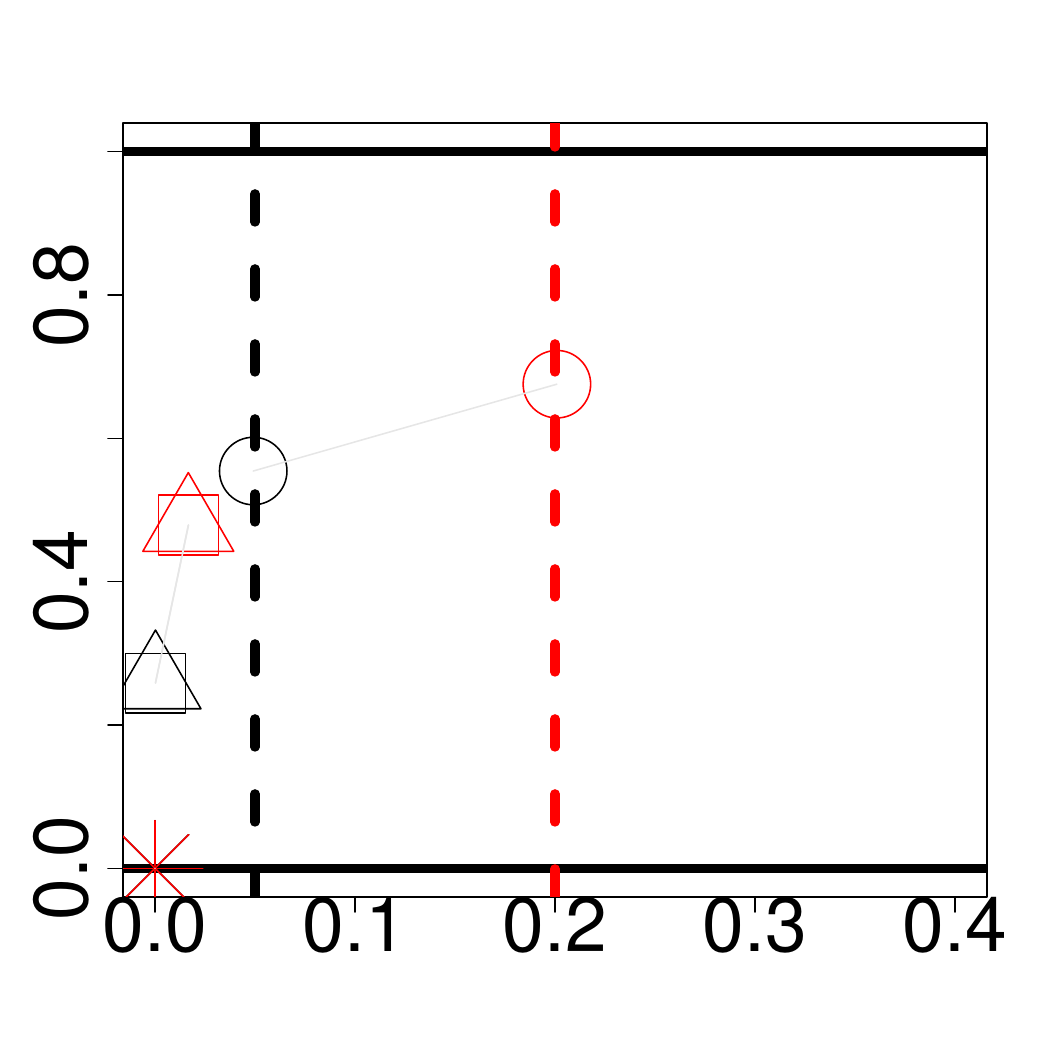}\\
\includegraphics[scale=0.17]{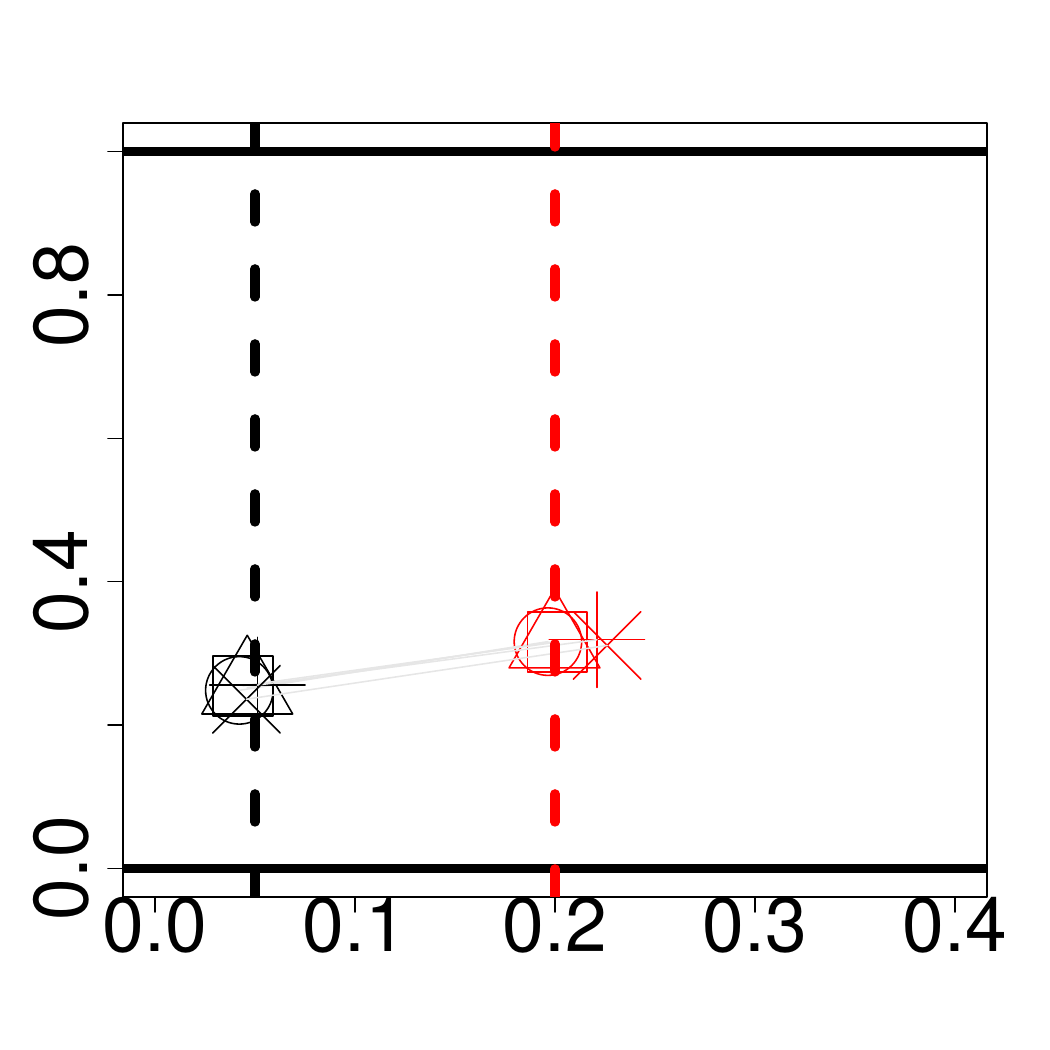}&
\hspace{-6mm}
\includegraphics[scale=0.17]{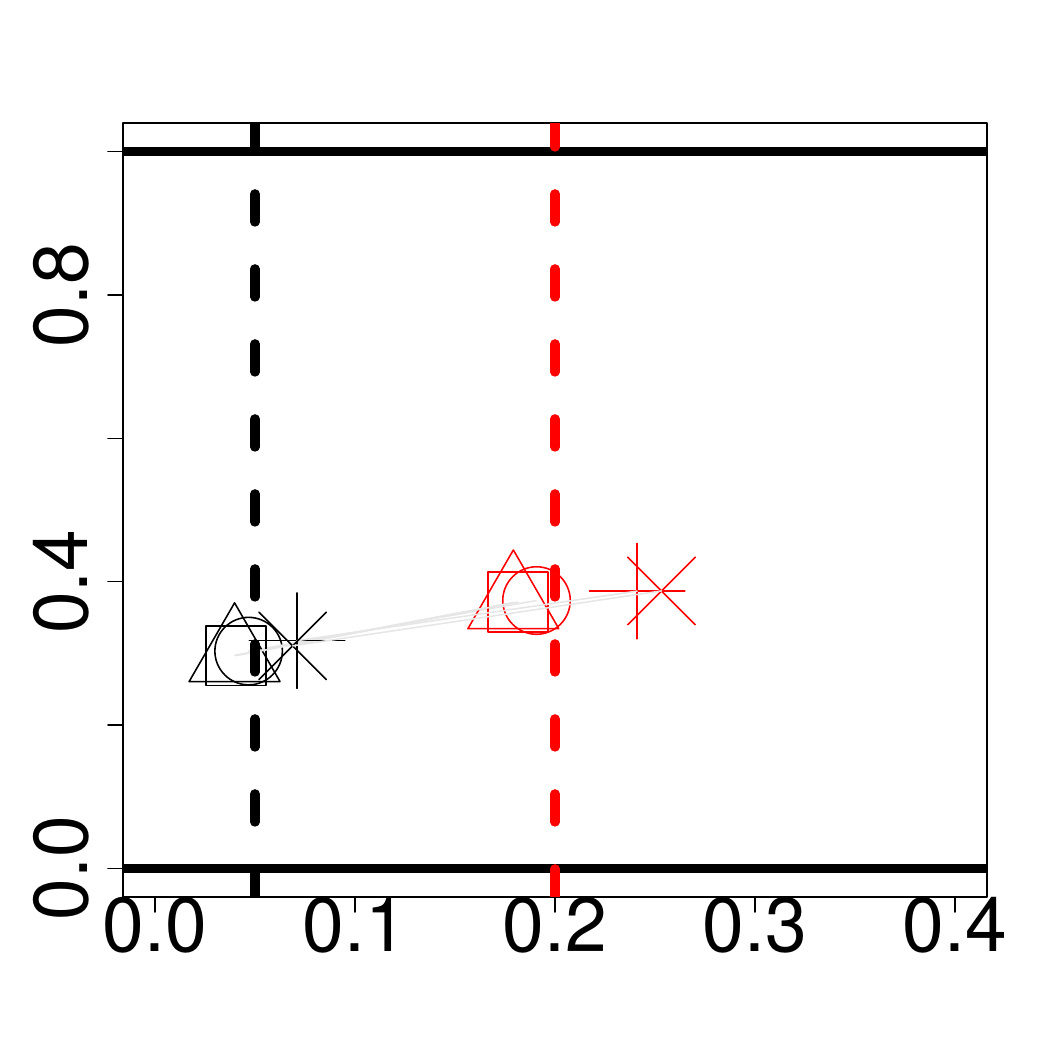}&
\hspace{-6mm}
\includegraphics[scale=0.17]{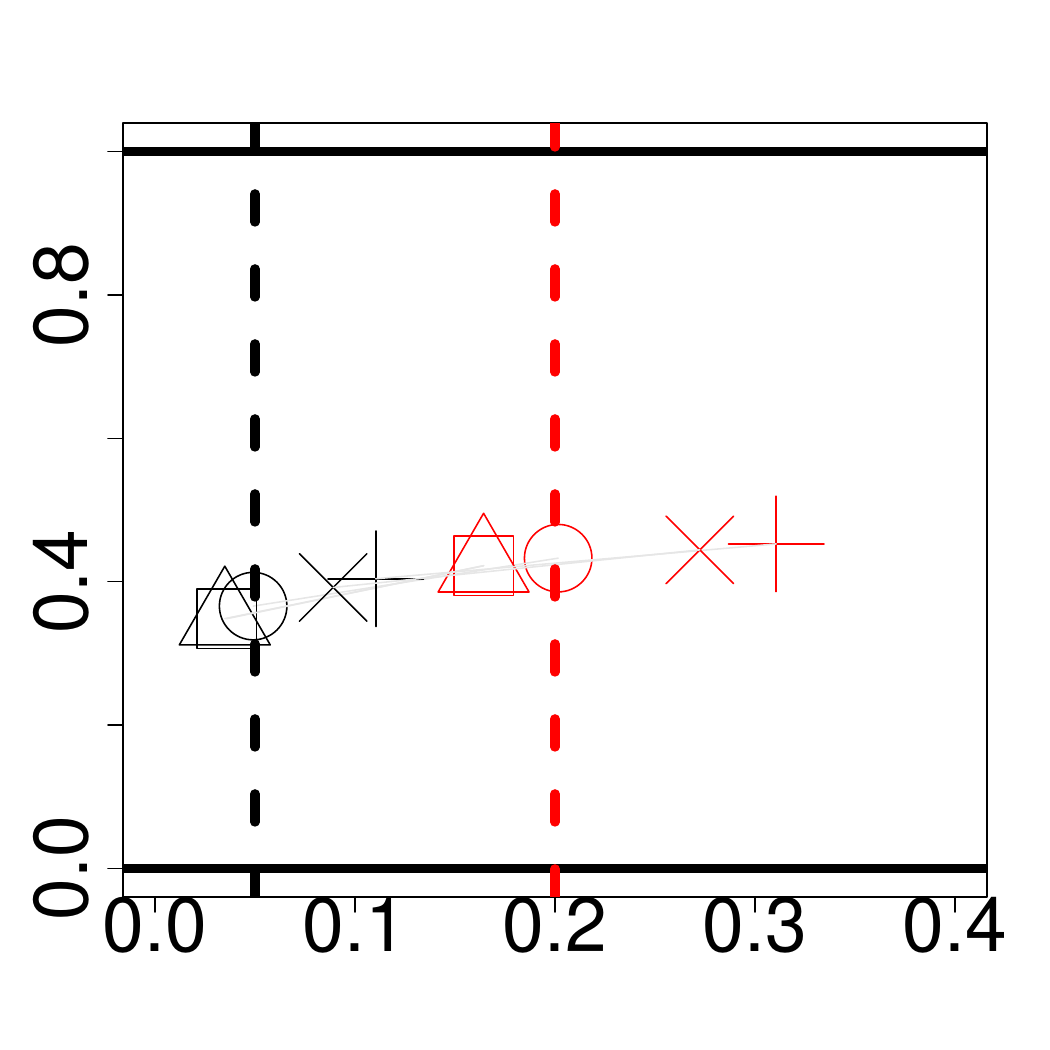}&
\hspace{-6mm}
\includegraphics[scale=0.17]{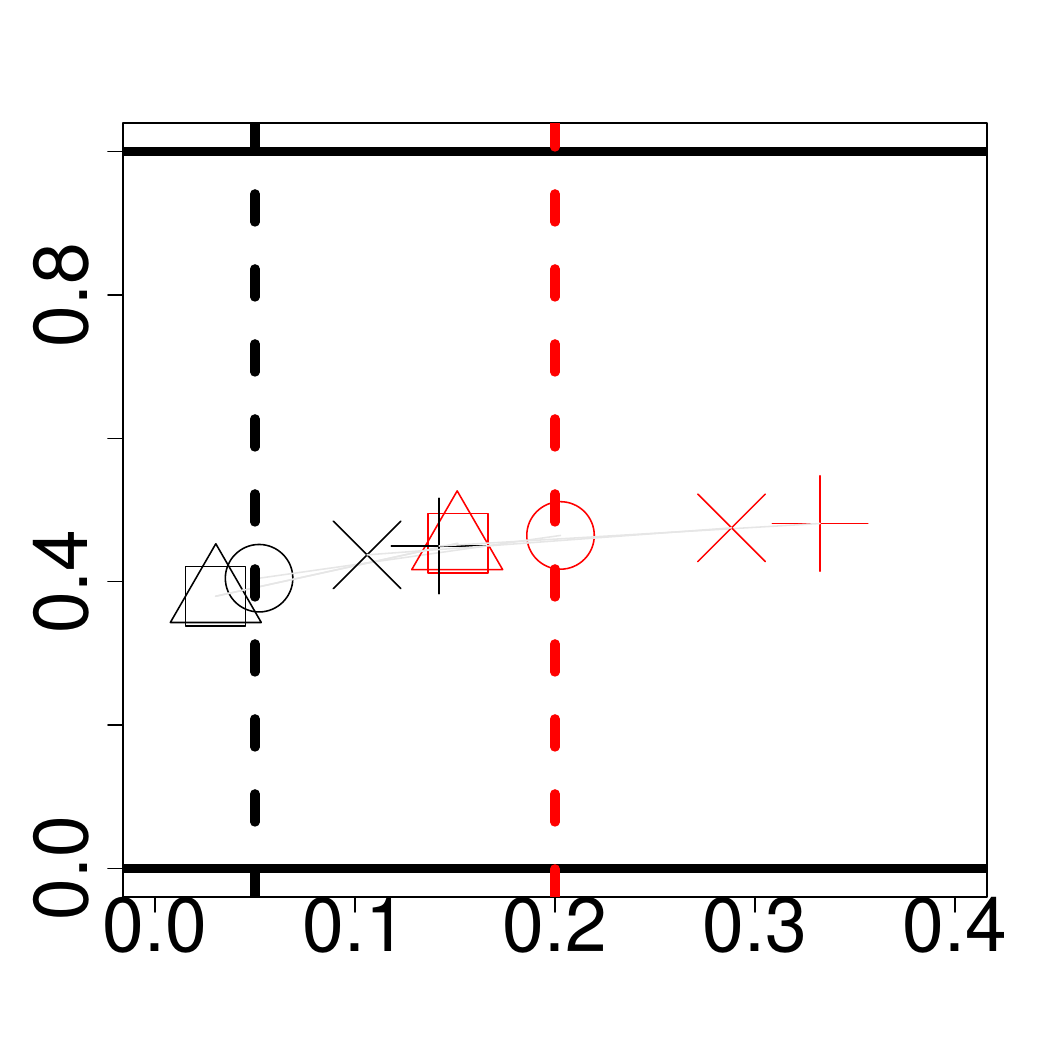}\\
\end{tabular}
\caption{$(\FDR,\TDR)$-plot in the independent case, for the $5$ different methods and $\alpha=0.05$ (black) or $\alpha=0.2$ (red).  Alternatives: standard (top), Cauchy (middle) or zero-located (bottom). The procedures are [Oracle] ($\circ$), [MedianMAD] ($\triangle$), [TrimMAD] ($\square$), [MLE] (+), LocfdrSC ($\times$).\label{fig:indep}}
\end{figure}

\begin{figure}[h!]
\begin{tabular}{cccc}
$k=30$ & $k=100$ & $k=300$ & $k=500$\\
\includegraphics[scale=0.17]{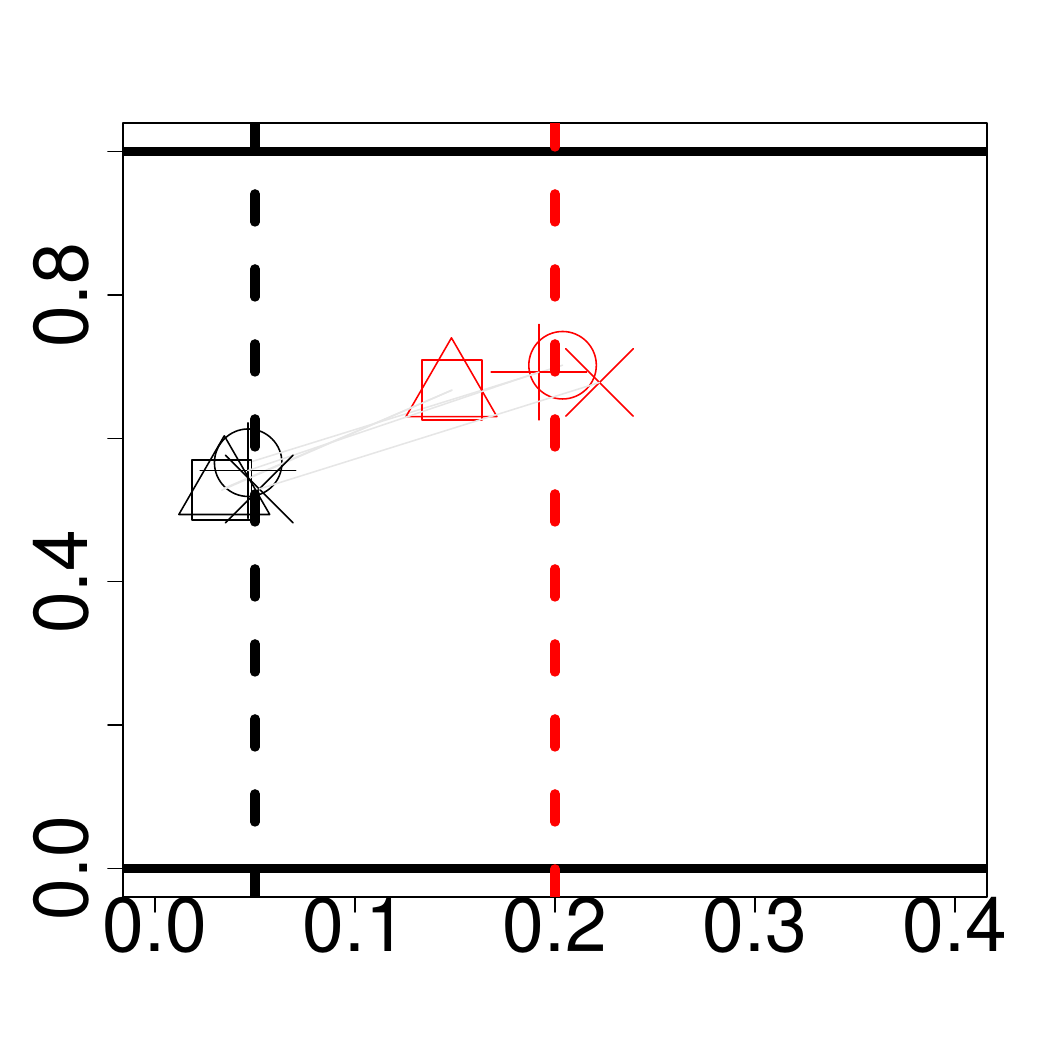}&
\hspace{-6mm}
\includegraphics[scale=0.17]{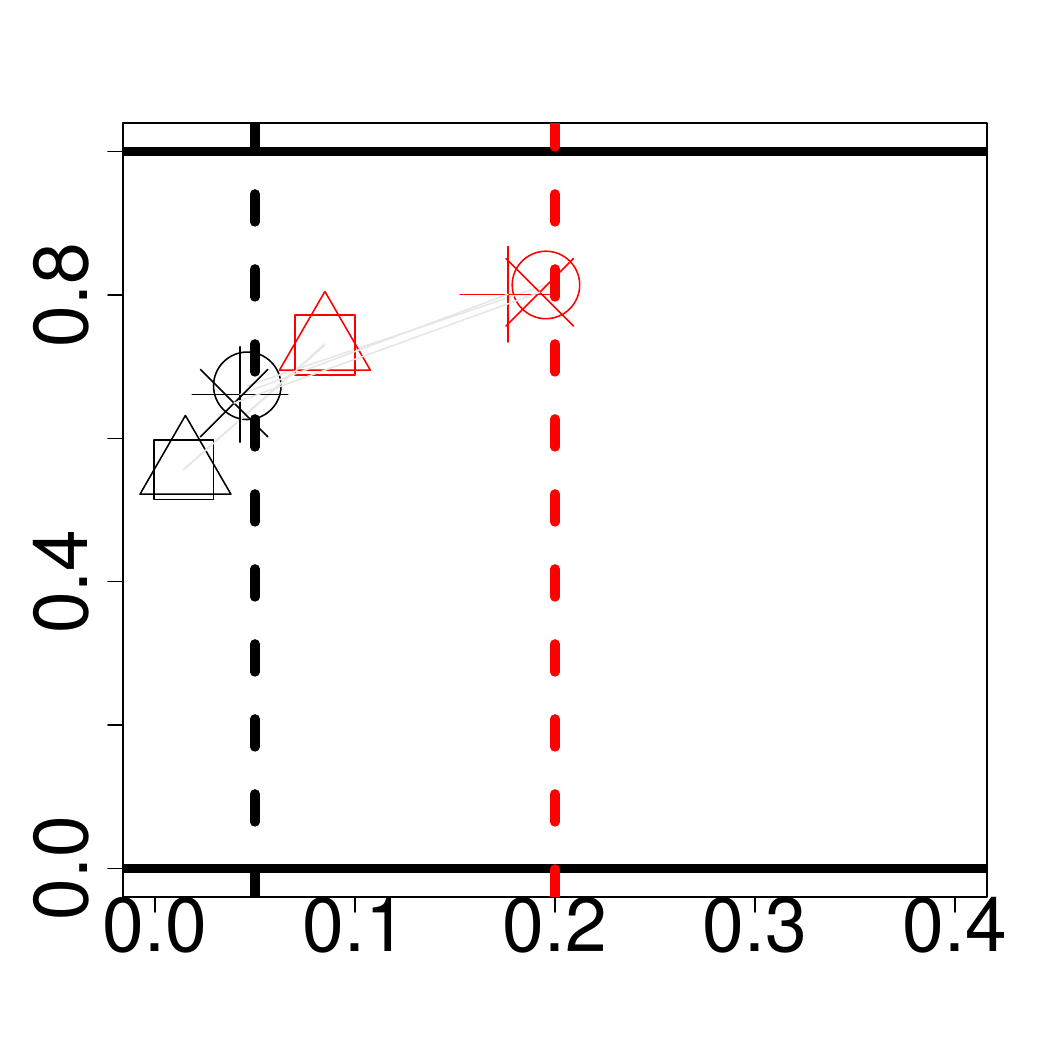}&
\hspace{-6mm}
\includegraphics[scale=0.17]{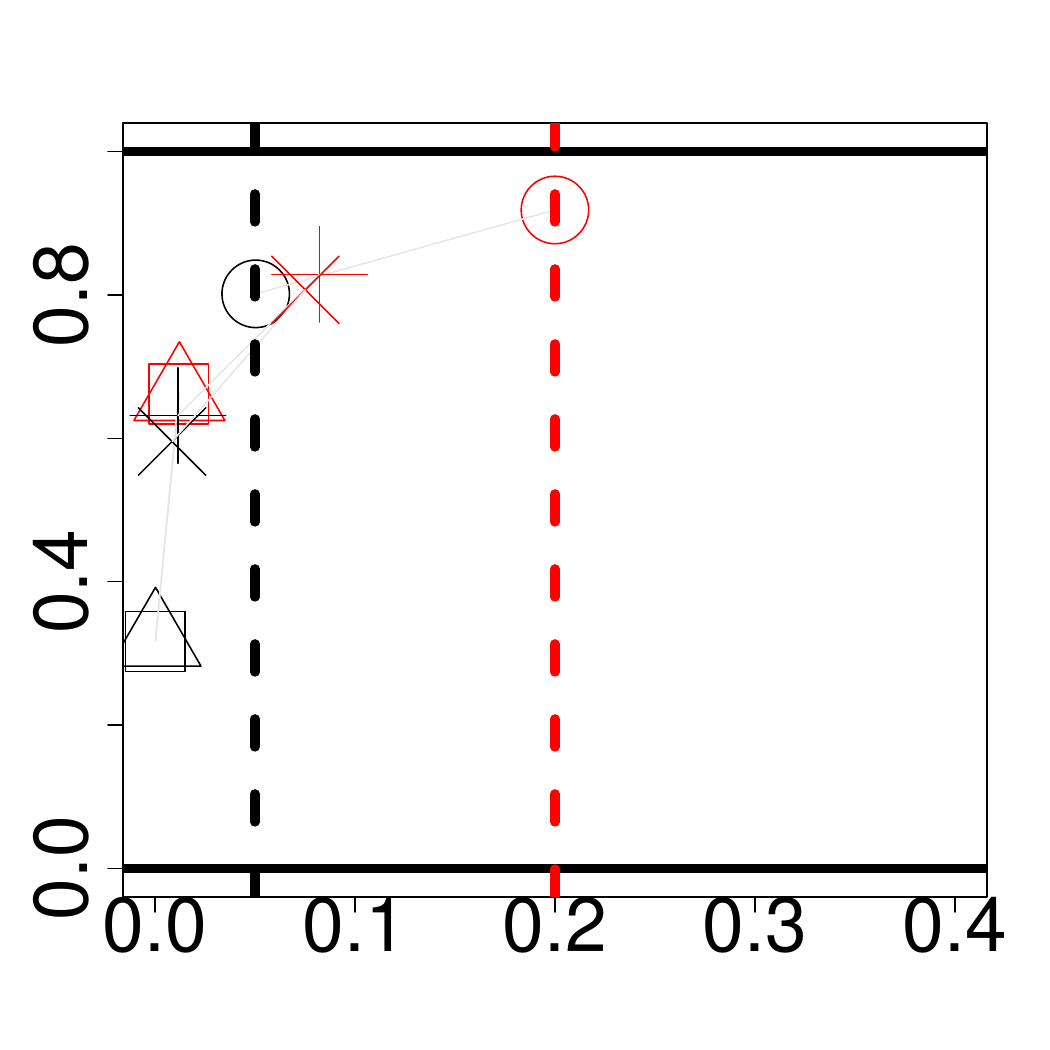}&
\hspace{-6mm}
\includegraphics[scale=0.17]{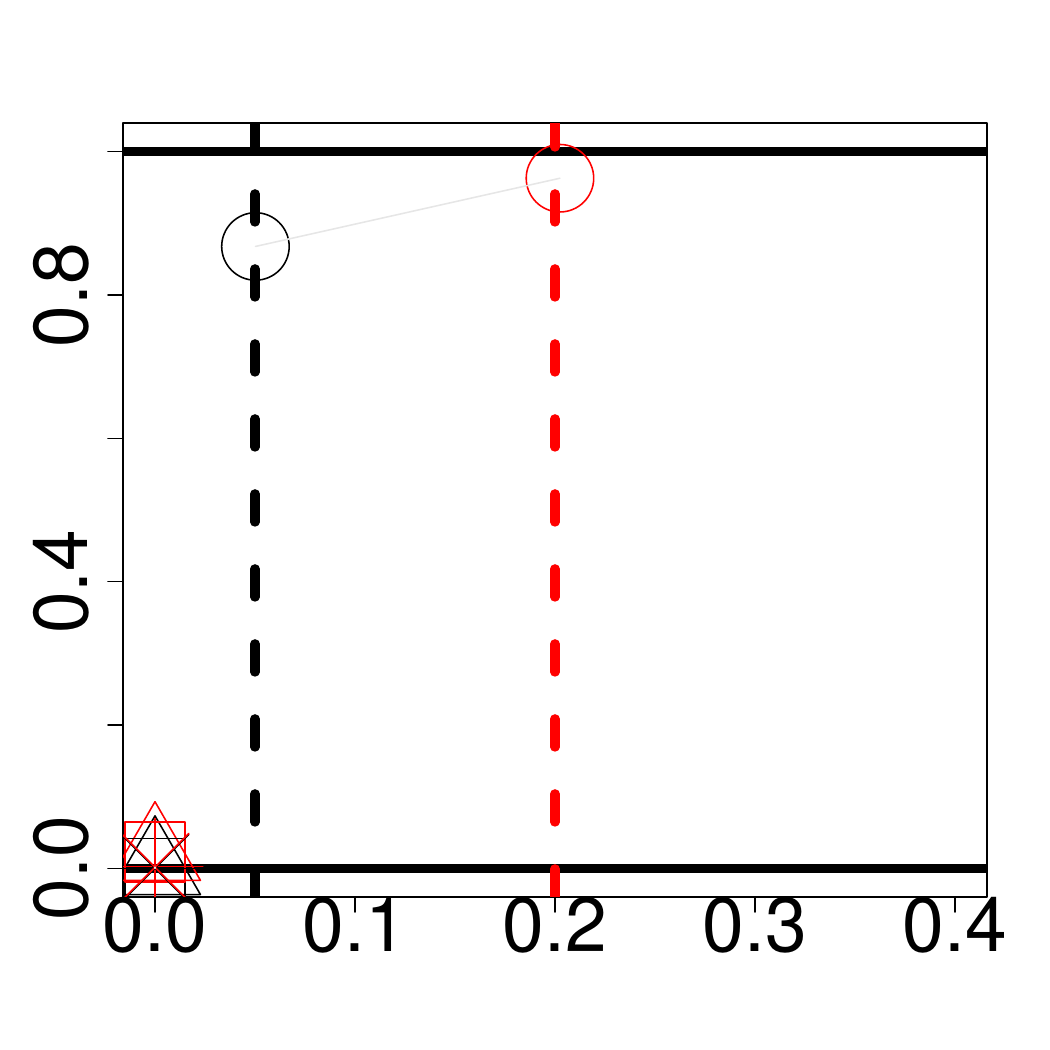}\\
\includegraphics[scale=0.17]{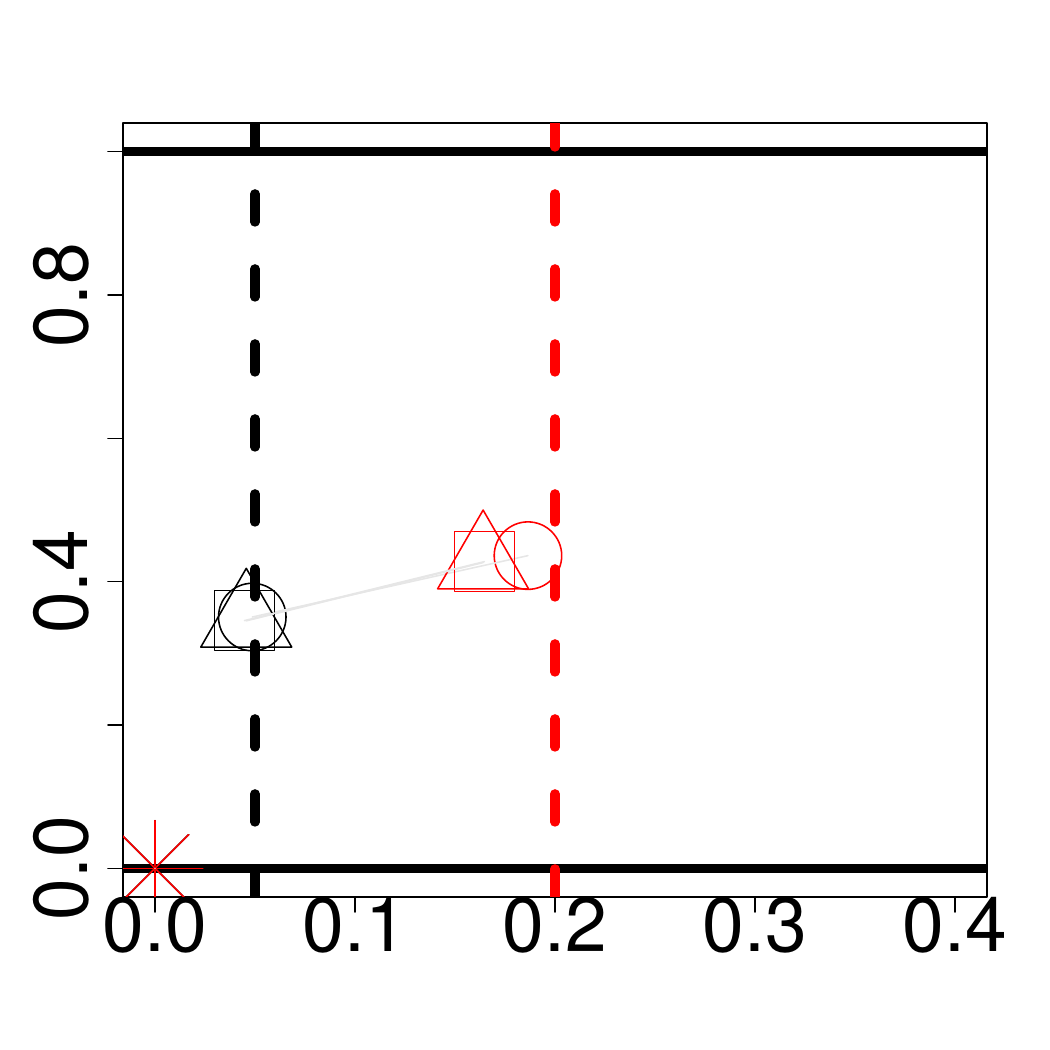}&
\hspace{-6mm}
\includegraphics[scale=0.17]{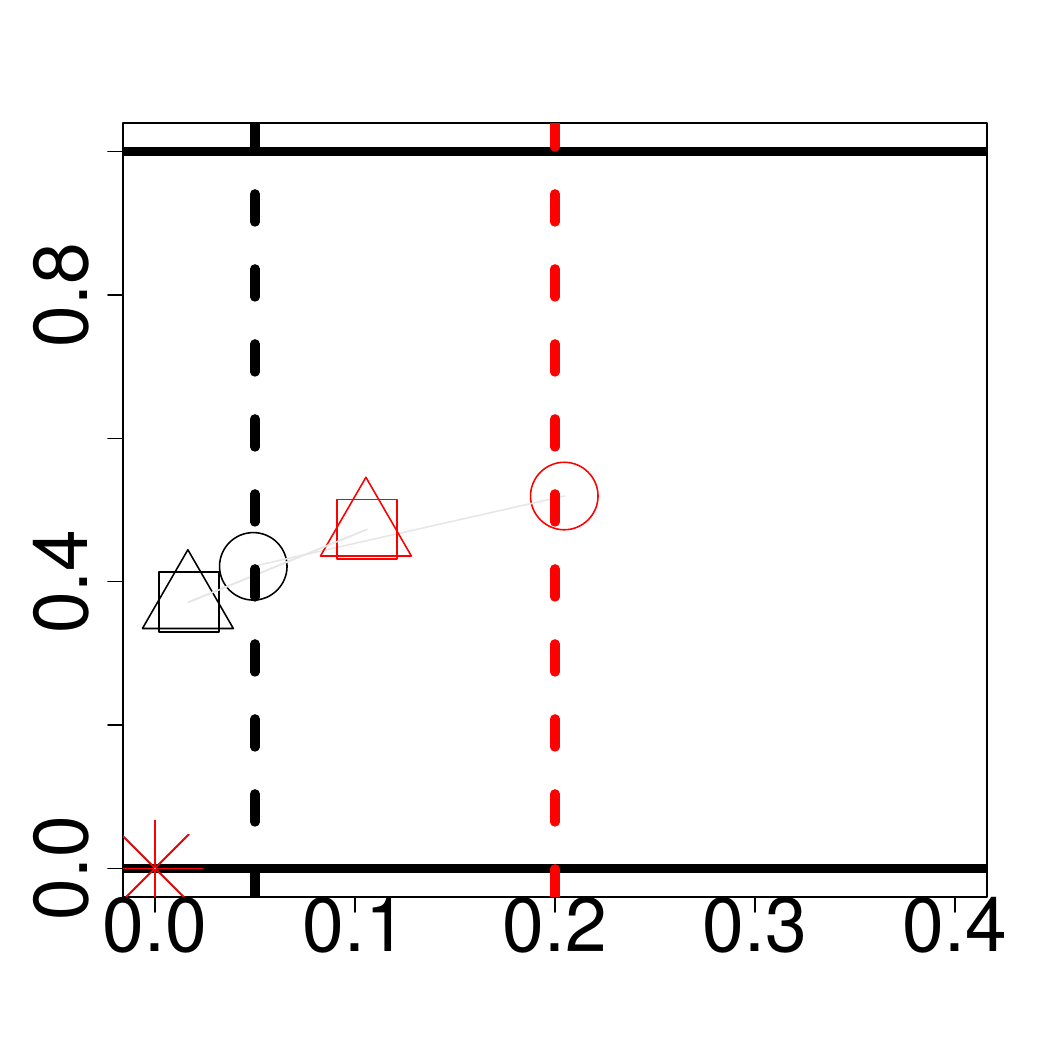}&
\hspace{-6mm}
\includegraphics[scale=0.17]{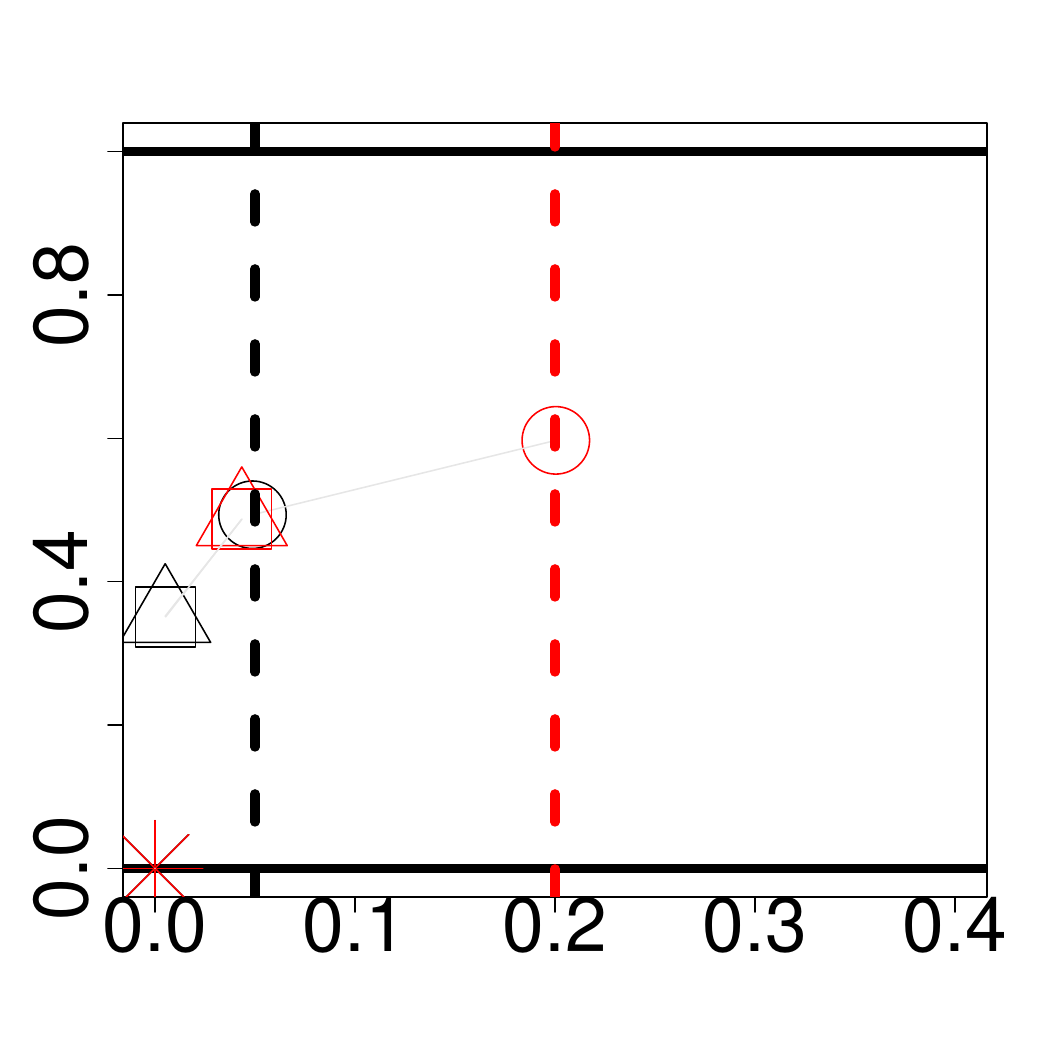}&
\hspace{-6mm}
\includegraphics[scale=0.17]{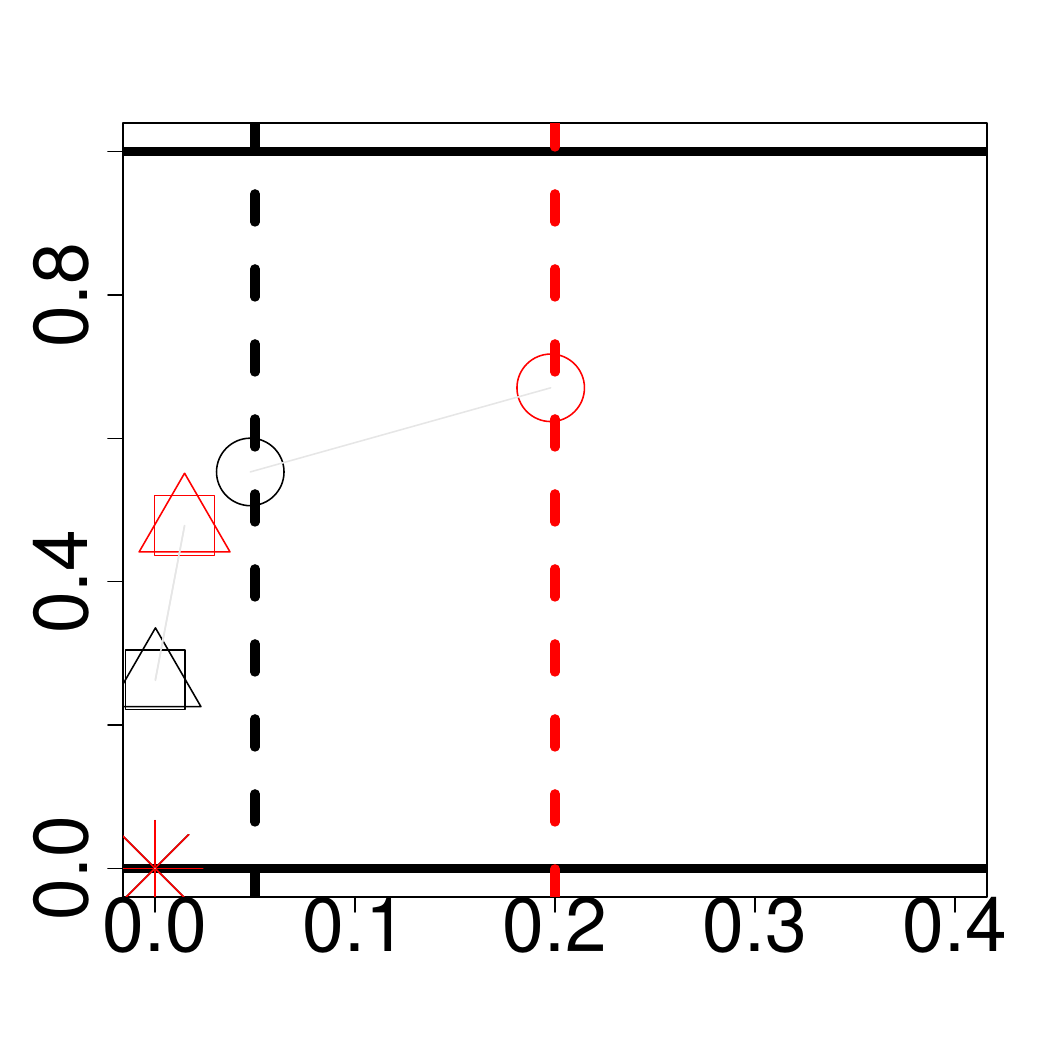}\\
\includegraphics[scale=0.17]{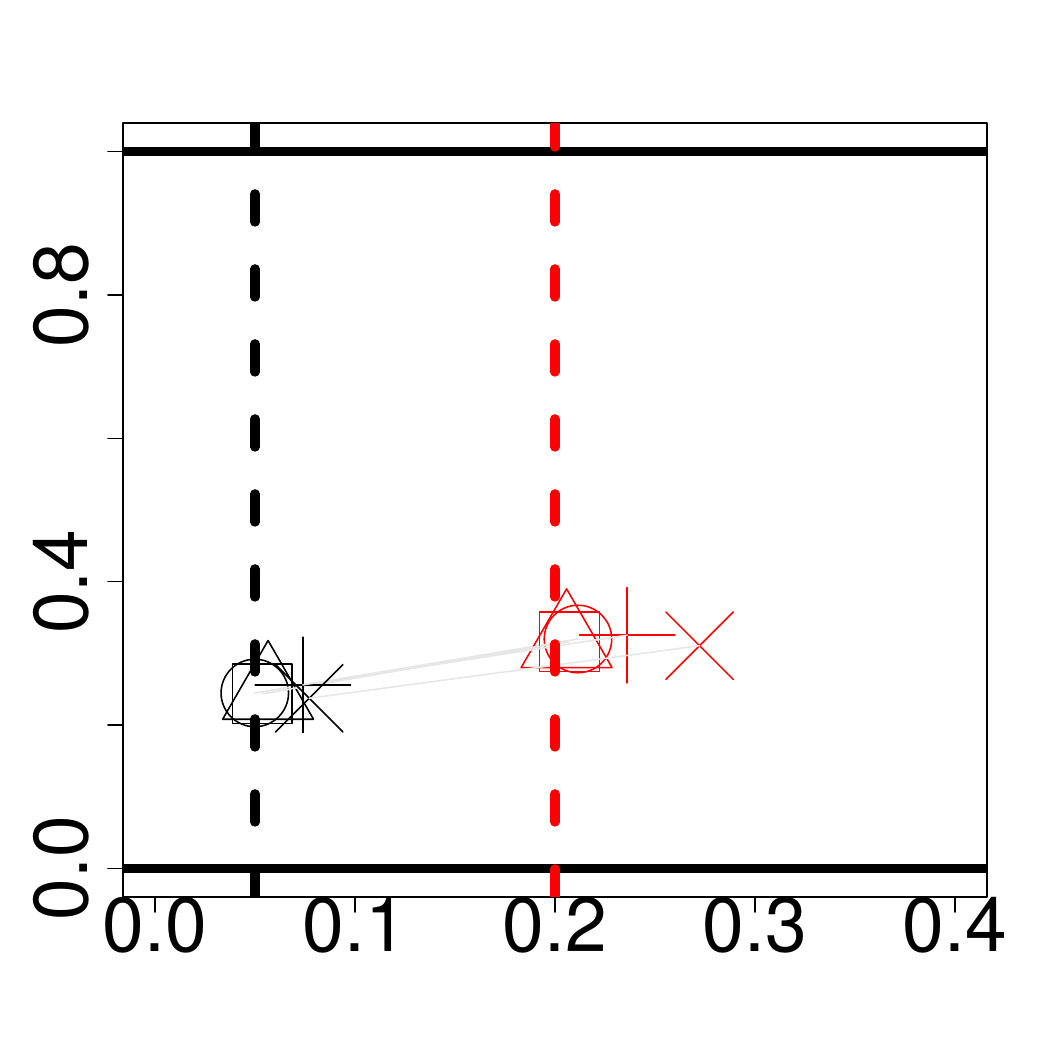}&
\hspace{-6mm}
\includegraphics[scale=0.17]{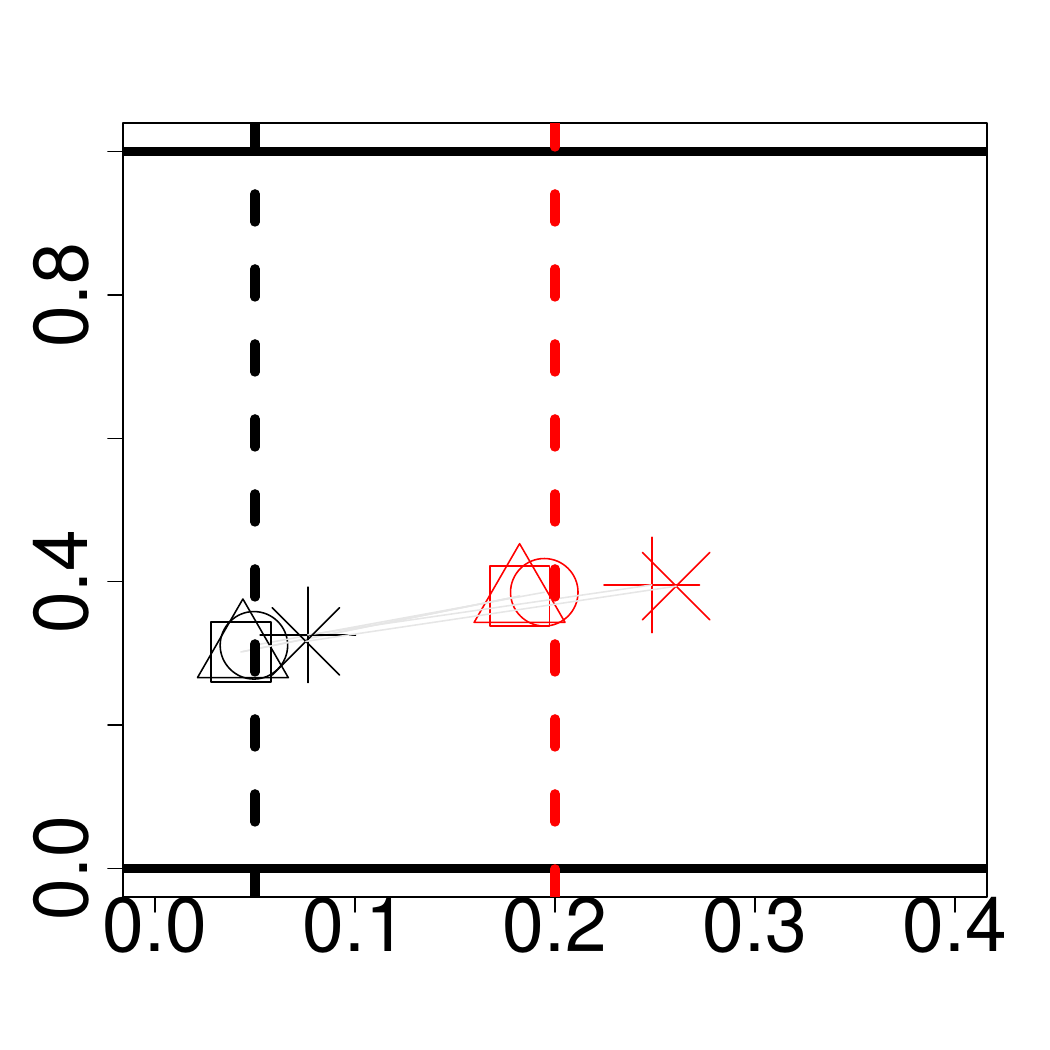}&
\hspace{-6mm}
\includegraphics[scale=0.17]{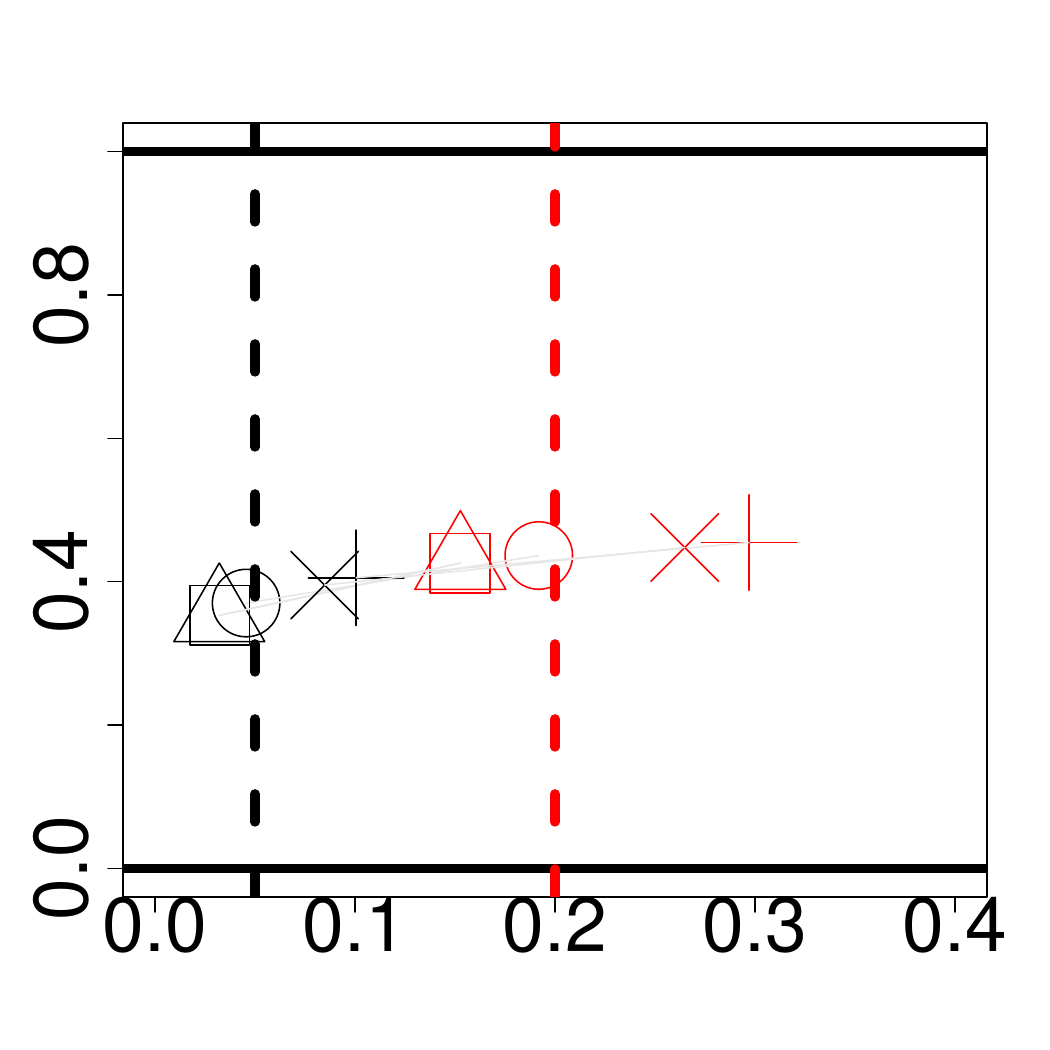}&
\hspace{-6mm}
\includegraphics[scale=0.17]{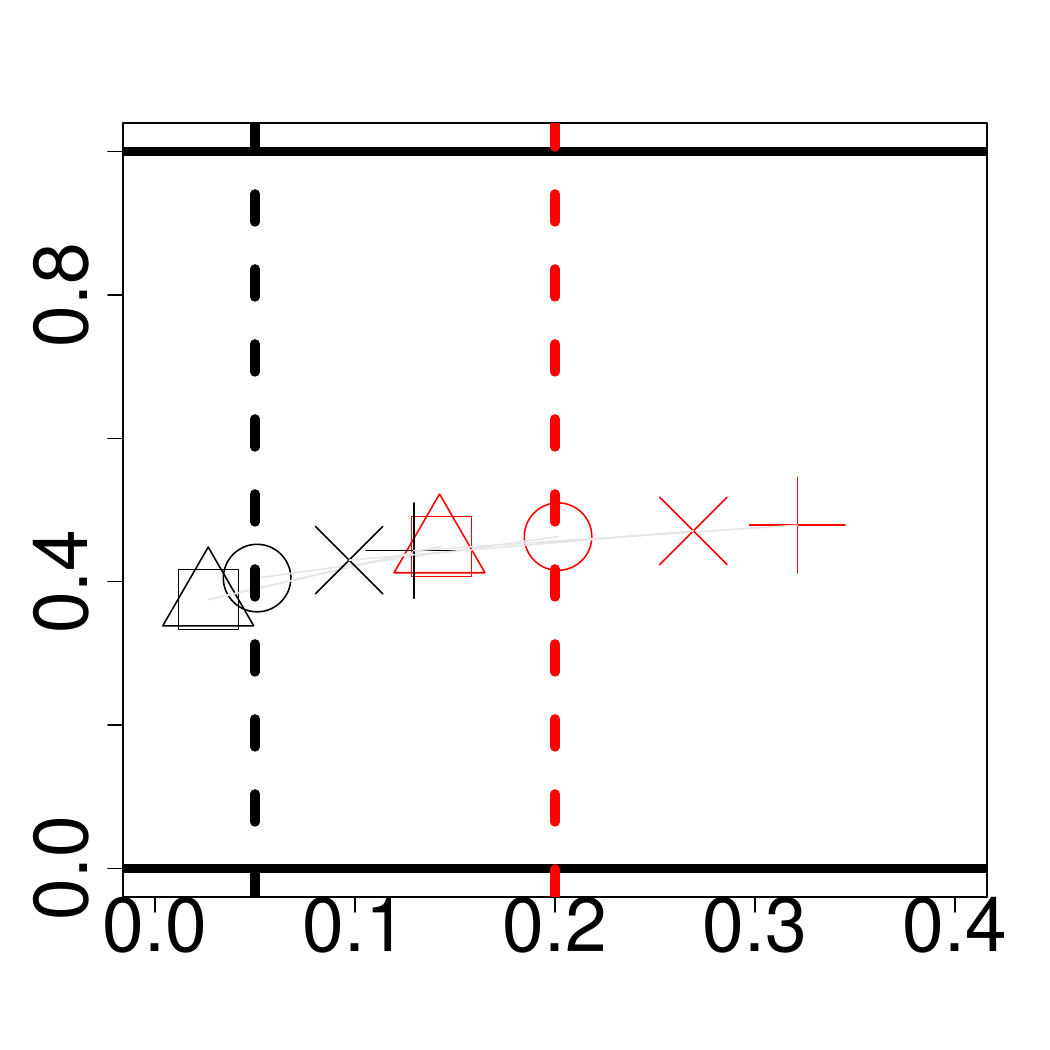}\\
\end{tabular}
\caption{Same as Figure~\ref{fig:indep} in the block-correlation case.\label{fig:block}}
\end{figure}

\begin{figure}[h!]
\begin{tabular}{cccc}
 $k=30$ & $k=100$ & $k=300$ & $k=500$\\
\includegraphics[scale=0.17]{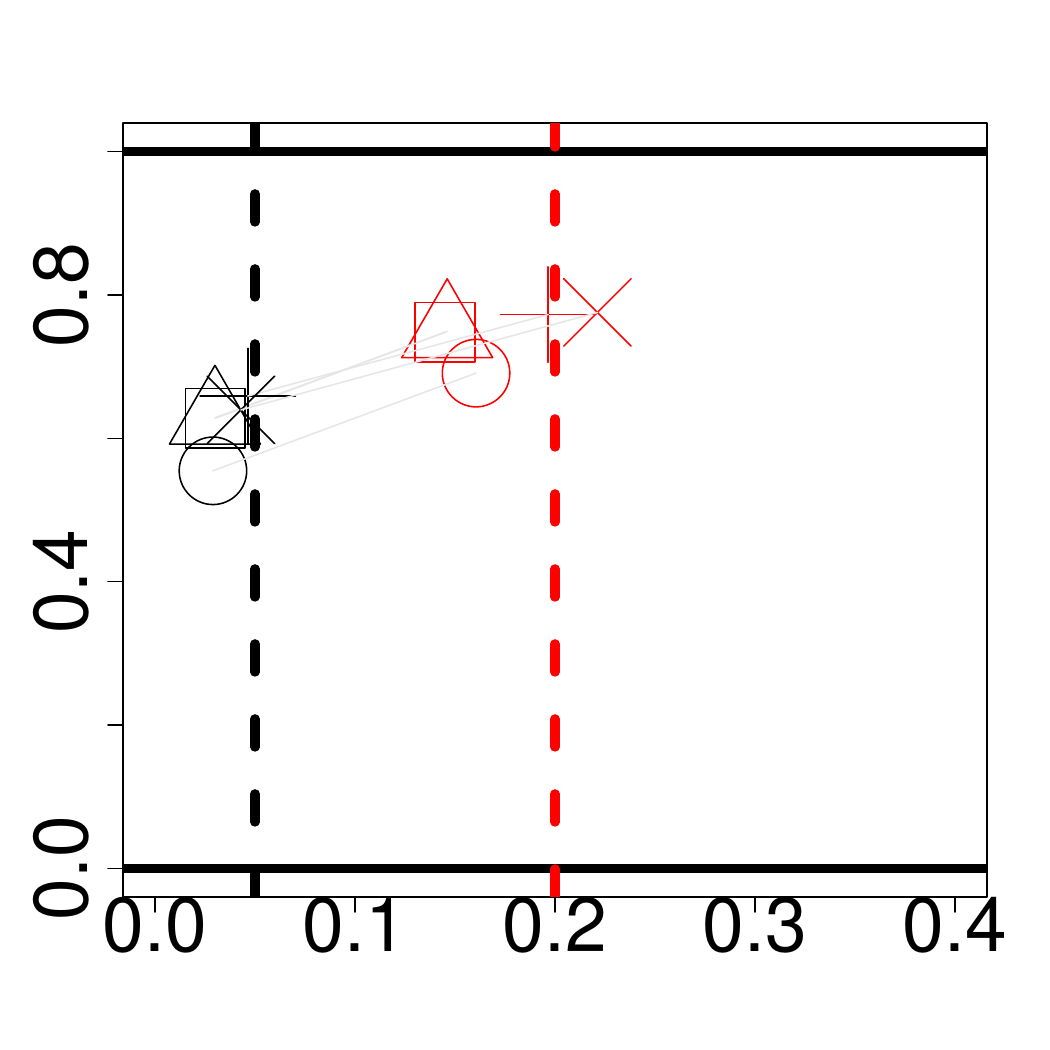}&
\hspace{-6mm}
\includegraphics[scale=0.17]{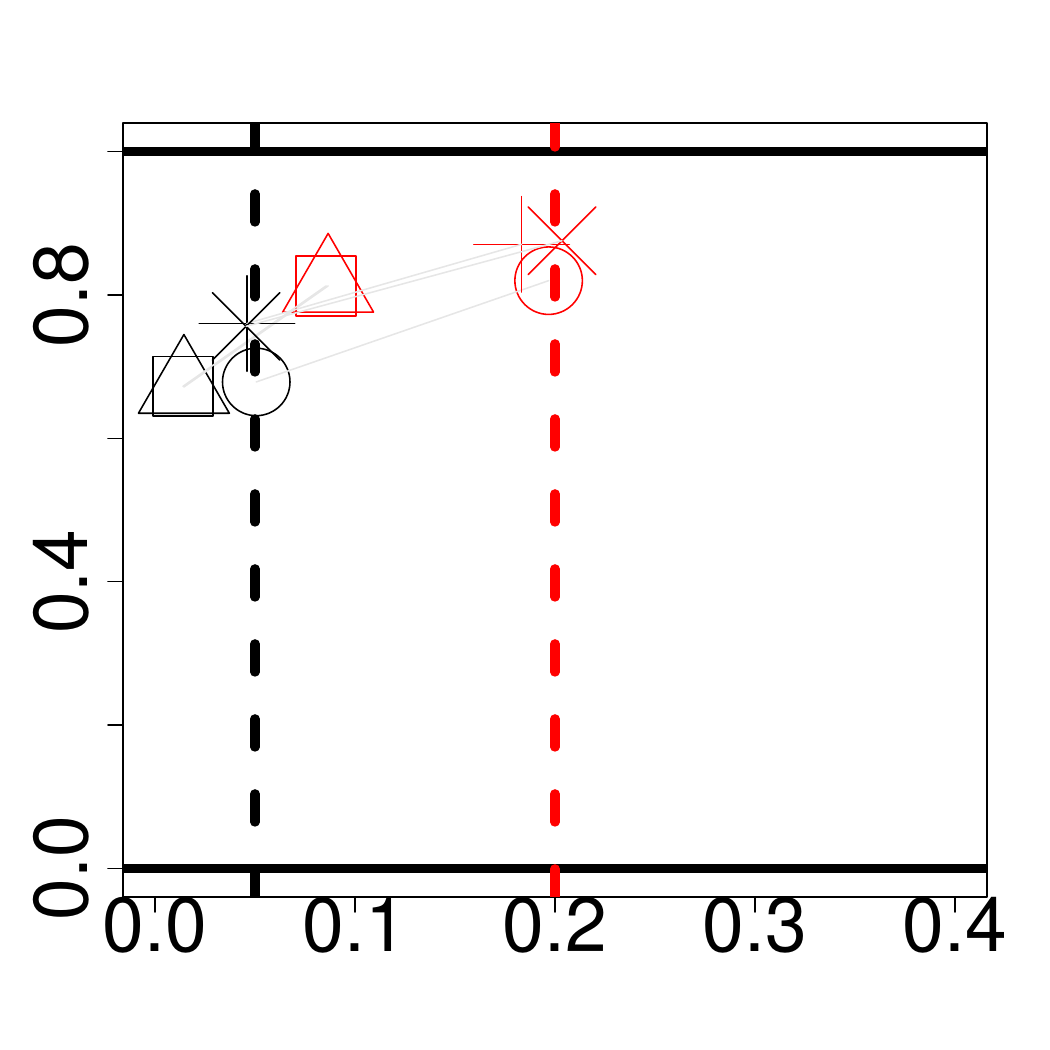}&
\hspace{-6mm}
\includegraphics[scale=0.17]{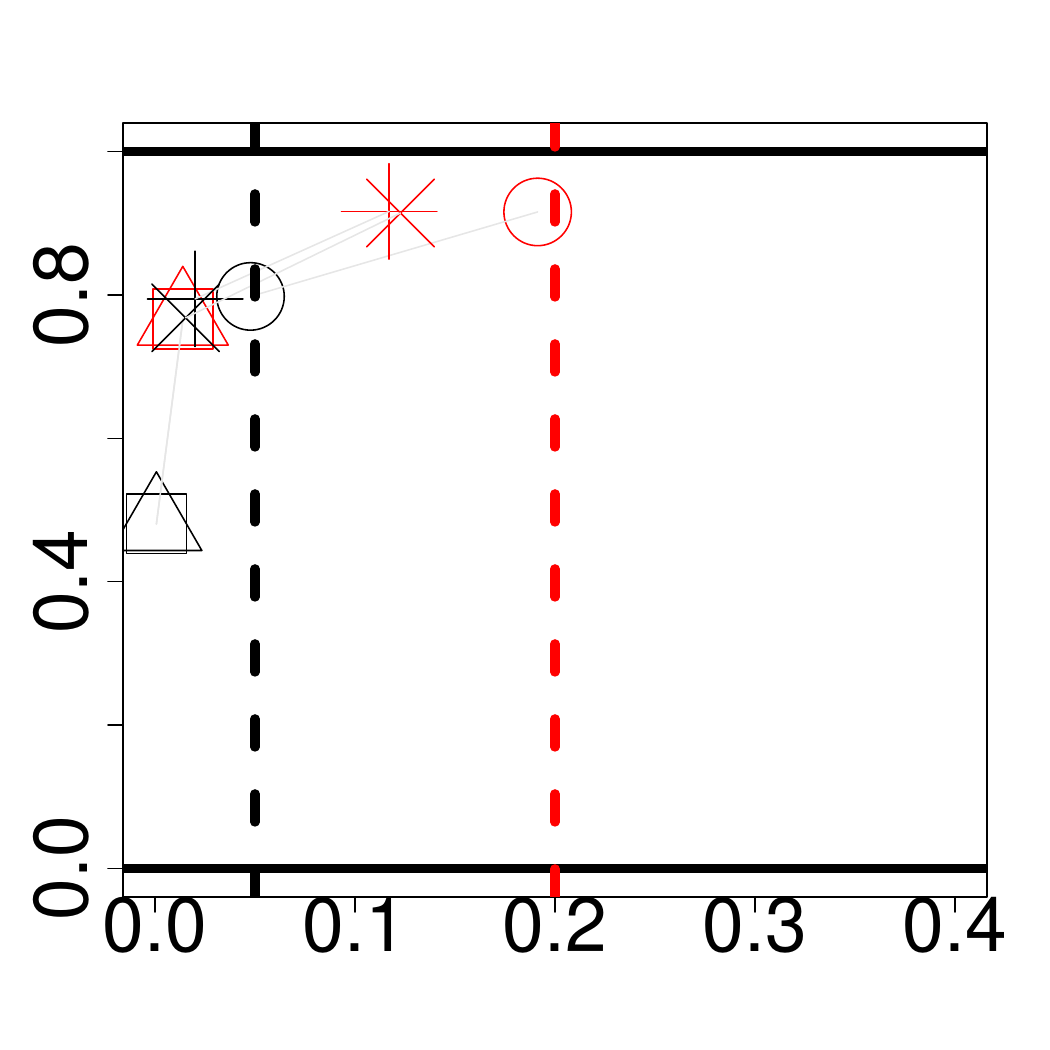}&
\hspace{-6mm}
\includegraphics[scale=0.17]{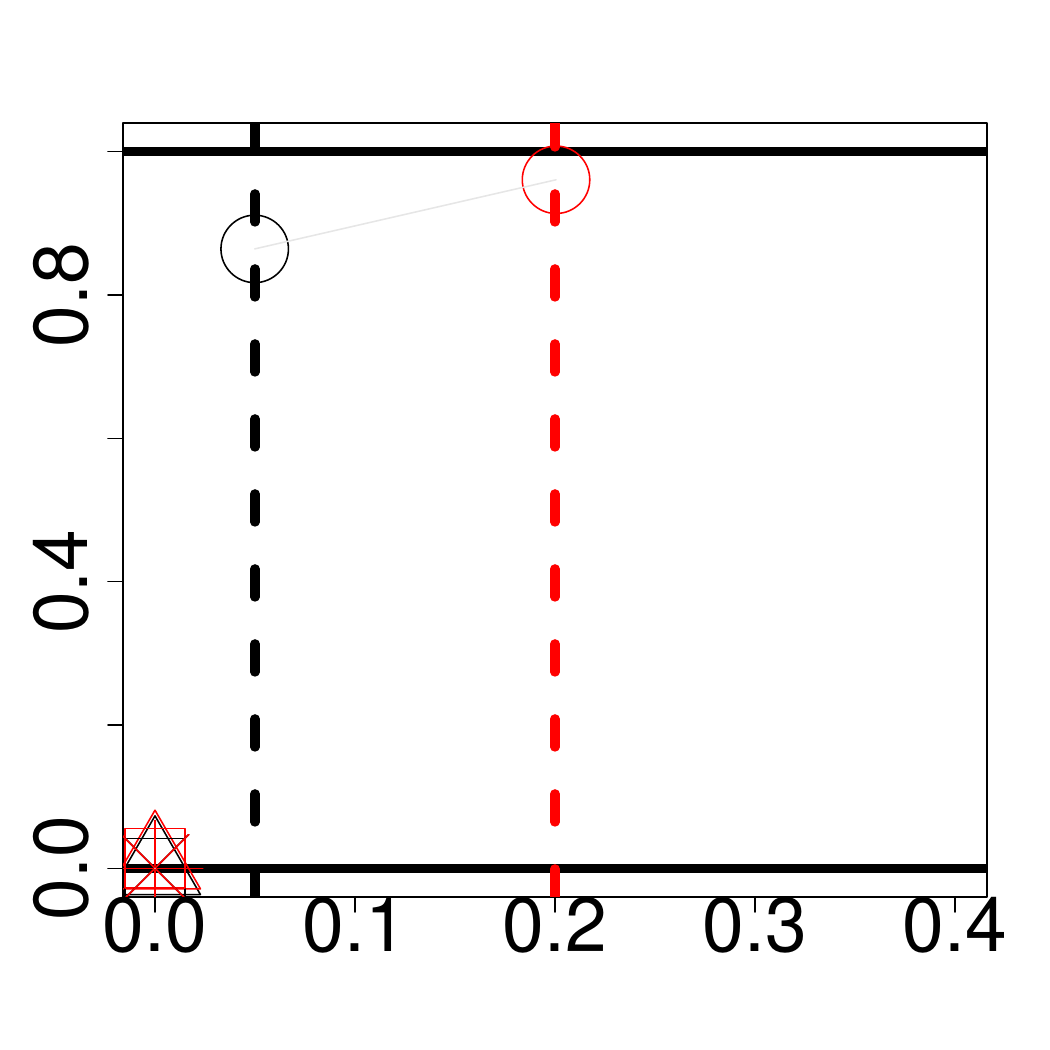}\\
\includegraphics[scale=0.17]{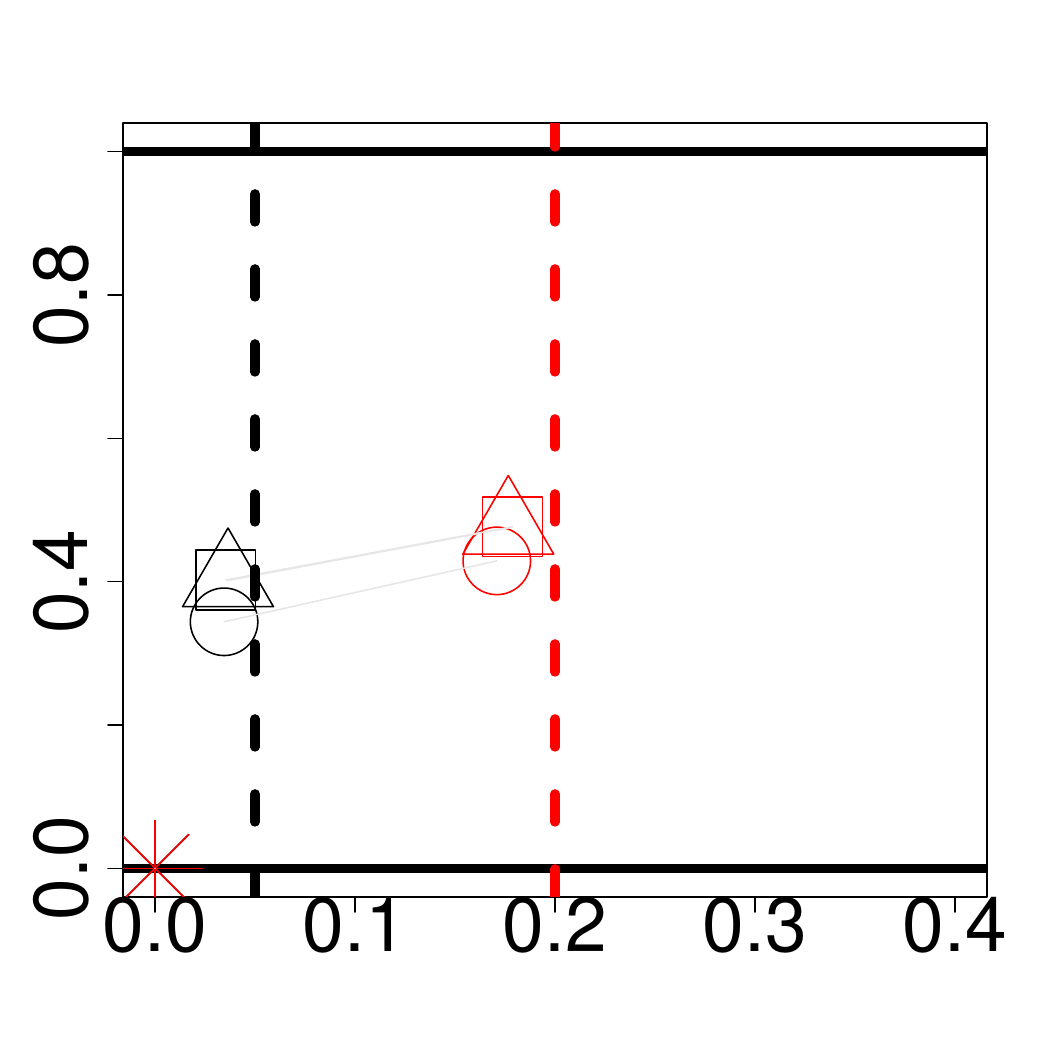}&
\hspace{-6mm}
\includegraphics[scale=0.17]{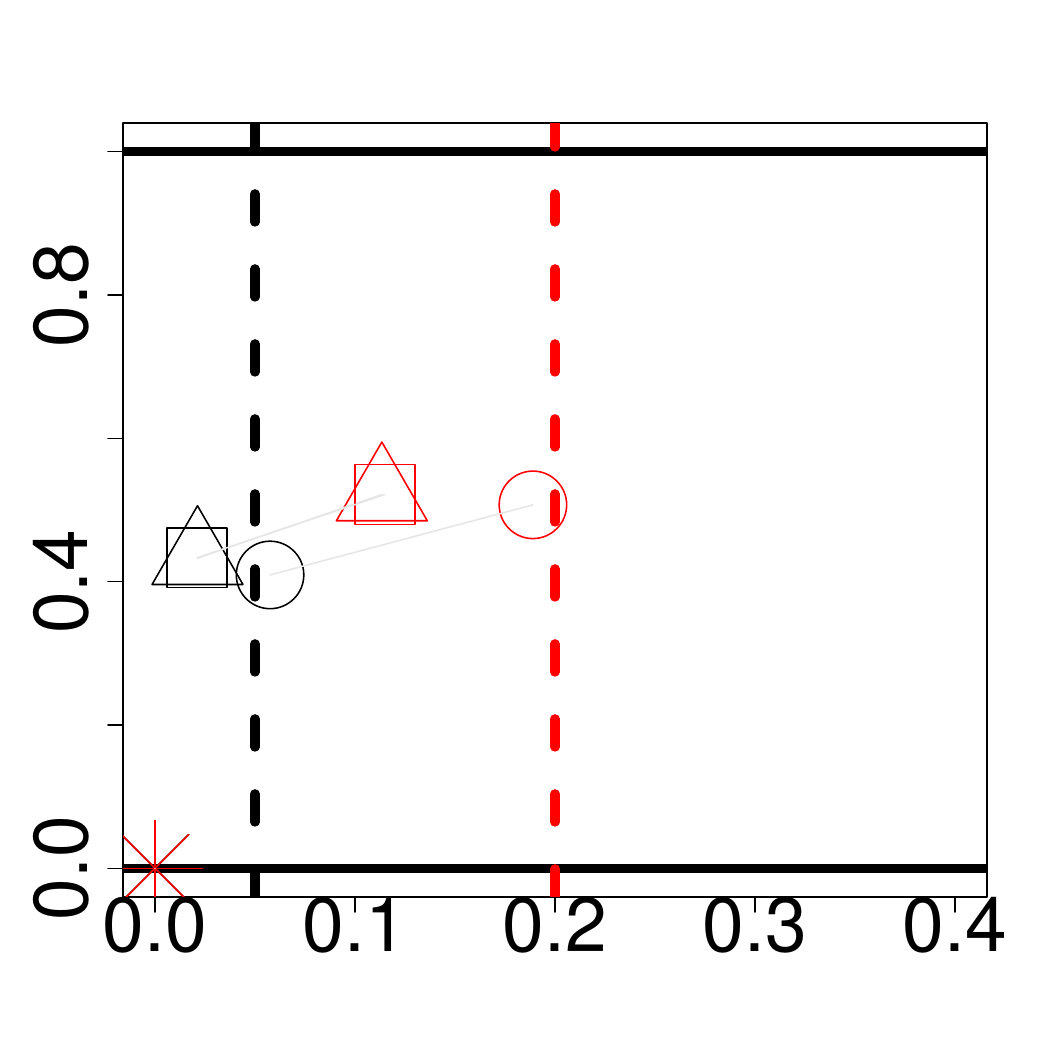}&
\hspace{-6mm}
\includegraphics[scale=0.17]{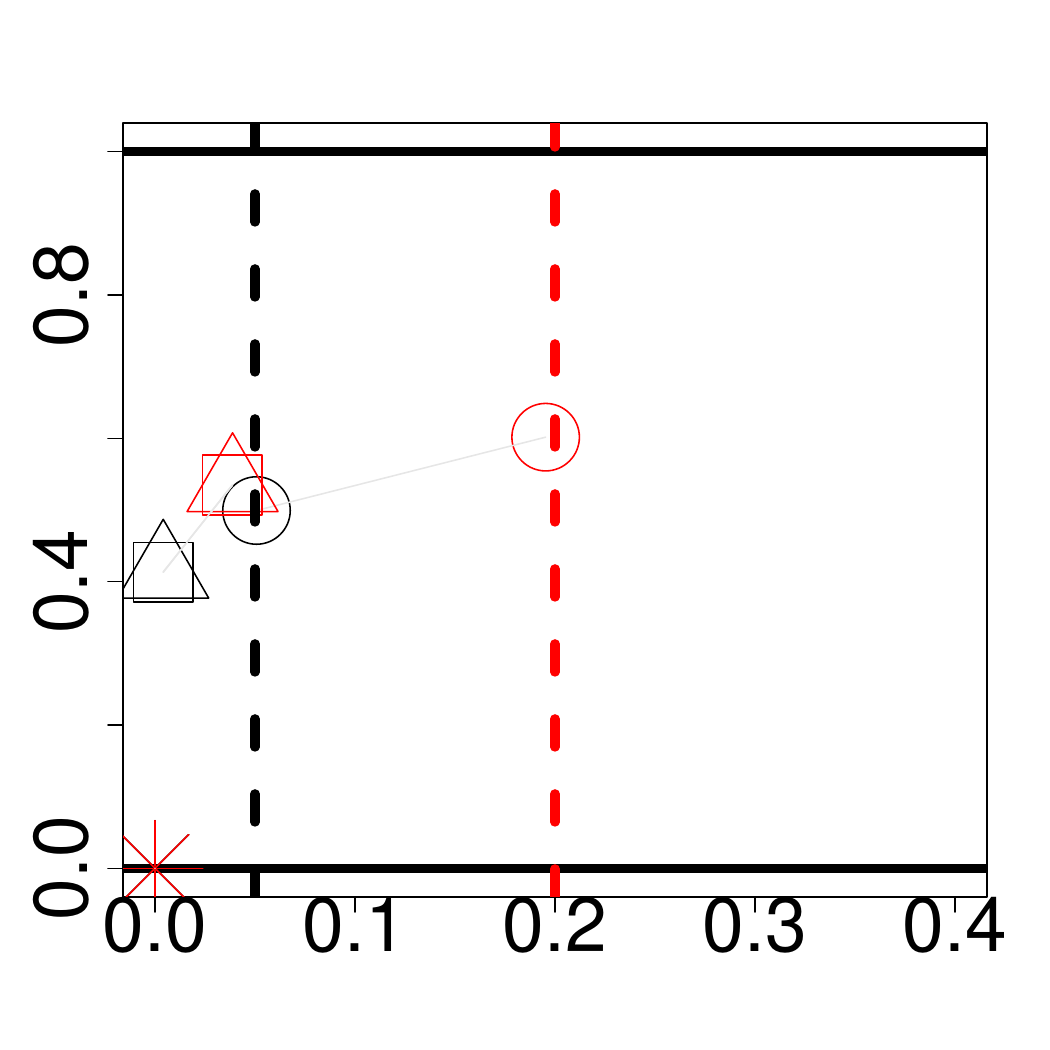}&
\hspace{-6mm}
\includegraphics[scale=0.17]{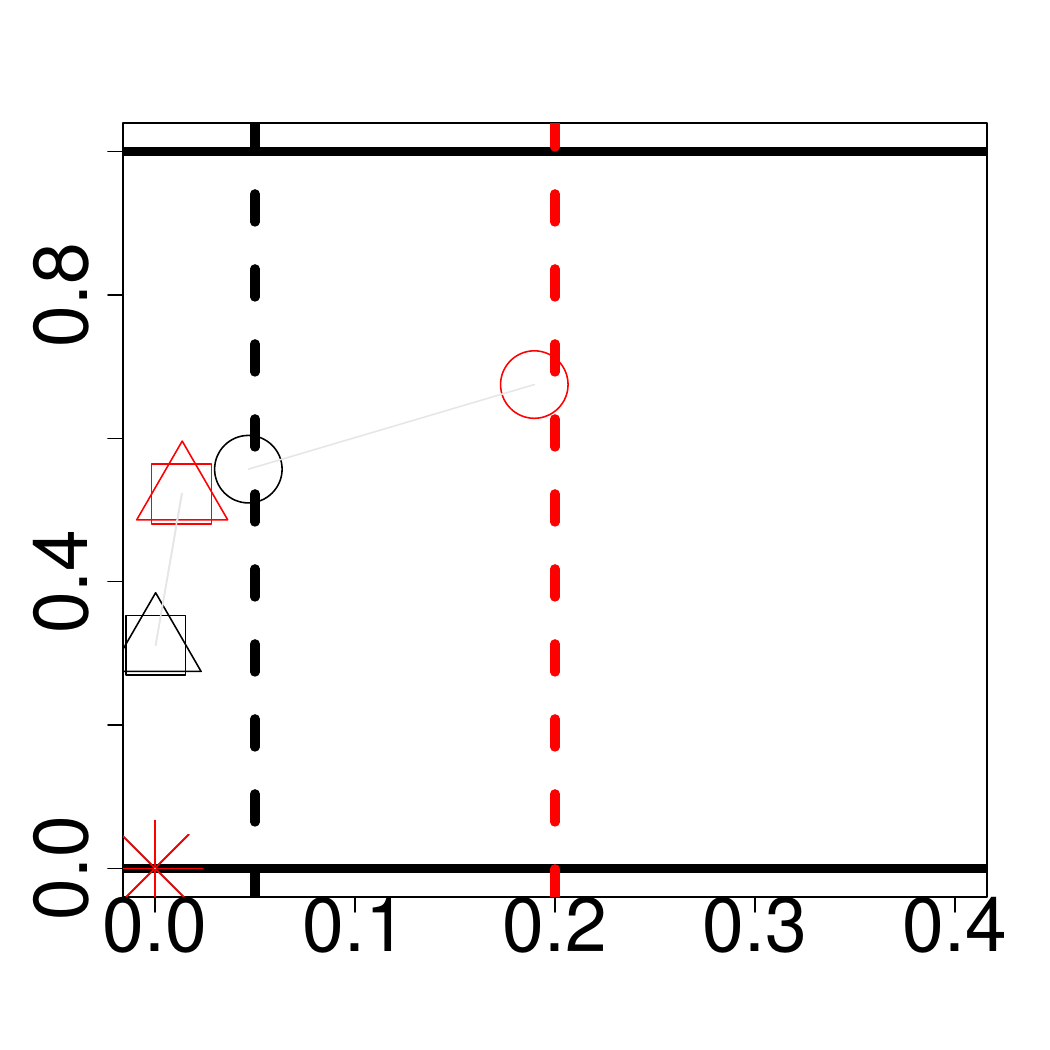}\\
\includegraphics[scale=0.17]{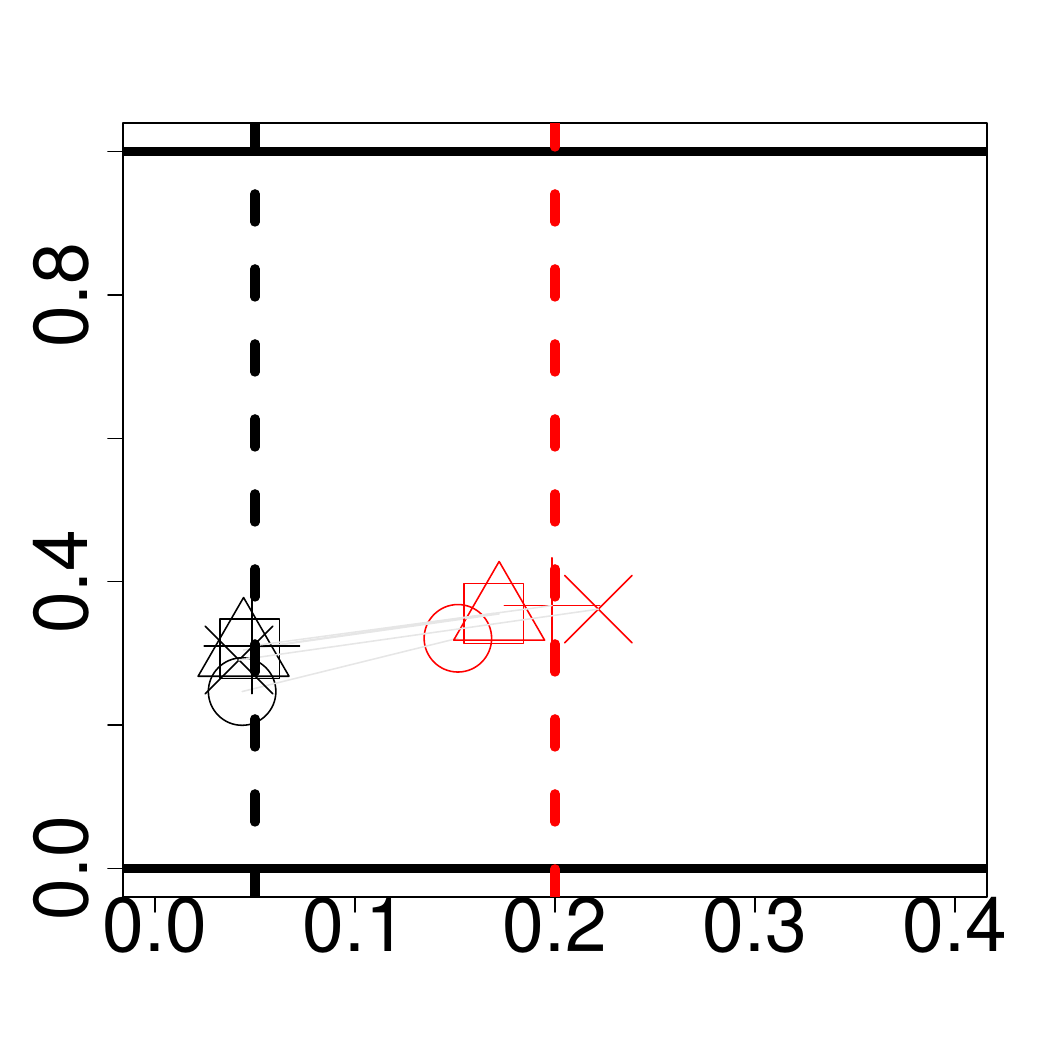}&
\hspace{-6mm}
\includegraphics[scale=0.17]{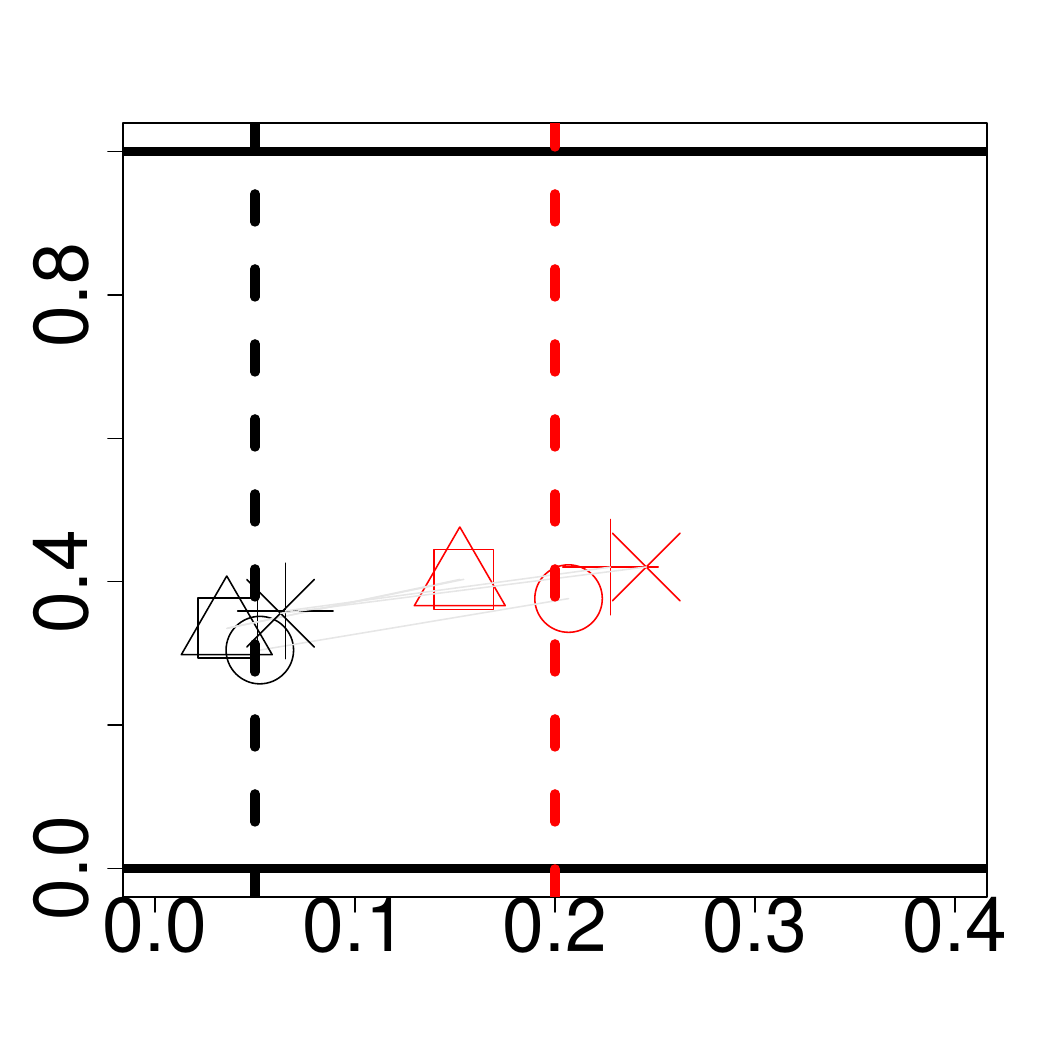}&
\hspace{-6mm}
\includegraphics[scale=0.17]{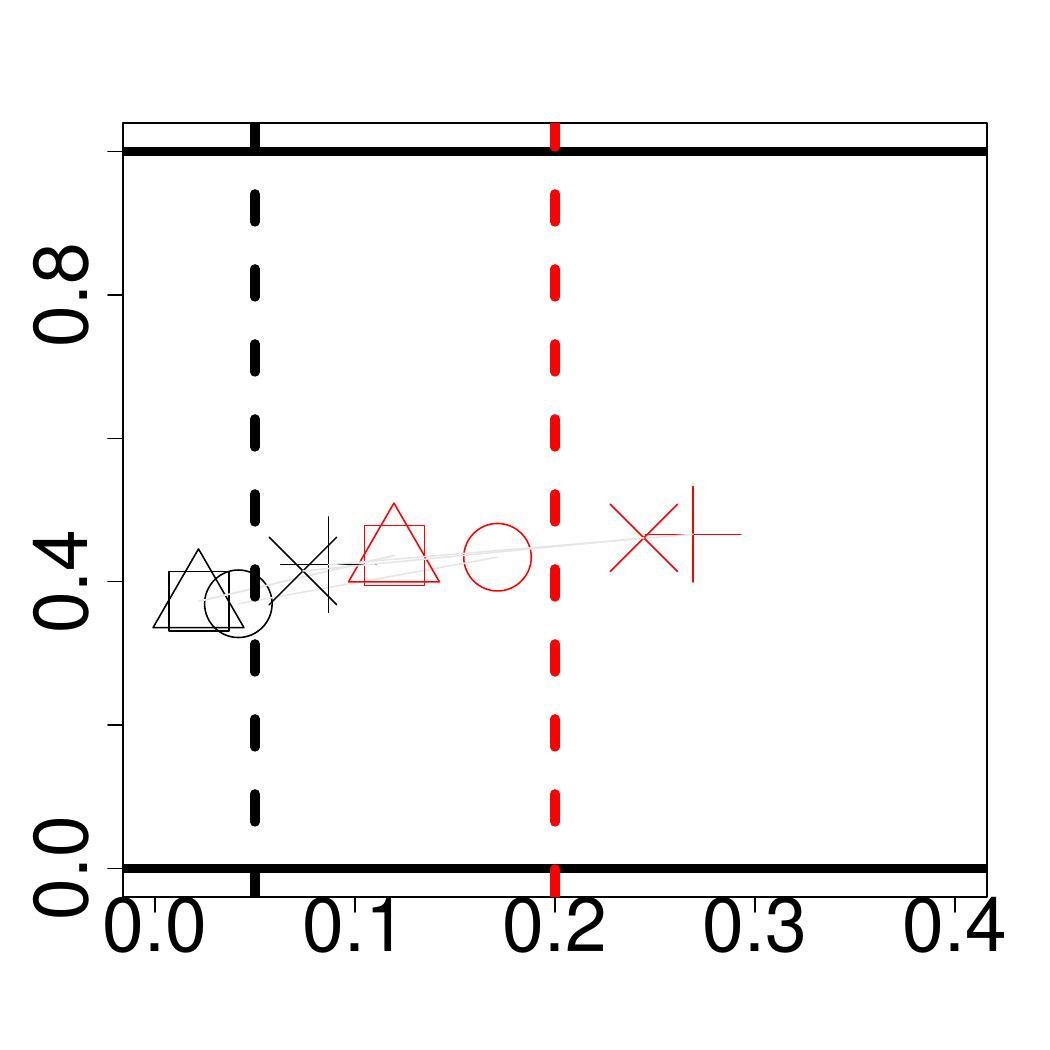}&
\hspace{-6mm}
\includegraphics[scale=0.17]{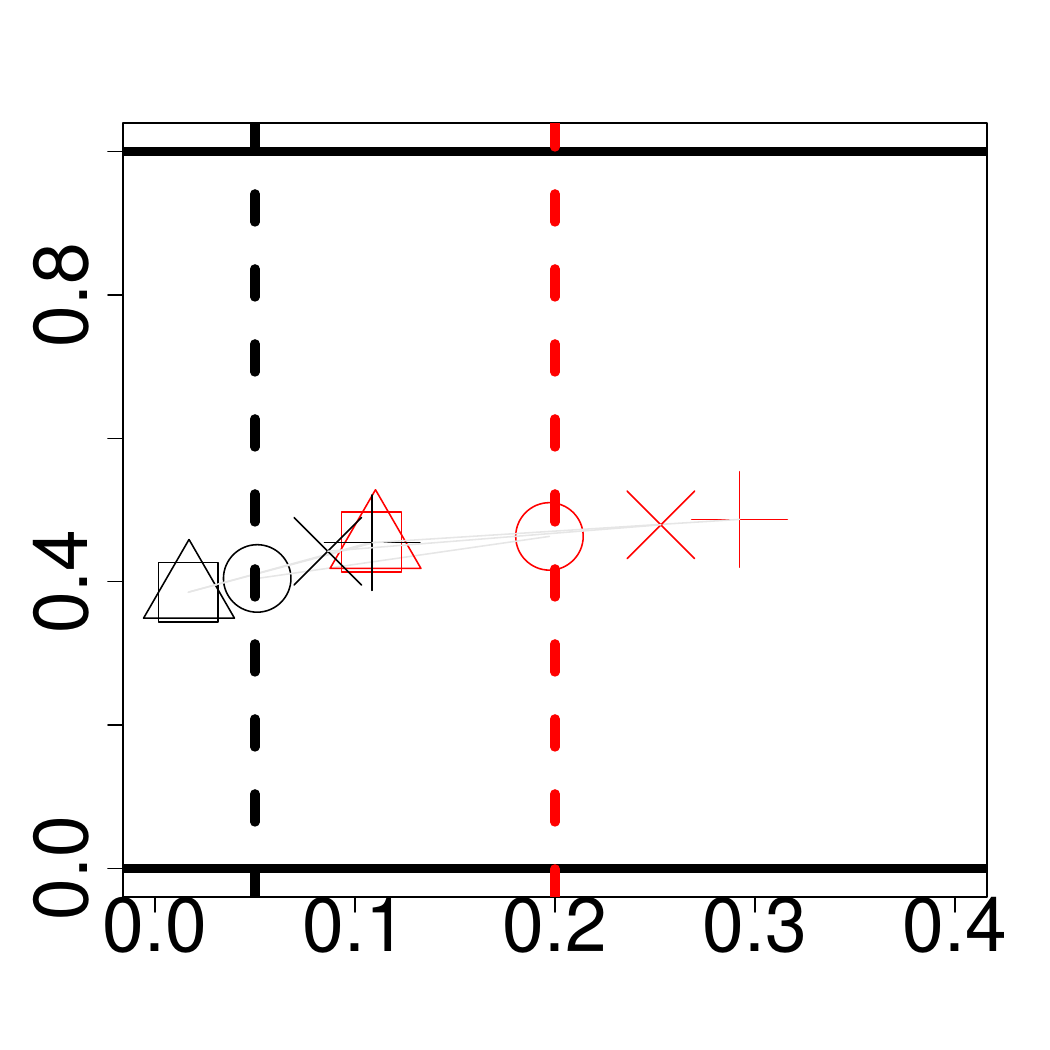}\\
\end{tabular}
\caption{Same as Figure~\ref{fig:indep} in the equi-correlation case.\label{fig:equi}}
\end{figure}

%

\section{Proof of Theorem \ref{THM:LOWER_GENERAL}} \label{sec:proofTHM:LOWER_GENERAL}

The proof  is based on the following argument. 
  We build two collections $\mathcal{P}_1$ and $\mathcal{P}_2$ of distributions. For any $P\in \cP_1$, the null distribution is $\mathcal{N}(0,1)$ and the distribution of the alternative is fairly separated from the null (see the red and blue curves in the left panel of Figure~\ref{lbsigma}). For any $P\in \cP_2$, the null distribution is $\mathcal{N}(0,\sigma_2^2)$ (with $\sigma_2>1$) and the alternative distribution is more concentrated around zero (right panel of Figure~\ref{lbsigma}). We first establish that any multiple testing procedure $R$ behaves similarly on $\cP_1$ and $\cP_2$. 
Then, we prove that, under $P\in \cP_2$,  if $|R(Y)|>0$, then its $\FDP$ is bounded away from zero. In contrast, under $P\in \cP_1$, if $|R(Y)|=0$, then its TDP is much smaller than that of oracle $\BH^{\star}_{\alpha}$. This will allow us to conclude that $R$ either does not control the FDR under some $P\in \cP_2$ or has a {too small} TDP under some $P\in \cP_1$.

\begin{figure}[h!]
\begin{tabular}{ccc}
$h=(1-\pi_1)\phi + \pi_1 f_1$  & $h=(1-\pi_2)\phi^{V}_{\sigma_2}+\pi_2 f_2$ \vspace{-0.5cm}\\
 \includegraphics[scale=0.45]{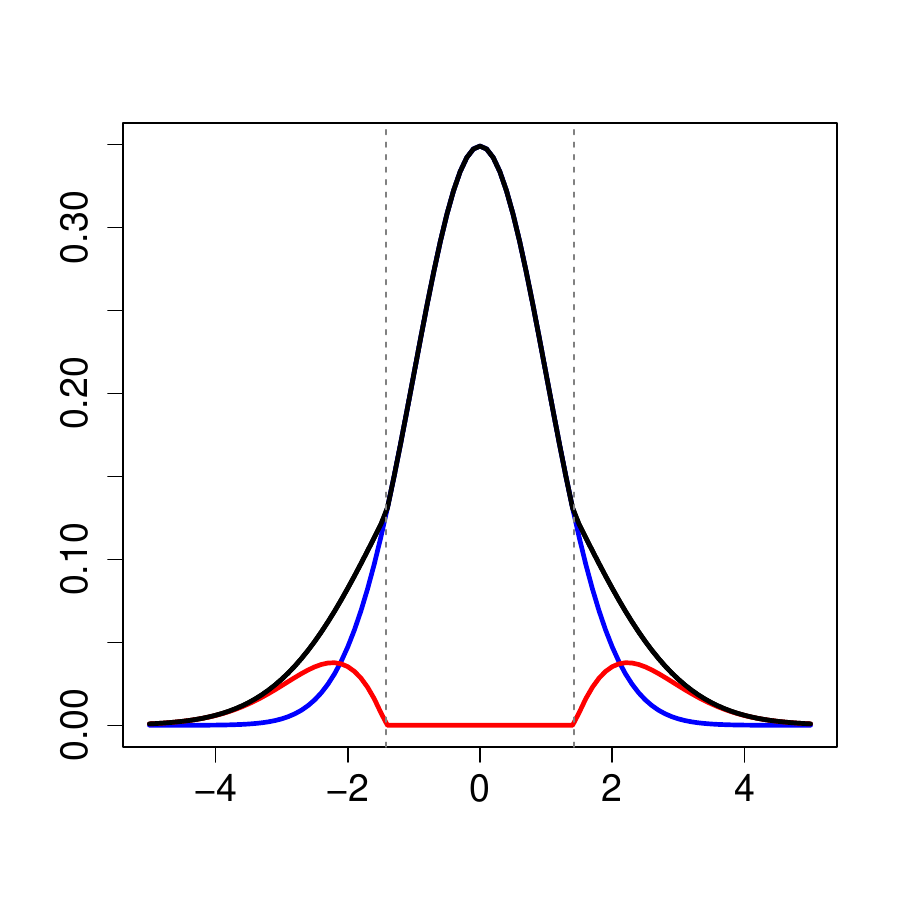}&\hspace{-0.5cm}\includegraphics[scale=0.45]{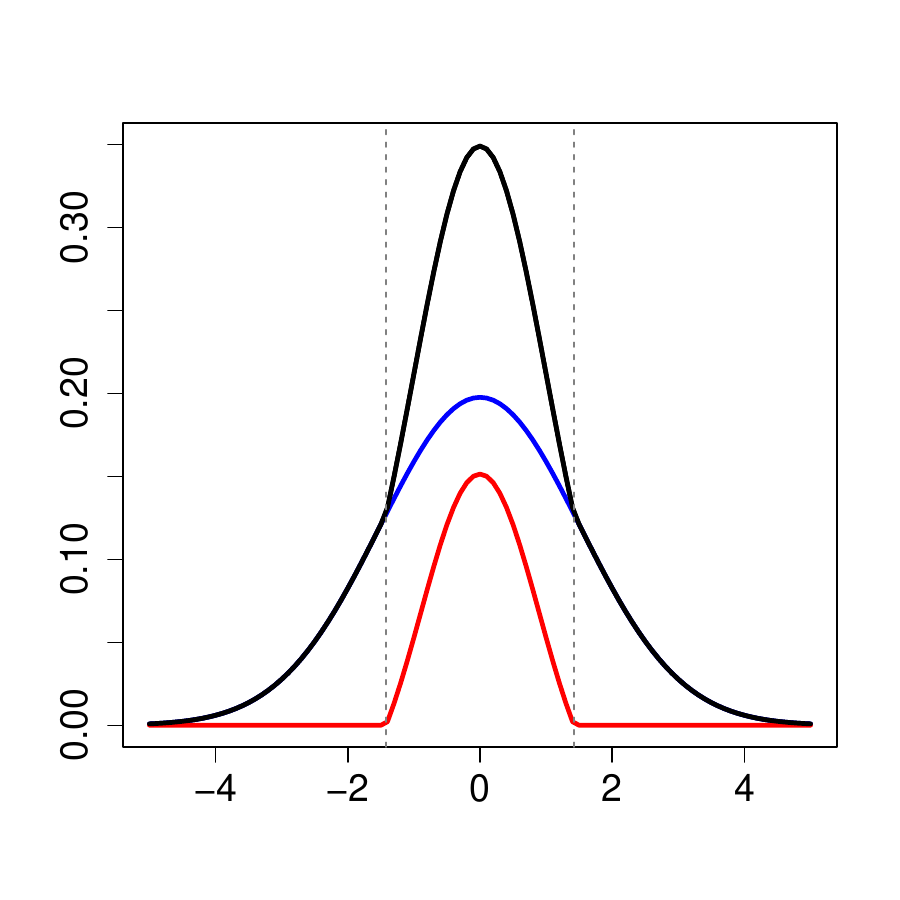}
 \end{tabular}
\vspace{-0.5cm}
\caption{Left: the density $h$ given by \eqref{eq:definition_g} (black), interpreted as a mixture between the null $\mathcal{N}(0,1)$ ($(1-\pi_1)\phi$ in blue) and the alternative $f_1$ ($\pi_1 f_1$ in red). Right: the same $h$ interpreted as a mixture between the null $\mathcal{N}(0,\sigma_2^2)$ ($(1-\pi_2)\phi^{V}_{\sigma} $ in blue) and the alternative $f_2$ ($\pi_2 f_2$ in red). $\pi_1=1/8$, $\pi_2=1/4$, $\sigma_2\approx 1.51$. The distance between the vertical dashed gray lines and $0$ is $t_0\approx 1.47$ given by \eqref{eq:definition_t0}. See the text for the definitions.
} \label{lbsigma}
\end{figure}

\subsection*{Step 1: Building a least favorable mixture distribution}

Let us denote $\phi^{V}_{\sigma}(x)=\phi(x/\sigma)/\sigma$ for all $x\in\R$ and $\sigma>0$. Fix some $\pi_1= k_1/(2n)$ and $\pi_2= k_2/(2n)$. For any $\sigma_2\geq 1$, define $\mu$, the real measure  with density 
\beq\label{eq:definition_g}
 h=(1-\pi_1)\phi + \pi_1 f_1= (1-\pi_2)\phi^{V}_{\sigma_2}+\pi_2 f_2 = \max((1-\pi_1)\phi, (1-\pi_2)\phi^{V}_{\sigma_2})\ ,
\eeq
where 
\beqn
f_1 =\frac{1}{\pi_1}\left[(1-\pi_2)\phi^{V}_{\sigma_2}- (1-\pi_1)\phi\right]_+\ ;\quad 
f_2 =\frac{1}{\pi_2}\left[(1-\pi_1)\phi-(1-\pi_2)\phi^{V}_{\sigma_2}\right]_+\ . 
\eeqn
Since $\int (1-\pi_1)\phi-(1-\pi_2)\phi^{V}_{\sigma_2}= \pi_2-\pi_1$, we deduce that, if $\sigma_2$ is chosen in such a way that $\int f_2(u)du=1$, we have $\int f_1(u)du=1$. Let us prove that $\int f_2(u)du=1$ for a suitable $\sigma_2>1$. For $\sigma_2=1$, we have $\int f_2(u)du= 1-\pi_1/\pi_2\in [0,1)$ (because $0<\pi_1\leq \pi_2$), whereas  
\[
\int  f_2(u)du\geq \frac{(1-\pi_2)}{\pi_2}\int [\phi(u) - \phi^{V}_{\sigma_2}(u)]_+du\geq 3  \int [\phi(u) - \phi^{V}_{\sigma_2}(u)]_+du\  , 
\] 
since $\pi_2\leq 1/4$. The above expression is larger than $1$ for $\sigma_2$ large enough (compared to some universal constant). 
 Since $\int f_2(u)du$ is a continuous function with respect to the variable $\sigma_2$, there exists at least one value of $\sigma_2>1$, depending only on $\pi_1$ and $\pi_2$, such that both $f_1$ and $f_2$ are densities. In the sequel, we fix $\sigma_2$ to one of these values. The above arguments also imply that $\sigma_2\leq c'_1$ for some positive universal constant $c'_1$.

\medskip

Recall that $\mu$ denotes the probability measure on $\R$ with density $h$ given by \eqref{eq:definition_g}.
Let $Q=\mu^{\otimes n}$ be the corresponding product distribution on $\R^n$. 
Let $Z_{1,i}$, $1\leq i\leq n$, be  i.i.d. and all following a Bernoulli distribution with parameter $\pi_1$. 
Let $Q_{1,z}$ the distribution on $\R^n$ of density $\prod_{i=1}^n  ((1-z_i) \phi+ z_i f_1)$  for $z\in\{0,1\}^n$, so that $Y\sim Q_{1,Z_1}$ is distributed as $Q$ unconditionally on $Z_1$.  
If $\sum_{i=1}^n z_i< n/2$, we have $\theta(Q_{1,z})= 0$, $\sigma_{1}(Q_{1,z})=1$ and $\cH_1(Q_{1,z})= \{1\leq i\leq n: z_i=1\}$. If $\sum_{i=1}^n z_i\geq n/2$, $P_{1,z}\notin \cP$. Nevertheless, we still let $\theta(Q_{1,z})=0$,  $\sigma_{1}(Q_{1,z})=1$ and $\cH_1(Q_{1,z})= \{1\leq i\leq n: z_i=1\}$ by convention. 
Define similarly $Z_2$ and $Q_{2,z}$, so that 
$Y\sim Q_{2,Z_2}$ has the same distribution as $Y\sim Q_{1,Z_1}$ (that is, $Q$) unconditionally on $Z_1$ and $Z_2$. 
In the sequel,  we denote $n_1(Z_1)=\sum_{i=1}^n \mathds{1}\{Z_{1,i}=1\} $ and $n_1(Z_2)=\sum_{i=1}^n \mathds{1}\{Z_{2,i}=1\}$, so that $n_1(Q_{1,Z_1})= n_1(Z_1)$ and $n_1(Q_{2,Z_2})= n_1(Z_2)$. 

As a consequence, $Q$ can be both interpreted as a mixture of $Q_{1,z}$ (with $\theta(Q_{1,z})=0$ and $\sigma(Q_{1,z})=1$) and as a mixture of $Q_{2,z}$ (with $\theta(Q_{2,z})=0$ and $\sigma(Q_{2,z})=\sigma_2$).

\medskip 

Consider any multiple testing procedure $R$ and define the event $A=\{|R(Y)|> 0\}$. 
Since
$
1=Q(A)+Q(A^c) = \E_{Z_1} Q_{1,Z_1}(A) + \E_{Z_2} Q_{2,Z_2}(A^c),
$
this entails that either $\E_{Z_1} Q_{1,Z_1}(A^c) \geq 1/2$ or $\E_{Z_2} Q_{2,Z_2}(A)\geq 1/2$. 
We show in Step 2 that, if $\E_{Z_2} Q_{2,Z_2}(A)\geq 1/2$, $R$ does not control the FDR under some $Q_{2,z}$ with $n_1(Q_{2,z})\leq k_2$, whereas we establish in Step 3 that, if $\E_{Z_1} Q_{1,Z_1}(A^c) \geq 1/2$, $R$ is over-conservative under some $Q_{1,z}$  with $n_1(Q_{1,z})\leq k_1$.

\subsection*{Step 2: if $\P_{Y\sim Q}(|R(Y)|>0)\geq  \frac{1}{2}$ then $\FDR(P_2,R)\geq c'_3$ for some $P_2$ with $n_1(P_2)\leq k_2$}
We consider the mixture distribution where $Z_2$ is sampled according to a Bernoulli distribution with parameters $\pi_2$ and $Y\sim Q_{2,Z_2}$. 
We have by the Fubini theorem,
\begin{align}
&\E_{Z_2}\left[\E_{Y\sim Q_{2,Z_2}}\left[ \FDP(Q_{2,Z_2}, R(Y)) \right] \right]\nonumber\\
&= \E_{Z_2}\left[\E_{Y\sim Q_{2,Z_2}}\left[ \frac{\sum_{i\in R(Y)} \mathds{1}\{Z_{2,i}=0\}}{|R(Y)|}\mathds{1}\{R(Y)>0\}\right] \right] \nonumber\\
 &= \E_{Y\sim Q}\left[\E_{Z_2}\left[ \frac{\sum_{i\in R(Y)} \mathds{1}\{Z_{2,i}=0\}}{|R(Y)|}\mathds{1}\{R(Y)>0\}\:\bigg|\: Y \right] \right] \nonumber\\
&= \E_{Y\sim Q}\left[ \frac{\sum_{i\in R(Y)}\P(Z_{2,i}=0\:|\:Y)}{|R(Y)|}\mathds{1}\{R(Y)>0\}\right] \label{equ:FDPswitch}.
\end{align}
Next, we have 
\begin{align*}
\P(Z_{2,i}=0\:|\:Y)&= \frac{(1-\pi_2)\phi^{V}_{\sigma_2}(Y_i)}{h(Y_i)}
=\frac{(1-\pi_2)\phi^{V}_{\sigma_2}(Y_i)}{\max((1-\pi_1)\phi(Y_i), (1-\pi_2)\phi^{V}_{\sigma_2}(Y_i))}\\
&\geq 1\wedge \frac{(1-\pi_2)\phi^{V}_{\sigma_2}(Y_i)}{(1-\pi_1)\phi(Y_i)}
\geq 1\wedge \frac{(1-\pi_2)}{(1-\pi_1)\sigma_2}\exp\left[\frac{Y_i^2}{2}\left(1-\frac{1}{\sigma_2^2}\right)\right]\ .
\end{align*}
Since $\sigma_2\geq 1$, $\pi_2\leq 1/4$ and since we proved in step 1 that $\sigma_2\leq c'_1$ (for some numerical constant $c'_1$), we obtain
\begin{align*}
\P(Z_{2,i}=0\:|\:Y)&\geq 1\wedge \frac{(1-\pi_2)}{(1-\pi_1)\sigma_2} \geq 1\wedge\frac{3}{4 c'_1}\ .
\end{align*}
Combining this with \eqref{equ:FDPswitch}, we obtain
\begin{align*}
&\E_{Z_2}\left[\E_{Y\sim Q_{2,Z_2}}\left[ \FDP(Q_{2,Z_2}, R(Y)) \right] \right]
\geq \left(1\wedge\frac{3}{4 c'_1}\right) \P_{Y\sim Q}( R(Y)>0)\geq \left(1\wedge\frac{3}{4 c'_1}\right)/2\ .
\end{align*}
Recall that $n_1(Z_2)$ follows a Binomial distribution with parameter $n$ and $\pi_2=k_2/(2n)$. 
By Chebychev inequality, we have 
\[
\P\left(\left|n_1(Z_2) - k_2/2\right|>  k_2/4 \right) \leq \frac{n \pi_2(1-\pi_2)}{( k_2/4)^2}\leq \frac{8}{k_2}. 
\]
This implies that there exists $z_2$ with $n_1(z_2)\in [k_2/4; k_2]$ such that 
\begin{align*}
\FDR(Q_{2,z_2},R)=\E_{Y\sim Q_{2,z_2}}\left[ \FDP(Q_{2,z_2}, R(Y)) \right] \geq \left(1\wedge\frac{3}{4 c'_1}\right)/2 -  \frac{8}{k_2},
\end{align*}
which is bounded away from zero for $k_2$ large enough, this last condition being ensured by~\eqref{eq:condition_lower_k_0_general} and the fact that $n$ is large enough. To summarize, we have proved that, for $P_2=Q_{2,z_2}$, we have   $\FDR(P_2,R(Y))\geq c'_3$ for some universal constant $c'_3\in (0,1)$  and $n_1(P_2)\leq k_2$.

\subsection*{Step 3: If $\P_{Y\sim Q}(|R(Y)|=0)\geq 1/2$, then $R$ is over-conservative, for some $P_1$ with $n_1(P_1)\leq k_1$} 

Applying Chebychev inequality as in Step 2, we deduce that $\P(n_1(Z_1)\in [k_1/4; k_1])\geq 1- 8/k_1$. Since $\E_{Z_1}\P_{Y\sim \P_{1,Z_1}}(|R(Y)|=0)\geq 1/2$, this implies that there exists $z_1$ with $n_1(z_1)\leq k_1$ such that $\P_{Y\sim \P_{1,z_1}}(|R(Y)|=0)\geq 1/2- 8/k_1$.
Since~\eqref{eq:condition_lower_k_0_general}  can be satisfied only if $\log(k_2/k_1)\leq (2 c_2\log 2)^{-1}\log(n)$, that is, $k_1 \geq k_2 n^{-(2 c_2\log 2)^{-1}}$, by choosing $c_1$ and $c_2$ large enough, we may assume that $1/2- 8/k_1\geq 2/5$.  In the sequel, we fix $P_1=P_{1,z_1}$. 
 For such $P_1$ with $n_1(P_1)\leq k_1$, we have therefore $\P_{Y\sim P_1}(|R(Y)|=0)\geq 2/5$. In contrast, we claim that $\BH_{\alpha/2}^*$ rejects many false null hypotheses with positive probability.

Before this, let us provide further properties of $\sigma_2$ and $h$. 
 The positive number $u_0$  satisfying $(1-\pi_1)\phi(u_0)= (1-\pi_2)\phi^{V}_{\sigma_2}(u_0)$ is defined as 
\beq\label{eq:definition_t0}
u_0^2= \frac{2\sigma_2^2}{\sigma_2^2-1}\log\left(\sigma_2\frac{1-\pi_1}{1-\pi_2}\right).
\eeq
We easily check that we have
$(1-\pi_1)\phi(u)> (1-\pi_2)\phi^{V}_{\sigma_2}(u)$ if and only if $|u|<u_0$, and 
$(1-\pi_1)\phi(u)< (1-\pi_2)\phi^{V}_{\sigma_2}(u)$ if and only if $|u|> u_0$, so that
$f_1(u)>0$  if and only if   $|u|>u_0$ and $f_2(u)>0$  if and only if  $|u|<u_0$.

\begin{lem}\label{lem:sigma_2}
There exists a numerical constant $c'_0\in (0,1)$ such that $\sigma_2$ satisfies
$$
\sigma_2 -1 \geq c'_0\pi_2 [1+\log( \pi_2/\pi_1)]^{-1}\ .
$$ 
Also, there exists another numerical constant $c^{''}_0\in (0,1)$ such that,  for $n$ large enough, 
we have $\overline{\Phi}(u_0/\sigma_2)\geq 10\sqrt{2\log(2n)/n}$ provided that 
$\log( \pi_2/\pi_1)\leq c_0^{''}\log(n)$.
\end{lem}

Let $u_1$ be the smallest number such that for all $u\geq u_1$,  one has 
\begin{equation}\label{equf1big}
 \pi_1 f_1(u) \geq 8\alpha^{-1}\phi(u)\ .
\end{equation}
This implies that for all $u\geq u_1$, $f_2(u)=0$. From the definition of $f_1$, we derive that 
$
 (1-\pi_2)\phi^{V}_{\sigma_2}(u_1) = \phi(u_1)[8\alpha^{-1}+ (1-\pi_1)]\ , 
$
which is again equivalent to 
\beq\label{eq:t1_obj}
u_1^2= \frac{2\sigma^2_2}{\sigma^2_2-1}\log\left[\sigma_2\frac{8\alpha^{-1}+ (1-\pi_1)}{1-\pi_2}\right]\ . 
\eeq

\begin{lem}\label{lem:t1}
There exists a positive numerical constant $c'_3$ such that the following holds for all $\alpha\in (0,1)$. If
\[
\frac{1+ \log(\pi_2/\pi_1)}{\pi_2}\log\left(\frac{2}{\alpha}\right)\leq c'_3\log(n) \ , 
\]
then, we have $u_1\leq \sqrt{\log(n)}$.
\end{lem}

Now, recall that $\BH_{\alpha/2}^*$ procedure does use some knowledge of the true underlying distribution $P_1$, namely, $\theta(P_1)=0$ and $\sigma(P_1)=1$. Hence, it can be written as $\BH_{\alpha/2}^*=\{i\in\{1,\dots,n\}\::\:|Y_i|\geq \hat{u}\}$ for 
\begin{equation}\label{equ:BHthresholdproof}
\hat{u}=\min\Big\{u\in \R_+\::\: \sum_{i=1}^n\mathds{1}\{|Y_i|\geq u\} \geq 4 n \alpha^{-1} \ol{\Phi}(u) \Big\}. 
\end{equation}

Hence, we shall prove  that $ \sum_{i\in \cH_1(P_1)}\mathds{1}\{|Y_i|\geq \hat{u}\}>0$ is large with high probability. For this, let us consider
$
N= \sum_{i\in \cH_1(P_1)}\mathds{1}\{|Y_i|\geq u_1\}
$ with $u_1$ as in~\eqref{eq:t1_obj}.
By \eqref{equf1big}, we have 
$
2\int_{u_1}^\infty f_1(u) du \geq 16 \pi_1^{-1}\alpha^{-1} \ol{\Phi}(u_1).
$
Since $n_1(P_1)\geq k_1/4=\pi_1 n/2$, $N$ stochastically dominates the binomial distribution with parameters $\lceil \pi_1 n/2\rceil$ and $16 \pi_1^{-1} \alpha^{-1} \ol{\Phi}(u_1)$. 
Applying Bernstein inequality yields 
$
\P_{Y\sim P_1}(N\leq q/2) \leq e^{-3q/28}
$
for $q=\pi_1 n/2 \times 16 \pi_1^{-1} \alpha^{-1} \ol{\Phi}(u_1) = 8 n  \alpha^{-1} \ol{\Phi}(u_1).$ By \eqref{equ:BHthresholdproof}, $N\geq q/2$ implies $\hat{u}\leq u_1$. This leads us to
\beq\label{eq:conclusion_step_3}
\P_{Y\sim P_1}(\hat{u}\leq u_1, N\geq 4 n  \alpha^{-1} \ol{\Phi}(u_1)) \geq 1-e^{-3q/28} = 1-e^{-(6/7)n  \alpha^{-1} \ol{\Phi}(u_1)} .
\eeq
In view of condition \eqref{eq:condition_lower_k_0_general},
we can apply Lemma~\ref{lem:t1} which gives $u_1\leq \sqrt{\log(n)}$. 
Next, by Lemma~\ref{lem:quantile}, for $n$ larger than a numerical constant, we have $n \ol{\Phi}(\sqrt{\log(n)})\geq c'\sqrt{n/(\log n)}$, for some other numerical constant $c'>0$.
 Hence, for $n$ larger than a numerical constant, with probability at least  $ 1-1/n$, we have 
 $$|\BH_{\alpha/2}^*\cap \cH_1(P_1)|= \sum_{i\in \cH_1(P_1)}\mathds{1}\{|Y_i|\geq \hat{u}\}\geq N \geq 4 n  \alpha^{-1} \ol{\Phi}(u_1)\geq c'\alpha^{-1}\sqrt{n/(\log n)}.$$

\subsection*{Conclusion} Step 1 entails that either $\E_{Z_2} Q_{2,Z_2}(A)\geq 1/2$  or $\E_{Z_1} Q_{1,Z_1}(A^c) \geq 1/2$.
In the former case, Step 2 implies that $ \sup_{P \in \mtc{P}, n_1(P)\leq k_2}   \FDR(P,R)\geq c_3$. In the latter case, we deduce from Step 3  that, for some $P\in \cP$ with $n_1(P)\leq k_1$, we have $\P_{Y\sim P_1}(|R(Y)|= 0)\geq 2/5$, whereas 
\[
\P_{Y\sim P_1}\left[|\BH^{\star}_{\alpha/2}\cap \cH_1(P_1)|\geq c'\alpha^{-1}\sqrt{n/(\log n)}\right]\geq 1 - e^{-c_5\alpha^{-1}\sqrt{n/\log(n)}}\ .
\]
 This concludes the proof by choosing appropriately the numerical positive constants  $c_1$--$c_4$. 

\section{Proof of Theorem~\ref{THMBORNESUP}}\label{sec:proofupperbound}

\subsection{Proof of  \eqref{eq:upper_FDR_thm} in Theorem~\ref{THMBORNESUP}}\label{sec:prooftha}

Fix $P\in\mathcal{P}$ with $n_0(P)/n\geq 0.9.$ First denote 
\begin{align*}
\delta&=c\left( \frac{n_1(P)+1}{n} + n^{-1/6}
\right)\\
\Omega&=\left\{\frac{|\wt{\theta}-\theta|}{\sigma} \leq \delta , \:
\frac{|\wt{\sigma}-\sigma|}{\sigma }\leq \delta,\:
 \wt{\sigma} \in [\sigma/2; 2\sigma],\:|\wt{\theta}-\theta|<0.3 \:\wt{\sigma} \right\},
\end{align*}
with $c>0$ being a universal constant chosen small enough so that 
$\P(\Omega^c)\leq 6 e^{-n^{2/3}}$.
 This is possible according to Lemma~\ref{lem:esti} (used with 
$x=n^{2/3}$).

For any $i\in\{1,\dots,n\}$, we introduce $Y^{(i)}$ as the vector of $\R^n$ such that $Y^{(i)}_j = Y_j$ for $j\neq i$ and   $Y^{(i)}_i=\mathrm{sign}(Y_i-\theta)\times \infty$. Hence $Y^{(i)}$ is such that the $i$-th observation has been set $-\infty$ or $+\infty$ depending on the sign of $Y_i-\theta$. 
The estimators based on the modified sample $Y^{(i)}$ are then defined by
\begin{equation}\label{equ:estimators(i)}
\wt{\theta}^{(i)} = Y^{(i)}_{(\lceil n/2\rceil)};\:\:
\wt{\sigma}^{(i)} =  U^{(i)}_{(\lceil n/2\rceil)}/\ol{\Phi}^{-1}(1/4),
\end{equation}
for $U^{(i)}_j=|Y^{(i)}_j-\wt{\theta}^{(i)}|$, $1\leq j\leq n$. 
As justified at the end of the proof, the purpose of these modified samples is to introduce some independence between the oracles $p$-values $p_i^{\star}$ and the estimators $(\wt{\theta}^{(i)},\wt{\sigma}^{(i)})$. 

It turns out that, for small rescaled $p$-values $p_i(\wt{\theta},\wt{\sigma})$, the estimators $\wt{\theta}^{(i)}$ and $\wt{\sigma}^{(i)}$ are not modified. Furthermore, the BH threshold does not change when replacing $Y$ by $Y^{(i)}$. These two facts lead to the following lemma.

\begin{lem}\label{cor:pvaluei}
Consider any $i\in\{1,\dots,n\}$ and any $\alpha \in (0,0.5)$. Provided that $|\wt{\theta}-\theta|<0.3 \:\wt{\sigma}$, we have
\begin{align*}
\mathds{1}\{p_i(\wt{\theta},\wt{\sigma}) \leq T_\alpha(Y;\wt{\theta},\wt{\sigma})\} &= \mathds{1}\{p_i(\wt{\theta}^{(i)},\wt{\sigma}^{(i)}) \leq T_\alpha(Y^{(i)};\wt{\theta}^{(i)},\wt{\sigma}^{(i)})\}.
\end{align*}
Moreover, if $p_i(\wt{\theta},\wt{\sigma}) \leq T_\alpha(Y;\wt{\theta},\wt{\sigma})$, we have $\wt{\theta}^{(i)}=\wt{\theta}$, $\wt{\sigma}^{(i)}=\wt{\sigma}$ and $T_\alpha(Y^{(i)};\wt{\theta}^{(i)},\wt{\sigma}^{(i)})=T_\alpha(Y;\wt{\theta},\wt{\sigma})\geq \alpha/n$.
\end{lem}

Combining this lemma with the definition of the FDP, we get
\begin{align*}
&\FDP(P,\BH_\alpha(Y;\wt{\theta},\wt{\sigma})) \mathds{1}_{\Omega} \\
&= \frac{\alpha}{n} \sum_{i\in \cH_0(P)}\frac{\mathds{1}{\{p_i(\wt{\theta},\wt{\sigma})\leq  T_\alpha(Y;\wt{\theta},\wt{\sigma})\}}}{T_\alpha(Y;\wt{\theta},\wt{\sigma})\vee (\alpha/n)}  \mathds{1}_{\Omega}  \\
&= \frac{\alpha}{n} \sum_{i\in \cH_0(P)}\frac{\mathds{1}{\{p_i(\wt{\theta}^{(i)},\wt{\sigma}^{(i)})\leq  T_\alpha(Y^{(i)};\wt{\theta}^{(i)},\wt{\sigma}^{(i)})\}}}{T_\alpha(Y^{(i)};\wt{\theta}^{(i)},\wt{\sigma}^{(i)})}  \mathds{1}_{\Omega}\mathds{1}_{\wt{\theta}^{(i)}=\wt{\theta},\wt{\sigma}^{(i)}=\wt{\sigma}} 
\end{align*}
Now, note that, on $\Omega$, when $\wt{\theta}^{(i)}=\wt{\theta}$, $\wt{\sigma}^{(i)}=\wt{\sigma}$, we have
$|\wh{\theta}^{(i)}-\theta|\leq \sigma \delta$, $|\wh{\sigma}^{(i)}-\sigma|\leq \sigma \delta$, $\wh{\sigma}^{(i)}\geq \sigma/2$.
The following key lemma compares the hypotheses rejected by the oracle BH procedure and the rescaled procedure. 

\begin{lem}\label{lem:forFNR} 
For arbitrary estimators $\wh{\theta},\wh{\sigma} $, any $\theta\in \mathbb{R}$, $\sigma>0$,  $\delta>0$,  $\alpha\in (0,0.8)$,  $t_0\in(0,\alpha)$, define	
$$
\eta = \delta  c \left( (2\log (1/t_0))^{1/2} +  2\log(1/t_0) \right)\ ,
$$
with the constant $c>0$ of Corollary~\ref{cor:pvaluemaj}.
 Assume that $\wh{\sigma}\in (\sigma/2; 2\sigma)$, $|\wh{\theta}-\theta|\leq (\sigma\wedge  \widehat{\sigma}) \delta$, $|\wh{\sigma}-\sigma|\leq (\sigma \wedge \widehat{\sigma})\delta$,  and  $\eta\leq 0.05$. Then, for all $i\in\{1,\dots,n\}$,
 \begin{itemize}
\item if $T_\alpha(\wh{\theta},\wh{\sigma})\vee (\alpha/n)\geq t_0$, we have
 \begin{align}
\mathds{1}\{p_i(\wh{\theta},\wh{\sigma})\leq T_{\alpha}(\wh{\theta},\wh{\sigma})\} &\leq \mathds{1}\{p_i(\theta,\sigma)\leq (1+\eta)T_{\alpha}(\wh{\theta},\wh{\sigma}) \}\nonumber\\
&\leq \mathds{1}\{p_i(\theta,\sigma)\leq T_{\alpha(1+\eta)}(\theta,\sigma)\}\ ;
\label{equ:forFDR}
\end{align}
\item if $T_{0.95\alpha}(\theta,\sigma)\vee (0.95\alpha/n)\geq t_0$, we have
 \begin{align}
\mathds{1}\{p_i(\theta,\sigma)\leq T_{\alpha(1-\eta)}(\theta,\sigma)\}  \leq \mathds{1}\{p_i(\wh{\theta},\wh{\sigma})\leq T_{\alpha}(\wh{\theta},\wh{\sigma})\}\ .
\label{equ:forFNR}
\end{align}
\end{itemize}
\end{lem}
Intuitively, \eqref{equ:forFDR} above implies that any rescaled procedure is more conservative than the oracle procedure $BH^\star_{\alpha(1+\eta)}$ with an enlarged parameter $\alpha(1+\eta)$. 

By definition of $\eta$ in the statement of Theorem~\ref{THMBORNESUP} and taking $\wh{\theta}=\wt{\theta}^{(i)}$, $\wh{\sigma}=\wt{\sigma}^{(i)}$,  $\theta=\theta(P)$, $\sigma=\sigma(P)$, $t_0=\alpha/n$, we are in position to apply \eqref{equ:forFDR}. For all $i\in\{1,\dots,n\}$, we have 
\begin{align*}
\mathds{1}\{ p_i(\wt{\theta}^{(i)},\wt{\sigma}^{(i)})\leq T_\alpha(Y^{(i)};\wt{\theta}^{(i)},\wt{\sigma}^{(i)}) \} &\leq \mathds{1}\{ p_i^\star\leq (1+\eta)T_\alpha(Y^{(i)};\wt{\theta}^{(i)},\wt{\sigma}^{(i)}) \},
\end{align*}
where we recall that $p_i^\star=p_i(\theta(P),\sigma(P))$ is the $i$-th oracle $p$-value.
This gives
\begin{align*}
\FDP(P,\BH_\alpha(Y;\wt{\theta},\wt{\sigma})) \mathds{1}_{\Omega} 
&\leq \frac{\alpha}{n} \sum_{i\in \cH_0(P)}\frac{\mathds{1}{\{p_i^\star\leq  (1+\eta)T_\alpha(Y^{(i)};\wt{\theta}^{(i)},\wt{\sigma}^{(i)})\}}}{T_\alpha(Y^{(i)};\wt{\theta}^{(i)},\wt{\sigma}^{(i)})} .
\end{align*}

The following lemma  stems from the symmetry of the normal distribution. 
\begin{lem}\label{lem:indep}
For any $P\in \mathcal{P}$, any $i\in \cH_0(P)$, $|Y_i-\theta(P)|$ is independent of  $Y^{(i)}$, and thus also of the estimators $(\wt{\theta}^{(i)},\wt{\sigma}^{(i)})$.
\end{lem}
By Lemma~\ref{lem:indep}, for $i\in\cH_0(P)$, the oracle $p$-value  $p_i^\star=2\ol{\Phi}\left(\frac{|Y_i-\theta(P)|}{\sigma(P)}\right)$ is independent of $(Y^{(i)},\wt{\theta}^{(i)},\wt{\sigma}^{(i)})$ and thus of $T_\alpha(Y^{(i)};\wt{\theta}^{(i)},\wt{\sigma}^{(i)})$. As a result, we obtain by integration
\begin{align*}
&\E_P\left[ \FDP(P,\BH_\alpha(Y;\wt{\theta},\wt{\sigma})) \mathds{1}_{\Omega} \right]\\ 
&\leq\frac{\alpha}{n}\sum_{i\in \cH_0(P)}\E_P\left[ \frac{ \mathds{1}{\{p_i^\star\leq  (1+\eta)T_\alpha(Y^{(i)};\wt{\theta}^{(i)},\wt{\sigma}^{(i)})\}}}{T_\alpha(Y^{(i)};\wt{\theta}^{(i)},\wt{\sigma}^{(i)})}  \right]
\\
&\leq \frac{\alpha}{n}\sum_{i\in \cH_0(P)}\E_P\left[ \frac{ \P\left[p_i^\star\leq  (1+\eta)T_\alpha(Y^{(i)};\wt{\theta}^{(i)},\wt{\sigma}^{(i)})\} \:|\:  Y^{(i)},\wt{\theta}^{(i)},\wt{\sigma}^{(i)}\right]}{T_\alpha(Y^{(i)};\wt{\theta}^{(i)},\wt{\sigma}^{(i)})}   \right]
\\
&\leq\frac{\alpha n_0(P)}{n} (1+\eta),\end{align*}
where we used that $p_i^\star\sim U(0,1)$ for $i\in\cH_0(P)$. This entails \eqref{eq:upper_FDR_thm} of Theorem~\ref{THMBORNESUP}.

\subsection{Proof of  \eqref{eq:lower_TDP_thm} in  Theorem~\ref{THMBORNESUP}}\label{sec:proofthb}

Take $P,\delta,\Omega$ as in the previous section. 
On the event $\Omega$, the conditions of Lemma~\ref{lem:forFNR} are satisfied with $\wh{\theta}=\wt{\theta}$, $\wh{\sigma}=\wt{\sigma}$,  $\theta=\theta(P)$, $\sigma=\sigma(P)$, and $t_0=0.95\alpha/n$. Hence,  \eqref{equ:forFNR} ensures that, for all $i\in\{1,\dots,n\}$, 
$
\mathds{1}\{p_i^\star\leq T_{\alpha(1-\eta)}(\theta(P),\sigma(P))\}  \leq \mathds{1}\{p_i(\wt{\theta},\wt{\sigma})\leq T_{\alpha}(\wt{\theta},\wt{\sigma})\}
$
and thus
$
\TDP(P,\BH_{\alpha(1-\eta)}^\star)\leq \TDP(P,\BH_{\alpha}(\wt{\theta},\wt{\sigma})).$
Hence, we have 
\begin{align*}
\P( \TDP(P,\BH_{\alpha(1-\eta)}^\star) > \TDP(P,\BH_{\alpha}(\wt{\theta},\wt{\sigma})))\leq  \P(\Omega^c) \leq 6 e^{-n^{2/3}}\leq  e^{-\sqrt{n}},
\end{align*}
 for $n$ larger than a universal constant.

\section{Additional proofs}

\subsection{Proof of Theorem~\ref{thm:lower_sigmaknown}}\label{sec:lower_sigmaknown}

We follow the same general  approach as for proving Theorem~\ref{THM:LOWER_GENERAL} (see Section~\ref{sec:proofTHM:LOWER_GENERAL}).

\subsection*{Step 1: Building a least favorable mixture distribution}

Given $\mu\in \R$, let $\phi_{\mu}$ be defined by $\phi_{\mu}(x)=\phi(x-\mu)$ for all $x\in\R$. Let us consider the real measure with density
\beq\label{eq:definition_gmu}
h=(1-\pi_1) \phi + \pi_1 f_1 = (1-\pi_2) \phi_{\mu} + \pi_2 f_2 = \max\{(1-\pi_1) \phi, (1-\pi_2) \phi_{\mu}\},
\eeq
for $\pi_1= k_1/(2n)$ and $\pi_2= k_2/(2n)$ (with $\pi_1\leq \pi_2$ by \eqref{eq:condition_lower_k_0_bus}) and where 
\beqn
f_1& =&\frac{1}{\pi_1}\left[(1-\pi_2)\phi_{\mu}- (1-\pi_1)\phi\right]_+\ ;\quad \quad 
f_2 =\frac{1}{\pi_2}\left[(1-\pi_1)\phi-(1-\pi_2)\phi_{\mu}\right]_+\ . 
\eeqn
Now, we can choose $\mu \in (0,2)$ (as a function of $\pi_1$ and $\pi_2$) such that  $f_1$, $f_2$ and $h$ are probability densities. To see this, it is sufficient to choose $\mu$ with $\int f_2(u)du=1$. Such a $\mu$ always exists because, as a function of $\mu\geq 0$, $\int f_2(u)du$ is continuous with value $(\pi_2-\pi_1)/\pi_2 <1$ for $\mu=0$ and value larger than $ \pi_2^{-1}(1-\pi_2) \int \left[\phi(u)-\phi_{\mu}(u)\right]_+du  \geq 3 \int_{-\infty}^{\mu/2} (\phi(u)-\phi_{\mu}(u)) du = 3({\Phi}(\mu/2) - {\Phi}(-\mu/2))>1 $ for $\mu\geq 2$. 
Hence, we fix in the sequel such a $\mu\in (0,2)$.
\begin{figure}[h!]
\begin{tabular}{ccc}
$h=(1-\pi_1)\phi + \pi_1 f_1$  & $h=(1-\pi_2)\phi_{\mu}+\pi_2 f_2$ \vspace{-0.5cm}\\
 \includegraphics[scale=0.45]{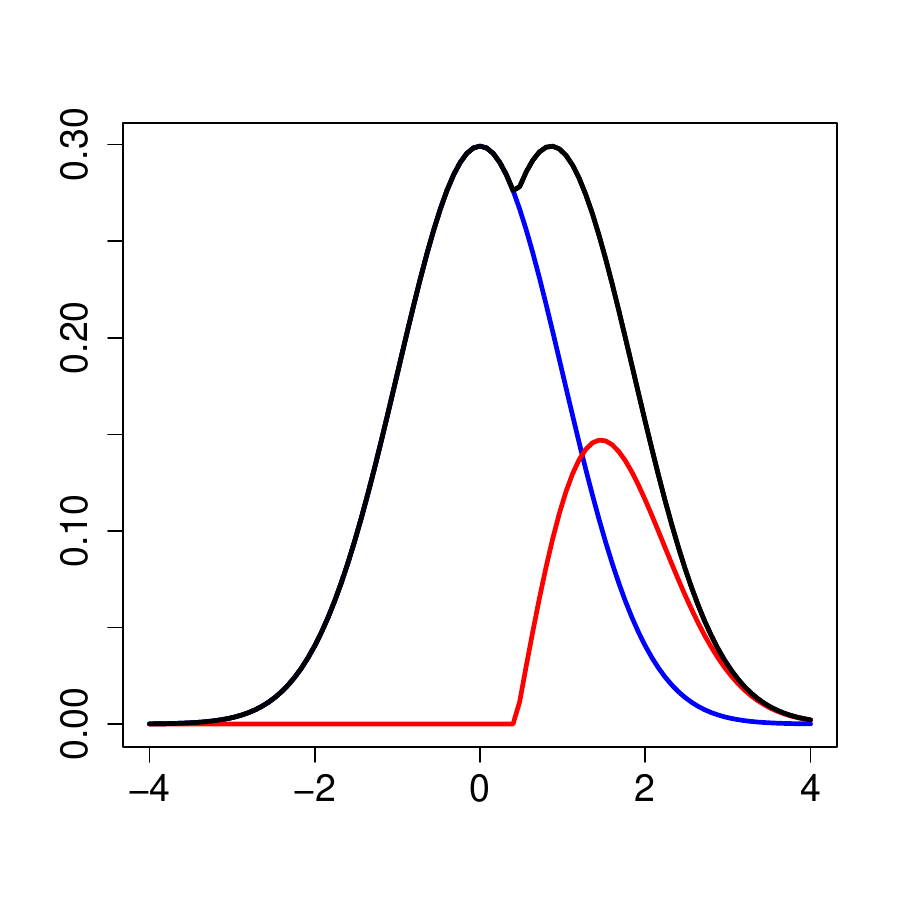}&\hspace{-0.5cm}\includegraphics[scale=0.45]{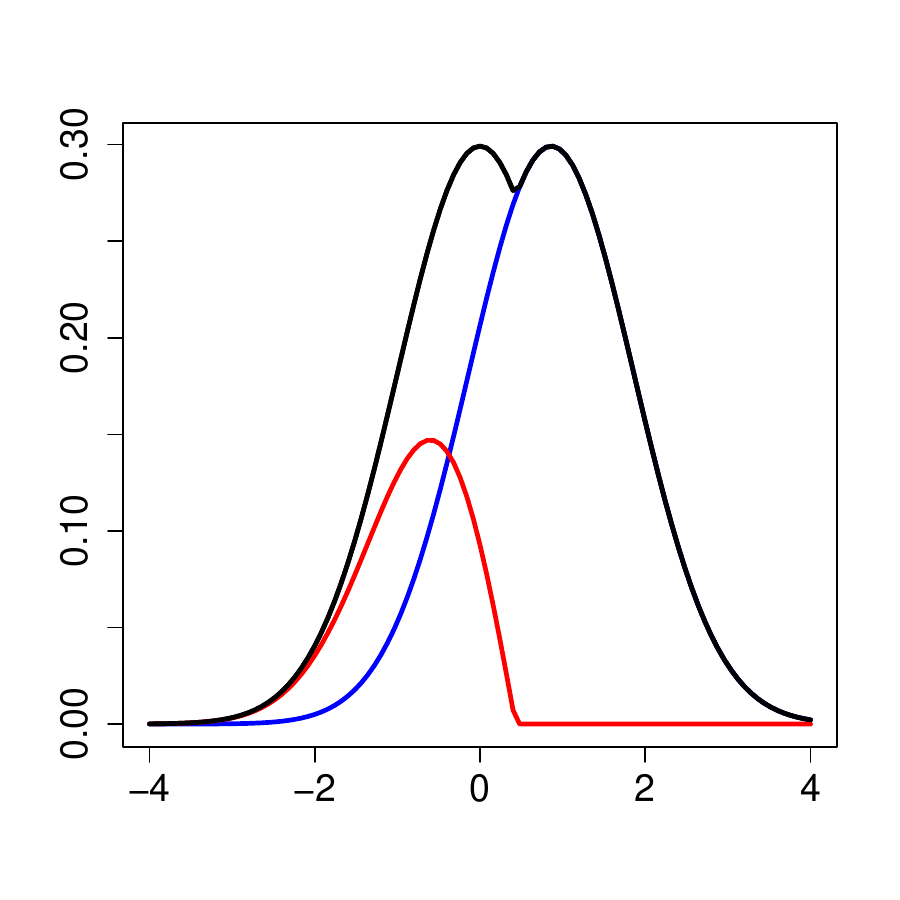}
 \end{tabular}
\vspace{-0.5cm}
\caption{Left: the density $h$ given by \eqref{eq:definition_gmu} (in black) interpreted as a mixture between the null $\mathcal{N}(0,1)$ ($(1-\pi_1) \phi$ in blue) and the alternative $f_1$ ($\pi_1 f_1$ in red). Right: the same $h$ interpreted as a mixture between the null $\mathcal{N}(\mu,1)$ ($(1-\pi_2) \phi_\mu$ in blue) and the alternative $f_2$ ($\pi_2 f_2$ in red). $\pi_1=\pi_2=1/4$, $\mu\approx 1.51$. See the text for the definitions.} \label{lbtheta}
\end{figure}

\medskip

Define $\kappa_0=  \mu^{-1}\log\left[(1-\pi_1)/(1-\pi_2)\right]\geq 0$ and $u_0= \kappa_0+\mu/2$. We deduce from straightforward computations that $f_1(x)>0$ if and only if $x >u_0$ and $f_2(x)>0$ if and only if $x< u_0$. The following lemma states a lower bound for $\mu$ (to be proved at the end of the section).

\begin{lem}\label{lem:lower_mu}
 There exists a numeric constant $c'>0$ such that 
 \[
  \mu \geq c' \frac{\pi_2}{\sqrt{1+\log\left(\frac{\pi_2}{\pi_1}\right)}}.
 \]
\end{lem}

Let $Q$ be the distribution on $\R^n$ associated to the  product density $\prod_{i=1}^nh(x_i)$.  
Let $Q_{1,z}$ (resp. $Q_{2,z}$) be the distribution on $\R^n$ of density $\prod_{i=1}^n  ((1-z_i) \phi+ z_i f_1)$ (resp. $\prod_{i=1}^n  ((1-z_i) \phi_\mu+ z_i f_2)$),  for $z\in\{0,1\}^n$. 
Let $Z_{1,i}$ (resp. $Z_{2,i}$), $1\leq i\leq n$, be  i.i.d. variables with common distribution being a Bernoulli distribution with parameter $\pi_1$ (resp. $\pi_2$). Hence, $Y\sim Q_{1,Z_1}$ (resp. $Y\sim Q_{2,Z_2}$) is distributed as $Q$ unconditionally on $Z_1$ (resp. $Z_2$).  Note that, for any $z$, $\theta(Q_{1,z})=0$ whereas $\theta(Q_{2,z})=\mu$. Besides, we have $\cH_1(Q_{j,z}=\{i: z_i=1\}$. {As in the proof of Theorem~\ref{THM:LOWER_GENERAL}, for $z$ such that $n_1(z)\geq n/2$, we have $Q_{j,z}\notin \cP$, but  we readily extend the definition of $\theta(Q_{1,z})$ and $\sigma(Q_{1,z})$ to that setting.}

\subsection*{Step 2: if $\P_{Y\sim Q}(|R(Y)|>0)\geq  \frac{1}{2}$ then $\FDR(P,R)\geq \frac{1}{5}$ for $n$ larger than a numeric constant, for some $P$ with $n_1(P)\leq k_2$}

Recall that, for any $j\in\{1,2\}$, we have $\FDP(Q_{j,z},R(Y))= \frac{\sum_{i\in R(Y)}\mathds{1}\{z_i=0\}}{|R(Y)|}\mathds{1}\{|R(Y)|>0\}$. We derive from the Fubini theorem that 
\begin{align}
&\E_{Z_1}\left[ \FDR(Q_{1,Z_1}, R) \right]+ \E_{Z_2}\left[ \FDR(Q_{2,Z_2}, R) \right]\\
&=\E_{Z_1}\left[\E_{Y\sim Q_{1,Z_1}}\left[ \FDP(Q_{1,Z_1}, R(Y)) \right] \right]+ \E_{Z_2}\left[\E_{Y\sim Q_{2,Z_2}}\left[ \FDP(Q_{2,Z_2}, R(Y)) \right] \right]\nonumber\\
&= \E_{Y\sim Q}\left[ \frac{\sum_{i\in R(Y)}(\P(Z_{1,i}=0\:|\:Y)  + \P(Z_{2,i}=0\:|\:Y) )}{|R(Y)|}\mathds{1}\{R(Y)>0\}  \right] \nonumber.
\end{align}
From Step 1, we deduce that  $\P(Z_{1,i}=0\:|\:Y)= 1- \pi_1 f_1(Y_i)/h(Y_i)$ and $f_1(y)=0$ for $y\leq u_0$. Similarly, we have $\P(Z_{2,i}=0\:|\:Y)= 1$ for $Y_i\geq u_0$. This entails that, for all $Y$, we have 
 $\P(Z_{1,i}=0\:|\:Y)  + \P(Z_{2,i}=0\:|\:Y) \geq 1$  for all $i$. Hence, we obtain
\begin{equation}\label{equ:borneinfstep1}
 \E_{Z_1}\left[ \FDR(Q_{1,Z_1}, R) \right]+ \E_{Z_2}\left[ \FDR(Q_{2,Z_2}, R) \right]\geq \: \P_{Y\sim Q}(R(Y)>0)\geq 1/2\ .
\end{equation}
Hence, we may assume that $\E_{Z_{j_0}}\left[ \FDR(Q_{j_0,Z_{j_0}}, R) \right]\geq 1/4$ for some $j_0\in\{1,2\}$. 
Then, we apply Chebychev's inequality to obtain
\begin{align*}
\E_{Z_{j_0}}\left[\E_{Y\sim Q_{j_0,Z_{j_0}}}\left[ \FDP(Q_{j_0,Z_{j_0}}, R(Y)) \right] \mathds{1}\{n_1(Z_{j_0})\in [k_{j_0}/4 ; k_{j_0}]\} \right]\geq  1/4 - 8/k_{j_0}.
\end{align*}
and thus 
\begin{align*}
\E_{Z_{j_0}}\left[\E_{Y\sim Q_{j_0,Z_{j_0}}}\left[ \FDP(Q_{j_0,Z_{j_0}}, R(Y)) \right] \mathds{1}\{n_1(Z_{j_0})\leq k_2\} \right]\geq  1/4 - 8/k_{1}\ .
\end{align*}
As a result, for $k_1$ large enough (by Condition \eqref{eq:condition_lower_k_0_bus} for $c_1,c_2$ large enough), there exists $z\in\{0,1\}^n$ such that 
$n_1(z)\leq k_2$
 and 
$\FDR(Q_{1,z}, R)  > 1/5$.

\subsection*{Step 3: If $\P_{Y\sim Q}(|R(Y)|=0)\geq 1/2$, then $R$ is over-conservative, for some $P$ with $n_1(P)\leq k_1$} 

Since $\E_{Z_1}[\P_{Y\sim Q_{1,Z_1}}(|R(Y)|=0)]\geq 1/2$, it follows again from Chebychev's inequality, that for some $z\in\{0,1\}^n$ such that $n_1(z)\in [k_1/4;k_1]$, we have 
$\P_{Y\sim Q_{1,z}}(|R(Y)|=0)\geq 1/2- 8/k_1\geq 2/5$ ($k_1$ being large enough).   In the sequel, we fix $P=Q_{1,z}$ for such a $z$, so that $\theta(P)=0$.

Let $u_1$ be the smallest number such that for all $u\geq u_1$,  one has 
\begin{equation}\label{equf1moy}
 \pi_1 f_1(u) \geq 16\alpha^{-1}\phi(u)\ .
\end{equation}
From the definition of $f_1$, we derive that 
$
 (1-\pi_2)\phi_{\mu}(u_1) = \phi(u_1)[16\alpha^{-1}+ (1-\pi_1)]\ , 
$
which is again equivalent to 
\[
u_1 = \frac{\mu}{2}+ \frac{1}{\mu}\log\left(\frac{1-\pi_1+16\alpha^{-1}}{1-\pi_2}\right) \leq \frac{\mu}{2}+ \frac{1}{\mu}\log\left(2+\frac{16}{\alpha}\right)\ , 
\]
Since $\mu\leq 2$ and by Lemma~\ref{lem:lower_mu}, we  have 
\beq\label{eq:maju1}
u_1 \leq 1+   \frac{\sqrt{1+\log\left(\frac{\pi_2}{\pi_1}\right)}} {c'\pi_2}\log\left(2+\frac{32}{\alpha}\right)= 1+  2n\frac{\sqrt{1+\log\left(\frac{k_2}{k_1}\right)}} {c'k_2}\log\left(2+\frac{32}{\alpha}\right).
\eeq
For a suitable constant $c_2$ in Condition~\eqref{eq:condition_lower_k_0_bus} and for $n$ large enough, we therefore have $u_1\leq \sqrt{\log(n)}$.

Then, it remains to prove that $\BH_{\alpha/2}^*$ rejects many false null hypotheses with probability close to one. Recall that 
 $\BH_{\alpha/2}^*=\{i\in\{1,\dots,n\}\::\:|Y_i|\geq \hat{u}\}$ for 
\[
\hat{u}=\min\Big\{u\in \R_+\::\: \sum_{i=1}^n\mathds{1}\{|Y_i|\geq u\} \geq 4 n \alpha^{-1} \ol{\Phi}(u) \Big\}. 
\]
Define $N= \sum_{i\in \cH_1(P_1)}\mathds{1}\{|Y_i|\geq u_1\}$.  Arguing exactly as in Step 3 of the proof of Theorem~\ref{THM:LOWER_GENERAL} (see Section~\ref{sec:proofTHM:LOWER_GENERAL}), we conclude as in 
\eqref{eq:conclusion_step_3} that 
\[
\P_{Y\sim P}(\hat{u}\leq u_1, N\geq 4 n  \alpha^{-1} \ol{\Phi}(u_1)) \geq 1-e^{-3q/28} = 1-e^{-c'n  \alpha^{-1} \ol{\Phi}(u_1)} \ ,
\]
where $n\ol{\Phi}(u_1)\geq c''\sqrt{n/\log(n)}$. We have proved that 
\[
\P_{Y\sim P}\left[|\BH_{\alpha/2}^*\cap \cH_1(P)|\geq c_4\alpha^{-1}\sqrt{\frac{n}{\log(n)}}\right]\geq 1 - e^{-c'  \alpha^{-1} c''\sqrt{n/\log(n)}}\ ,
\]
whereas $|R(Y)|=0$ with probability higher than $2/5$.

\begin{proof}[Proof of Lemma~\ref{lem:lower_mu}]

Since $\int f_1(x)dx=1$, we deduce from the definition of $\kappa_0$ that 
\begin{eqnarray}
 \pi_1&= &\int \left[(1-\pi_2)\phi_{\mu}(x)- (1-\pi_1)\phi(x)\right]_+dx \nonumber\\ &=&  [1- \pi_2]\overline{\Phi}[\kappa_0- \mu/2] - [1-\pi_1] \overline{\Phi}[\kappa_0+ \mu/2] \nonumber \\
 & = & -[\pi_2-\pi_1]\overline{\Phi}[\kappa_0+ \mu/2]+ [1- \pi_2]\int_{\kappa_0- \mu/2}^{\kappa_0+ \mu/2}\phi(x)dx\ . \label{eq:upper_1_pi_1}
\end{eqnarray}

Recall that $\pi_2> \pi_1$. By integration, we derive that 
\beqn 
\frac{\pi_1}{1-\pi_2}&\leq& \int_{\kappa_0- \mu/2}^{\kappa_0+ \mu/2}\phi(x)dx
\leq \phi(\kappa_0)\int_{-\mu/2}^{\mu/2} e^{\kappa_0 x}dx\\ &\leq & \frac{\phi(\kappa_0)}{\kappa_0}\frac{\pi_2-\pi_1}{\sqrt{(1-\pi_1)(1-\pi_2)}}
\leq \frac{\phi(\kappa_0)}{\kappa_0}\frac{\pi_2-\pi_1}{1-\pi_2}\ . 
\eeqn 
Hence, we conclude that 
\beq
\frac{\phi(\kappa_0)}{\kappa_0} \geq \frac{\pi_1}{\pi_2-\pi_1}\ .  \label{eq:upper_1_pi_1_2}
\eeq 

\medskip 

\noindent 
  {\bf Case 1}: $\frac{\pi_1}{\pi_2-\pi_1}\geq  \phi(0)e^{-1/2}$. Since $\phi(\kappa_0)\leq \phi(0)$, we deduce from the definition of $\kappa_0$ 
\beq\label{eq:conclusition_minoration_pi}
 \mu \geq   \frac{\pi_1\log\left(1+\frac{\pi_2-\pi_1}{1-\pi_2}\right)}{\phi(0)(\pi_2-\pi_1)}\geq \pi_1\geq c'\frac{\pi_2}{\sqrt{1+\log(\pi_2/\pi_1)}} \ , 
\eeq
for a suitable constant $c'$ since $\pi_2\leq 1/2$, $\log(1+x)\geq  x/2$ for $x\in[0,1]$ and we assume that $\pi_2/\pi_1\leq 1+ e^{1/2}/\phi(0)$. 
Note that, for $\pi_1= \pi_2$, we also easily derive from~\eqref{eq:upper_1_pi_1} that \eqref{eq:conclusition_minoration_pi} also holds. 

\medskip

\noindent 
{\bf Case 2}: $\frac{\pi_1}{\pi_2-\pi_1}< \phi(0)e^{-1/2}$. We deduce from~\eqref{eq:upper_1_pi_1_2}  that either $\kappa_0\leq 1$ or 
 $\phi(\kappa_0)\geq \frac{\pi_1}{\pi_2-\pi_1}$, which in turn implies that
\[
\kappa_0 \leq  \sqrt{2\log\left(\frac{\phi(0)[\pi_2-\pi_1]}{ \pi_1}\right)}\ . 
\]
From the definition of $\kappa_0$, we derive that 
\[
\mu \geq  \frac{\log\left(1+ \frac{\pi_2-\pi_1}{1-\pi_2}\right)}{\sqrt{2\log\left( \frac{\phi(0)[\pi_2-\pi_1]}{\pi_1}\right)}}\geq \frac{\pi_2-\pi_1}{2\sqrt{2[1+ \log(\pi_2/\pi_1)]}}\geq c' \frac{\pi_2}{2\sqrt{2[1+ \log(\pi_2/\pi_1)]}}\ ,
\]
for a suitable constant $c'$ since $\pi_2/\pi_1$ is bounded away from one. 
 
\end{proof}

\subsection{Proof of Theorem~\ref{THM:THETACHAPFORFDR}}

Without loss of generality, we restrict ourselves to distributions $P$ such  that $\sigma(P)=1$. 
Let $\wh{\theta}$ be any estimator of $\theta(P)$ and assume that Condition \eqref{eq:condition_lower_k_0} holds.

\subsection*{Step 1: building a least favorable mixture distribution}

We use the same mixture distribution as in the proof of Theorem~\ref{thm:lower_sigmaknown}. Consider the density $h$~\eqref{eq:definition_gmu} with $\pi_1=\pi_2=\pi=k/2n$, and
\[
h=(1-\pi) \phi + \pi f_1 = (1-\pi) \phi_{\mu} + \pi f_2 = \max\{(1-\pi) \phi, (1-\pi) \phi_{\mu}\}, 
\]
 where 
\beqn
f_1& =&\frac{1}{\pi}\left[(1-\pi)\phi_{\mu}- (1-\pi)\phi\right]_+\ ;\quad \quad 
f_2 =\frac{1}{\pi}\left[(1-\pi)\phi-(1-\pi)\phi_{\mu}\right]_+\ . 
\eeqn
Recall that $\mu\in (0,2)$ is chosen in such a way that $f_1$ and $f_2$ are densities. 
Also recall the probability measures $Q$, $Q_{1,z}$, and $Q_{2,z}$ introduced in the previous proof (see Section~\ref{sec:lower_sigmaknown}). Also let $Z_{i}$, $1\leq i\leq n$, be  i.i.d. variables with common distribution being a Bernoulli distribution with parameter $\pi$.
 For any event $A$, we have
$
 \P_{Y\sim Q}[A]= \E_{Z}[\P_{Y\sim Q_{1,Z}}(A)]= \E_{Z}[\P_{Y\sim Q_{2,Z}}(A)]
$. 
Consider the  events  
$$\Omega^{-} = \{\widehat{\theta}(Y)\geq \mu/2\}\,;\:\:\:
 \Omega^{+}=\{\widehat{\theta}(Y)\leq \mu/2\}
.$$  Either $\E_{Z}[\P_{Y\sim Q_{1,Z}}(\Omega^{-})]\geq 1/2$ or $\E_{Z}[\P_{Y\sim Q_{2,Z}}(\Omega^{+})]\geq 1/2$. We assume without loss generality that $\E_{Z}[\P_{Y\sim Q_{1,Z}}(\Omega^{-})]\geq 1/2$ the other case being handled similarly. 

Since $n_1(Z)$ follows a Binomial distribution with parameters $n$ and $\pi$, it follows from Bernstein's inequality that 
\[
|n_1(Z)-\pi n|\leq   \sqrt{2n\pi \log(n) } + \frac{\log(n)}{3}\leq n/4\ , 
\]
with probability higher than $1-2/n$, for $n$ large enough. Hence, there exists $z\in\{0,1\}^n$ such that for $P=Q_{1,z}$ we have 
\beq \label{eq:lower_bound_n1(P)}
n_1(P)\in \Big[\pi n -  \sqrt{2n\pi \log(n) } - \frac{\log(n)}{3}; n/2\Big]\quad \text{ and }\quad \P_{Y\sim P}[\Omega^{-}]\geq \frac{1}{2}- \frac{2}{n}.
\eeq
Note that $\theta(P)= 0$ whereas, on $\Omega^{-}$, $\widehat{\theta}$ is larger or equal to $\mu/2$. In the remainder of the proof, we quantify how this estimation error shifts the distribution of the rescaled $p$-values.

\subsection*{Step 2: translated $p$-values} Since, under $\Omega^-$, we have $\theta(P)= 0$ and $\widehat{\theta} \geq \mu/2$, the rescaled $p$-values are shifted and do not follow an uniform distribution. Let us characterize this shift.
We have for all $t\in[0,1]$, and all $i\in\{1,\dots,n\}$,
\begin{align*}
 \mathds{1}\{p_i(Y;\wh{\theta},1) \leq t\}
=&\: \mathds{1}\{|Y_i-\wh{\theta}| \geq \ol{\Phi}^{-1}(t/2)\}\\
\geq & \mathds{1}\{Y_i-\wh{\theta} \leq - \ol{\Phi}^{-1}(t/2)\}\\
\geq&\:   \mathds{1}\{Y_i  \leq \wh{\theta}- \ol{\Phi}^{-1}(t/2)\}\\
\geq &\:   \mathds{1}\{p_i^- \leq \ol{\Phi}\big[\ol{\Phi}^{-1}(t/2) -\wh{\theta}\big]\}\ ,
\end{align*}
where we have denoted $p_i^-=\ol{\Phi}(-Y_i)$.
Let $\Psi(t)=\ol{\Phi}(\ol{\Phi}^{-1}(t/2)- \widehat{\theta})$ and $\Psi_1(t)=\ol{\Phi}(\ol{\Phi}^{-1}(t/2)-\mu/2)$, $t\in[0,1]$. 
On the event $\Omega^-$, we have $\Psi(t)\geq \Psi_1(t)$ for any $t\in [0,1]$.
 This entails that, for all $i\in\{1,\dots,n\}$ and $t\in [0,1]$,
\begin{align}\label{equprooflowerboundfalsereject}
& \mathds{1}\{p_i(Y;\wh{\theta},1) \leq t\}\geq  \mathds{1}\{p_i^- \leq \Psi(t)\}\geq  \mathds{1}\{p_i^- \leq \Psi_1(t)\}  \mbox{ for $Y\in\Omega^-$}\ .
\end{align}
{Interestingly, for $i\in \cH_0$, the $p_i^{-}$'s  are all i.i.d. $U(0,1)$. In contrast to $\Psi$, the function $\Psi_1$ does not depend on the $Y_i$'s. }

\subsection*{Step 3:  with high probability, on $\Omega^-$,  the threshold of $\BH_\alpha(Y;\wh{\theta},1)$ is large}
{Since the rescaled $p$-values are shifted, one should expect that a large number of them are small enough so that  $\BH_\alpha(Y;\wh{\theta},1)$ rejects many hypotheses. }
Let us denote by $\wh{T}_\alpha=T_\alpha(Y;\wh{\theta},\sigma(P)) \vee (\alpha/n)$ the $p$-value  threshold of $\BH_\alpha(Y;\wh{\theta},1)$. {In view of~\eqref{equprooflowerboundfalsereject}, we  consider the empirical distribution function of $p_i^{-}$, $i\in \cH_0$, given by }  
\[
\wh{\mathbb{G}}_0^-(t)=(n_0(P))^{-1} \sum_{i\in \cH_0(P)}\mathds{1}\{p_i^- \leq t\}, \:\:\:t\in[0,1]\ . 
\]
Relying on the  DKW inequality (Lemma~\ref{lem:DKW}), we derive that this process is uniformly bounded. Precisely, we have $\P(\Omega_0^-)\geq 1-1/n$, where
\[
\Omega_0^- = \left\{\sup_{t\in[0,1]}|\wh{\mathbb{G}}_0^-(t)-t| \leq \sqrt{\log (2n)/(2n_0(P))}\right\}\ . 
\]
 Now, a consequence of \eqref{equprooflowerboundfalsereject} is that 
\begin{align}
\alpha \geq \wh{T}_\alpha &= 
\max\left\{ t \in[0,1]\::\: \sum_{i=1}^n \mathds{1}\{p_i(Y;\wh{\theta},\sigma(P)) \leq t\} \geq  n t/\alpha\right\}
\nonumber\\
&\geq  \max\left\{ t \in[0,1]\::\: \sum_{i=1}^n \mathds{1}\{p_i^- \leq \Psi_1(t)\} \geq   nt/\alpha\right\}
\nonumber\\
&\geq  \max\left\{ t \in[0,1]\::\: (n_0(P)/n) \wh{\mathbb{G}}_0^-(\Psi_1(t)) \geq   t/\alpha\right\}
\geq T_0^-
,\label{equTchapalpha}
\end{align}
by letting $T_0^{-}= \max\{ t \in[0,1]\::\:  \wh{\mathbb{G}}_0^-(\Psi_1(t)) \geq   2 t/\alpha\}$. On $\Omega_-$, $\wh{\mathbb{G}}_0^-(\Psi_1(t))$ is uniformly close to $\Psi_1(t)$, which will allow us to get a lower bound of $\Psi_1(T_0^-)$. This argument is formalized in 
 Lemma~\ref{lem:convunifprocess} below. 
 
 \begin{lem}\label{lem:convunifprocess}
 There exists an integer $N=N(\alpha)$ such that if $n_0(P) \geq N$ and
\begin{equation}
\mu\geq 
\frac{4\log(32/\alpha)}{\sqrt{0.25\log (n_0(P)) + \log(8/\alpha)}}\ ,\label{equ-prooflowerboundeta}
\end{equation}
 we have  on the event  $\Omega_0^-$, 
\begin{equation}\label{equ:lowerboundT0m}
\Psi_1(T_0^-)\geq n_0(P)^{-1/4}\ . 
\end{equation}

\end{lem}

By Lemma~\ref{lem:lower_mu}, we have $\mu\geq c'\pi= c'k/(2n)$. Hence,  Condition \eqref{equ-prooflowerboundeta} is satisfied by 
Condition \eqref{eq:condition_lower_k_0} together with $n_0(P)\geq 0.5n$.  Combining \eqref{equTchapalpha} and \eqref{equ:lowerboundT0m} and since $\Psi_1$ is increasing, 
 we finally have on the event $\Omega^-\cap \Omega^-_0$,
\beq \label{eq:condition_lower_p_value_empirique}
 \Psi(\widehat{T}_{\alpha}) \geq \Psi_1(\widehat{T}_{\alpha})\geq \Psi_1(T_0^-) \geq n_0(P)^{-1/4}\ .
\eeq

\subsection*{Step 4: $\BH_\alpha(Y;\wh{\theta},1)$  makes  many false rejections}
Since the threshold $\widehat{T}_{\alpha}$ is large enough, one can then prove that $\BH_\alpha(Y;\wh{\theta},1)$  makes  many false rejections.
By \eqref{equprooflowerboundfalsereject}, we have, on the event $\Omega^-\cap \Omega^-_0$, that for all $t\in[0,1]$,
\beqn 
\sum_{i\in \cH_0(P)}\mathds{1}{\{p_i(Y;\wh{\theta},1)\leq t\}}  &=& \sum_{i\in \cH_0(P)}\mathds{1}{\{Y_i \leq - \overline{\Phi}^{-1}(t/2)+  \widehat{\theta}\}}+ \mathds{1}{\{Y_i \geq \overline{\Phi}^{-1}(t/2)+ \widehat{\theta}\}} \\
&=  & \sum_{i\in \cH_0(P)} \mathds{1}{\{p_i^{-}\leq \Psi(t)\}}+ \mathds{1}{\{p_i^{+}\leq \Psi^+(t)\}}\ , 
\eeqn 
where $\Psi^{+}(t)= \overline{\Phi}[\overline{\Phi}^{-1}(t/2)+ \widehat{\theta}]$ and $p_i^+= \overline{\Phi}(Y_i)$. Define the process 
\[
\wh{\mathbb{G}}_0^+(t)=(n_0(P))^{-1} \sum_{i\in \cH_0(P)}\mathds{1}\{p_i^+ \leq t\}, \:\:\:t\in[0,1]\ . 
\]
Relying again on the DKW inequality (Lemma~\ref{lem:DKW})), we derive that this process is uniformly bounded in the sense that $\P(\Omega_0^-)\geq 1-1/n$, where
\[
\Omega_0^+ = \left\{\sup_{t\in[0,1]}|\wh{\mathbb{G}}_0^+(t)-t| \leq \sqrt{\log (2n)/(2n_0(P))}\right\}\ . 
\]
Hence, on $\Omega^{-}\cap \Omega_0^{-}\cap\Omega_0^{+}$, we have, uniformly over all $t\in [0,1]$,
\[
\sum_{i\in \cH_0(P)}\mathds{1}{\{p_i(Y;\wh{\theta},1)\leq t\}}  \geq  n_0(P) [\Psi(t)+\Psi^{+}(t)] - \sqrt{2n\log (2n) }\ . 
\]
By~\eqref{eq:lower_bound_n1(P)}, $n_0(P)\geq n(1-\pi) - \sqrt{2n\pi\log(n)}- \log(n)/3\geq n(1-\pi)- 2\sqrt{2n\log(n)}$. Hence, we conclude that, uniformly over all $t\in [0,1]$,  
\beq
\sum_{i\in \cH_0(P)}\mathds{1}{\{p_i(Y;\wh{\theta},1)\leq t\}} 
\geq  n(1-\pi) [\Psi(t)+\Psi^{+}(t)]- 5\sqrt{2n\log(2n)}\ .  \label{prooflowerboundlbH0}   
\eeq

In the previous step (see \eqref{eq:condition_lower_p_value_empirique}), we have proved that 
$\Psi(\widehat{T}_{\alpha}) \geq  n_0(P)^{-1/4}$. This implies that 
\[
\big|\BH_\alpha(Y;\wh{\theta},1)\cap \cH_0(P)\big|\geq \frac{3}{4}n^{3/4}-  5\sqrt{2n\log(2n)}\geq \frac{1}{2}n^{3/4}\ ,
\]
for $n$ large enough. This proves the first statement of the theorem.

\subsection*{Step 5: with high probability, on $\Omega^-$, $\BH_\alpha(Y;\wh{\theta},1)$ cannot make too many true rejections}
In this step, we  bound the number of true rejections uniformly with  respect to the threshold $t$ of the testing procedure. 
We have for all $t\in[0,1]$, and all $i\in\{1,\dots,n\}$,
\beqn 
\mathds{1}{\{p_i(Y;\wh{\theta},1)\leq t\}}&=& \mathds{1}{\{Y_i \leq - \overline{\Phi}^{-1}(t/2)+  \widehat{\theta}\}}+ \mathds{1}{\{Y_i \geq \overline{\Phi}^{-1}(t/2)+ \widehat{\theta}\}} \\
&=  & \mathds{1}{\{Y_i \leq - \overline{\Phi}^{-1}(t/2)+  \widehat{\theta}\}}+ \mathds{1}{\{Y_i \geq \overline{\Phi}^{-1}(t/2)+ \widehat{\theta}\}}.
\eeqn 
Now, recall that the variables $Y_i, i\in \cH_1(P),$ are  i.i.d. with common density $\frac{1-\pi}{\pi}(\phi_{\mu}-\phi)_+$. For $t\in[0,1]$,  define 
\begin{align*}
\wh{\mathbb{G}}_1^-(t)&=(n_1(P))^{-1} \sum_{i\in \cH_1(P)}\mathds{1}\{\ol{\Phi}(-Y_i) \leq t\}, \:\:\:{G}_1^-(t)=\E_{P} [ \wh{\mathbb{G}}_1^-(t)]\ ; \\
\Omega_1^- &= \left\{\sup_{t\in[0,1]}(\wh{\mathbb{G}}_1^-(t)-{G}_1^-(t)) \leq \sqrt{\log (n)/(2n_1(P))}\right\}\ ;\\
\wh{\mathbb{G}}_1^+(t)&=(n_1(P))^{-1} \sum_{i\in \cH_1(P)}\mathds{1}\{\ol{\Phi}(Y_i) \leq t\}, \:\:
{G}_1^+(t)=\E_{P} [ \wh{\mathbb{G}}_1^+(t)]\ ; \\
\Omega_1^+ &= \left\{\sup_{t\in[0,1]}(\wh{\mathbb{G}}_1^+(t)-{G}_1^+(t)) \leq \sqrt{\log (n)/(2n_1(P))} \right\}\ .
\end{align*}
Applying twice the DKW inequality (Lemma~\ref{lem:DKW}),   we have $\P(\Omega_1^-)\geq 1 - 1/n$ and $\P(\Omega_1^+)\geq 1 - 1/n$, which gives that the event $\Omega_1= \Omega_1^+\cap \Omega_1^-$ is such that $\P(\Omega_1)\geq 1- 2/n$.
Furthermore, we have for all $t\in[0,1]$, and $i\in\cH_1(P)$, 
\begin{align*}
{G}_1^-(t)&=\P(\ol{\Phi}(-Y_i) \leq t )=\P(Y_i \leq -\ol{\Phi}^{-1}(t) )= \frac{1-\pi}{\pi}\int_{-\infty}^{-\ol{\Phi}^{-1}(t)} (\phi(x-\mu)-\phi(x))_+ dx\\ & \leq  \frac{1-\pi}{\pi}\overline{\Phi}[\overline{\Phi}^{-1}(t)+\mu]\ ;
\\
{G}_1^+(t)&=\P(\ol{\Phi}(Y_i) \leq t )=\P(Y_i \geq  \ol{\Phi}^{-1}(t) )= \frac{1-\pi}{\pi}\int_{\ol{\Phi}^{-1}(t)}^{\infty} (\phi(x-\mu)-\phi(x))_+ dx \\&  \leq  \frac{1-\pi}{\pi}\overline{\Phi}[\overline{\Phi}^{-1}(t)-\mu]\ . 
\end{align*}
 As a result, on the event $\Omega_1$, we have for all $t\in[0,1]$,
\beqn
&&\hspace{-1cm}\sum_{i\in \cH_1(P)}\mathds{1}{\{p_i(Y;\wh{\theta},1)\leq t\}}\\
&=  & n_1(P)\: \wh{\mathbb{G}}_1^-\left(\Psi(t)\right) + n_1(P) \:\wh{\mathbb{G}}_1^+\left(\Psi^{+}(t)\right)\\
&\leq & n_1(P)\left[ G_1^{-}\left(\Psi(t)   \right)+ G_1^{+}\left(\Psi^{+}(t) \right)\right]  +\sqrt{2n_1(P)\log(n)}\\
&\leq & n_1(P)\frac{1-\pi}{\pi}\left[ \overline{\Phi}[\ol{\Phi}^{-1}(t/2)- \widehat{\theta}+\mu]  +\overline{\Phi}[\overline{\Phi}^{-1}(t/2)+ \widehat{\theta}-\mu] \right] \\ &&  +\sqrt{2n_1(P)\log(n)}\ . 
\eeqn

Now, since for any fixed $t\in[0,1]$, the map $$u\in [0,\infty] \mapsto \overline{\Phi}\left[\overline{\Phi}^{-1}(t/2)+u\right]+ \overline{\Phi}\left[ \overline{\Phi}^{-1}(t/2) -  u \right]$$ is nondecreasing, we have on  $\Omega^-\cap \Omega^-_0\cap\Omega_1$ that 
\beqn 
\lefteqn{\overline{\Phi}\left[\overline{\Phi}^{-1}(t/2)+ |\widehat{\theta} - \mu|\right]+ \overline{\Phi}\left[ \overline{\Phi}^{-1}(t/2) -   |\widehat{\theta} -  \mu|  \right]}&&\\ &\leq& \overline{\Phi}\left[\overline{\Phi}^{-1}(t/2)+ \widehat{\theta} \right]+ \overline{\Phi}\left[ \overline{\Phi}^{-1}(t/2) -   \widehat{\theta}   \right] = \Psi(t)+ \Psi^{+}(t)\ ,
\eeqn 
by using $|\widehat{\theta} - \mu|\leq \widehat{\theta}$ (since $\wh{\theta}\geq \mu/2$) and the definition of $\Psi(t)$.
Finally, on $\Omega^-\cap  \Omega_0^{-}\cap \Omega_1$, we obtain that, for all $t\in[0,1]$,
\[
\sum_{i\in \cH_1(P)}\mathds{1}{\{p_i(Y;\wh{\theta},1)\leq t\}} 
\leq  \frac{n_1(P)(1-\pi)}{\pi} [\Psi(t)+\Psi^{+}(t)]+ \sqrt{2\log(n)}\ . 
\]
From~\eqref{eq:lower_bound_n1(P)} and since $\pi\geq \log^{-1/2}(n)$, we deduce that 
\[
\frac{n_1(P)(1-\pi)}{\pi} \leq n(1-\pi)+  \sqrt{\frac{2n\log(n)}{\pi}}+ \frac{\log(n)}{3\pi}\leq n(1-\pi)+2\sqrt{2n}\log(n)\ . 
\]
Hence, we conclude that, for all $t\in[0,1]$,
\beq
\sum_{i\in \cH_1(P)}\mathds{1}{\{p_i(Y;\wh{\theta},1)\leq t\}} 
\leq  n(1-\pi)  [\Psi(t)+\Psi^{+}(t)]+ 5\sqrt{2n}\log(n)\ .
\label{prooflowerboundmajH1}
\eeq

\subsection*{Step 6: On $\Omega^-\cap \Omega_0^-\cap \Omega_1$, the FDP of $\BH_\alpha(Y;\wh{\theta},1)$ is  close to $1/2$}
We deduce from the previous steps a lower bound for the FDP. First, by definition of the FDP and of the threshold $\widehat{T}_{\alpha}$  of $\BH_\alpha(Y;\wh{\theta},1)$ (note that the procedure rejects at least one true null hypotheses by \eqref{eq:condition_lower_p_value_empirique} and \eqref{prooflowerboundlbH0}), we deduce from \eqref{prooflowerboundlbH0} and   \eqref{prooflowerboundmajH1} that, for $Y\in\Omega^-\cap \Omega_0^-\cap \Omega_1$,
\beqn
&&\hspace{-1cm}\FDP(P,\BH_\alpha(Y;\wh{\theta},1)) \\
&\geq& \inf_{t\in [\widehat{T}_{\alpha}, \alpha]}
\frac{\sum_{i\in \cH_0(P)}\mathds{1}{\{p_i(Y;\wh{\theta},1))\leq t\}}}{\sum_{i=1}^n\mathds{1}{\{p_i(Y;\wh{\theta},1)\leq t\}}} \\
&=   & \inf_{t\in [\widehat{T}_{\alpha}, \alpha]} \left[1+ \frac{\sum_{i\in \cH_1(P)}\mathds{1}{\{p_i(Y,\widehat{\theta},1)\leq t\}} }{\sum_{i\in \cH_0(P)}\mathds{1}{\{p_i(Y,\widehat{\theta},1)\leq t\}}} \right]^{-1} \\
&\geq & \inf_{t\in [\widehat{T}_{\alpha}, \alpha]}\left[1+ \frac{n(1-\pi)[\Psi(t)+\Psi^{+}(t)]+ 5\sqrt{2n}\log(n) }{(n(1-\pi)[\Psi(t)+\Psi^{+}(t)]- 5\sqrt{2n}\log(n)))_+ }\right]^{-1}\\
&\geq & \inf_{t\in [\widehat{T}_{\alpha}, \alpha]}\left[2 + \frac{10\sqrt{2n}\log(n)}{(n(1-\pi)[\Psi(t)+\Psi^{+}(t)]- 5\sqrt{2n}\log(n)))_+}\right]^{-1}\ .
\eeqn 
We have proved in \eqref{equ:lowerboundT0m} that $\Psi(t)\geq \Psi(\widehat{T}_{\alpha})\geq n_0(P)^{-1/4}\geq (n/2)^{-1/4}$. Hence, for $n$ large enough, we obtain
\[
 \FDP(P,\BH_\alpha(Y;\wh{\theta},1)) \geq \frac{1}{2+ c'n^{-1/5}}\ .
\]
Since the event $\Omega^{-}\cap \Omega_0^-\cap \Omega_1$ occurs with probability higher than $1/2-c''/n$, we have proved the second statement of the theorem.

\subsection{Proof of Theorem~\ref{mainthsigmaknown} (ii)}\label{sec:proof_mainth_sigma_known_2}

As explained above, Theorem~\ref{mainthsigmaknown} (ii) is shown by adapting Theorem~\ref{THMBORNESUP} to the case $\wh{\sigma}=\sigma(P)$, the only difference in the proof being
that we now quantify the impact of the mean $\theta(P)$  estimation error on the $p$-values and the corresponding threshold $T_{\alpha}(\wh{\theta},\sigma(P))$ of plug-in BH. In other words, Lemma~\ref{lem:forFNR} has to be replaced by  Lemma~\ref{lem:forFNRsigmaknown} below (to be proved in Appendix~\ref{sec:remaining}).

\begin{lem}\label{lem:forFNRsigmaknown}
For any estimator $\wh{\theta} $, let $\delta>0$, $\theta\in \R$, $\sigma>0$, $\alpha\in (0,0.8)$, $t_0\in(0,\alpha)$ and 
\[
\eta =  \delta c \sqrt{ 2\log (1/t_0)} \ ,
\]
with the constant $c>0$ of Corollary~\ref{cor:pvaluemaj}.
 Assume $|\wh{\theta}-\theta|\leq \sigma \delta$ and  $\eta\leq 0.05$. Then, for all $i\in\{1,\dots,n\}$,
 \begin{itemize}
\item if $T_\alpha(\wh{\theta},\wh{\sigma})\vee (\alpha/n)\geq t_0$, we have
 \begin{align}
\mathds{1}\{p_i(\wh{\theta},\sigma)\leq T_{\alpha}(\wh{\theta},\sigma)\} 
&\leq \mathds{1}\{p_i(\theta,\sigma)\leq (1+\eta)T_{\alpha}(\wh{\theta},\sigma) \}\nonumber\\
&\leq \mathds{1}\{p_i(\theta,\sigma)\leq T_{\alpha(1+\eta)}(\theta,\sigma)\}\ ;
\label{equ:forFDR2}
\end{align}
\item if $T_{0.95\alpha}(\theta,\sigma)\vee (0.95\alpha/n)\geq t_0$, we have
 \begin{align}
\mathds{1}\{p_i(\theta,\sigma)\leq T_{\alpha(1-\eta)}(\theta,\sigma)\}  \leq \mathds{1}\{p_i(\wh{\theta},\sigma)\leq T_{\alpha}(\wh{\theta},\sigma)\}\ .
\label{equ:forFNR2}
\end{align}
\end{itemize}\end{lem}

\subsection{Proofs for location model}\label{proof:locationsection}

\begin{proof}[Proof of Theorem~\ref{thm:lower_location}]
Consider any procedure $R$. 
We extend the proof of Theorem~\ref{thm:lower_sigmaknown} (see Section~\ref{sec:lower_sigmaknown}) to general function $g$  but in the specific case $k_1=k_2=k$. Take $\pi=k/(2n)$. Let $\mu>0$ be such that 
\[
\int_{0}^{\mu/2}g(x)dx= \frac{\pi}{2(1-\pi)}< 1/2\ .
\]
Since $g$ is non-increasing on $\R_+$, we have $\mu= 2 \ol{G}^{-1}(\frac{1-2\pi}{2(1-\pi)})\geq c'_{g}\pi$.  
Then, one can check that 
\[
h=(1-\pi) g + \pi f_1 = (1-\pi) g_{\mu} + \pi f_2 = \max\{(1-\pi) g, (1-\pi) g_{\mu}\}\ , 
\]
is a density where 
\beqn
f_1& =&\frac{1}{\pi}\left[(1-\pi)g_{\mu}- (1-\pi)g\right]_+\ ;\quad \quad 
f_2 =\frac{1}{\pi}\left[(1-\pi)g-(1-\pi)g_{\mu}\right]_+\ . 
\eeqn
Since $g$ is non-increasing on $\R_+$, $f_1(x)> 0$ for $x> \mu/2$ and $f_2(x)>0$ for $x < \mu/2$.

Defining the distributions $Q$, $Q_{1,z}$, $Q_{2,z}$ and $Z$ as in the proof of Theorem~\ref{thm:lower_sigmaknown} (with $k_1=k_2$, $Z=Z_1=Z_2$), we derive that 
either $$\E_{Z}[\E_{Y\sim Q_{1,Z}}[|R(Y)|>0]=\E_{Z}[\E_{Y\sim Q_{2,Z}}[|R(Y)|>0]\geq 1/2$$ or $$\E_{Z}[\E_{Y\sim Q_{1,Z}}[|R(Y)|=0]=\E_{Z}[\E_{Y\sim Q_{2,Z}}[|R(Y)|=0]> 1/2.$$ In the former case, we obtain by arguing exactly as in step $2$ of the previous proof that there exists $P$ with $n_1(P)\leq k$ and $\FDR(P,R)\geq 2/5$. 

\medskip 

Let us now turn  to the case $\E_{Z}[\E_{Y\sim Q_{1,Z}}[|R(Y)|=0]]\geq  1/2$. Hence, there exists $P$ with $n_1(P)\in [k/4,k]$ and $\theta(P)=0$ such that $\P_{Y\sim P}(|R(Y)|=0]\geq 2/5$  (for $k$ large enough). It remains to prove that the oracle Benjamini-Hochberg $\BH^{\star}_{\alpha/2}$ rejects many null hypotheses with probability close to one. It suffices to prove that many oracle $p$-values $p_i^{\star}=2\ol{G}(|Y_i|)$ are small enough. Consider any $i\in \cH_1(P)$. The corresponding density of $Y_i$ is given by $f_1(x)= \frac{1-\pi}{\pi}[g(x-\mu)- g(x)]_+$, which is positive for $x\geq \mu/2$.
Consider  some $t\in [\alpha/(2n);\alpha/2]$ whose value will be fixed later.
The event $\{p_i^{\star}\leq t\}\subset \{Y_i\geq \ol{G}^{-1}\left(\frac{t}{2}\right)\}$ occurs with probability higher  or equal to 
\[
r_{\mu}(t)=  \frac{1-\pi}{\pi}\left[ \ol{G}\left[\ol{G}^{-1}\left(\frac{t}{2}\right)-\mu\right]- \frac{t}{2}\right]\ .
\]
Applying  Bernstein inequality, we deduce that  
\beq\label{eq:Bernstein_t_fixe}
 \P_{Y\sim P}\left[\sum_{i\in \cH_1(p)}\ind\{p_i^{\star}\leq t\}\geq \frac{n_1(P)}{2}r_{\mu}(t)\right]\geq 1 - e^{-3n_1(P)r_{\mu}(t)/28}\ . 
\eeq
Observe that 
\[
 \frac{n_1(P)r_{\mu}(t)}{2} \geq \frac{1-\pi}{4} n \left[ \ol{G}\left[\ol{G}^{-1}\left(\frac{t}{2}\right)-\mu\right]- \frac{t}{2}\right]\geq \frac{3 n}{16} \left[\overline{G}\left(\ol{G}^{-1}\left(\frac{t}{2}\right)-c'_{g}\pi\right)- \frac{t}{2}\right]
\]
By definition of BH procedure, on the event within \eqref{eq:Bernstein_t_fixe}, $\BH^{\star}_{\alpha/2}$  rejects each null hypothesis corresponding to a $p$-value $p_i^{\star}\leq t$ as long as this last expression is higher than $2nt/\alpha$. Putting everything together we have proved that 
\[
\P_{Y\sim P}\left[|\BH^{\star}_{\alpha/2}\cap \cH_1(P)|\geq \frac{2nt}{\alpha}\right]\geq 1-e^{-c' nt/\alpha}\ ,
\]
if some $t\in [\frac{\alpha}{2n}; \alpha/2]$ satisfies
\beq\label{eq:condition_t_small}
\overline{G}\left(\ol{G}^{-1}\left(\frac{t}{2}\right)-c'_{g}\pi\right)\geq \frac{12t}{\alpha}\ . 
\eeq
Such a $t$ exists by Condition~\eqref{eq:condition_scaling}. Fixing $t=t_0$ leads to the desired conclusion.

\end{proof}

\begin{proof}[Proof of Theorem~\ref{thm:lower_location2}]
We consider the exact same density $h$ as above, but we now fix $\pi= \ol{\pi} - n^{-1/3}$.  Arguing as in the proof of Theorem~\ref{thm:lower_sigmaknown} (see \eqref{equ:borneinfstep1} therein), we have
\[
\E_{Z}\FDR[Q_{1,Z},R]+\E_{Z}\FDR[Q_{2,Z},R] \geq \P_{Y\sim Q}[|R(Y)|>0]\ ,
\]
where $Z \sim \mathcal{B}(\pi)^{\otimes n} $ and with $Q$, $Q_{1,z}$, $Q_{2,z}$ defined therein.

If $\P_{Y\sim Q}[|R(Y)|>0]\geq 1/2$, we have either $\E_{Z}\FDR[Q_{1,Z},R]\geq 1/4$ or $\E_{Z}\FDR[Q_{2,Z},R]\geq 1/4$. Since $n_1(Z)$ follows a Binomial distribution with parameter $n$ and  $\pi=\ol{\pi} - n^{-1/3}$, it follows from Bernstein inequality, that 
\begin{equation}\label{bernsteinZ1}
\P_{Z}[n_1(Z)\geq \bar \pi n] \leq \exp(-c_{\alpha}' n^{1/3})\ , 
\end{equation}
for some constant $c'_{\alpha}>0$ that only depends on $\alpha$ (through $\pi_{\alpha}$). As a consequence, there exists $i\in \{0,1\}$ and $z\in\{0,1\}^n$ with $n_1(z)\leq \bar \pi n$ such that 
\[
 \FDR[Q_{i,z},R]\geq 1/4 - \exp(-c_{\alpha}' n^{1/3}).
\]

\medskip

Now assume that 
\beq\label{eq:lower_proba_non_rejet}
\P_{Y\sim Q}[|R(Y)|=0]= \E_{Z}[P_{Y\sim Q_{1,Z}}=0]\geq 1/2\ .
\eeq
We consider the behavior of $|\BH^{\star}_{\alpha}|$, the number of rejections of $\BH^{\star}_{\alpha}$, under $Q_{1,Z}$. Fix $t_0= 2\overline{G}(\mu/2)= \frac{1-2\pi}{1-\pi}\in (0,1)$, so that $\overline{G}^{-1}(t_0/2)=\mu/2$.
From Fubini's Theorem and the definition of $|\BH^{\star}_{\alpha}|$, we derive that
\beqn 
\E_{Z}\P_{Y\sim Q_{1,Z}}\left[|\BH^{\star}_{\alpha}|\geq \frac{nt_0}{\alpha}\right]&= &\P_{Y\sim Q} \left[|\BH^{\star}_{\alpha}|\geq n\frac{t_0}{\alpha}\right]
\\
&\geq & \P_{Y\sim Q} \left[\sum_{i=1}^n \ind\{p_i^{\star} \leq t_0\}\geq \frac{nt_0}{\alpha}\right].
\eeqn 
Under $Q$, the random variables $\ind\{p_i^{\star} \leq t_0\}$, $1\leq i \leq n$, are i.i.d. and follow a Bernoulli distribution with parameter 
\[
 \P[p_i^{\star} \leq t_0]= (1-\pi) \left[\ol{G}(\mu/2)+ 1- \ol{G}(\mu/2)\right]= 1-\pi\ . 
\]

Define the function $\psi: x \mapsto (1-x) - \tfrac{1-2x}{\alpha(1-x)}$. For any $\alpha\in (0,1)$, $\psi'(x)=-1+1/[\alpha(1-x)^2]$. Hence, $\psi$ is convex and strictly increasing in $[0,1/2]$. Recall the definition of $\pi_{\alpha}= \frac{\sqrt{1-\alpha} -(1-\alpha)}{\alpha}\in (0,1/2)$. One then shows that $\psi(\pi_{\alpha})=0$ and $\psi(x)\in (0, 1/2)$ for any $x\in (\pi_{\alpha},1/2)$. 
Since $\lim \pi = \bar \pi > \pi_{\alpha}$, it follows that, for $n$ larger than a constant depending only on $\bar \pi$ and $\alpha$, we have $\pi >\pi_{\alpha}$. Hence, we derive  from Bernstein inequality that 
\begin{eqnarray} \nonumber
  \E_{Z}\P_{Y\sim Q_{1,Z}}\left[|\BH^{\star}_{\alpha}|\geq \frac{nt_0}{\alpha}\right]&\geq& 1- \exp\left[-n\frac{\psi^2(\pi)}{2[1-\pi]+ \psi(\pi)/3} \right]\\
 &\geq& 1 - \exp\left[- c_{\alpha}n (\pi-\pi_{\alpha})^2\right]\nonumber\\ 
 &\geq& 1 - \exp\left[- c_{\alpha}n (\bar \pi-\pi_{\alpha})^2\right]\ , \label{eq:upper_EZ1}
\end{eqnarray}
since  $\psi(\pi)=\psi(\pi)- \psi(\pi_{\alpha})\geq \psi'(\pi_{\alpha})(\pi -\pi_{\alpha})$, $\psi'(\pi_{\alpha})>0$ only depends on $\alpha$.
Moreover, provided that $|\BH^{\star}_{\alpha}|\geq \frac{nt_0}{\alpha}$, we have $|\BH^{\star}_{\alpha}\cap \cH_1(Q_{1,Z})| \geq  \sum_{i=1}^n \ind\{Z_i=1 \} \ind\{p_i^{\star} \leq t_0\}$. When $Z_i=1$, it follows from the definition of $t_0$ that we have $p_i^{\star} \leq t_0$ almost surely. As a consequence, for each $i$, $\ind\{Z_i=1 \} \ind\{p_i^{\star} \leq t_0\}$ is a Bernoulli variable of parameter $\pi$, we derive from Bernstein inequality that 
\begin{eqnarray} \nonumber
  \E_{Z}\P_{Y\sim Q_{1,Z}}\left[\sum_{i=1}^n \ind\{Z_i=1 \} \ind\{p_i^{\star} \leq t_0\} \geq \frac{n\pi}{2}\right]&\geq& 1- \exp\left[-n c \pi^2 \right]\\
 &\geq& 1 - \exp\left[- c_2 n \bar \pi^2\right] , \label{eq:upper_EZ2}
\end{eqnarray}
for $n$ larger than a constant depending on $\bar \pi$. Since $\pi\geq \ol{\pi}/2$ for $n$ large enough, we deduce that by combining the two previous inequalities that 
\[
  \E_{Z}\P_{Y\sim Q_{1,Z}}\left[|\BH^{\star}_{\alpha}\cap \cH_1(P)|\geq\frac{n\overline{\pi}}{4}  \right]\geq 1 -  2\exp\left[- c'_{\alpha}n (\bar \pi-\pi_{\alpha})^2\right]\ ,
\]
and therefore that 
\[
 \P_{Z}\left[\P_{Y\sim Q_{1,Z}}\left[|\BH^{\star}_{\alpha}\cap \cH_1(P)|\geq\frac{n\ol{\pi}}{4}  \right]\geq 1 -  10 \exp\left[- c'_{\alpha}n (\bar \pi-\pi_{\alpha})^2 \right]\right] \geq \frac{4}{5}\ . 
\]
Also, it follows from \eqref{eq:lower_proba_non_rejet} that 
\[
 \P_{Z}\left[P_{Y\sim Q_{1,Z}}[|R(Y)|=0]> 1/3\right]\geq \frac{1}{4}\ . 
\]
Combining the two previous inequalities with \eqref{bernsteinZ1}, we conclude that there exists $P\in\mathcal{P}_g$ with $n_1(P)\leq \ol{\pi}n$ such that $ P_{Y\sim P}[|R(Y)|=0]\geq 1/3$ and 
\[
 \P_{Y\sim P}\left[|\BH^{\star}_{\alpha}\cap \cH_1(P)|\geq\frac{n\ol{\pi}}{4}  \right]\geq 1 -  10 \exp\left[- c'_{\alpha}n (\bar \pi-\pi_{\alpha})^2 \right]\ . 
\]
The result follows. 
\end{proof}

\begin{proof}[Proof of Proposition~\ref{prp:lower_laplace}]
We follow the same approach as in the previous proof, but we can sharpen the bounds using the explicit form of $g$. As in the previous proofs, we consider the same  density $h$, which now takes the form $h(y)=(1-\pi) \max\{ e^{-|y|}/2, e^{-|y-\mu|}/2\}$ with $\pi=\ol{\pi}-n^{-1/3}$. We also consider the same $Q$, $Q_{1,z}$, $Q_{2,z}$, $z\in\{0,1\}^n$, and $Z$ with i.i.d. $\mathcal{B}(\pi)$ coordinates. 
Recall that $g(x)=e^{-|x|}/2$, $\ol{G}(x)= e^{-x}/2$ for $x\geq 0$ and  $\ol{G}^{-1}(t)= \log(1/(2t))$ for $t\leq 1/2$.
As a consequence, $\mu= 2\log(\frac{1-\pi}{1-2\pi})$. 
Note that for any $x >\mu$, we have 
$
g(x-\mu)/g(x)= e^{-\mu}. 
$
Now consider any multiple testing procedure $R$.

\medskip 

\noindent 
{\bf Step 1}: Controlling $\E_{Z}\FDR(Q_{1,Z},R)+ \E_{Z}\FDR(Q_{2,Z},R)$. 
For $Y_i\geq \mu$,  then
\[
\P_{Z, Y\sim Q_{1,Z}}[Z_i=0 \:|\:Y_i] = \frac{(1-\pi) e^{-Y_i}/2}{h(Y_i)} =  \frac{e^{-Y_i}}{\max\{ e^{-Y_i}, e^{-Y_i+\mu}\}}  =  e^{-\mu}= \frac{(1-2\pi)^2}{(1-\pi)^2}\ . 
\]
If $Y_i\in [\mu/2; \mu]$, then 
\[
\P_{Z, Y\sim Q_{1,Z}}[Z_i=0 \:|\:Y_i]=
\frac{e^{-Y_i}}{\max\{ e^{-Y_i},e^{-\mu+Y_i}\}}=  e^{\mu -2Y_i}\geq e^{-\mu}=\frac{(1-2\pi)^2}{(1-\pi)^2}\ . 
\]

Finally, if $Y_i\leq \mu/2$, then $\P_{Z, Y\sim Q_{1,Z}}[Z_i=0 \:|\:Y_i]=1$.
Arguing similarly for $Q_{2,Z}$, we derive that 
\[
 \P_{Z, Y\sim Q_{1,Z}}[Z_i=0 \:|\:Y_i]+ \P_{Z, Y\sim Q_{2,Z}}[Z_i=0 \:|\:Y_i] \geq 1+ e^{-\mu}
\]
This allows us to derive that 
\[
\E_{Z}[\FDR(Q_{1,Z},R)]+ \E_{Z}[\FDR(Q_{2,Z},R)]\geq (1+e^{-\mu})\P_{Y\sim Q}[|R(Y)|>0]\ .  
\]
By symmetry, we may assume henceforth that 
\beq \label{eq:lower_FDR_Laplace}
\E_{Z}[\FDR(Q_{1,Z},R)]\geq (1+e^{-\mu})\frac{\P_{Y\sim Q}[|R(Y)|>0]}{2}\ .  
\eeq 

\noindent 
{\bf Step 2}: Controlling the behavior of $|\BH^{\star}_{\alpha}|$. 
Under $Q_{1,z}$, the oracle $p$-value $p_i^{\star}$ is simply $2\ol{G}(|Y_i|)$.
Consider any $t\leq e^{-\mu}$. Under the mixture distribution $Q$, we have 
\begin{align} 
\P_{Y\sim Q}[p_i^{\star}\leq t]&=(1-\pi)\left[\int_{-\infty}^{-\ol{G}^{-1}(t/2)} g(x)dx+\int_{\ol{G}^{-1}(t/2)}^{\infty} g(x-\mu)dx\right]\nonumber\\
& = (1-\pi) 
\left[t/2+ e^{\mu} t/2\right] = t\eta(\pi)\ ,\label{equintermlocation}
\end{align}
where
\beq\label{eq:def_eta}
\eta(\pi)= \frac{(1-\pi)}{2}[1+e^{\mu}]=\frac{1-\pi}{2}\left[1 +\frac{(1-\pi)^2}{(1-2\pi)^2}\right]\ . 
\eeq
The function $\eta$ is increasing and is larger than $1$ for  $\pi\in (0,1/2)$ (to see this, we check that $\eta'(\pi)=(\pi/2)(10\pi^2-15\pi+6)/(1-2\pi)^3)$ ). Besides, $\eta$ goes to $+\infty$ when $\pi$ converges to $1/2$. 
In addition, observe that $\pi^{*}_{\alpha}$ given by \eqref{eq2} satisfy $\frac{(1-2\pi^{*}_{\alpha})^2}{(1-\pi^{*}_{\alpha})^2}=\alpha$, so that $\alpha< e^{-\mu}$. This implies $\alpha\eta(\pi)<(1-\pi)< 1$.

Now denote $T^*_\alpha(Y) = \max\left\{ t \in[0,1]\::\: \sum_{i=1}^n \mathds{1}\{p_i^{\star} \leq t\} \geq  n t/\alpha\right\}$ the threshold of $\BH^{\star}_{\alpha}$. We have
\begin{align*}
\P_{Y\sim Q}[|\BH^{\star}_{\alpha}|>0]&= (\alpha/n) \sum_{i=1}^n E_{Y\sim Q}\left[ \frac{\mathds{1}\{p_i^{*}\leq T^*_\alpha(Y)\}}{T^*_\alpha(Y)\vee (\alpha/n)}\right] \\
&=(\alpha/n) \sum_{i=1}^n E_{Y\sim Q}\left[ \frac{\P(p_i^{*}\leq T^*_\alpha(Y^{(i)}) \:\big|\: Y^{(i)})}{T^*_\alpha(Y^{(i)})}\right] = \alpha\eta(\pi)
\end{align*}
where we used  Lemma~\ref{lem:BHYi} (and the notation therein), the independence between $p_i^*$ and $Y^{(i)}$, combined with the fact that $T^*_\alpha(Y^{(i)}) \leq \alpha< e^{-\mu}$ and \eqref{equintermlocation}.

Next, as assumed in the statement of the theorem, let us suppose that $R$ is such that $$\sup_{\substack {P\in \cP_g :  n_1(P)/n\leq \ol{\pi}}}\{\FDR(P,R) - \alpha n_0(P)/n\}\leq 0.$$ 
Then, it follows from~\eqref{eq:lower_FDR_Laplace} and Bernstein inequality that 
\begin{align*}
& \E_{Z}\P_{Y\sim P_{1,Z}}[|R(Y)|>0]\\
&\leq \frac{2}{1+e^{-\mu}}\E_{Z}[\FDR(Q_{1,Z},R)]\\
&\leq   \frac{2}{1+e^{-\mu}} \left[ \E\left[(\alpha n_0(Z)/n)\mathds{1}\{\ol{\pi}- 2n^{-1/3} \leq n_1(Z)/n\leq \ol{\pi}\}\right]\right.\\
&\left.\quad\quad\quad\quad\quad+ \P( |n_1(Z)/n - \pi|> n^{-1/3})\right]\\
&\leq  \frac{2}{1+e^{-\mu}}\alpha(1-\pi)+ 2n^{-1/3}+ 4e^{-c_{\ol{\pi}} n^{1/3} }\\
&\leq  \frac{2}{1+e^{-\mu}}\alpha(1-\pi)+ c'_{\ol{\pi}}n^{-1/3}.
\end{align*} 
Combining the above bounds yields
\begin{align*}
&\E_{Z}\left[\P_{Y\sim P_{1,Z}}[|\BH^{\star}_{\alpha}|>0] - \P_{Y\sim P_{1,Z}}[|R(Y)|>0]\right]\\
&\geq 
\alpha\eta(\pi) - \frac{2}{1+e^{-\mu}}\alpha(1-\pi)- c'_{\ol{\pi}}n^{-1/3}\\
&=  \alpha \zeta(\pi)- c'_{\ol{\pi}} n^{-1/3} \ , 
\end{align*}
where  $\zeta(u)= \frac{1-u}{2}\big[1+\frac{(1-u)^2}{(1- 2 u)^2} - 4/\left(1+\frac{(1-2 u)^2}{(1-u)^2}\right)\big]$, for $u\in(0,1/2)$. Since $1+x> 4/(1+1/x)$ for any $x>1$, it follows that $\zeta(u)>0$ for any $u\in (0,1/2)$. Besides, one can check that $\zeta$ is increasing on $(0,1/2)$.  Since the functions $\eta$ and $\mu$ are continuously differentiable in $\pi$, we conclude that
\[
\E_{Z}\left[\P_{Y\sim P_{1,Z}}[|\BH^{\star}_{\alpha}|>0] - \P_{Y\sim P_{1,Z}}[|R(Y)|>0]\right]\geq  \alpha \zeta(\ol{\pi})- c''_{\ol{\pi}} n^{-1/3} \ , 
\]
 Applying again Bernstein inequality to $n_1(Z)$, we conclude that there exists $P$ with $n_1(P)/n\leq \ol{\pi}$ such that 
\[
 \P_{Y\sim P}[|\BH^{\star}_{\alpha}|>0] - \P_{Y\sim P}[|R(Y)|>0]\geq \alpha \zeta(\ol{\pi})- c'''_{\ol{\pi}} n^{-1/3}\ . 
\]
The result follows.

 \end{proof}

\begin{proof}[Proof of Theorem~\ref{THMBORNESUPgeneral}]

We adapt the proof of Theorem~\ref{THMBORNESUP} (see Sections~\ref{sec:prooftha}~and~\ref{sec:proofthb}) to the location model. 
Following exactly the same proof,  \eqref{eq:upper_FDR_thmgeneral} and \eqref{eq:lower_TDP_thmgeneral} 
hold provided that we modify the four following ingredients: Lemma~\ref{lem:esti}, Lemma~\ref{cor:pvaluei}, Lemma~\ref{lem:forFNR} and Lemma~\ref{lem:indep}. This is done below.

\noindent {\bf Modification of Lemma~\ref{lem:esti}:} 
Arguing  as for Lemma~\ref{lem:esti}, we can prove that the empirical median is close to $\theta(P)$. 
Precisely, there exist constants $c'_1,c'_2,c'_3>0$, only depending on $g$, such that for $n\geq 1$, all $P\in\mathcal{P}$ such that $n_0(P)/n\geq 0.9$, and all $x\in (0,c'_1 n)$,
\begin{align}
\P\left(|\wt{\theta}-\theta(P)|\geq  c'_2  \frac{n_1(P)+2}{n} + c'_3 \sqrt{\frac{ x}{n}}\right) &\leq 2 e^{-x}\label{equ:thetachapgeneral}\ .
\end{align}
The proof is mainly based on the fact $\ol{G}$ and $\ol{G}^{-1}$ are continuously differentiable and therefore locally lipschitz around $0$ and $1/2$. 
Hence, for $x=n^{2/3}$, we get for $n$ large enough, $\P(\Omega^c)\leq 2e^{-n^{2/3}}$ for
$$
\Omega=\left\{|\wt{\theta}-\theta(P)|\leq \delta\right\}, \:\:\delta = c_2  \left(\frac{n_1(P)+2}{n} +  n^{-1/6}\right),
$$
for some constant $c_2$ only depending on $g$. 

\noindent {\bf Modification of Lemma~\ref{cor:pvaluei}:}
with the additional assumption $\eta\leq 1/2$, we easily check that the same result holds under the condition $|\wt{\theta}-\theta(P)| \leq \ol{G}^{-1}(1/4)/2$, which is ensured on $\Omega$ for $n$ larger than some constant (only depending on $g$).
 
\noindent  {\bf Modification of Lemma~\ref{lem:forFNR}:} the following lemma is a modification of Lemma~\ref{lem:forFNR}.

\begin{lem}\label{lem:forFNRlocation} 
For an arbitrary estimator $\wh{\theta}$, let $\delta>0$, $\theta\in \R$, $\alpha\in (0,0.5)$, $t_0\in(0,\alpha)$ and 
$$
 \eta=\delta\max_{t\in [t_0,\alpha]} \left\{ \frac{1}{\ol{G}^{-1}(t/2)- \ol{G}^{-1}(t )}   \frac{g(\ol{G}^{-1}(t))}{g(\ol{G}^{-1}(t/2))}\right\} 
$$
 Assume that  $|\wh{\theta}-\theta|\leq \delta$ and $\eta\leq 0.05$. Then, for all $i\in\{1,\dots,n\}$,
 \begin{itemize}
\item if $T_\alpha(\wh{\theta})\vee (\alpha/n)\geq t_0$, we have
 \begin{align}
\mathds{1}\{p_i(\wh{\theta})\leq T_{\alpha}(\wh{\theta})\} \leq \mathds{1}\{p_i(\theta)\leq (1+\eta)T_{\alpha}(\wh{\theta}) \}\leq \mathds{1}\{p_i(\theta)\leq T_{\alpha(1+\eta)}(\theta)\}\ ;
\label{equ:forFDRlocation}
\end{align}
\item if $T_{0.95\alpha}(\theta)\vee (0.95\alpha/n)\geq t_0$, we have
 \begin{align}
\mathds{1}\{p_i(\theta)\leq T_{\alpha(1-\eta)}(\theta)\}  \leq \mathds{1}\{p_i(\wh{\theta})\leq T_{\alpha}(\wh{\theta})\}\ .
\label{equ:forFNRlocation}
\end{align}
\end{itemize}
\end{lem}

To prove Lemma~\ref{lem:forFNRlocation}, we follow the same strategy as in Section~\ref{sec:prooflem:forFNR}. 
For all $u,u'\in \R$, $i\in\{1,\dots,n\}$, $\alpha\in (0,0.5)$, $t_0\in(0,\alpha)$, for all $t\in [t_0,\alpha]$, we have
\begin{eqnarray*}
\mathds{1}\{ p_i(u')\leq t\} &= &\mathds{1}\{ 2\ol{G}(|Y_i-u'|)\leq t\}  \\
& \leq &\mathds{1}\{ 2\ol{G}(|Y_i-u| +|u-u'|)\leq t\}  \\
&&= \mathds{1}\{ 2\ol{G}(|Y_i-u|)\leq 2\ol{G}( \ol{G}^{-1}(t/2) - |u-u'|)\}\\
&\leq& \mathds{1}\{ 2\ol{G}(|Y_i-u|)\leq t(1+ \eta')\},
\end{eqnarray*}
for 
$$
\eta'= \max_{t\in [t_0,\alpha]} \left\{\frac{\ol{G}( \ol{G}^{-1}(t/2) - \delta) - t/2}{t/2}\right\},
$$
provided that $|u-u'|\leq \delta$.
Now, if $\eta\leq 0.05$, we can prove that 
\begin{equation}\label{condbornsuplocation2}
\eta'\leq \eta. 
\end{equation}
Indeed,
$\eta\leq 1$ implies that the following inequality holds 
\begin{equation}\label{condbornsuplocation}
 \delta\leq \min_{t\in [t_0,\alpha]} \left\{\ol{G}^{-1}(t/2)- \ol{G}^{-1}(t )\right\},
\end{equation}
Furthermore, since $\delta\leq \ol{G}^{-1}(\alpha/2)$, for all $t\in [t_0,\alpha]$,
\begin{align*}
 \ol{G}(\ol{G}^{-1}(t/2)-\delta) -t/2
&\leq \delta \max_{u\in [\ol{G}^{-1}(t/2)-\delta;\ol{G}^{-1}(t/2)]} \{g(u)\}\\
&= \delta g(\ol{G}^{-1}(t/2)-\delta)\leq \delta g(\ol{G}^{-1}(t)),
\end{align*}
 by the monotonic property of $g$ and \eqref{condbornsuplocation}. Similarly, for all $t\in [t_0,\alpha]$,
 $$
 \ol{G}^{-1}(t/2)- \ol{G}^{-1}(t) \leq \frac{t/2}{ g(\ol{G}^{-1}(t/2))}. 
 $$
Combining these inequalities leads to  \eqref{condbornsuplocation2}.
In turn, \eqref{equ:forFDRlocation} and \eqref{equ:forFNRlocation} hold provided that $|\wh{\theta}-\theta|\leq \delta$.

\noindent {\bf  Modification of  Lemma~\ref{lem:indep}:} The same results holds because $g$ is symmetric.

\end{proof}

\begin{proof}[Proof of Corollaries~\ref{cor:subottin} and \ref{cor:laplace}]

First, we assume that $\zeta>1$. Define $g(x)=L^{-1}_\zeta \:e^{-|x|^\zeta/\zeta}$, $\zeta>1$, with the normalization constant $L_\zeta=2 \Gamma(1/\zeta)\zeta^{1/\zeta-1}$. In that case, we have (see Lemma~S-5.1 in the supplement of \cite{NR2012}),
\begin{align}
\forall q\in (0,\ol{G}(1)),\:\: \ol{G}^{-1}(q)&\leq \left(-\zeta \log q -\zeta \log L_\zeta\right)^{1/\zeta} \label{eq:1};\\
\forall y>0, \:\: \frac{g(y)}{y^{\zeta-1}}\geq \ol{G}(y)&\geq \frac{g(y)}{y^{\zeta-1}}\frac{y^{\zeta}}{y^{\zeta}+\zeta-1} .\nonumber
\end{align}
The last inequality (used with $y=\ol{G}^{-1}(q)$) implies that 
\beq\label{eq:2}
 [\ol{G}^{-1}(q)]^{\zeta}+  \zeta \log(q)  +\zeta \log L_\zeta+   \zeta(\zeta-1) \log(\ol{G}^{-1}(q)) \in \left[\zeta\log\left[\frac{(\ol{G}^{-1}(q))^{\zeta}}{(\ol{G}^{-1}(q))^{\zeta}+ \zeta-1}\right];0\right]\ ,
\eeq
and therefore that $[\ol{G}^{-1}(t)]\sim  [\zeta\log(1/t)]^{1/\zeta}$ for $t$ going to zero. 
As a consequence, for $t$ small enough,  we have 
\beq\label{eq:difference_quantile_t2}
{ \ol{G}^{-1}(t/2)- \ol{G}^{-1}(t)=\zeta^{-1} [\ol{G}^{-1}(t')]^{1-\zeta}\left[\ol{G}^{-1}(t/2)]^{\zeta}- [\ol{G}^{-1}(t)]^{\zeta}\right]\ , }
\eeq
where $t'\in [t/2;t]$. In view of \eqref{eq:2}, this is of the order $\log^{{1/\zeta}-1}(1/t)$. Besides, \eqref{eq:2} implies that $g(\ol{G}^{-1}(t/2))/g(\ol{G}^{-1}(t)$ is bounded away from $0$. Since $[\ol{G}^{-1}(t/2)]- [\ol{G}^{-1}(t)]$ is bounded away from zero for large $t$, we deduce that 
$\eta$ in~\eqref{relationGbar} is of the order 
\[
 \left(\frac{k_n}{n}\vee n^{-1/6}\right) \log^{{1- 1/\zeta}}(n) \ . 
\]
Hence, $\eta$ goes to $0$ when $k_n\ll n / \log^{{1- 1/\zeta}}(n) $ and we deduce from Theorem~\ref{THMBORNESUPgeneral} that the scaling $\widehat{\theta}$ is asymptotically optimal. Conversely, for $k_n\gg n / \log^{{1- 1/\zeta}}(n)$, Condition~\eqref{eq:condition_scaling} is satisfied and we deduce from Theorem~\ref{thm:lower_location} that optimal scaling is impossible 

\medskip

 Let us turn to the case $\zeta=1$ (Laplace distribution). 
 In that case,  $g(x)=0.5 \:e^{-|x|}$; $\ol{G}(y)=0.5 e^{-y}$ for $y\geq 0$; $\ol{G}^{-1}(q)=-\log( 2q)$ for $q\in[0,1/2]$. Hence, 
 $\eta$ in~\eqref{relationGbar} satisfies
 \[
  \eta= c_2(g) \left(\frac{k_n}{n}\vee n^{-1/6}\right)\frac{2}{\log(2)}\ . 
 \]
Hence, we deduce from Theorem~\ref{THMBORNESUPgeneral} that that the scaling $\widehat{\theta}$ is asymptotically optimal as long as $k_n\ll n$.

\end{proof}

\subsection{Proof of Theorem~\ref{thmconfidence}}

Since the result is trivial if $\mathcal{F}_{0,k}(P)$ is empty, one can fix $F_0\in \mathcal{F}_{0,k}(P)$.
Let us denote 
$$
\wh{F}_{0}(y)=n_0^{-1}\sum_{i\in \cH_0} \ind\{ Y_i\leq y \},\:\: 
y\in\R,
$$
where $\cH_0=\{i\in\{1,\dots,n\}\::\: F_i=F_0\}$, $n_0=|\cH_0|$, which is larger than or equal to $n-k$ by assumption.
By using the DKW inequality \cite{Mass1990}, on an event with probability at least $1-\alpha$, we have for all $y\in\R$, 
$${F}_{0}(y)-d\leq  \wh{F}_0(y) \leq {F}_{0}(y)+d,$$
where $d=\{-\log(\alpha/2)/(2(n-k))\}^{1/2}$.
Now, since 
$$
\wh{F}_n = (1-k/n) \: \wh{F}_0 + (k/n) \: \wh{F}_{1},
$$
for some c.d.f. $\wh{F}_{1}$, we have for all $y\in \R$,
$$
\frac{\wh{F}_n(y)-(1-k/n) F_0(y)-(1-k/n) d}{ k/n } \leq \wh{F}_{1}(y)\leq \frac{\wh{F}_n(y)-(1-k/n) F_0(y)+(1-k/n) d}{ k/n }.
$$
Using the monotonicity of $\wh{F}_1$, this gives for all $y\in\R$,
\begin{align*}
&0\vee \frac{\sup_{x\leq y}\{\wh{F}_n(x)-(1-k/n) F_0(x)\}-c_{n,\alpha}}{ k/n }\\
&\leq \wh{F}_1(y)\leq 1\wedge \frac{\inf_{x\geq y}\{\wh{F}_n(x)-(1-k/n) F_0(x)\}+c_{n,\alpha}}{ k/n },
\end{align*}
where we used $c_{n,\alpha}=(1-k/n) d 
$.
Since, for all $y\in [Y_{(\l)},Y_{(\l+1)})$, we have $\sup_{Y_{(\l)}\leq x\leq y}\{\wh{F}_n(x)-(1-k/n) F_0(x)\}=\l/n-(1-k/n)F_0(Y_{(\l)})$ and $\inf_{y\leq x<Y_{(\l+1)}}\{\wh{F}_n(x)-(1-k/n) F_0(x)\}=\l/n-(1-k/n)F_0(Y_{(\l+1)}^-)$, the result follows.

\section{Technical lemmas for Theorem~\ref{THM:LOWER_GENERAL}}\label{sec:proof:lower_general}

\begin{proof}[Proof of Lemma~\ref{lem:sigma_2}]
We deduce from the definition of $f_1$ that 
\beqn 
\frac{\pi_1}{2}= \pi_1 \frac{\int f_1(u)du}{2}&=&(1-\pi_2)\overline{\Phi}(t_0/\sigma_2) - (1-\pi_1)\overline{\Phi}(t_0) \\
&=&(1-\pi_2)\left(\overline{\Phi}(t_0/\sigma_2)-\overline{\Phi}(t_0)\right) - (\pi_2-\pi_1)\overline{\Phi}(t_0) \\
&\leq & (1-\pi_2)\phi(t_0/\sigma_2)t_0\frac{\sigma_2-1}{\sigma_2}
=  (1-\pi_1)\phi(t_0)t_0(\sigma_2-1)
\eeqn 
where we used the definition of $t_0$ in the last line. Let $c'_0\in (0,1)$ be an absolute constant that will be fixed later. 
We prove the first result by contradiction.  Assume that 
\beq\label{eq:lower-sigma_2}
\sigma_2-1\leq c'_0\pi_2 [1+\log( \pi_2/\pi_1)]^{-1}\leq 1/4\ ,
\eeq 
which, in view of the previous inequality, implies
\beq\label{eq:condition_lower_t0}
 t_0\phi(t_0)\geq \frac{\pi_1[1+\log( \pi_2/\pi_1)]}{2c'_0\pi_2(1-\pi_1)}\ . 
\eeq

\noindent 
{\bf Case 1}: $\pi_2\leq 2\pi_1$. Then, \eqref{eq:condition_lower_t0} implies that $t_0\phi(t_0)\geq (4c'_0)^{-1}(1+\log(2))$, because $x\in[0,1]\mapsto x(1+\log(1/x))$ is nondecreasing. Since $x\phi(x)\leq (2\pi)^{-1/2}e^{-1/2}$ for any $x\in \mathbb{R}$, this last inequality cannot hold for $c'_0$ sufficiently small and we have therefore
\[
 \sigma_2 -1 \geq c'_0\pi_2 [1+\log( \pi_2/\pi_1)]^{-1}\ . 
\]

\bigskip

\noindent
{\bf Case 2}: $\pi_2> 2\pi_1$.  We deduce from~\eqref{eq:lower-sigma_2}, the definition \eqref{eq:definition_t0} of $t_0$, $\log(1+x)\geq x/(1+x)$ for $x\geq 0$, and $\sigma_2\geq 1$  that 
\beqn 
t_0 & = & \sqrt{\frac{2\log\left(\frac{1-\pi_1}{1-\pi_2}\right)\sigma_2^2}{\sigma_2^2-1}+ \frac{2\log(\sigma_2)\sigma_2^2}{\sigma_2^2-1}}\\
&\geq& \sqrt{\frac{2[1+\log( \pi_2/\pi_1)]\log(\frac{1-\pi_1}{1-\pi_2})\sigma^2_2 }{c'_0\pi_2(\sigma_2+1)}+  \frac{2\sigma_2}{\sigma_2+1}} \\
&\geq &  \sqrt{\frac{2[1+\log( \pi_2/\pi_1)](\pi_2-\pi_1)}{c'_0\pi_2(1-\pi_1)}+ 1}\\
&\geq & \sqrt{\frac{1+\log( \pi_2/\pi_1)}{c'_0}+ 1}\ .
\eeqn
Recall that $x\phi(x)$ is decreasing for $x\geq 1$. 
Provided that we chose $c'_0\leq 1/2$, we have therefore 
$\phi(t_0)t_0\leq \frac{\pi_1}{\pi_2}\sqrt{\frac{1+\log( \pi_2/\pi_1)}{c'_0}+1}$ which contradicts \eqref{eq:condition_lower_t0} provided that $c'_0$ is small enough (independently of $\pi_1$ and $\pi_2$). As in Case 1, we conclude that 
\[
 \sigma_2-1\geq c'_0\pi_2 [1+\log( \pi_2/\pi_1)]^{-1}\ . 
\] 
\medskip 

Let us turn to the second part of the lemma. By concavity, we have  $\log(1+x)/x\in [1/(1+x),1]$ for any $x>0$. From~\eqref{eq:definition_t0}, $\pi_1\leq 1/4$, and the last bound of $\sigma_2-1$, we deduce that 
\beqn 
  \frac{t_0}{\sigma_2}&\leq& \sqrt{ \frac{2\log(\sigma_2)}{\sigma_2^2-1}+ \frac{2\log\left(\frac{1-\pi_1}{1-\pi_2}\right)}{\sigma_2^2-1}}
  \\
  &\leq & 
  1+ \sqrt{\frac{2(\pi_2-\pi_1)}{(1-\pi_1)(\sigma_2^2-1)}}\leq 1+ \sqrt{\frac{4 (\pi_2-\pi_1)}{3(\sigma_2-1)}}
  \\
&  \leq &1+ \sqrt{\frac{4 [1+\log( \pi_2/\pi_1)](\pi_2-\pi_1)}{3c'_0\pi_2}}
\leq  1+ \sqrt{\frac{4 [1+\log( \pi_2/\pi_1)]}{3c'_0}}\ .
\eeqn 
Provided that $\log( \pi_2/\pi_1)\leq c^{''}_0 \log(n)$, we have $t_0/\sigma_2\leq \sqrt{0.5\log(n)}$ for $n$ large enough and $c^{''}_0$ sufficiently small. Hence, we derive from Lemma~\ref{lem:quantile} that 
\[
 \overline{\Phi}(t_0/\sigma_2)\geq \frac{\sqrt{0.5\log(n)}}{1+0.5\log(n)}\phi(\sqrt{0.5\log(n)})\geq   (4\pi \log n)^{-1/2} n^{-1/4}\geq   10\sqrt{2\log(2n)/n}\ ,
\]
for $n$ large enough.
\end{proof}

\begin{proof}[Proof of Lemma \ref{lem:t1}]
We check that $u_1\leq  \sqrt{\log(n)}$. 
From the definition~\eqref{eq:t1_obj} of $u_1$, $\pi_1\leq \pi_2\leq 1/4$, and Lemma~\ref{lem:sigma_2}, we deduce that 
\beqn 
u_1^2 &\leq &  2\frac{\sigma^2_2}{\sigma_2^2 -1}\log(\sigma_2) + 2\frac{\sigma_2^2}{\sigma_2^2-1}\log\left(\frac{9}{\alpha(1-\pi_2)}\right)\\ 
&\leq& 2\sigma_2 \left[1+ \frac{1+ \log(\pi_2/\pi_1)}{c_0 \pi_2}\log\left(\frac{12}{\alpha}\right)\right]\ , 
\eeqn 
where we used that  $\pi_2\leq 1/4$. Besides, we have shown above \eqref{eq:definition_t0} that $\sigma_2\leq c_1$ for some universal constant $c_1$. All in all, we have proved that 
\[
 u_1^2\leq  c'_1 + c'_2 \frac{1+ \log(\pi_2/\pi_1)}{c_0 \pi_2}\log\left(\frac{68}{3\alpha}\right)\ , 
\]
which, by assumption, is smaller than $\log n$. The result follows. 
\end{proof}

\section{Remaining proofs for Theorems~\ref{THMBORNESUP} and \ref{THM:THETACHAPFORFDR}}\label{sec:remaining}

\subsection{Estimation of $\theta(P)$, $\sigma(P)$}\label{sec:estimationrate}

The following results are close to those of \cite{chen2018robust} in dimension $1$. The setting here is slightly different, because we are not considering a mixture model, so we provide a proof for completeness.

\begin{lem}\label{lem:esti}
Consider the estimators defined by \eqref{equ:estimators}. Then there exists a constant $c>0$ such that for $n\geq 16$, all $P\in\mathcal{P}$ such that $n_0(P)/n\geq 0.9$, and all $x\in (0,c n)$,
\begin{align}
\P\left(\frac{|\wt{\theta}-\theta(P)|}{\sigma}\geq  2 \frac{n_1(P)+2}{n} + 5\sqrt{\frac{ x}{n}}\right) &\leq 2 e^{-x}\label{equ:thetachap};\\
\P\left(\frac{|\wt{\sigma}-\sigma(P)|}{\sigma }\geq   6\frac{n_1(P)+2}{n}+16\sqrt{\frac{ x}{n}} \right) &\leq 4e^{-x} \label{equ:sigmachap}.
\end{align}

\end{lem}
\begin{proof}[Proof of Lemma~\ref{lem:esti}]

Let $\xi = (Y-\theta)/\sigma$, so that $\xi_i$, $i\in\cH_0$ are i.i.d. $\mathcal{N}(0,1)$. Let $T=\frac{\sqrt{18\lceil n/2\rceil x}}{n_0} + 2 \frac{n_1+2}{n}$. Since $n_0\geq 0.9n$, to prove \eqref{equ:thetachap}, it suffices to show that 
we prove $|\wt{\theta}- \theta|\geq \sigma T$ with probability smaller than $2e^{-x}$.
We have
\begin{align}
\P\left(\wt{\theta}- \theta\geq \sigma T\right)&=
\P\left(\frac{Y_{(\lceil n/2\rceil)}-\theta}{\sigma} \geq  T \right)=  
\P\left(\xi_{(\lceil n/2\rceil)}\geq  T\right)\nonumber\\
& \leq  \P\left(\xi_{( \lceil n/2\rceil:\cH_0)}\geq  T\right),\label{equinterm}
\end{align}
Note that $0.3 < 0.5\leq \lceil n/2\rceil/n_0 \leq  (n/2+1)/(0.9n)< 0.7 $ by assumption.
Hence, from Lemma~\ref{lem:difference_quantile}, we have 
$$
0\leq - \overline{\Phi}^{-1}(\lceil n/2\rceil/n_0) \leq 3.6 (\lceil n/2\rceil/n_0-1/2)\leq 1.8 \frac{n+2-n_0}{n_0}\leq 
2 \frac{n_1+2}{n}.
$$
Using this and applying Lemma~\ref{lem:quantile_empirique}, we get that for $x\leq c n $ (for some constant $c>0$) the rhs  in \eqref{equinterm} is bounded by
\begin{eqnarray} \nonumber
\P\left(\xi_{( \lceil n/2\rceil:\cH_0)}\geq  T\right) &\leq & \P\Big[\xi_{(\lceil n/2\rceil:\cH_0)}+ \overline{\Phi}^{-1}(\lceil n/2\rceil/n_0) \geq   \sqrt{18 x/n} \Big]\\
&\leq& \P\Big[\xi_{(\lceil n/2\rceil:\cH_0)}+ \overline{\Phi}^{-1}(\lceil n/2\rceil/n_0) \geq  3\frac{\sqrt{2 \lceil n/2\rceil x}}{n_0}\Big]
\leq  e^{-x}\ .\label{eq:upper_1_esti}
 \end{eqnarray}
This gives for all $x\in (0,c n)$, $\P(\wt{\theta} - \theta\geq \sigma T) \leq e^{-x}$.
Conversely, we have
\begin{eqnarray*}
    \P\left(\theta- \wt{\theta}\geq \sigma T\right)&=& \P\left((- \xi)_{(\lceil n/2\rceil)} \geq  T\right) \nonumber \\
&\leq&  \P\left((-\xi)_{(\lceil n/2\rceil:\cH_0)} >    T\right)=  \P\left(\xi_{(\lceil n/2\rceil:\cH_0)} >    T\right)\ ,
\end{eqnarray*}
by symmetry of the Gaussian distribution.  Bounding again $\P(\xi_{(\lceil n/2\rceil:\cH_0)} >    T)$, we obtain  \eqref{equ:thetachap}.

\medskip 

Let us now prove \eqref{equ:sigmachap}.  Let $u_0=\ol{\Phi}^{-1}(1/4)\in (0.6,0.7)$ and $ T'=  (1+n_1)/n +   \frac{\sqrt{8(n+1) x}}{n_0}$. Since $n_0\leq 0.9n$, we only have to prove that, with probability higher than $1-4e^{-x}$, we have $|\sigma-\wt{\sigma}|\geq \sigma \left(2T+2T'\right)$. 
By Definition \eqref{equ:estimators} of $\tilde{\sigma}$, we have 
\begin{align}
\P\left(|\sigma-\wt{\sigma}|\geq \sigma \left(2T+2T'\right)\right)&=
\P\left(\frac{ |u_0- U_{(\lceil n/2\rceil)}|/\sigma }{ u_0} \geq  2T+2T'\right)\nonumber.
\end{align}
Since $|\xi_i|- |\xi_{(\lceil n/2\rceil)}| \leq |\xi_i - \xi_{(\lceil n/2\rceil)}|\leq |\xi_i|+ |\xi_{(\lceil n/2\rceil)}|$, we have
$|U_{(\lceil n/2\rceil)}|/\sigma|- |\xi|_{(n/2)}\in [-  |\xi_{(\lceil n/2\rceil)}|;  |\xi_{(\lceil n/2\rceil)}|]$. 
 Thus, we have
\begin{eqnarray} \nonumber
  \P\left(|\sigma-\wt{\sigma}|\geq \sigma \left(2T+2T'\right)\right)&\leq & \P\left( \frac{|\xi_{(\lceil n/2\rceil)}|}{u_0} \geq 2T\right)+ \P\left(\left|\frac{ u_0-|\xi|_{(\lceil n/2\rceil)}}{ u_0}\right|  \geq 2T'\right)\\
 &\leq & 2e^{-x}  
+ \P\left(|u_0-|\xi|_{(\lceil n/2\rceil)}|  \geq T'\right)\ ,\label{eq:upper_deviation_sigma1}
\end{eqnarray}
 where we used~\eqref{eq:upper_1_esti} and $2 u_0 \geq 1$. Since 
 \[
 |\xi|_{(n_0-n + \lceil  n/2\rceil :\cH_0)} = -(-|\xi|)_{(n+1- \lceil n/2\rceil: \cH_0)}\leq |\xi|_{(\lceil n/2\rceil)}\leq |\xi|_{(\lceil n/2\rceil: \cH_0)} 
 \]
 we have 
 \beqn 
 \P\left(|u_0-|\xi|_{(\lceil n/2\rceil)}|  \geq T'\right)&\leq & \P\left(|\xi|_{(\lceil n/2\rceil: \cH_0)}\geq u_0+T'\right)+ \P\left(|\xi|_{n_0-n + \lceil  n/2\rceil :\cH_0}\leq u_0-T'\right)\ . 
 \eeqn 
We now apply Lemma~\ref{lem:quantile_empirique2} to control the deviations of these order statistics. We easily check that $0.4n_0\leq \lfloor n/2\rfloor \leq \lceil n/2 \rceil  \leq  0.6 n_0$. 
Hence, for all  $x\leq cn$(for $c$ small enough), we have
\beqn 
 \P\left[|\xi|_{(\lceil n/2\rceil: \cH_0)}\geq  \overline{\Phi}^{-1}\left(\frac{n_0- \lceil n/2 \rceil}{2n_0} \right) + 4\frac{\sqrt{\lceil n/2\rceil x}}{n_0}\right]\leq e^{-x} \ ; \\
  \P\left[|\xi|_{(n_0-n + \lceil  n/2\rceil :\cH_0)}\leq  \overline{\Phi}^{-1}\left(\frac{n- \lceil n/2 \rceil}{2n_0} \right) - 2\frac{\sqrt{2(n_0-n +\lceil n/2\rceil ) x}}{n_0}\right]\leq e^{-x}\ . 
\eeqn 
It remains to compare these two rhs expression with $T'$. By Lemma~\ref{lem:difference_quantile} and since $n_0\geq 0.9 n$, both  $|\overline{\Phi}^{-1}\left(\frac{n_0- \lceil n/2 \rceil}{2n_0} \right)- \overline{\Phi}^{-1}(1/4)|$ and $|\overline{\Phi}^{-1}\left(\frac{n- \lceil n/2 \rceil}{2n_0} \right)- \overline{\Phi}^{-1}(1/4)|$ are less or equal $(n_1+1)/n$. The deviation terms in the above deviation inequalities are also smaller than $\sqrt{6(n+1)x}/n_0$. This concludes the proof.

\end{proof}

\subsection{Proof of Lemma~\ref{cor:pvaluei}}\label{sec:prooflemmacor:pvaluei}

We start with two lemmas. The first one ensures that $\wt{\theta}$ and $\wt{\sigma}$ are not perturbed when $p_i(\wt{\theta},\wt{\sigma})$ is small. The second one compares the thresholds of plug-in BH procedures based on $Y$ and $Y^{(i)}$. 

\begin{lem}\label{lem:dropi}
For any $i\in\{1,\dots,n\}$, and $t\in (0,0.5)$, if $p_i(\wt{\theta},\wt{\sigma}) \leq t$ and $|\wt{\theta}-\theta|<0.3 \:\wt{\sigma}$, then we have both $\wt{\theta}^{(i)}=\wt{\theta}$ and
$ \wt{\sigma}^{(i)}=\wt{\sigma}$.
\end{lem}

\begin{lem}\label{lem:BHYi}
For  all $u\in \R$, $s>0$, any $i\in \{1,\dots,n\}$, we have for all $\alpha\in (0,1)$,
\begin{align*}
\mathds{1}\{p_i(u,s)\leq T_\alpha(Y;u,s)\} = \mathds{1}\{T_\alpha(Y^{(i)};u,s)= T_\alpha(Y;u,s)\} = \mathds{1}\{p_i(u,s)\leq T_\alpha(Y^{(i)};u,s)\},
\end{align*}
where $T(\cdot)$ is defined by \eqref{Tdef}.
\end{lem}

From  Lemma~\ref{lem:BHYi} with $u=\wt{\theta}$ and $s=\wt{\sigma}$ and  Lemma~\ref{lem:dropi} with $t=T_\alpha(Y^{(i)};\wt{\theta},\wt{\sigma}) \leq \alpha < 0.5$, we deduce that 
\begin{align*}
\mathds{1}\{p_i(\wt{\theta},\wt{\sigma}) \leq T_\alpha(Y;\wt{\theta},\wt{\sigma})\} 
&=\mathds{1}\{p_i(\wt{\theta},\wt{\sigma}) \leq T_\alpha(Y^{(i)};\wt{\theta},\wt{\sigma})\} \\
&= 
\mathds{1}\{p_i(\wt{\theta},\wt{\sigma}) \leq T_\alpha(Y^{(i)};\wt{\theta},\wt{\sigma}), \wt{\theta}=\wt{\theta}^{(i)}, \wt{\sigma} = \wt{\sigma}^{(i)}\}\\
&= \mathds{1}\{p_i(\wt{\theta}^{(i)},\wt{\sigma}^{(i)}) \leq T_\alpha(Y^{(i)};\wt{\theta}^{(i)},\wt{\sigma}^{(i)}), \wt{\theta}=\wt{\theta}^{(i)}, \wt{\sigma} = \wt{\sigma}^{(i)}\}\\
&= \mathds{1}\{p_i(\wt{\theta}^{(i)},\wt{\sigma}^{(i)}) \leq T_\alpha(Y^{(i)};\wt{\theta}^{(i)},\wt{\sigma}^{(i)})\}.
\end{align*}
The fourth equality  uses again Lemma~\ref{lem:dropi} with $t=T_\alpha(Y^{(i)};\wt{\theta}^{(i)},\wt{\sigma}^{(i)}) \leq \alpha < 0.5$. Finally, provided that the above event is true, we clearly have $T_\alpha(Y;u,s)=T_\alpha(Y^{(i)};\wt{\theta}^{(i)},\wt{\sigma}^{(i)})$ by Lemma~\ref{lem:BHYi}. We have proved Lemma~\ref{cor:pvaluei}

\medskip

\begin{proof}[Proof of Lemma~\ref{lem:dropi}]

Assume that $p_i(\wt{\theta},\wt{\sigma}) \leq t$. 
This gives $|Y_i -\wt{\theta} | \geq \wt{\sigma} \:\ol{\Phi}^{-1}(t/2) > \wt{\sigma} \: \ol{\Phi}^{-1}(1/4)$ (because $t<1/2$).
If we further assume that  $|\wt{\theta}-\theta|\leq   0.3 \wt{\sigma} < \wt{\sigma} \: \ol{\Phi}^{-1}(1/4)/2$, we have 
\begin{itemize}
\item either 
$Y_i - \wt{\theta}  >  \wt{\sigma} \:\ol{\Phi}^{-1}(1/4)$, which gives $Y_i -\theta  > \wt{\sigma}\:\ol{\Phi}^{-1}(1/4)/2$ and thus $Y_i > Y_{(\lceil n/2\rceil)} \vee \theta$. In this case, $Y^{(i)}_i=\mathrm{sign}(Y_i-\theta)\times \infty = \infty$ and $\wt{\theta}^{(i)}=\wt{\theta}$;
\item 
 or $Y_i - \wt{\theta}  <  -\wt{\sigma} \:\ol{\Phi}^{-1}(1/4)$, which gives $Y_i -\theta  < -\wt{\sigma}\:\ol{\Phi}^{-1}(1/4)/2$ and thus $Y_i < Y_{(\lceil n/2\rceil)} \wedge \theta$. In this case, $Y^{(i)}_i=\mathrm{sign}(Y_i-\theta)\times \infty = -\infty$ and $\wt{\theta}^{(i)}=\wt{\theta}$.
\end{itemize}
Hence, in both cases, we have $\wt{\theta}^{(i)}=\wt{\theta}$. This implies that $ \ol{\Phi}^{-1}(1/4)\: \wt{\sigma}$ is the empirical median of the $|Y_j-\wt{\theta}^{(i)}|$, $1\leq j\leq n$ and that $|Y_i -\wt{\theta}^{(i)} |> \ol{\Phi}^{-1}(1/4)\:\wt{\sigma} $. Hence, $ \ol{\Phi}^{-1}(1/4)\: \wt{\sigma}$  is also the empirical median of the $|Y_j^{(i)}-\wt{\theta}^{(i)}|$, $1\leq j\leq n$ (whose element $j=i$ is infinite).  Hence, $\wt{\sigma}=\wt{\sigma}^{(i)}$ and the result is proved.

\end{proof}

\begin{proof}[Proof of Lemma~\ref{lem:BHYi}]

Remember that 
$$T_{\alpha}(Y;u,s) = \max\left\{ t \in[0,1]\::\: \sum_{j=1}^n \mathds{1}\{p_j(u,s) \leq t\} \geq  n t/\alpha\right\}$$
with $ \sum_{j=1}^n \mathds{1}\{p_j(u,s) \leq T_{\alpha}(Y;u,s)\} =  n T_{\alpha}(Y;u,s)/\alpha$.
Since $Y^{(i)}_i=\mathrm{sign}(Y_i-\theta)\times \infty$, the corresponding $p$-value is equal to $0$ and 
$$T_{\alpha}(Y^{(i)};u,s) = \max\left\{ t \in[0,1]\::\: 1+\sum_{j\neq i} \mathds{1}\{p_j(u,s) \leq t\} \geq  n t/\alpha\right\}.$$
Hence, $T_{\alpha}(Y^{(i)};u,s)\geq T_{\alpha}(Y;u,s) $ always holds.

Assume now that $p_i(u,s) \leq T_{\alpha}(Y^{(i)};u,s)$. This gives $\sum_{j=1}^n \mathds{1}\{p_j(u,s) \leq T_{\alpha}(Y^{(i)};u,s)\}=1+\sum_{j\neq i} \mathds{1}\{p_j(u,s) \leq T_{\alpha}(Y^{(i)};u,s)\} \geq  n T_{\alpha}(Y^{(i)};u,s)/\alpha$ and thus the reverse inequality  $T_{\alpha}(Y^{(i)};u,s)\leq T_{\alpha}(Y;u,s) $ is also true, which gives $T_{\alpha}(Y^{(i)};u,s) = T_{\alpha}(Y;u,s) $. This in turn implies $p_i(u,s) \leq T_{\alpha}(Y;u,s)$. 

To conclude, it remains to check that  $T_{\alpha}(Y^{(i)};u,s)= T_{\alpha}(Y;u,s)$ implies $p_i(u,s) \leq T_{\alpha}(Y^{(i)};u,s)$. If both thresholds are equal, we have
\[
 1+  \sum_{j\neq i } \mathds{1}\{p_j(u,s) \leq T_{\alpha}(Y^{(i)};u,s)\} =  n T_{\alpha}(Y^{(i)};u,s)/\alpha= \sum_{j=1 }^n \mathds{1}\{p_j(u,s) \leq T_{\alpha}(Y^{(i)};u,s)\}\ ,
\]
which implies $p_i(u,s) \leq T_{\alpha}(Y^{(i)};u,s)$. The result follows.

\end{proof}

 \subsection{Proof of Lemma~\ref{lem:forFNR}}\label{sec:prooflem:forFNR}

We start by gathering a few lemmas on the rescaled $p$-values process. Those are proved at the end of the section. 
The first lemma to quantifies how the rescaling affects the $p$-value process. 
For $x,y\geq 0$ and $t\in [0,1)$, define
\begin{equation}\label{equ-It}
I_t(x,y)=2\ol{\Phi}\left( \ol{\Phi}^{-1}(t/2) - x  -  y\:\ol{\Phi}^{-1}(t/2)\right)\ . 
\end{equation}
The following lemma quantifies how the process $\mathds{1}\{p_i(u,s)\leq t\}$ fluctuates in $u$ and $s$, according to the functional $I_t(\cdot,\cdot)$.
\begin{lem}\label{lem:pvalueprocess}
For all $u,u'\in \R$, $s,s'>0$, $i\in\{1,\dots,n\}$ and $t\in [0,1)$, we have
\begin{align}
\mathds{1}\{p_i(u',s')\leq t\}\leq \mathds{1}\{p_i(u,s)\leq I_t( |u'-u| s^{-1},|s'-s|s^{-1})\} \label{equ-pvalueprocess}.
\end{align}
\end{lem}

Interestingly, $t\mapsto I_t(x,y)$ is close to the identity function when $x$ and $y$ are small, as the following lemma shows.
\begin{lem}\label{lem:majorationnull}
There exists a universal constant $c>1$ 
 such that the following holds. For all $\alpha\in (0,0.8)$, for all $x,y\geq 0 $ and $t_0\in(0,\alpha)$, we have
\begin{align}\label{equ:majorationnull}
\max_{t_0 \leq t \leq \alpha}\left\{ \frac{I_t(x,y)-t}{t}  \right\}&\leq c \left(x(2\log (1/t_0))^{1/2} +2 y \log(1/t_0) \right)\ .
\end{align}
provided that this upper bound is smaller than $0.05$.
\end{lem}

Combining Lemma~\ref{lem:pvalueprocess} and Lemma~\ref{lem:majorationnull}, we obtain the following corollary.

\begin{cor}\label{cor:pvaluemaj}
There exists a universal constant $c>1$ 
 such that the following holds. For all $u,u'\in \R$, $s,s'>0$, $i\in\{1,\dots,n\}$, $\alpha\in (0,0.8)$, $t_0\in(0,\alpha)$, let
\begin{equation}\label{equeta}
\eta=c \left(|u'-u| s^{-1} (2\log (1/t_0))^{1/2} + |s'-s|s^{-1} 2\log(1/t_0) \right).
\end{equation}
Provided $\eta\leq 0.05$, we have for all $t\in [t_0,\alpha]$,
\begin{align}\label{majmargpvalue}
\mathds{1}\{ p_i(u',s')\leq t\} \leq \mathds{1}\{ p_i(u,s)\leq t (1+\eta)\}.  
\end{align}
\end{cor}

Let us first prove \eqref{equ:forFDR}. 
First, if $T_\alpha(\wh{\theta},\wh{\sigma}) < \alpha/n$, then $\mathds{1}\{ p_i(\wh{\theta},\wh{\sigma})\leq T_\alpha(\wh{\theta},\wh{\sigma})\}=0$ for all $i$ and the result is trivial. Now assume that $T_\alpha(\wh{\theta},\wh{\sigma})\geq t_0$. 
By \eqref{majmargpvalue} ($u'=\wh{\theta}$, $u=\theta$, $s'=\wh{\sigma}$, $s=\sigma$, $t=T_\alpha(\wh{\theta},\wh{\sigma})$), we have for all $i$,
\begin{equation}\label{equ:forprooflemforFNR}
\mathds{1}\{ p_i(\wh{\theta},\wh{\sigma})\leq T_\alpha(\wh{\theta},\wh{\sigma})\} \leq \mathds{1}\{ p_i(\theta,\sigma)\leq (1+\eta) T_\alpha(\wh{\theta},\wh{\sigma}) \},
\end{equation}
so we only have to prove 
$(1+\eta) T_\alpha(\wh{\theta},\wh{\sigma})\leq T_{\alpha(1+\eta)}(\theta,\sigma)$.
Since by definition
$$T_{\alpha(1+\eta)}(\theta,\sigma) = \max\left\{ t \in[0,1]\::\: \sum_{i=1}^n \mathds{1}\{p_i(\theta,\sigma) \leq t\} \geq  n t/(\alpha(1+\eta))\right\},
$$
we only have to prove 
$
\sum_{i=1}^n \mathds{1}\{p_i(\theta,\sigma) \leq (1+\eta)T_\alpha(\wh{\theta},\wh{\sigma})\} \geq  n T_\alpha(\wh{\theta},\wh{\sigma})/\alpha.
$
For this, apply again \eqref{equ:forprooflemforFNR}, to get
\begin{align*}
\sum_{i=1}^n \mathds{1}\{p_i(\theta,\sigma) \leq (1+\eta)T_\alpha(\wh{\theta},\wh{\sigma})\} &\geq 
\sum_{i=1}^n \mathds{1}\{ p_i(\wh{\theta},\wh{\sigma})\leq T_\alpha(\wh{\theta},\wh{\sigma})\}\\
& = n T_\alpha(\wh{\theta},\wh{\sigma})/\alpha ,
\end{align*}
by using the definition of $T_\alpha(\wh{\theta},\wh{\sigma})$. Hence, the first result is proved.

Exchanging $\wh{\theta},\wh{\sigma}$ by $\theta,\sigma$ and replacing $\alpha$ by $\alpha(1-\eta)$, \eqref{equ:forFDR} implies that if $T_{\alpha(1-\eta)}(\theta,\sigma)\vee (\alpha(1-\eta)/n)\geq t_0$,
$\mathds{1}\{p_i(\theta,\sigma)\leq T_{\alpha(1-\eta)}(\theta,\sigma)\}  \leq \mathds{1}\{p_i(\wh{\theta},\wh{\sigma})\leq T_{\alpha(1-\eta)(1+\eta)}(\wh{\theta},\wh{\sigma})\} $. Since $T_{\alpha(1-\eta)(1+\eta)}(\wh{\theta},\wh{\sigma})\leq T_{\alpha}(\wh{\theta},\wh{\sigma})$, this gives in turn \eqref{equ:forFNR}, which concludes the proof. 

\begin{proof}[Proof of Lemma~\ref{lem:pvalueprocess}]

First, we can  assume that $p_i(u',s')\leq t$ otherwise the inequality is trivial. By definition \eqref{equ-pvalues}, this implies 
$
|Y_i-u'|/s'\geq  \ol{\Phi}^{-1}(t/2).
$
By triangular inequality, we have 
\begin{align*}
\frac{|Y_i-u|}{s} &\geq  \frac{s'}{s} \:\ol{\Phi}^{-1}(t/2)- \frac{1}{s}|u'-u| \ ;  \\
&=\ol{\Phi}^{-1}(t/2) - \frac{1}{s}|u'-u| - \frac{|s'-s|}{s} \:\ol{\Phi}^{-1}(t/2)\ ,
\end{align*}
which entails the upper bound in \eqref{equ-pvalueprocess}. 
\end{proof}

\begin{proof}[Proof of Lemma~\ref{lem:majorationnull}]
 We have
\begin{align*}
I_t(x,y)= 2\ol{\Phi}\left(\ol{\Phi}^{-1}\left(t/2\right)- z(t)\right)\ , \:\:z(t)=x+y\ol{\Phi}^{-1}\left(t/2\right) .\end{align*}
By Lemma \ref{lem:quantile}, we have for all $t\in[t_0,\alpha]$
$$
z(t) \leq x + y (2\log(1/t))^{1/2}  \leq (0.05/c) (2\log (1/t))^{-1/2} \leq 0.05/ \ol{\Phi}^{-1}\left(t/2\right)\ ,
$$
 since the rhs of \eqref{equ:majorationnull} is smaller than or equal to $0.05$.  Now using that $\ol{\Phi}(\sqrt{0.05})\geq 0.4 \geq t/2$, we deduce
$
z(t)\leq \ol{\Phi}^{-1}\left(t/2\right)
$
for all $t\in[t_0,\alpha]$. Also deduce that for such a value of $t$,
$$
\frac{\phi\left(\ol{\Phi}^{-1}(t/2)-z(t)\right) }{\phi\left(\ol{\Phi}^{-1}(t/2)\right) } = e^{-z^2(t)/2} e^{z(t) \ol{\Phi}^{-1}(t/2)} \leq e^{z(t) \ol{\Phi}^{-1}(t/2)}  \leq e^{0.05} \leq 2\ .
$$
Since $\ol{\Phi}$ is decreasing and its derivative is  $-\phi$, we have for all $t\in[t_0,\alpha]$ 
\begin{align*}
I_t(x,y) -t & = 2\ol{\Phi}(\ol{\Phi}^{-1}(t/2)-z(t)) - 2\ol{\Phi}(\ol{\Phi}^{-1}(t/2))\\
&\leq z(t)\: 2\phi\left(\ol{\Phi}^{-1}(t/2)-z(t)\right) \\
&\leq z(t)\: 2\frac{\phi\left(\ol{\Phi}^{-1}(t/2)-z(t)\right) }{\phi\left(\ol{\Phi}^{-1}(t/2)\right) }  \phi\left(\ol{\Phi}^{-1}(t/2)\right)\\
&\leq 4z(t)  \phi\left(\ol{\Phi}^{-1}(t/2)\right)\\
& \leq 2tz(t)\left(1 + \left(\ol{\Phi}^{-1}(t/2)\right)^{-2}\right)\ol{\Phi}^{-1}(t/2)\\
&\leq tz(t_0)\ol{\Phi}^{-1}(t_0/2) \left(1 + \left( \ol{\Phi}^{-1}(0.4)\right)^{-2}\right)\ ,
\end{align*}
where we used inequality~\eqref{eq:maj-fonctionrepgauss} of Lemma \ref{lem:quantile}  and $t\in [t_0,\alpha]$ in the last line. Finally, we invoke Lemma~\ref{lem:quantile} again to obtain that 
\[
z(t_0)\ol{\Phi}^{-1}(t_0/2)\leq x  (2\log(1/t_0))^{1/2}+ 2y\log(1/t_0)\ ,
\]
which concludes the proof. 
\end{proof}

\subsection{Proof of Lemma~\ref{lem:indep}}
By assumption, the $Y_i$'s are all mutually independent. In addition, when $i\in\cH_0$, $Y_i-\theta$ is a Gaussian distribution, which is symmetric. This implies that the variables of $\{(\mathrm{sign}(Y_i-\theta),  i \in \cH_0), (|Y_i-\theta|,  i \in \cH_0), (Y_i, i \in \cH_1)\} $ are all mutually independent. In particular, for all fixed $i\in\cH_0$, the variables of  
$\{\mathrm{sign}(Y_i-\theta), |Y_i-\theta|, (Y_{j}, j\neq i)\} $ are mutually independent.
Since $Y^{(i)}$ is a measurable function of $\{\mathrm{sign}(Y_i-\theta),  (Y_{j}, j\neq i)\} $, it is in particular independent of $|Y_i-\theta|$.

 \subsection{Proof of Lemma~\ref{lem:forFNRsigmaknown}}

It is analogous to the proof of Lemma~\ref{lem:forFNR}. The only change is that \eqref{majmargpvalue} is now used with $s=\sigma$ (so $s=s'$) instead of $s=\wh{\sigma}$.

\subsection{Proof of Lemma~\ref{lem:convunifprocess}}

To simplify the notation, we write $\delta=\mu/2$ and $n_0$ for $n_0(P)$ 
\begin{lem}\label{lem:psi}
The function  $t\mapsto \Psi_1(t)/t$ is continuous and strictly decreasing on $(0,1]$ and $\Psi_1(t)/t$ goes to infinity when $t$ converges to zero.
\end{lem}

Since $\Psi_1(1)= \ol{\Phi}(- \delta)\leq 1 < 2/\alpha$,   the equation $\Psi_1(t)=2t/\alpha$ has only one solution on  $(0,1)$, denoted $t^\star_{\alpha}$. 
Write $t_0= \tfrac{\alpha}{4} n_0^{-1/4} $, we claim that, for $n_0\geq N(\alpha)$, $\Psi_1(t_0)\geq 4t_0/\alpha$.  This claim is justified at the end of the proof.  
This implies $\Psi_1(t_0)/t_0\geq 4/\alpha = \Psi_1(t^\star_{\alpha/2})/t^\star_{\alpha/2}$. Hence, we have $t^\star_{\alpha/2}\geq t_0$.
On the event $\Omega_0^-=\left\{\sup_{t\in[0,1]}|\wh{\mathbb{G}}_0^-(t)-t| \leq \sqrt{\log (2n)/(2n_0)}\right\}$, we have $\wh{\mathbb{G}}_0^-(\Psi_1(t))\geq \Psi_1(t) - \sqrt{\log (2n)/(2n_0)}$ for all $t\in[0,1]$, hence $T_0^{-}= \max\{ t \in[0,1]\::\:  \wh{\mathbb{G}}_0^-(\Psi_1(t)) \geq   2 t/\alpha\}$ is such that
\begin{align*}
 T_0^- &\geq\max\left\{t\in[0,1]\::\:\Psi_1(t) \geq 2t/\alpha + \sqrt{\log (2n)/(2n_0)}\right\}\\
   &\geq\max\left\{t\in[0,1]\::\:\Psi_1(t) \geq 2t/\alpha + 2t_0/\alpha\right\}\\
 &\geq \max\left\{t\in[0,1]\::\:\Psi_1(t) \geq 2t/\alpha + 2t^\star_{\alpha/2}/\alpha \right\} \\
 & \geq t^\star_{\alpha/2}\geq t_0,
 \end{align*}
 since  $\Psi_1(t^\star_{\alpha/2})=4 t^\star_{\alpha/2}/\alpha$. Since $\Psi_1$ is non-decreasing, we conclude that 
 \[
 \Psi_1(T_0^-)\geq \Psi_1(t_0)\geq \frac{4t_0}{\alpha}= n_0^{-1/4}\ ,
 \]
which is the statement of the lemma.

It remains to prove the claim  $\Psi_1(t_0)\geq 4t_0/\alpha$ for $n_0\geq N(\alpha)$ and all $\delta\geq \delta_0=2\log(32/\alpha)/ \sqrt{\log (2/t_0)}$ (Condition~\eqref{equ-prooflowerboundeta}).
 Define $x_0=\overline{\Phi}^{-1}(t_0/2)$ 
so that, by \eqref{eq:encadrement_quantile_1plus} of Lemma~\ref{lem:quantile}, $\delta_0 x_0\geq 2\log(32/\alpha)$ for $n_0\geq N(\alpha)$. 
Since $\Psi_1(t)$ is increasing with respect to $\delta$, it suffices to prove the claim for $\delta= \delta_0$. For $n_0\geq N(\alpha)$, we have $\delta_0\leq 1$ and $x_0\geq 2$. 
Applying twice Lemma~\ref{lem:quantile}, we obtain 
\beqn 
\Psi_1(t_0)&\geq& \frac{(x_0-\delta_0)}{1+(x_0-\delta_0)^2}\phi(x_0-\delta_0)\geq \frac{(x_0-\delta_0)e^{x_0\delta_0-\delta_0^2/2}}{1+(x_0-\delta_0)^2}\phi(x_0)
\\
&\geq&  t_0 \frac{x_0(x_0-\delta_0)}{2[1+(x_0-\delta_0)^2]}e^{x_0\delta_0-\delta_0^2/2}\geq \frac{t_0x_0^2}{4(1+x_0^2)}e^{x_0\delta_0/2}\\
&\geq & \frac{t_0}{8}e^{x_0\delta_0/2}\geq \frac{4t_0}{\alpha}\ , 
\eeqn 
where we used $\delta_0 x_0\geq 2\log(32/\alpha)$ in the last line.

\begin{proof}[Proof of Lemma~\ref{lem:psi}]
For $t$ going to $0$,  $\Phi^{-1}(t/2)$ goes to infinity. Furthermore,  Lemma~\ref{lem:quantile} ensures that 
$\ol{\Phi}(x)\sim \phi(x)/x$. Hence, for $t$ converging to $0$, we have 
\[
\frac{ \Psi_1(t)}{t} = \frac{ \Psi_1(t)}{\ol{\Phi}(\ol{\Phi}^{-1}(t/2))}\sim \frac{\phi(\ol{\Phi}^{-1}(t/2)-\delta)}{\phi(\ol{\Phi}^{-1}(t/2))}=  e^{\delta \ol{\Phi}^{-1}(t/2))- \delta^2/2 }\rightarrow \infty\ . 
\]
To show that $t\in(0,1]\mapsto \Psi_1(t)/t$ is decreasing, we prove that $t\in(0,1]\mapsto \Psi_1(t)$ is strictly concave. This holds because
$$
\Psi'_1(t)= \frac{1}{2}\frac{\phi\left(\ol{\Phi}^{-1}(t/2)-\delta \right)}{\phi\left(\ol{\Phi}^{-1}(t/2)\right)} = e^{\delta \ol{\Phi}^{-1}(t/2))- \delta^2/2 }
$$
is decreasing in $t$.
\end{proof}

\section{Auxiliary results}\label{sec:aux}

\begin{lem}[DKW inequality~\cite{Mass1990}]\label{lem:DKW}
 Let $X_1,\ldots, X_n$ be i.i.d. distributed with cumulative function $F$. Denote $F_n$ the empirical distribution function defined by $F_n(x)= n^{-1}\sum_{i=1}^n \ind\{X_i\leq x\}$. Then, for any $t\geq \log(2)$, we have 
 \[
  \P\left[\sup_{x\in \mathbb{R}}(F_n(x)- F(x))\geq \sqrt{\frac{t}{2n}} \right]\leq e^{-t}\ . 
 \]

\end{lem}

\begin{lem}[\cite{CDRV2018}]\label{lem:quantile}
We have 
\beq 
\max\left(\frac{t\phi(t)}{1+t^2}, \frac{1}{2}- \frac{t}{\sqrt{2\pi}}\right) \leq \ol{\Phi}(t)\leq \phi(t) \min\left(\frac{1}{t}, \sqrt{\frac{\pi}{2}}\right) , \:\:\:\:\mbox{ for all $t>0$}\label{eq:maj-fonctionrepgauss}\ . 
\eeq
As a consequence, for any $x< 0.5$, we have
\begin{align}
\sqrt{2\pi}(1/2- x)&\leq \overline{\Phi}^{-1}(x)\leq \sqrt{2\log\left(\frac{1}{2x}\right)} \ , \label{eq:encadrement_quantile_1}\\
   \log\left( \frac{[\overline{\Phi}^{-1}(x)]^2}{[\overline{\Phi}^{-1}(x)]^2+1}\right) &\leq \frac{[\overline{\Phi}^{-1}(x)]^2}{2}  -  \log\left(\frac{1}{x}\right) + \log\left(\sqrt{2\pi} \overline{\Phi}^{-1}(x)\right)\leq 0\ ,\label{eq:encadrement_quantile}
\end{align}
and if additionally $x\leq 0.004$, we have
\begin{align}
\overline{\Phi}^{-1}(x)\geq \sqrt{\log\left(\frac{1}{x}\right)} \label{eq:encadrement_quantile_1plus}\ .
\end{align}

\end{lem}

\begin{lem}\label{lem:difference_quantile}
 For $0.2\leq x \leq y \leq 0.8$, we have 
\beq\label{eq:upper_difference_quantile}
 \overline{\Phi}^{-1}(x)- \overline{\Phi}^{-1}(y)\leq
 3.6 (y-x).
\eeq
\end{lem}

\begin{proof}[Proof of Lemma \ref{lem:difference_quantile}]
By the mean-value theorem, we have 
\beq\label{eq:lower_mean_value_theorem}
\frac{y-x}{\sup_{z\in [x,y]}\phi(\overline{\Phi}^{-1}(z))} \leq \overline{\Phi}^{-1}(x) - \overline{\Phi}^{-1}(y) \leq \frac{y-x}{\inf_{z\in [x,y]}\phi(\overline{\Phi}^{-1}(z))}\ . 
\eeq
The function $t\mapsto \phi(\overline{\Phi}^{-1}(t+1/2))$ defined on $[-1/2,1/2]$  is symmetric and increasing on $[-1/2,0]$. Thus if $0.2\leq x\leq y\leq 0.8$, the above  infimum equals 
$\phi(\overline{\Phi}^{-1}(0.2))$ which is larger than $1/3.6$. 
\end{proof}

\begin{lem}[\cite{CDRV2018}] \label{lem:quantile_empirique}
 Let $\xi=(\xi_{1},\ldots, \xi_{n})$ be a standard Gaussian vector of size $n$. 
 For any integer $q\in (0.3n, 0.7n)$ and  for all $0< x\leq \frac{8}{225}q \wedge\big( \frac{n^2}{18q}[\overline{\Phi}^{-1}(q/n)-\overline{\Phi}^{-1}(0.7)]^{2}\big)$, we have
\begin{align}
& \P\Big[\xi_{(q)}+ \overline{\Phi}^{-1}(q/n) \geq  3\frac{\sqrt{2qx}}{n}\Big]\leq  e^{-x} \ ,\label{equ1Nico}
 \end{align}
 where we denote $\xi_{(1)}\geq \dots \geq \xi_{(n)}$ the values of $\xi_{1},\ldots, \xi_{n}$ ordered decreasingly.
 \end{lem}
 
 
 \begin{lem} \label{lem:quantile_empirique2}
 Let $\xi=(\xi_{1},\ldots, \xi_{n})$ be a standard Gaussian vector of size $n$ and denote $|\xi|_{(1)}\leq \dots \leq |\xi|_{(n)}$ the values of $|\xi_{1}|,\ldots, |\xi_{n}|$ ordered increasingly.
 For any integer $q\in [0.2n,0.6 n]$ and for all $0< x\leq  0.04q \wedge\big( \frac{n^2}{14q}[ \overline{\Phi}^{-1}(0.2)-\overline{\Phi}^{-1}((1-q/n)/2) ]^{2}\big)$, we have
\begin{align}
& \P\Big[|\xi|_{(q)}- \overline{\Phi}^{-1}((1- q/n)/2) \geq  \frac{4\sqrt{qx}}{n}\Big]\leq  e^{-x}\  .\label{newequ1Nico}
 \end{align}
 For any integer $q\in (0.4 n, n)$ and  for all $0< x\leq  \frac{n^2}{8q}[ \overline{\Phi}^{-1}((1-q/n)/2) - \overline{\Phi}^{-1}(0.3) ]^{2} $, we have
\begin{align}
& \P\Big[ |\xi|_{(q)}-\overline{\Phi}^{-1}((1- q/n)/2) \leq - 2\frac{\sqrt{2 qx}}{n}\Big]\leq  e^{-x}\  .\label{newequ2Nico}
 \end{align}

 \end{lem}
 
\begin{proof}[Proof of Lemma \ref{lem:quantile_empirique2}]
Consider any  $t>0$ and denote $p= 1-2\overline{\Phi}[\overline{\Phi}^{-1}((1-q/n)/2)+ t]$ which belongs to $[q/n,1)$. Denote $\cB(n,p)$ the binomial distribution with parameters $n$ and $p$. 
We have 
 \beq\label{eq:first_upper_proba}
 \P\Big[|\xi|_{(q)}\geq \overline{\Phi}^{-1}((1-q/n)/2)+t\Big]= \P\left[\mathcal{B}(n,p)\leq q-1\right]\leq\P\left[\mathcal{B}(n,p)\leq q\right]\ . 
 \eeq
 By the mean value theorem, we have 
 $$ p-q/n \geq t \inf_{x\in [0,t]}\phi[\overline{\Phi}^{-1}((1-q/n)/2)+x]=t \phi[\overline{\Phi}^{-1}((1-q/n)/2)+t],$$
 because $\overline{\Phi}^{-1}((1-q/n)/2)\geq 0$.
 Assume that $\overline{\Phi}^{-1}((1-q/n)/2)+t \leq \overline{\Phi}^{-1}(0.2)$. Then, it follows from the previous inequality that 
 $p-q/n \geq 2 t \phi[\overline{\Phi}^{-1}(0.2)]\geq t/2$. Together with Bernstein's inequality, we obtain  
\beqn 
 \P\Big[|\xi|_{(q)}\geq \overline{\Phi}^{-1}((1-q/n)/2)+t\Big]&\leq & \P\left[\mathcal{B}(n,q/n+ t/2)\leq q\right]\\
 &\leq & \exp\left[- \frac{n^2t^2}{8[(q+nt/2)(1-q/n)+ nt/3]}\right]
 \eeqn 
 Since $q\geq 0.2n$ and further assuming that $nt\leq 0.8q$, we conclude that 
\[
 \P\Big[|\xi|_{(q)}\geq \overline{\Phi}^{-1}((1-q/n)/2)+t\Big]\leq e^{-n^2t^2/14}\ , 
\]
for any $0< t\leq 0.8q/n \wedge [ \overline{\Phi}^{-1}(0.2)-\overline{\Phi}^{-1}((1-q/n)/2)]$. We have proved  \eqref{newequ1Nico}.

Next, we consider the left deviations. 
Assume that  $q/n>0.4$ (so that $(1-q/n)/2)<0.3$) and take $0\leq t \leq \overline{\Phi}^{-1}((1-q/n)/2) - \overline{\Phi}^{-1}(0.3)$.
Write $p= 1-2\overline{\Phi}[\overline{\Phi}^{-1}((1-q/n)/2) - t]$. We have $p\in [0.4,q/n]$ and 
$$\P[|\xi|_{(q)}\leq \overline{\Phi}^{-1}((1-q/n)/2)-t]= \P[\mathcal{B}(n,p)\geq   q].$$ 
By the mean value theorem,
$$
  q/n-p \geq 2 t \inf_{x\in [0,t]}\phi\Big[\overline{\Phi}^{-1}((1-q/n)/2) - x\Big]\geq 2 t \phi\Big[\overline{\Phi}^{-1}((1-q/n)/2) \Big] \geq 2 t \phi(\overline{\Phi}^{-1}(0.2)) \geq t/2,
$$ 
Then, Bernstein's inequality yields
\beq\label{eq:second_upper_proba} 
\P[|\xi|_{(q)}\leq \overline{\Phi}^{-1}((1-q/n)/2)-t]\leq
  \exp\left[- \frac{(q-np)^2}{2np(1-p)+ 2(q-np)/3}\right]\leq  \exp\left[- \frac{(q-np)^2}{2q}\right]\ ,
 \eeq
because $2np(1-p)+ 2(q-np)/3\leq (2-2/3)np+ 2q/3\leq 2q$ since $p\leq q/n$.
which implies 
\[
 \P\Big[\xi_{(q)}\leq -\overline{\Phi}^{-1}((1-q/n)/2)-t\Big]\leq \exp\left[- \frac{n^2 t^2 }{8q}\right]\ . 
\]
We have shown \eqref{newequ2Nico}. 

\end{proof}

\end{document}